\documentclass[preprint]{imsart}

\usepackage{geometry}
\RequirePackage{amsthm,amsmath,amsfonts,amssymb}
\RequirePackage[numbers]{natbib}
\RequirePackage[colorlinks,citecolor=blue,urlcolor=blue]{hyperref}
\RequirePackage{graphicx}

\RequirePackage{amsthm,amsmath,amsfonts,amssymb}
\RequirePackage[numbers]{natbib}
\RequirePackage[colorlinks,citecolor=blue,urlcolor=blue]{hyperref}
\usepackage[colorlinks,citecolor=blue,linkcolor=blue,urlcolor=blue]{hyperref}
\usepackage{latexsym, epsfig, amssymb, amsmath, amsthm, graphicx, mathrsfs, booktabs}

\usepackage{color}
\usepackage{subfig}
\usepackage{graphicx}
\usepackage[OT1]{fontenc}
\usepackage{lscape}
\usepackage{multirow,multicol}
\usepackage{array}

\usepackage{tikz}
\usepackage{mathtools}
\usepackage{enumitem}
\usepackage{bbm}
\usepackage{bm}

\startlocaldefs

\theoremstyle{plain}
\newtheorem{theorem}{Theorem}
\newtheorem{corollary}{Corollary}
\newtheorem{lemma}{Lemma}
\newtheorem*{lemma*}{Lemma}
\newtheorem{proposition}{Proposition}
\theoremstyle{remark}
\newtheorem{remark}{Remark}

\newtheorem{exmp}{Example} 
\newtheorem{condi}{Condition} 
\newtheorem*{condi*}{Condition} 
\newtheorem{defn}{Definition} 
\newtheorem*{defn*}{Definition} 
\DeclareMathOperator*{\argmin}{argmin}
\renewcommand{\hat}{\widehat}
\newlist{steps}{enumerate}{1}
\setlist[steps, 1]{label = Step \arabic*:}
\DeclarePairedDelimiter\ceil{\lceil}{\rceil}
\DeclarePairedDelimiter\floor{\lfloor}{\rfloor}
\renewcommand{\top}{^T}
\renewcommand{\dots}{\cdots}
\let\oldFootnote\footnote
\newcommand\nextToken\relax

\renewcommand\footnote[1]{%
	\oldFootnote{#1}\futurelet\nextToken\isFootnote}

\newcommand\isFootnote{%
	\ifx\footnote\nextToken\textsuperscript{,}\fi}

\newcommand\blfootnote[1]{%
	\begingroup
	\renewcommand\thefootnote{}\footnote{#1}%
	\addtocounter{footnote}{-1}%
	\endgroup
}

\endlocaldefs

\begin{document}
	
	\begin{frontmatter}
		\title{Asymptotic Properties of High-Dimensional Random Forests}
		
		\runtitle{High-Dimensional Random Forests}

  			\author{\fnms{Chien-Ming} \snm{Chi$^1$}}, 
  			\author{\fnms{Patrick} \snm{Vossler$^1$}}, 
  			\author{\fnms{Yingying} \snm{Fan$^1$}}, 
  			\and 
  			\author{\fnms{Jinchi} \snm{Lv$^1$}}
  			\address{$^1$Data Sciences and Operations Department, Marshall School of Business, University of Southern California, Los Angeles, CA 90089,} 
			
			\blfootnote{This work was supported by NSF Grant DMS-1953356 and a grant from the Simons Foundation.  Co-corresponding authors: Yingying Fan (fanyingy@usc.edu) and Jinchi Lv (jinchilv@usc.edu). The authors sincerely thank the Co-Editor, Associate Editor, and referees for their constructive comments that have helped improve the paper substantially.}
 			
		
		\begin{abstract}
			As a flexible nonparametric learning tool,  the random forests algorithm has been widely applied to various real applications with appealing empirical performance, even in the presence of high-dimensional feature space. Unveiling the underlying mechanisms has led to some important recent theoretical results on the consistency of the random forests algorithm and its variants. However, to our knowledge, almost all existing works concerning random forests consistency in high dimensional setting were established for various modified random forests models where the splitting rules are independent of the response; a few exceptions assume simple data generating models with binary features. In light of this, in this paper we derive the consistency rates for the random forests algorithm associated with the sample CART splitting criterion, which is the one used in the original version of the algorithm \cite{Breiman2001}, in a general high-dimensional nonparametric regression setting through a bias-variance decomposition analysis. Our new theoretical results show that random forests can indeed adapt to high dimensionality and allow for discontinuous regression function. Our bias analysis characterizes explicitly how the random forests bias depends on the sample size, tree height, and column subsampling parameter. Some limitations on our current results are also discussed.
		\end{abstract}
		
		
		
	\end{frontmatter}
	
	\section{Introduction} \label{Sec1}

	As an ensemble method for prediction and classification tasks first introduced in \cite{Breiman2001, Breiman2002}, the random forests algorithm has received a rapidly growing amount of attention from many researchers and practitioners over recent years. It has been well demonstrated as a flexible,  nonparametric learning tool with appealing empirical performance in various real applications involving high-dimensional feature spaces. The main idea of random forests is to build a large number of decision trees independently using the training sample, and output the average of predictions from individual trees as the forests prediction at a test point. Such an intuitive algorithm has been applied successfully to many areas such as finance~\cite{khaidem2016predicting}, bioinformatics~\cite{qi2012random, diaz2006gene}, and multi-source remote sensing~\cite{gislason2006random}, to name a few. The fact that random forests has only a few tuning parameters also makes it often favored in practice \cite{varian2014big, howard2012two}. In addition to prediction and classification, random forests has also been exploited for other statistical applications such as feature selection with importance measures~\cite{Goldstein2011, Mentch2014} and survival analysis~\cite{Ishwaran2008, ishwaran2010consistency}. See, e.g., \cite{Biau2016} for a recent overview of different applications of random forests.
	
	The empirical success and popularity of random forests raise a natural question of how to understand its underlying mechanisms from the theoretical perspective. There is a relatively limited but important line of recent work on the consistency of random forests. Some of the earlier consistency results in \cite{Biau2012, bai2005maxima, genuer2012variance, ishwaran2010consistency, zhu2015reinforcement} usually considered certain simplified versions of the original random forests algorithm, where the splitting rules are assumed to be independent of the response. Recently, \cite{Biau2015} made an important contribution to the consistency of the original version of the random forests algorithm in the classical setting of fixed-dimensional ambient feature space. As mentioned before, random forests can deal with high-dimensional feature space with promising empirical performance. To understand such a phenomenon, \cite{Biau2012, klusowski2019sharp} established consistency results for simplified versions of the random forests algorithm where the rates of convergence depend on the number of informative features in sparse models. In a special case where all features are binary,  \cite{syrgkanis2020estimation} derived high-dimensional consistency rates for random forests  without column subsampling. 
	Additional results along this line include the pointwise consistency~\cite{klusowski2019sharp, Wager2018}, asymptotic distribution~\cite{Wager2018}, and confidence intervals for predictions~\cite{wager2014confidence}.
	
	Despite the aforementioned existing theory for random forests, it remains largely unclear how to characterize the consistency rate for the original version of the random forests algorithm in a general high-dimensional nonparametric regression setting. In this work, the ``original version of random forests,'' or simply ``random forests" hereafter, refers to the random forests algorithm that 1) grows trees with Breiman's CART~\cite{Breiman2001}, and 2) utilizes column subsampling (the key distinction of random forests~\cite{Breiman2001} from bagging \cite{breiman1996bagging}). Our main contribution is to characterize such a consistency rate for random forests with non-fully grown trees. To this end, we introduce a new condition,  the sufficient impurity decrease (SID), on the underlying regression function and the feature distribution, to assist our technical analysis. Assuming regularity conditions and SID, we show that the random forests estimator can be consistent with a rate of polynomial order of sample size,  provided that the feature dimensionality increases at most polynomially with the sample size. To the best of our knowledge, this is the first result on high-dimensional consistency rate for the original version of random forests. Thanks to the bias-variance decomposition of the prediction loss, we establish upper bounds on the squared bias (i.e., the approximation error) and variance separately.    Our bias analysis reveals some new and interesting understanding of how the bias depends on the sample size, column subsampling parameter, and forest height. The latter two are the most important model complexity parameters.  We also establish the convergence rate of the estimation variance,  which is less precise than our bias results in terms of characterizing the effects of model complexity parameters.  Such limitation is largely due to technical challenges.

	The SID, formally defined in Condition~\ref{P}, is a key assumption for obtaining the desired consistency rates.  
	We discover that, if conditional on  each cell in the feature space  there exists one global split of the cell such that a sufficient amount of estimation bias (we use the terms ``impurity," ``bias," and ``approximation error" interchangeably in this work) can be reduced, then the desired convergence rate results follow. This finding greatly helps us understand how random forests controls the bias.  
	The SID condition can accommodate  discontinuous regression functions and dependent covariates, a nice flexibility making our results applicable to a wide range of applications. The SID condition is new to the literature, and a concurrent work \cite{syrgkanis2020estimation} exploits it independently in a simpler setting with all binary features. At a high level, the idea of SID roots in a frequently used concept called impurity decrease for measuring importance of individual features; we use it in this paper for a different purpose of proving random forests consistency.  A related well-known feature importance is the mean decrease impurity (MDI)~\cite{louppe2013understanding}, which evaluates the importance of a feature by calculating how much variance of the response can be reduced by using this feature in random forests learning.  Details on the MDI and other feature importance measures can be found in~\cite{louppe2013understanding, scornet2020trees}. 
	
	In terms of technical innovations, we take a global view of biases from all individual trees in the forests, and then with the SID condition, we obtain a precise characterization on how the column subsampling affects bias. This global approach is one of our major technical innovations in obtaining high-dimensional random forests consistency rate. Our variance analysis of the forests uses the ``grid" discretization approach which is also new to the random forests literature. Despite the technical innovation in our variance analysis, we take a local view and bound the forests variance by establishing variance bound for individual trees. A caveat of this local approach is that the resulting variance upper bound is the most conservative one working for all column-subsampling parameters, and hence less precise. It remains open on how to establish the global control of the forests variance.     
	
	We provide in Table~\ref{table:comparison} a comparison of our consistency theory with some closely related results in the literature. The consistency of the original version of the random forests algorithm was first investigated in the seminal work \cite{Biau2015} under the setting of a continuous additive regression function and independent covariates with fixed dimensionsality, and no explicit rate of convergence was provided.  The results therein cover random forests with fully-grown trees where each terminal node contains exactly one data point, and demonstrate the importance of row subsampling in achieving random forests consistency.  By considering a variant of random forests, \cite{Biau2012} analyzed the consistency of {centered random forests} in a fixed-dimensional feature space and derived the rate of convergence which depends on the number of relevant features,  assuming a Lipschitz continuous regression function.  
	Recently, \cite{klusowski2019sharp} improved over \cite{Biau2012} on the consistency of centered random forests. \cite{mourtada2018minimax} established the minimax rate of convergence for a variant of random forests, Mondrian random forests, under the assumptions of fixed dimensionality, H\"{o}lder continuous regression function, and dependent covariates. A fundamental difference between the original random forests algorithm and these variants is that the original version uses the response to guide the splits, while these variants do not.

	\begin{table}[h]
		\centering	
		\caption{Comparison of 
			consistency rates}
		\begin{tabular}{ | m{4.1em} | m{1.3cm}| m{2cm} | m{6cm} | m{1.85cm}|} 
			\hline
			& $p \gg n$ & Consistency rate & Conditions & Original algorithm\\ 
			\hline
			Our work & Yes& Yes & The SID assumption (Section \ref{Sec3.1}) & Yes\\ 
			\hline
			\cite{Biau2015} & No & No & Independent covariates and continuous additive regression function& Yes\\ 
			\hline
			\cite{klusowski2019sharp} & No & Yes & Dependent covariates and Lipschitz continuous function& No\\ 
			\hline
			\cite{mourtada2018minimax} & No & Yes & Dependent covariates and H\"{o}lder continuous functions& No\\ 
			\hline
		\end{tabular}
		\label{table:comparison}
	\end{table}

	The rest of the paper is organized as follows. Sections~\ref{Sec2.1}--\ref{Sec2.2} introduce the model setting and the random forests algorithm. Section~\ref{Sec2.3} gives a roadmap of the bias-variance decomposition analysis of random forests, which is fundamental to our main results. We then present  SID in Section~\ref{Sec3.1} with examples justifying its usefulness. The main results on the consistency rates are provided in Section~\ref{Sec3.2}. To further appreciate our main results, we also give consistency rates under a simple example with binary features in Section~\ref{Sec3.32}. There, with restrictive model assumptions, we derive sharper  convergence rates. As a way of further motivating SID, in Section~\ref{Sec3.4}, we discuss the relationships between SID and the relevance of active features. In Section~\ref{Sec3.5}, we compare our results to recent related works. We detail our analysis of approximation error and estimation variance in Sections \ref{Sec4} and \ref{Sec5}, respectively. Section \ref{Sec6} discusses some implications and extensions of our work. All the proofs and technical details are provided in the Supplementary Material.

	\subsection{Notation} \label{Sec1.1}
	To facilitate the technical presentation, we first introduce some necessary notation that will be used throughout the paper. Let $(\Omega, \mathcal{F}, \mathop{{}\mathbb{P}})$ be the underlying probability space.   Denote by $a_{n} = o(b_{n})$ if $\lim_{n} a_{n}/ b_{n} = 0$  for some real $a_{n}$ and $b_{n}$, and $a_{n} = O(b_{n})$ if $\lim\sup_{n} |a_{n}| / |b_{n}| < \infty$. The number of elements in a set $S$ is denoted as $\# S$, and for an interval $t$, we define $|t| \coloneqq \sup t - \inf t$. When the summation is over an empty set, we define its value as zero; also, we define $\frac{0}{0} = 0$. For simplicity, we frequently denote a sequence of elements $A_{1}, \dots, A_{k}$ as $A_{1:k}$. Unless otherwise noted, all logarithms used in this work are logarithms with base $2$.
	
	\section{Random forests}\label{Sec2}
	
	\subsection{Model setting and random forests algorithm}\label{Sec2.1}
	Let us denote by $m(\boldsymbol{X})$ the measurable nonparametric regression function with $p$-dimensional random vector $\boldsymbol{X}$ taking values in $[0, 1]^{p}$. The random forests algorithm aims to learn the regression function in a nonparametric fashion based on the observations $\boldsymbol{x}_{i} \in [0, 1]^{p}, y_{i}\in \mathbb{R}, i = 1, \dots, n$, from the nonparametric regression model 
	\begin{equation} \label{new.eq.001}
		y_{i} = m(\boldsymbol{x}_{i}) + \varepsilon_{i},
	\end{equation}
	where $\boldsymbol{X}, \boldsymbol{x}_{i}, \varepsilon_{i}, i = 1,\dots, n$, are independent, and $\{\boldsymbol{x}_{i}\}$ and  $\{\varepsilon_{i}\}$ are two sequences of identically distributed random variables. In addition, $\boldsymbol{x}_{1}$ is distributed identically as $\boldsymbol{X}$.


	\begin{figure}[h]\centering
		{  \label{h1}
			\begin{tikzpicture}[level distance = 4em,
				every node/.style = {shape=rectangle, rounded corners,
					draw, align=center,
					top color=white, bottom color=blue!0}]]
				\tikzstyle{level 1}=[sibling distance=60mm] 
				\tikzstyle{level 2}=[sibling distance=30mm] 
				\node { Level $0$\\$\boldsymbol{t}_{0} \coloneqq [0, 1]^{p}$\\ Split: $(j_{1}\in\Theta_{1, 1}, c_{1})$}
				child { node { Cell $\boldsymbol{t}_{1, 1}$\\ Split: $(j_{2} \in \Theta_{2, 1}, c_{2})$} 
					child { node {Cell $\boldsymbol{t}_{2, 1}$} }
					child { node {Cell $\boldsymbol{t}_{2, 2}$ } } }
				child { node {Cell $\boldsymbol{t}_{1, 2}$ \\ Split: $(j_{3} \in \Theta_{2, 2}, c_{3})$}
					child { node {Cell $\boldsymbol{t}_{2, 3}$} }
					child { node {Cell $\boldsymbol{t}_{2, 4}$ } }
				};
		\end{tikzpicture}}		
		\caption{A level $2$ (height $3$) tree example. Each node represents a cell and defines the point where we split the current cell and produce new cells. We denote the sets of features eligible for splitting cells at level $k-1$ as $\Theta_{k} \coloneqq \{\Theta_{k, 1}, \dots, \Theta_{k, 2^{k-1}}\}$ with $\Theta_{k, s}\subset \{1, \dots, p\}$.}  \label{fig:1}
	\end{figure}
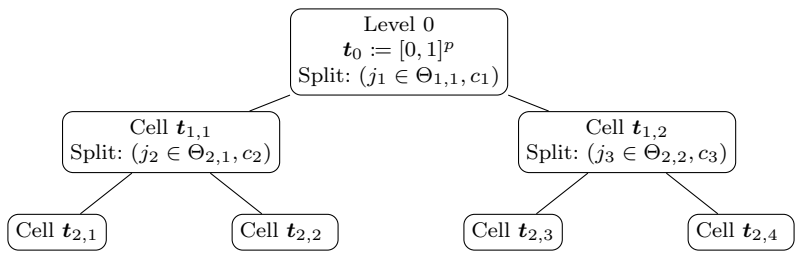

	In what follows, we introduce our random forests estimates and begin with a quick review of how a tree algorithm grows a tree using a given splitting criterion. The algorithm recursively partitions the root cell, using the splitting criterion to determine where to split each cell. This split involves two components: the direction or feature to split $j$ and the feature's value to split on $c$. In addition, the algorithm restricts the set of available features used to decide the direction of each split. This procedure repeats until the tree height reaches a predetermined level; the last grown cells are called the tree's end cells. 
	
	Next, we introduce the notation for the structure of a tree and its cells. A cell is defined as a rectangle such that  $\boldsymbol{t} = \times_{ j=1}^{p}t_{j}\coloneqq t_{1}\times\dots \times t_{p}$, where $\times$ denotes the Cartesian product and each $t_{j}$ is a closed or half closed interval in $[0, 1]$. For a cell $\boldsymbol{t}$, its two daughter cells, $t_{1}\times \dots \times t_{j-1}\times (t_{j}\cap[0, c))\times t_{j+1}\times \dots \times t_{p}$ 
	and  $t_{1}\times \dots \times t_{j-1}\times (t_{j}\cap[ c, 1]) \times t_{j+1}\times \dots \times t_{p}$,
	are obtained by splitting $\boldsymbol{t}$ according to the split $(j, c)$ with direction $j\in \{1, \dots, p\}$ and point $c\in t_{j}$.  We use $\Theta_{k} \coloneqq \{\Theta_{k, 1}, \dots, \Theta_{k, 2^{k-1}}\}$ to denote the sets of available features for the $2^{k-1}$ splits at level $k-1$ that grow the $2^{k}$ cells at level $k$ of the tree. See Figure~\ref{fig:1} for a graphical example. A split is also referred to as a cut, and we use $\boldsymbol{t}(j, c)$ to denote one of the daughter cells of $\boldsymbol{t}$ after the split $(j, c)$.

	Among the many existing splitting criteria, we are particularly interested in analyzing the statistical properties of the classification and regression tree (CART)-split criterion~\cite{Breiman2001, Breiman2002} used in the original random forests algorithm and formally introduced in Section~\ref{Sec2.2}. In this work, we use deterministic splitting criteria to characterize the statistical properties of CART-splits. A deterministic splitting criterion gives a split for a given cell $\boldsymbol{t}$ and set of available features $\Theta\subset\{1, \dots, p\}$ without being subject to  variation of the observed sample $(\boldsymbol{x}_{i}, y_{i})$'s. An important example of deterministic splitting criterion, known as the theoretical CART-split criterion in the literature~\cite{Biau2015,scornet2020trees}, is introduced in Section~\ref{Sec3.1}; the splits made by the theoretical CART-split criterion are not affected by the observed sample.  Some other deterministic splitting criteria are introduced in Section~\ref{Sec4.1}.  Next, we introduce the notation needed for our technical analysis.

	In light of how a tree algorithm grows trees, given any (deterministic) splitting criterion and a set of $\Theta_{1:k}$, we can grow all cells in a tree at each level until level $k$. We introduce a \emph{tree growing rule} for recording these cells. A tree growing rule denoted as $T$ is associated with this splitting criterion and given $\Theta_{1:k}$,  $T(\Theta_{1:k})$ denotes the collection of all sequences of cells connecting the root to the end cells at level $k$ of this tree.  Precisely, for each end cell of this tree, we can list a unique $k$-dimensional tuple of cells  that connects the root cell $\boldsymbol{t}_{0}\coloneqq[0, 1]^{p}$ to this end cell. These tuples of cells can be thought of as ``tree branches" that trace down from the root cell. An example of a tree branch in Figure~\ref{fig:1} is $( \boldsymbol{t}_{1, 1}, \boldsymbol{t}_{2, 2})$ that connects the root cell to end cell $\boldsymbol{t}_{2, 2}$.  In particular, the collection $T(\Theta_{1:k})$ contains $2^{k}$ such $k$-dimensional  tuples of cells.
	
	Given any $T$ (and hence the associated splitting criterion) and $\Theta_{1:k}$, the tree estimate denoted as $\hat{m}_{T(\Theta_{1:k})}$ for a test point $\boldsymbol{c}\in [0, 1]^{p}$ is defined as 
	\begin{equation}
		\label{eq.new.2}
		\hat{m}_{T(\Theta_{1:k})}(\boldsymbol{c}, \mathcal{X}_{n}) \coloneqq \sum_{(\boldsymbol{t}_{1}, \dots, \boldsymbol{t}_{k})\in T(\Theta_{1:k})} \boldsymbol{1}_{\boldsymbol{c} \in \boldsymbol{t}_{k}} \left( \frac{\sum_{i \in \{i : \boldsymbol{x}_{i} \in \boldsymbol{t}_{k}\}} y_{i}}{\# \{i : \boldsymbol{x}_{i} \in \boldsymbol{t}_{k}\}}\right),
	\end{equation}
	where $\mathcal{X}_{n}\coloneqq \{\boldsymbol{x}_{i}, y_{i}\}_{i=1}^{n}$, the fraction is defined as zero when no sample is in the cell $\boldsymbol{t}_{k}$, and $\boldsymbol{1}_{\boldsymbol{c} \in \boldsymbol{t}_{k}}$ is an indicator function taking value 1 if $\boldsymbol{c} \in \boldsymbol{t}_{k}$ and 0 otherwise. In (\ref{eq.new.2}), each test point in $[0, 1]^{p}$ belongs to one end cell since $\{\boldsymbol{t}_{k}: (\boldsymbol{t}_{1}, \dots, \boldsymbol{t}_{k}) \in T(\Theta_{1:k})\}$ for each integer $k > 0$ is a partition of $[0, 1]^{p}$.

	
	Now that we have introduced individual trees, we discuss how to  combine trees to form a forest. As mentioned previously, for each cell of a tree, only a subset of all features are used to choose a split. Every set of available features $\Theta_{l, s}, l = 1, \dots, k, s = 1, \dots, 2^{l-1}$, in this work has $\ceil{\gamma_{0}p}$ distinct integers among $1, \dots, p$ with $\ceil{\cdot}$ the ceiling function for some $0 < \gamma_{0} \le 1$, which is a predetermined constant parameter. The default parameter value $\gamma_{0} = 1/3$ is used in most implementations of the random forests algorithm. Given a growing rule $T$, each sequence of sets of available features $\Theta_{1:k}$ can be used for growing a level $k$ tree, and a sequence of distinct $\Theta_{1:k}$ results in a distinct tree. Given $k$, $p$, and $\gamma_{0}$, the forests considered in this work consist of all possible distinct trees, in the sense that all possible sets of available features, $\Theta_{1:k}$,  are considered.
	
	The prediction of random forests is the average over predictions of all tree models in the forests. To have a precise definition, we introduce the boldface random mappings $\boldsymbol{\Theta}_{1:k}$, which are independent and uniformly distributed over all possible $\Theta_{1:k}$ for each integer $k$. The random forests estimate for $\boldsymbol{c}$ with the observations $\mathcal{X}_{n}$ is given by
	\begin{equation}
		\label{eq.new.4}
		\begin{aligned}
			& \mathbb{E}(\hat{m}_{T(\boldsymbol{\Theta}_{1:k})}(\boldsymbol{c}, \mathcal{X}_{n}) \ \vert \ \mathcal{X}_{n}) = \sum_{\Theta_{1:k}}\mathbb{P}( \cap_{s=1}^{k}\{\boldsymbol{\Theta}_{s}= \Theta_{s}\})\hat{m}_{T(\Theta_{1:k})}(\boldsymbol{c}, \mathcal{X}_{n}).
		\end{aligned}
	\end{equation}
	That is, we take expectation over sets of available features\footnote{For clarity, given the tree height $k + 1$, the number of features $p$, and $\gamma_{0}$, the number of distinct $\Theta_{1:k}$ is $\binom{p}{\ceil{\gamma_{0}p}}^{2^{k}-1}$. Moreover, for each set of available features, $\mathbb{P}( \cap_{s=1}^{k}\{\boldsymbol{\Theta}_{s} = \Theta_{s}\} ) =  \binom{p}{\ceil{\gamma_{0}p}}^{1-2^{k}}$.
	}\footnote{When $\gamma_{0} =1$, the expectation is redundant.}. This step of tree aggregation is called column subsampling. In practice, there can be more than one way of using column subsampling; in this work, we particularly use all possible trees in tree aggregation for  column subsampling in (\ref{eq.new.4}). 
	\begin{remark}\label{r1}
		In contrast to the conventional notation for estimates, the notation for forests estimates (\ref{eq.new.4}) is with the conditional expectation. Alternatively, we can define the finite discrete parameter space $\mathcal{Q}$ and write (\ref{eq.new.4}) as $(\#\mathcal{Q})^{-1}\sum_{\Theta \in \mathcal{Q}}\hat{m}_{T(\Theta)}(\boldsymbol{c}, \mathcal{X}_{n})$, where $\Theta$ denotes a sequence of sets of available features. However, we choose to use the definition stated in (\ref{eq.new.4}) since the exact definition of the discrete parameter space and the cardinality of this space are not strictly relevant to our technical analysis. We define other forests estimates likewise.
	\end{remark}
	\begin{remark}
		We use the conditional expectation for compact notation for random forests estimates. However, by definition, for the conditional expectation to exist, the first moment of the integrand is required to exist. Rigorously, some regularity conditions are needed for this purpose. For details, see Section~\ref{SecA.3} of Supplementary Material.
	\end{remark}
	
	In addition to column subsampling, the random forests algorithm also resamples observations for making predictions. Let $A = \{a_{1}, \dots, a_{B}\}$ be a set of subsamples with each $a_{i}$ consisting of $\ceil{bn}$ observations (indices)  drawn without replacement from $\{1, \dots, n\}$ for some positive integer $B$ and $0 < b\le 1$; in addition, each $a_{i}$  is independent of model training. The default values of parameters $B$ and $b$ are $500$ and $0.632$, respectively, in the \texttt{randomForest} R package~\cite{rf}\footnote{Another default setup sets $b = 1$ but draws observations with replacement.}. The tree estimate using subsample $a$ is defined as 
	\begin{equation}
		\label{eq.new.6}\hat{m}_{T(\Theta_{1:k}), a}(\boldsymbol{c}, \mathcal{X}_{n}) \coloneqq \sum_{(\boldsymbol{t}_{1}, \dots, \boldsymbol{t}_{k})\in T(\Theta_{1:k})} \boldsymbol{1}_{\boldsymbol{c} \in \boldsymbol{t}_{k}}  \left(\frac{\sum_{i \in a\cap\{i : \boldsymbol{x}_{i} \in \boldsymbol{t}_{k}\}} y_{i}}{\# (a\cap\{i : \boldsymbol{x}_{i} \in \boldsymbol{t}_{k}\}) }\right).
	\end{equation}
	The random forests estimate given $A$ is then defined as 
	\begin{equation}
		\label{eq.new.5}
		B^{-1}\sum_{a \in A}\mathbb{E}(\hat{m}_{T, a}(\boldsymbol{\Theta}_{1:k}, \boldsymbol{c}, \mathcal{X}_{n}) \ \vert \ \mathcal{X}_{n}) \coloneqq B^{-1}\sum_{a \in A}\mathbb{E}(\hat{m}_{T(\boldsymbol{\Theta}_{1:k}), a}(\boldsymbol{c}, \mathcal{X}_{n}) \ \vert \ \mathcal{X}_{n}),
	\end{equation}
	where we move the boldface notation up into the parenthesis for simplicity and the conditional expectation above 
	is with respect to $\boldsymbol{\Theta}_{1:k}$ \footnote{
		In particular, for independent random variables such as $\boldsymbol{\Theta}_{k}$'s, we use an expectation with a subscript to indicate which variables the expectation is with respect to, which is equivalent to the expectation conditional on all other variables. We use conditional expectations to make the expressions in alignment with those in the technical proofs, where we repeatedly manipulate the conditional expectations.}.
	
	The estimate in (\ref{eq.new.5}) is an abstract random forests estimate since a generic tree growing rule $T$ is used. The benefit of using abstract random forests estimates can be seen in Theorem~\ref{theorem3} in Section~\ref{Sec4.1} for analyzing the bias of random forests.  In addition to abstract random forests estimates, we consider the sample random forests estimate introduced in (\ref{estimate2}) in Section~\ref{Sec2.2}. For simplicity, we refer to both versions as random forests estimates unless the distinction is necessary. 

	\subsection{CART-split criterion}\label{Sec2.2}

	Given  a cell $\boldsymbol{t}$, a subset of observation indices $a$, and a set of available features $\Theta\subset\{1, \dots, p\}$, the CART-split is defined as 
	\begin{equation}
		\begin{split} 
			\label{new.eq.004}
			(\,\widehat{j}, \widehat{c}\,) \coloneqq {} \argmin_{j\in \Theta, \, c \in \{x_{ij}   : \ \boldsymbol{x}_{i}\in\boldsymbol{t}, i \in a   \}} \left[
			\sum_{i \in a \cap P_{L} } \left(\bar{y}_{L}    - y_{i}\right)^{2}  + \sum_{i \in a\cap P_{R} } \left( \bar{y}_{R}  - y_{i}\right)^{2}  \right],
		\end{split}
	\end{equation} 
	where $P_{L} \coloneqq \{i : \ \boldsymbol{x}_{i}\in\boldsymbol{t}, \, x_{ij} < c  \}$, $P_{R} \coloneqq \{i : \ \boldsymbol{x}_{i}\in\boldsymbol{t}, \, x_{ij} \geq c  \}$, and
	\begin{equation*}
		\begin{split}
			\bar{y}_{L}  \coloneqq\ \sum_{i \in a\cap P_{L}} \frac{y_{i} }{  \#(a\cap P_{L} ) }, \quad \bar{y}_{R}  \coloneqq \sum_{i \in a\cap P_{R}}\frac{y_{i} }{ \#(a\cap P_{R}) }.
		\end{split}
	\end{equation*}
	The criterion breaks ties randomly; to simplify the analysis, we assume that the criterion splits on a random point if $\# \{x_{ij}   :  \boldsymbol{x}_{i}\in\boldsymbol{t}, i \in a   \}\le 1$, a situation where both summations in \eqref{new.eq.004} are zeros (summation over an empty index set is defined as zero). Furthermore,  we define the criterion in the way such that every split results in two non-empty (an empty cell has zero volume) daughter cells\footnote{This statement can be made rigorous with a more sophisticated definition of CART-splits, but we omit the details for simplicity.}. It is worth mentioning that given a sample, the CART-split criterion conditional on the sample is a deterministic (except for random splits due to ties) splitting criterion; conditioning on another sample leads to another deterministic splitting criterion.

	Let us define $\hat{T}_{a}$ as the sample tree growing rule that is associated with a splitting criterion following (\ref{new.eq.004}).  In (\ref{eq.new.2}) and (\ref{eq.new.6}), we have introduced the tree estimates based on tree growing rules associated with deterministic splitting criteria. The tree estimates using $\hat{T}_{a}$ can be similarly defined  because the sample tree growing rule is a deterministic tree growing rule when the sample $\mathcal{X}_{n}$ is given.  Specifically, we have\footnote{The notation for \eqref{estimate2} is not suitable for cases such as the honest trees~\cite{Wager2018} where the sample for growing trees and that for prediction are different.}
	\begin{equation}\label{estimate2}
		\hat{m}_{\widehat{T}_{a}(\Theta_{1:k})}(\boldsymbol{c}, \mathcal{X}_{n}) \coloneqq \sum_{(\boldsymbol{t}_{1}, \dots, \boldsymbol{t}_{k})\in \widehat{T}_{a}(\Theta_{1:k})} \boldsymbol{1}_{\boldsymbol{c} \in \boldsymbol{t}_{k}} \left( \frac{\sum_{i \in \{i : \boldsymbol{x}_{i} \in \boldsymbol{t}_{k}\}} y_{i}}{\# \{i : \boldsymbol{x}_{i} \in \boldsymbol{t}_{k}\}}\right),
	\end{equation}
	and the definition is the same for $\hat{m}_{\hat{T}_{a}, a}$. Hence, the random forests estimate for a test point $\boldsymbol{c} \in [0, 1]^{p}$ is given by
	\begin{equation}
		\label{estimate1}
		B^{-1}\sum_{a \in A}\mathbb{E} \Big(\hat{m}_{\hat{T}_{a}, a}(\boldsymbol{\Theta}_{1:k}, \boldsymbol{c}, \mathcal{X}_{n}) \ \vert \  \mathcal{X}_{n} \Big),
	\end{equation}
	where the average and conditional expectation correspond to the sample and column subsamplings, respectively. Note that the average and conditional expectation are interchangeable.

	\subsection{Roadmap of bias-variance decomposition analysis}\label{Sec2.3}
	
	We are now ready to introduce the bias-variance decomposition analysis for random forests, whose $\mathbb{L}^{2}$ prediction loss is defined as
	\begin{equation}\label{eq.new.7}
		\mathbb{E} \Big(m(\boldsymbol{X}) - B^{-1}\sum_{a \in A}\mathbb{E} \Big(\hat{m}_{\hat{T}_{a}, a}(\boldsymbol{\Theta}_{1:k}, \boldsymbol{X}, \mathcal{X}_{n}) \ \vert \ \boldsymbol{X}, \mathcal{X}_{n} \Big) \Big)^{2}.
	\end{equation}
	Let us define some notation for a further illustration. For a tree growing rule $T$ and $\Theta_{1:k}$, the population version of (\ref{eq.new.2})  is defined as 
	\begin{equation}\label{eq.new.1}
		m_{T(\Theta_{1:k})}^{*}(\boldsymbol{c}) \coloneqq \sum_{(\boldsymbol{t}_{1}, \dots, \boldsymbol{t}_{k})\in T(\Theta_{1:k})} \boldsymbol{1}_{\boldsymbol{c} \in \boldsymbol{t}_{k}} \mathbb{E}(m(\boldsymbol{X})\ \vert \ \boldsymbol{X} \in \boldsymbol{t}_{k})
	\end{equation}
	for each test point $\boldsymbol{c} \in [0, 1]^{p}$. For simplicity, we use $m_{T}^{*}(\Theta_{1:k}, \boldsymbol{c})$ to denote the left-hand side (LHS) of (\ref{eq.new.1}). The population estimate (\ref{eq.new.1}) can also be used with the sample tree growing rule; $m_{\widehat{T}_{a}}^{*}$ is defined likewise with $\hat{T}_{a}$ in place of $T$ in (\ref{eq.new.1}). To simplify the notation, let us temporarily consider the case that uses the full sample $a = \{1, \dots, n\}$ and denote $\hat{T}_{a}$ and $\hat{m}_{\hat{T}_{a}, a}$ as $\hat{T}$ and $\hat{m}_{\hat{T}}$, respectively. Since we utilize the full sample, the sample subsampling and the average $B^{-1}\sum_{a\in A}(\cdot)$ in the random forests estimate in (\ref{eq.new.7}) are no longer needed. Thus, \eqref{eq.new.7} becomes
	\begin{equation*}
		\begin{split}
			\mathbb{E} \Big(m(\boldsymbol{X}) - \mathbb{E} \Big(\hat{m}_{\hat{T}}(\boldsymbol{\Theta}_{1:k}, \boldsymbol{X}, \mathcal{X}_{n}) \ \vert \ \boldsymbol{X}, \mathcal{X}_{n} \Big) \Big)^{2}.
		\end{split}
	\end{equation*}
	
	By Jensen's inequality and the Cauchy--Schwarz inequality (precisely, we need conditional Jensen's inequality; for simplicity, we omit ``conditional'' when no confusion arises),  we can deduce that 
	\begin{equation}
		\begin{split}
			\label{decom1}
			& \frac{1}{2}	\mathbb{E} \Big(m(\boldsymbol{X}) - \mathbb{E} \Big(\hat{m}_{\hat{T}}(\boldsymbol{\Theta}_{1:k}, \boldsymbol{X}, \mathcal{X}_{n}) \ \vert \ \boldsymbol{X}, \mathcal{X}_{n} \Big) \Big)^{2} \\ & \leq \mathbb{E} \Big(m(\boldsymbol{X})  -m_{ \widehat{T} }^{*} (\boldsymbol{\Theta}_{1:k}, \boldsymbol{X})\Big)^{2} +  \mathbb{E}\Big(m_{ \widehat{T} }^{*} (\boldsymbol{\Theta}_{1:k} , \boldsymbol{X})- \widehat{m}_{ \widehat{T}}(\boldsymbol{\Theta}_{1:k} , \boldsymbol{X}, \mathcal{X}_{n} ) \Big)^{2},
		\end{split}
	\end{equation}
	where the right-hand side (RHS) is a summation of approximation error (the first term, which is also referred to as squared bias) and  estimation variance (the second term). Details on deriving (\ref{decom1}) can be found in Section~\ref{SecA.2} of Supplementary Material. The first term on the RHS of (\ref{decom1}) is referred to as the approximation error.  By the definition of $m_{T}^{*}$ in (\ref{eq.new.1}), it holds that for any tree growing rule $T$ and each $\Theta_{1:k}$, on $\cap_{l =1}^{k}\{\boldsymbol{\Theta}_{l} = \Theta_{l}\}$,
	\begin{equation}
		\begin{split}\label{t.new.5}
			&\mathbb{E}\Big(\Big(m(\boldsymbol{X}) - {m}_{T}^{*}\left(\boldsymbol{\Theta}_{1:k}, \boldsymbol{X}\right) \Big)^{2}  \Big\vert   \boldsymbol{\Theta}_{1:k}= \Theta_{1:k}\Big) \\
			& =  \sum_{ (\boldsymbol{t}_{1}, \dots, \boldsymbol{t}_{k}) \in T(\Theta_{1:k})  }  \mathbb{P}(\boldsymbol{X}\in \boldsymbol{t}_{k} )\textnormal{Var} (m(\boldsymbol{X}) \vert  \boldsymbol{X}\in\boldsymbol{t}_{k}).
	\end{split}\end{equation}
	The RHS of (\ref{t.new.5}) is the average approximation error resulting from $\mathbb{L}^{2}$-approximating $m(\boldsymbol{X})$  by the class of step functions $\{f(\boldsymbol{X}) : f(\boldsymbol{X}) =  \sum_{(\boldsymbol{t}_{1}, \dots, \boldsymbol{t}_{k}) \in T(\Theta_{1:k})} c(\boldsymbol{t}_{k}) \boldsymbol{1}_{\boldsymbol{X}\in \boldsymbol{t}_{k}}, c(\boldsymbol{t}_{k}) \in \mathbb{R}\}$. Observe that the approximation error in \eqref{decom1} is also subject to sample variation because sample CART-splits are used to build trees.

	In Section~\ref{Sec3.2}, we will obtain the desired convergence rates for random forests consistency by bounding the two terms in (\ref{decom1}), and introduce these results in Theorem~\ref{theorem1} and Corollary~\ref{corollary1}. In Section~\ref{Sec4}, we will analyze and establish an upper bound for the average approximation error in Lemma~\ref{lemma1}. The estimation variance term is then analyzed in Section~\ref{Sec5}, where we will introduce the high-dimensional random forests estimation foundation to establish the convergence rate for estimation variance in Lemma~\ref{con3}. Furthermore, $\gamma_{0}$ is a predetermined constant parameter and we do not specify the value of $\gamma_{0}$ if it is not directly relevant to the results (e.g., theorems or lemmas).  


	\section{Main results} \label{Sec3}

	\subsection{Definitions and technical conditions} \label{Sec3.1}
	For a cell $\boldsymbol{t}$ and its two daughter cells $\boldsymbol{t}^{'}$ and $\boldsymbol{t}^{''}$, let us define \begin{equation}
		\begin{split}
			\label{I1}
			& (I)_{\boldsymbol{t}, \boldsymbol{t}^{'}}  \coloneqq \mathbb{P} (\boldsymbol{X}\in\boldsymbol{t}^{'}   \vert  \boldsymbol{X} \in \boldsymbol{t} )\textnormal{Var}(m(\boldsymbol{X})  \vert  \boldsymbol{X}\in\boldsymbol{t}^{'}) \\
			&\qquad+ \mathbb{P} (\boldsymbol{X}\in\boldsymbol{t}^{''}   \vert  \boldsymbol{X} \in \boldsymbol{t} ) \textnormal{Var}(m(\boldsymbol{X})  \vert  \boldsymbol{X}\in\boldsymbol{t}^{''}),
	\end{split}\end{equation}
	and
	\begin{equation}\begin{split}
			\label{I2}
			& (II)_{\boldsymbol{t}, \boldsymbol{t}^{'}}\coloneqq \mathbb{P} (\boldsymbol{X}\in\boldsymbol{t}^{'}   \vert  \boldsymbol{X} \in \boldsymbol{t} )
			\Big(\mathbb{E} (m(\boldsymbol{X})  \vert  \boldsymbol{X} \in \boldsymbol{t}^{'}) - \mathbb{E} (m(\boldsymbol{X})  \vert  \boldsymbol{X} \in \boldsymbol{t} ) \Big)^{2} \\
			& +  \mathbb{P} (\boldsymbol{X}\in\boldsymbol{t}^{''}   \vert  \boldsymbol{X} \in \boldsymbol{t} )
			\Big(\mathbb{E} (m(\boldsymbol{X})  \vert  \boldsymbol{X} \in \boldsymbol{t}^{''}) - \mathbb{E} (m(\boldsymbol{X})  \vert  \boldsymbol{X} \in \boldsymbol{t} ) \Big)^{2},
	\end{split}\end{equation}
	and have $(I)_{\boldsymbol{t}, \boldsymbol{t}^{''}}$ and $(II)_{\boldsymbol{t}, \boldsymbol{t}^{''}}$ defined the same as $(I)_{\boldsymbol{t}, \boldsymbol{t}^{'}}$ and $(II)_{\boldsymbol{t}, \boldsymbol{t}^{'}}$, respectively. To facilitate our technical analysis, we introduce some natural regularity conditions and their intuitions below. 
	\begin{condi} 
		\label{P}
		There exists some $\alpha_{1}\ge  1$  such that for each cell $\boldsymbol{t} = t_{1}\times \dots\times t_{p}$,
		\begin{equation*}
			\textnormal{Var}(m(\boldsymbol{X}) \ \vert \ \boldsymbol{X} \in \boldsymbol{t} ) \le \alpha_{1} \sup_{j\in \{1, \dots, p\}, c\in t_{j}} (II)_{\boldsymbol{t}, \boldsymbol{t}(j, c)}.
		\end{equation*}
	\end{condi}

	\begin{condi} 
		\label{abc}
		The distribution of $\boldsymbol{X}$ has a density function $f$ that is bounded away from $0$ and $\infty$.
	\end{condi}
	
	\begin{condi} 
		\label{ME}
		Assume model (\ref{new.eq.001}) and $p=O(n^{K_{0}})$ for some positive constant $K_{0}$.  In addition, assume a symmetric distribution around $0$ for $\varepsilon_{1}$ and  $\mathbb{E}|\varepsilon_{1}|^{q} < \infty$ for sufficiently large $q > 0$ whose value will be specified depending on the contexts. 
	\end{condi}

	\begin{condi} 
		\label{BO}
		Assume that $\sup_{\boldsymbol{c} \in [0, 1]^{p}}|m(\boldsymbol{c})| \le M_{0}$ for some $M_{0} > 0$.
	\end{condi}
	
We name Condition~\ref{P} above as the sufficient impurity decrease (SID), and it is new to the literature. Overall, Conditions \ref{abc}--\ref{BO} are some basic assumptions in nonparametric regression models. In particular, our technical analysis allows for polynomially growing dimensionality $p$. The symmetric distribution on the model error is a technical assumption that can be relaxed. The SID assumption introduced in Condition~\ref{P} plays a key role in our technical analysis and we motivate the need for this condition as follows. Consider two tree models: $f_{1}(\boldsymbol{X}) = \boldsymbol{1}_{\boldsymbol{X}\in \boldsymbol{t}_{0}}\mathbb{E}(m(\boldsymbol{X})| \boldsymbol{X} \in \boldsymbol{t}_{0})$ and $f_{2}(\boldsymbol{X}) = \boldsymbol{1}_{\boldsymbol{X}\in \boldsymbol{t}^{'}}\mathbb{E}(m(\boldsymbol{X})| \boldsymbol{X} \in \boldsymbol{t}^{'}) + \boldsymbol{1}_{\boldsymbol{X}\in \boldsymbol{t}^{''}}\mathbb{E}(m(\boldsymbol{X})| \boldsymbol{X} \in \boldsymbol{t}^{''})$, where $\boldsymbol{t}^{'}$ and $\boldsymbol{t}^{''}$ are daughter cells of $\boldsymbol{t}_{0}$ after some split. The squared biases given these tree models are respectively $\mathbb{E} (m(\boldsymbol{X}) - f_{1}(\boldsymbol{X}))^2 = \textnormal{Var}(m(\boldsymbol{X})|\boldsymbol{X} \in \boldsymbol{t}_{0})$ and $\mathbb{E} (m(\boldsymbol{X}) - f_{2}(\boldsymbol{X}))^2  = (I)_{\boldsymbol{t}_{0},\boldsymbol{t}^{'}}$. We see that $(I)_{\boldsymbol{t}_{0},\boldsymbol{t}^{'}}$ is the squared bias for approximating $m(\boldsymbol{X})$ with $f_{2}(\boldsymbol{X})$. Since it is the  squared bias remaining after the split on $\boldsymbol{t}_{0}$, it is also called the remaining bias; analogously, we extend the definition to an arbitrary cell $\boldsymbol{t}$ and one of its daughter cell $\boldsymbol{t}^{'}$ and use $(I)_{\boldsymbol{t},\boldsymbol{t}^{'}}$ to denote the ``conditional remaining bias''. The term $(II)_{\boldsymbol{t},\boldsymbol{t}^{'}}$ and $\textnormal{Var}( m(\boldsymbol{X})  \vert \boldsymbol{X} \in \boldsymbol{t})$ are respectively called the ``conditional bias decrease'' (or conditional impurity decrease) and ``conditional total bias'' because $\textnormal{Var}( m(\boldsymbol{X})  \vert \boldsymbol{X} \in \boldsymbol{t}) = (I)_{\boldsymbol{t}, \boldsymbol{t}^{'}} + (II)_{\boldsymbol{t}, \boldsymbol{t}^{'}}$. 
	
Intuitively, having large conditional bias decrease on each cell is a desired property for achieving a good control of the squared bias of random forests estimate. 
This naturally motivates the SID condition. Notice that SID only requires a nontrivial lower bound for the maximum conditional bias decrease, and that the split $(j^*, c^*)=\arg\sup_{j\in \Theta, c\in t_{j}} (II)_{\boldsymbol{t}, \boldsymbol{t}(j, c)}$ with the column restriction $\Theta$ is the theoretical CART~\cite{Biau2015, scornet2020trees, klusowski2019analyzing}.

In Section~\ref{Sec3.2}, we establish the convergence rate for random forests estimate with $m(\boldsymbol x)$ coming from the functional class 
	$$\textnormal{SID}(\alpha)\coloneqq\{ m(\boldsymbol{X}): m(\boldsymbol{X})\textnormal{ satisfies SID with } \alpha_{1} \le \alpha \}.$$ 
The size of $\textnormal{SID}(\alpha)$ is non-decreasing in $\alpha\ge 1$: if $m(\boldsymbol{X})\in \textnormal{SID}(\alpha - c)$ for some $\alpha - c \ge 1$ and $c>0$, then $m(\boldsymbol{X})\in \textnormal{SID}(\alpha )$. In Section~\ref{Sec3.1.1}, we verify that many popular regression functions can belong to the above functional class and derive the corresponding values of $\alpha$; 
these examples show that the SID condition can accommodate non-additive and/or discontinuous regression functions, and allows for dependent features. In Section~\ref{Sec3.3}, we illustrate an important relation between SID and the model sparsity.

	\subsubsection{Examples satisfying SID}\label{Sec3.1.1}

	\begin{exmp}\label{new.example1}
		Consider $m(\boldsymbol{X}) = \boldsymbol{1}_{X_{1} \in [b, 1]}$ for some $0\le b\le 1$, and the distribution of $\boldsymbol{X}$ is arbitrary. Then, $m(\boldsymbol{X}) \in \textnormal{SID}(1).$ 
	\end{exmp}

	\begin{exmp}\label{new.example2}
		Let $\boldsymbol{X}$ have uniform distribution over $[0, 1]^p$ and $0<s^*\le p$ be a given integer. Consider the regression function defined as 
		$$m(\boldsymbol{X}) = \beta_{0} + \sum_{j=1}^{s^*} \Big(\beta_{j}X_{j} + \sum_{l \ge j}^{s^*}\beta_{lj}X_{l}X_{j} \Big),$$ 
		where if  $\beta_{lj} \not = 0$ for some $j< l \le s^*$, then $a\times b\ge 0$ for every $\{a, b\}\subset \{\beta_{j}, \beta_{1j}, \dots, \beta_{s^*j} \}$, where $\beta_{j_{2}j_{1}} = \beta_{j_{1}j_{2}}$; the coefficients are otherwise arbitrary. Then, $m(\boldsymbol{X}) \in \textnormal{SID}(c s^*)$ for some constant $c>0$ independent of model coefficients here. 
	\end{exmp}
	
	\begin{exmp}\label{new.example3}
Let $\boldsymbol{X}$ be uniformly distributed over $[0, 1]^p$. Consider $m(\boldsymbol{X})$ with $(\frac{\partial m(\boldsymbol{z})}{\partial z_{1}}, \dots, \frac{\partial m(\boldsymbol{z})}{\partial z_{p}})$ being continuous in $[0, 1]^p$. In addition, \textnormal{1)} for each $j\in S^*$, either $M_{1}\le \frac{\partial m(\boldsymbol{z})}{\partial z_{j}}\le M_{2}$ for every $\boldsymbol{z}\in[0, 1]^p$ or $ M_{1}\le -\frac{\partial m(\boldsymbol{z})}{\partial z_{j}}\le M_{2}$ for every $\boldsymbol{z}\in[0, 1]^p$, where $M_{2}\ge M_{1}>0$ are constants and $S^*$ is some subset of $\{1, \dots, p\}$, and \textnormal{2)} for each $j\not\in S^*$, $\frac{\partial m(\boldsymbol{z})}{\partial z_{j}} = 0$. Then,
		$m(\boldsymbol{X}) \in \textnormal{SID} (c(\# S^*)^2 )$ for some constant $c>0$ depending only on $M_{1}, M_{2}$. Furthermore, if $m(\boldsymbol{X}) = \sum_{j=1}^{s^*} m_{j}(X_{j})$ for some positive integer $s^*$, then
		$m(\boldsymbol{X}) \in \textnormal{SID} (cs^*)$.
	\end{exmp}

	\begin{exmp}\label{new.example4}
		Let $\boldsymbol{X}$ be uniformly distributed over $[0, 1]^p$. Consider  $m(\boldsymbol X) = \sum_{j=1}^{s^*}m_j(X_{j})$ with $1\le s^*\le p$ an integer. Suppose that for some $c_{0}>0$ and $\frac{1}{2}<\lambda < 1$, it holds that for every $1\le j\le s^*$ and every $(a,b)\subset [0,1]$,			
		\begin{equation}
			\begin{split}			 					 
				\label{tool.6}
				&\sup_{x\in \Lambda(a,b)} \left(\frac{1}{x-a}\int_a^x\left(m_j\left(\left(\frac{b-x}{x-a}\right)(z-a)+x\right)-m_j(z)\right)dz \right)^2 \\
				& \geq c_{0}\textnormal{Var}(m_{j}(X_{j})|X_{j}\in (a, b)),
			\end{split}
		\end{equation}
		where $\Lambda(a, b) = [\lambda a + (1-\lambda) b, (1 - \lambda)a + \lambda b]$. Then, $m(\boldsymbol{X})\in \textnormal{SID}(cs^*)$ for some $c>0$ depending on $c_{0}, \lambda$. In particular, \eqref{tool.6} holds if $m_{j}(z)$ is differentiable on $[0, 1]$ and that for some $c_{1}>0$ and every $(a, b)\subset[0, 1]$,
		\begin{equation}
			\begin{split}			 					 
				\label{tool.7}
				&\textnormal{LHS of \eqref{tool.6}} \geq c_{1}\sup_{z\in (a,b)}|m_j'(z)|^2 (b-a)^2.
			\end{split}
		\end{equation}
	\end{exmp}


	
	Example~\ref{new.example1} assures that SID allows for dependent features. 
	Example~\ref{new.example2} shows that SID is satisfied in high-dimensional sparse quadratic models. 
	Example~\ref{new.example3} 
	considers a general structure for $m(\boldsymbol{X})$ which can include some special cases of cumulative distribution functions, linear functions, logistic functions, and non-additive polynomial functions. Example~\ref{new.example4} provides some sufficient conditions ensuring SID with uniform $\boldsymbol{X}$ under the sparse additive model setting. In particular, if $m(\boldsymbol X)$ in Examples~\ref{new.example2}--\ref{new.example3} are also additive, then they can be verified to satisfy the requirements of Example~\ref{new.example4} as well.  
	Similar sufficient conditions to those in Example \ref{new.example4} can be derived for nonadditive models but take much more complicated forms; the additive structure in Example \ref{new.example4} is imposed for simplified forms of these conditions.
	More examples, including logistic regression functions, higher order polynomial functions with interactions, additive piecewise linear models, and linear combinations of indicator functions of hyperrectangles can be found in Section~\ref{SecA.7} of the Supplementary Material. Proofs for these examples and proofs for the example in Remark~\ref{remark3} below are respectively in Sections~\ref{SecE.3}--\ref{SecE.7} of the Supplementary Material.

	\begin{remark}	\label{remark3}

		The coefficient restriction is necessary for SID to hold in Example~\ref{new.example2}. A counterexample violating the coefficient restriction of Example~\ref{new.example2} is $m(\boldsymbol{X}) = X_{1}X_{2} -0.5X_{1} - 0.5X_{2} + 0.25$ with uniformly distributed  $\boldsymbol{X}$; see Supplementary Material for a formal proof that this example violates SID. 
		Section 5 of~\cite{syrgkanis2020estimation} studied inconsistency of random forests under similar model settings, suggesting that certain coefficient restriction is necessary for good performance of CART in these cases.  Identifying the necessary condition for random forests consistency,  and studying how far SID is from such a condition is an open question for future study.
		
	\end{remark}

	\subsubsection{SID and model sparsity: sparsity parameter $\alpha_{1}$}\label{Sec3.3}
	
	A smaller value of  $\alpha_{1}$ in SID implies that the optimal split can reduce more impurity in terms of conditional total bias given each cell. On the other hand, in sparse models, the examples in the previous section show that the required value of $\alpha_{1}$ for $m(\boldsymbol{X}) \in \textnormal{SID}(\alpha_{1})$ is at least linearly proportional to the number of active features in $m(\boldsymbol{X})$. These results echo the intuition on how CART works in reducing impurity: more active features 
	implies that on average each split contributes less (proportionally) to impurity reduction. Here, we remark that  the previous examples aim for generality and hence the derived $\alpha_{1}$ may not be optimal. For example, in Example~\ref{new.example3}, with the additional assumption of additive model,  the value of $\alpha_{1}$ depends linearly  on the number of active features, compared to the quadratic dependence without such an assumption. In addition, the values of $\alpha_{1}$ in these examples may be smaller for certain specific model coefficients.

	\subsection{Convergence rates}\label{Sec3.2}
	We are now ready to characterize the explicit convergence rates for the consistency of random forests in a fairly general high-dimensional nonparametric regression setting.  Recall that $\hat{T}_{a}$ and $\gamma_{0} $ are defined in Sections \ref{Sec2.2} and \ref{Sec2.1}, respectively; $B$ is the number of trees regarding row subsampling and $0 <b\le 1$ is the proportion of training data used for row subsampling, whose definitions can be found right before  \eqref{eq.new.6}. Details on the random forests estimates $\mathbb{E}( \widehat{m}_{ \widehat{T}_{a}, a}(\dots )\ \big\vert \ \boldsymbol{X}, \mathcal{X}_{n} )$ and $\frac{1}{B}\sum_{a \in A}\mathbb{E}( \widehat{m}_{ \widehat{T}_{a}, a}(\dots )\ \big\vert \ \boldsymbol{X}, \mathcal{X}_{n} )$ can be found in  Section~\ref{Sec2.2}.  
	\begin{theorem}\label{theorem1}
		Assume that Conditions \ref{P}--\ref{BO}  hold and let $0<b\le 1 $, $0 < \gamma_{0}\le 1$, $\alpha_{2} > 1$, $0<  \eta < 1/8$, $0 < c< 1/4$, and $\delta > 0$ be given with $2\eta<\delta <\frac{1}{4}$. 
Let $A = \{a_{1}, \dots, a_{B}\}$ with $\# a_{i}=\ceil{bn}$ for $i=1,\cdots, B$ and $a\in A$ be given. Then, there exists some $C>0$ such that  for all large $n$ and each $1 \le k \le   c\log_{2}{\ceil{bn}}$,
		\begin{equation}\label{main1}
			\begin{split}&  \mathbb{E}\Big(m(\boldsymbol{X}) - \mathbb{E}\Big( \widehat{m}_{ \widehat{T}_{a}, a}(\boldsymbol{\Theta}_{1:k}, \boldsymbol{X}, \mathcal{X}_{n} )\ \Big\vert \ \boldsymbol{X}, \mathcal{X}_{n} \Big) \Big)^{2} \\
				&\quad \le C\Big(  \alpha_{1} (\ceil{bn})^{-\eta}+ (1 - \gamma_{0} (\alpha_{1} \alpha_{2})^{-1})^{k} +   (\ceil{bn})^{-\delta + c}   \Big).
		\end{split}\end{equation}
		In addition, when we also aggregate over row subsamples (i.e., over $a \in A$), we have
		\begin{equation}\label{main2}\begin{split}& \mathbb{E} \Big(m(\boldsymbol{X}) - \frac{1}{B}\sum_{a\in A}\mathbb{E}\Big(\widehat{m}_{ \widehat{T}_{a}, a}(\boldsymbol{\Theta}_{1:k}, \boldsymbol{X}, \mathcal{X}_{n} )\ \Big\vert \ \boldsymbol{X}, \mathcal{X}_{n} \Big) \Big)^{2} \\
				&\quad \le C\Big(  \alpha_{1} (\ceil{bn})^{-\eta} + (1 - \gamma_{0} (\alpha_{1} \alpha_{2})^{-1})^{k} + (\ceil{bn})^{-\delta + c}  \Big).
		\end{split}\end{equation}

	\end{theorem}
	To the best of our knowledge, Theorem \ref{theorem1} above is the first result on the consistency rates for the original version of the random forests algorithm; see Table~\ref{table:comparison} in the Introduction for detailed discussions. The constant $\alpha_2  > 1$ is arbitrary and is needed to account for the estimation error from using sample CART-splits. Although Condition \ref{con3} restricts the feature dimensionality $p$, the upper bounds in Theorem \ref{theorem1} do not depend on $p$ explicitly. See Remark~\ref{p} for the implicit dependence of the rates on $p$; also, see Section~\ref{Sec3.32} for more informative convergence rates depending on $p$ for models with binary features. Our results provide no interesting information about the tuning parameter $b$.  The technical reason is that we have used the Cauchy--Schwarz inequality when deriving \eqref{main2} from \eqref{main1},  which holds even in the worst case when all trees are highly correlated with each other.  In this sense, \eqref{main2} only gives a highly conservative upper bound.  Because random forests improves upon bagging~\cite{breiman1996bagging} using column subsampling, in what follows we focus on random forests using only column subsampling with the full sample (i.e., $b = 1$ and $a = \{1, \dots, n\}$).  The notation for the case with $b= 1$ can be found in Section~\ref{Sec2.3}. 
	
	To gain more in-depth understandings on the upper bounds in Theorem \ref{theorem1},  we provide the following Corollary~\ref{corollary1} restating the results in Theorem~\ref{theorem1} with more emphasis on the bias-variance decomposition. 
	\begin{corollary}\label{corollary1}
		Under all the conditions of Theorem~\ref{theorem1},  for all large $n$ and each $1 \le k \le   c\log_{2}{n}$, it holds for the two terms on the RHS of \eqref{decom1} that
		\begin{equation}\begin{split}\label{main4}
				\textnormal{Squared bias} & \coloneqq \mathbb{E} \Big(m(\boldsymbol{X})  -m_{ \widehat{T} }^{*} (\boldsymbol{\Theta}_{1:k}, \boldsymbol{X})\Big)^{2} \\
				& \le  O\Big(n^{-\eta} + \underbrace{(1 - \gamma_{0} (\alpha_{1} \alpha_{2})^{-1})^{k}}_{\textnormal{Main term of bias}}\Big) + \underbrace{O(n^{-\delta + c})}_{\textnormal{Uninteresting error}}
		\end{split}\end{equation}
		and 
		\begin{equation}\begin{split}\label{main5}		
				\textnormal{Estimation variance} & \coloneqq  \mathbb{E}\Big(m_{ \widehat{T} }^{*} (\boldsymbol{\Theta}_{1:k} , \boldsymbol{X})- \widehat{m}_{ \widehat{T}}(\boldsymbol{\Theta}_{1:k} , \boldsymbol{X}, \mathcal{X}_{n} ) \Big)^{2}\\& \le O(n^{-\eta} ) + \underbrace{O(n^{-\delta + c})}_{\textnormal{Uninteresting error}}.
		\end{split}\end{equation}
	\end{corollary}

	The third term in \eqref{main1}, which is also displayed in \eqref{main4} and \eqref{main5} as the ``uninteresting error,'' is caused by technical difficulty and does not carry too much meaning; see Remark \ref{reason1} below for details.  It is possible that such term can be eliminated by more refined technical analysis. Thus, the discussion in this section ignores this term for simplicity.

	Our results provide a fresh understanding of random forests,  especially in terms of how random forests controls the bias. The second term on the RHS of \eqref{main4} is the main term of random forests bias, whereas the first term $n^{-\eta}$ upper-bounds the error caused by the sample CART-splits. If theoretical CART-splits (see Section~\ref{Sec3.1} for its formal definition) are used, then the first term on the RHS of \eqref{main4} vanishes and the second term has $\alpha_2=1$; see Remark \ref{good1} in Section \ref{Sec4.1}.   We contribute to characterizing quantitatively how the use of column subsampling controls the bias when $k\le c\log_{2}{n}$; the exponential decay rate in the second term on the RHS of \eqref{main4} is derived through a \textit{global control} on the bias of all trees in the forests and the SID condition makes such precise quantification possible.  For fixed sample size $n$,  the first term $n^{-\eta}$ in the  bias upper bound does not vanish as tree height increases.  For the bias upper bound to decrease to zero, it requires both $n, k\rightarrow\infty$. Tuning a higher $\gamma_{0}$ decreases the second term in the bias decomposition, because a larger $\gamma_0$ makes each split more likely to be on relevant features (see Definition \ref{nonparametric_relevance} for formal definition) and hence decreases the bias faster. On the other hand, it is likely that the first term on the RHS of \eqref{main4}, $O(n^{-\eta})$, is a conservative upper bound on the error caused by sample CART-splits because this term is invariant to $\gamma_{0}$ and $k$. Accurately characterizing how this error  term depends on $\gamma_0$,  $k$ as well as $n$ is an interesting future research topic. Moreover, it is necessary to point out that this error term and the first term of estimation variance can generally have different convergence rates; $O(n^{-\eta})$ is a common upper bound for both terms for simplicity of the technical presentation.

	Our upper bound of estimation variance in \eqref{main5} is less informative compared to our bias upper bound in the sense that it does not reflect any effects of $\gamma_{0}$ and $k$.  In fact, the variance upper bound is conservative due to certain technical difficulties -- we bound the variance of random forests estimate by establishing a uniform upper bound for the variances of individual trees, a great distinction from the global approach in our bias analysis using SID.   It was discovered in the literature that random forests with column subsampling 
	is closely related to the adaptive nearest neighbors method \citep{lin2006random} and adaptive weighted estimation \citep{athey2019generalized}.  Let us draw analogy to the adaptive nearest neighbors method to assist the understanding.  As revealed in \cite{lin2006random}, the forest estimate has the representation of adaptive weighted nearest neighbors estimator.  The column subsampling rate $\gamma_0$ controls the adaptive weights assigned to the nearest neighbors,  with $\gamma_0=1$ producing uniform weights on only the nearest neighbors (under some distance metric) and zero weights on further-away ones, and smaller $\gamma_0$ producing more adaptive and distributed weights extended even to some further-away neighbors.  This suggests that the estimation variance should depend on $\gamma_0$.   However, our upper bound is a conservative one holding for all values of $\gamma_0$.  It is also commonly believed that larger $k$ usually causes larger variance for the forest estimate because the end cells tend to contain smaller numbers of observations.  In Section~\ref{Sec3.32}, we illustrate the effect of $k$ on estimation variance under a simplified model setting. In general, it is a challenging and interesting future research topic to precisely characterize how variance depends on $\gamma_0$ and $k$. As a result, the fact that the upper bound in Theorem \ref{theorem1} decreases with $\gamma_0$ is a coincidence rather than reality; it is caused by the crude upper bounds for several parts in the bias-variance decomposition.    
	

Theorem \ref{theorem1} holds (nonuniformly) for each combination of $(k, \eta,\delta, c)$ satisfying $1\leq k\leq c\log_2 n$, $0<\eta<1/8$, $\delta\in (2\eta, 1/4)$ and $0<c<1/4$.
If we set $\eta = \frac{1}{8} - \epsilon$, $\delta = \frac{1}{4} - \epsilon$, $c = \frac{1}{8}$, and $k= \floor{\frac{1}{8}\log_{2}(n)}$ in Theorem~\ref{theorem1}, we obtain more informative convergence rate as shown in Corollary \ref{corollary2} below.  
	\begin{corollary}\label{corollary2}
		Assume that Conditions \ref{abc}--\ref{BO}  hold and let  $0<\epsilon<\frac{1}{8}$, $\alpha_{1}\ge 1$, $\alpha_{2} >1$, and $0< \gamma_{0}\le1$ be given. For all large $n$ and tree height $k=\floor{ \frac{1}{8}\log_{2}{n}}$,
		\begin{equation*}\begin{split} &\sup_{m(\boldsymbol{X})\in\textnormal{SID}(\alpha_{1})} \left[\mathbb{E} \Big(m(\boldsymbol{X}) - \mathbb{E}\Big(\widehat{m}_{ \widehat{T}}(\boldsymbol{\Theta}_{1:k}, \boldsymbol{X}, \mathcal{X}_{n} )\ \Big\vert \ \boldsymbol{X}, \mathcal{X}_{n} \Big) \Big)^{2} \right]\\
				&\qquad\le O\Big(n^{-\frac{1}{8} +  \epsilon} + \underbrace{(1-\gamma_{0}(\alpha_{1}\alpha_{2})^{-1})^{\floor{\frac{\log_{2}(n)}{8} }}}_{\textnormal{Main term of bias}} \Big)
				\le O\Big(n^{-\frac{1}{8} +  \epsilon} + n^{-\frac{\log_{2}(e)}{8}\times\frac{\gamma_{0}}{\alpha_{1}\alpha_{2}}}  \Big),
		\end{split}\end{equation*}
		where the term $n^{-1/8+\epsilon}$ above is a common  bound of  all terms on the RHS of \eqref{main4} and \eqref{main5} except for the ``Main term of bias.''
	\end{corollary}

	\begin{remark}\label{p}
		
		The feature dimensionality $p$ and tree height $k$ decide the number of all possible cells when growing trees.  In deriving the consistency rates in Theorem \ref{theorem1},  we have to account for the probabilistic deviations between the sample moments (e.g., means) conditional on all possible cells and their corresponding population counterparts.  This is easy to understand for the variance analysis.  For the bias analysis,  we also need to account for such deviations because of the use of sample CART-splits. Roughly speaking, the deviations between the sample and population moments on each cell can be quantified by Hoeffding's inequality.   But to achieve uniform control over all possible cells,  we need to restrict the number of cells which is upper-bounded by $2^{c\log_{2}{n}}\left(p (\lceil n^{1+\rho_{1}}\rceil + 1)\right)^{c\log_{2}{n}}$ for some $\rho_{1} > 0$ when height $k\le c\log_{2}{n}$; see \eqref{numberOf} in Supplementary Material for details.   The condition of $p=O(n^{K_{0}})$ is sufficient for restricting the number of all possible cells for our analyses.  It also shows that  both bias and variance depend implicitly on $p$ through $n$ in our upper bounds.  
	\end{remark}
	
	\begin{remark}\label{reason1}
		We explain  briefly how the third term in the upper bound in Theorem \ref{theorem1} results from our technical analysis. When a tree is grown up to level $k$, there are $2^{k}$ cells, which form a partition of the $p$-dimensional hypercube. Among these cells, if a cell $\boldsymbol{t}$ is too small such that $\mathbb{P}(\boldsymbol{X} \in \boldsymbol{t}) \le c_{n}$ with $c_n$ depending only on $n$,  there will not be enough observations in $\bm{t}$ with high probability,  and consequently we cannot use Hoeffding's inequality to control the probabilistic difference between $\mathbb{E}(m(\boldsymbol{X}) | \boldsymbol{X} \in \boldsymbol{t})$ and its sample counterpart. Since the total probability of all these small cells is less than $c_{n}2^{k}$, when $c_{n}$ and $k\le c\log_{2}{n}$ are sufficiently small,  we use  Conditions~\ref{ME} and \ref{BO} to establish upper bounds for the mean differences on all these small cells.  
		This is one of the reasons why we need the brute-force analysis method; for details, see Lemma~\ref{lemma1} in Section~\ref{Sec4}. 
		We note that this is also one of the reasons why we limit the tree height parameter $c<1/4$. In addition, another reason for the third term in the rates is due to the use of our high-dimensional estimation framework. Details on this can be found in Lemma~\ref{con3} in Section~\ref{Sec5}.	Whether it is possible and how to eliminate this term by a finer analysis is an interesting future research topic. 
	\end{remark}

	\subsection{Sharper convergence rates with binary features}\label{Sec3.32}
	
	In this section, we demonstrate that our bias-variance decomposition analysis technique can yield a sharper upper bound under some simplified model setting. We confine ourselves to Example~\ref{binary} below,  and assume the absence of row subsampling in this section for simplicity. 
	\begin{exmp}
		\label{binary}
		Assume that $X_{1}, \dots, X_{p}$ are independent and $\mathbb{P} (X_{j} = 1)= \mathbb{P} (X_{j} = 0) = \frac{1}{2}$ for all $j$, and
		that $m(\boldsymbol{X}) = \sum_{j=1}^{s^*} \beta\boldsymbol{1}_{X_{j} = 1}$ for some $|\beta|>0$ and $s^*\le p$. 
	\end{exmp}

Let us gain some intuitions about how binary features can greatly simplify the problem. Consider any end cell $\boldsymbol t_k$ and the branch $(\boldsymbol t_{0} \ldots, \boldsymbol t_{k})$ connecting it to the root cell $\boldsymbol{t}_0$. Along this branch,    once a coordinate $j$ gets a split with  $c\in (0,1]$, any further split on this coordinate results in zero decrease in population impurity. In addition, since each sample CART split is on some data point, the split can only be either $(j, 1)$ or  $(j, 0)$ for some $j$, with the latter $(j, 0)$ resulting in an empty daughter cell and zero population impurity decrease; thus, split $(j,0)$ can be safely removed from the consideration. In summary, the CART essentially tries to minimize \eqref{new.eq.004} only with regard to $j\in \{1,\cdots, p\}$ for finding a split of the form $(j,1)$. For these reasons, the problem is much simplified. 

Nevertheless, by definition, the sample CART could still lead to a split of form $(j,0)$ when the event $\#\{i:x_{ij} = 1\} = 0$ occurs for some $j$, making the analysis tedious. To avoid this, the CART in this section is redefined as follows: for a cell $\boldsymbol{t}$ and feature restriction $\Theta$, the split is $(\widehat{j}, 1)$ such that 
\begin{equation*}
		\begin{split} 
			\widehat{j} \coloneqq  \argmin_{j\in \Theta} \left[
			\sum_{i \in P_{L} } \left( \frac{\sum_{i \in  P_{L}} y_{i} }{  \# P_{L}  }    - y_{i}\right)^{2}  + \sum_{i \in P_{R} } \left( \frac{\sum_{i \in  P_{R}} y_{i} }{ \# P_{R} }  - y_{i}\right)^{2}  \right],
		\end{split}
	\end{equation*} 
where $P_{L} \coloneqq \{i : \ \boldsymbol{x}_{i}\in\boldsymbol{t}, \, x_{ij} < 1  \}$, $P_{R} \coloneqq \{i : \ \boldsymbol{x}_{i}\in\boldsymbol{t}, \, x_{ij} \geq 1  \}$. CART stops splitting when all available coordinates have been split. To have equal height for each tree branch, we may consider trivial splits that give empty daughter cells; see the proof of Proposition~\ref{proposition2} for details.

	\begin{proposition}\label{proposition2}
Consider Example~\ref{binary} and  i.i.d. observations from  \eqref{new.eq.001} with $|\varepsilon_{1}|\le M_{\varepsilon}$ for some $M_{\varepsilon}>0$ and $\mathbb{E}(\varepsilon_{1}) = 0$.  Let $0<\gamma_{0}\le 1$, $0 < \eta< 1$, and $\epsilon>0$ be given, and suppose $(\log_{e}p)^{2+\epsilon} = o(n^{1-\eta})$. Then,  \textnormal{1)} $m(\boldsymbol{X})\in  \textnormal{SID}(s^*)$, and \textnormal{2)} for all large $n$ and every $0\le k\le \eta\log_{2}(n)$,
		{\small\begin{equation}
			\begin{split}\label{proposition2.1}
				&\mathbb{E} \Big(m(\boldsymbol{X}) - \mathbb{E}\Big(\widehat{m}_{ \widehat{T}}(\boldsymbol{\Theta}_{1:k}, \boldsymbol{X}, \mathcal{X}_{n} )\ \Big\vert \ \boldsymbol{X}, \mathcal{X}_{n} \Big) \Big)^{2} \\
				& \le \underbrace{2(1-\gamma_{0}(s^*)^{-1})^k \textnormal{Var} (m(\boldsymbol{X}))}_{\textnormal{Squared bias}} + \underbrace{2( 3M_{0} + 2M_{\varepsilon})^2\frac{2^{k}\big(\log_{e}({\max\{n, p\}})\big)^{2+\epsilon}}{n}}_{\textnormal{Estimation variance}}  + \ o(n^{-1}),
			\end{split} 
		\end{equation}}
		
		\noindent where $M_{0} = \sup_{\boldsymbol{c}\in[0, 1]^p}|m(\boldsymbol{c})|$, and  $\widehat{m}_{ T } (\boldsymbol{\Theta}_{1:0} , \boldsymbol{X}, \mathcal{X}_{n} )$ is defined to be $n^{-1}\sum_{i=1}^{n} y_{i}$. Particularly, if $\gamma_{0} = 1$, then \eqref{proposition2.1} holds with the squared bias bound replaced with $\max\big\{(s^*-k)\frac{\beta^2}{2}, 0\big\}$.
	\end{proposition}

	Let us compare Proposition \ref{proposition2} with Theorem \ref{theorem1} (the continuous feature assumption in Theorem \ref{theorem1} can be easily relaxed to accommodate binary features). The upper bound here is free of uninteresting errors, and the estimation variance here depends on tree level $k$ more explicitly, thanks to the simpler model setting with binary features.  The squared bias term on the RHS of \eqref{proposition2.1} explicitly depends on the column subsampling parameter $\gamma_{0}$, tree level $k$,  and  sparsity parameter $s^*$ (note that $\alpha_{1}=s^*$ here). It is worth mentioning that the squared bias here does not depend on the sample size $n$ and $\alpha_{2}$, because under this simplified model setting we can show that sample CART approximates theoretical CART perfectly on a high probability event. Our upper-bound in Proposition \ref{proposition2} shows the explicit bias-variance trade-off with respect to tree level $k$, which is not present in Theorem \ref{theorem1}. The trade-off with respect to $\gamma_{0}$, however, is still not reflected in the over all upper bound because our estimation variance bound is universal for all $0 < \gamma_{0}\le 1$; such caveat is caused by different technical approaches used in our bias and variance analyses, where we globally control bias via SID but bound forests estimation variance via bounding individual tree estimation variance. See more detailed discussion at the end of the proof of Proposition~\ref{proposition2}. 
	
	The special upper-bound 
	when $\gamma_0=1$ corresponds to the regression tree model; in such case the optimal convergence rate (according to our Proposition \ref{proposition2}) is achieved when  $k=s^*$ (i.e., squared bias = 0)  and is of order $\frac{2^{s^*}}{n}(\log_{e}({\max\{n, p\}}))^{2+\epsilon}$. This optimal rate is faster than the best rate obtained by minimizing the RHS of \eqref{proposition2.1} with respect to $k$. This is because \eqref{proposition2.1} is proved by applying our general bias analysis technique developed for proving Theorem \ref{theorem1} and hence makes no use of the specific model structure in Example~\ref{binary}. Our optimal rate when $\gamma_0=1$ is consistent with those in Theorems 3.3 and 4.4 of~\cite{syrgkanis2020estimation}, up to a logarithmic factor. Their main results concern a more general model than the linear model in Example~\ref{binary}, but rely on the same binary features assumption and are confined to the case of $\gamma_{0}=1$.

	\subsection{Role of relevant features}\label{Sec3.4}
	
	In this section, we formally study the role of relevant features for SID. We show that if the regularity conditions and SID are assumed, then for some cells, only the splits along the relevant feature directions can reduce a sufficient amount of bias. To be precise,  we introduce a variant of SID with some $S_{0}\subset \{1,\dots, p\}$ below. 
	\begin{condi*}
		\label{P2}
		There exists some $\alpha_{1}\ge  1$  such that for each cell $\boldsymbol{t} = t_{1}\times \dots \times t_{p}$,
		\begin{equation*}
			\textnormal{Var}(m(\boldsymbol{X}) \ \vert \ \boldsymbol{X} \in \boldsymbol{t} ) \le \alpha_{1} \sup_{j\in S_{0}, c\in t_{j}} (II)_{\boldsymbol{t}, \boldsymbol{t}(j, c)}.
		\end{equation*}
	\end{condi*}	
	
	For simplicity, we refer to the above condition as SID2. The difference from SID is that the supremum in SID2  is taken only over the features in $S_{0}$. In what follows, we will see that when the regularity conditions on the underlying regression function and SID are assumed, SID2 holds only if $S_{0}$ includes all relevant features. We begin with a formal definition of relevant features.

	\begin{defn}\label{nonparametric_relevance}
		A feature $j$ is said to be relevant for regression function $m(\boldsymbol{X})$ if and only if there exists some constant $\iota>0$ such that
		\[\mathop{{}\mathbb{E}} (\textnormal{Var} (m(\boldsymbol{X}) \ \vert \ X_{s}, s \in \{1, \dots, p\} \backslash \{j\})) >\iota.\]
	\end{defn}
	
	In Theorem \ref{theorem2} below, we characterize the magnitude of the $\mathbb{L}^{2}$ loss when a relevant feature is left out during the model training. 
	\begin{theorem}
		\label{theorem2}
		Assume that Conditions \ref{ME}--\ref{BO} hold and some relevant feature $j$ is not involved in the random forests model training procedure. Then we have 
		\[\mathbb{E} \Big(m(\boldsymbol{X}) - \frac{1}{B}\sum_{a\in A}\mathbb{E}\Big(\widehat{m}_{\widehat{T}_{a}, a}(\boldsymbol{\Theta}_{1}, \dots, \boldsymbol{\Theta}_{k}, \boldsymbol{X}, \mathcal{X}_{n} )\ \Big\vert \ \boldsymbol{X}, \mathcal{X}_{n} \Big) \Big)^{2} \ge \iota. \]
	\end{theorem}

	By Theorem~\ref{theorem1}, the regularity conditions and SID2 are sufficient for high-dimensional random forests consistency using only features in set $S_{0}$. Furthermore, if SID2 holds with some $S_{0}$, it holds with each $S_{1}$ such that $S_{0}\subset S_{1}$. Thus, the result in Theorem~\ref{theorem2} suggests that with the regularity conditions and SID assumed, for SID2 to hold with some $S_{0}$, all relevant features must be included in $S_{0}$. Otherwise, let us assume that $j^{*}\not\in S_{0}$ is the index of some relevant feature. From previous discussions, we see that SID holds and that SID2 holds with $S_{1} = \{1, \dots, p\}\backslash \{j^{*}\}$, which imply random forests consistency with or without the relevant feature $j^{*}$ and then lead to a contradiction to Theorem~\ref{theorem2}.  
	
	The non-inclusion assumption of a relevant feature in Theorem \ref{theorem2} is a convenient way of assuming that a relevant feature never gets split in random forests training (which is a random event depending on the training data complicatedly), and thus may be unnecessarily strong for our purpose. Nevertheless, our goal is to delieve the key message that random forests needs to split on every relevant feature direction to control the bias. 
	
	\begin{remark}
		Definition~\ref{nonparametric_relevance} provides a natural measurement of feature importance. Alternative definitions have been considered in the literature. For example, for nonparametric feature screening, \cite{Fan2011} assumed that $\mathbb{E}(m(\boldsymbol{X})) = 0$ and for some constant $\iota > 0$, it holds that for each relevant feature $j$,
		\begin{equation}\label{fan1}\textnormal{Var}(m(\boldsymbol{X})) - \mathbb{E} \Big( \textnormal{Var}(m(\boldsymbol{X})\ \vert \ X_{j}) \Big) \ge \iota. \end{equation}		
		The difference is that Definition~\ref{nonparametric_relevance} measures the conditional importance of each feature given all other features, while (\ref{fan1}) measures the marginal importance of features. 
	\end{remark}
	
	\subsection{Related works}\label{Sec3.5}
	We outline the difference between our random forests estimate and standard random forests software packages, and provide a detailed comparison of our consistency results with the existing results from the recent literature. Our setting differs from standard random forests software packages in the number of trees grown and the height of trees. In practice, random forests packages first randomly draw a set of subsamples $a$ with $\#a = \ceil{bn}$ (two subsampling modes are available; see~\cite{rf} for details) and available columns for splitting. Then these packages follow (\ref{new.eq.004}) to split cells and they stop splitting a cell if and only if the cell contains one observation. By default, these packages grow $500$ such independent  trees.  As for our work, for each $l \in \{1,\dots, B\}$ with an arbitrary integer $B>0$, $a_{l}$ contains $\ceil{bn}$ distinct indices  in $\{1, \dots, n\}$; these indices can be chosen in any way independent of the training sample. Then we grow a forest with all possible trees defined in Section~\ref{Sec2.1} for each $a_{l}$. Besides, we consider trees of height at most $c\log_{2}{n}$ for some possibly small $c > 0$ and our sample trees defined in Section~\ref{Sec2.2} continue to grow cells for a cell with one observation.

	Next, we compare our consistency results in Section~\ref{Sec3.2} to the existing ones from the recent literature. For easier comparison, let us focus on the case of $k=c\log_{2}{n}$ and drop the uninteresting error (i.e., the third term of $(\ceil{bn})^{-\delta + c} $) from the upper bounds in Theorem \ref{theorem1}.  With such convention,  noting that  $(1 - \gamma_{0}(\alpha_{1}\alpha_{2})^{-1})^{c\log_{2}{n} }  \approx n^{-c\gamma_{0} (\alpha_{1} \alpha_{2})^{-1}}$,  our rate of convergence becomes $n^{-\frac{c\gamma_{0}}{\alpha_{1}\alpha_{2}}} + n^{-\eta}$. Table~\ref{table:rate} summarizes our rate of convergence and the ones for two modified versions of the random forests algorithm,  the centered random forests~\cite{klusowski2019sharp} and Mondrian random forests~\cite{mourtada2018minimax}, where $s$ represents the number of informative features and $\beta > 0$ denotes the exponent for the H\"{o}lder continuity condition.  As mentioned in the Introduction, these modified versions of the random forests algorithm use splitting methods that are independent of the response in the training sample, which is a departure from the original version of the random forests algorithm proposed 
	in \cite{Breiman2001}.

	We see that both Theorem~\ref{theorem1} and~\cite{klusowski2019sharp} have taken into account the sparsity parameter in a similar fashion, 
	whereas~\cite{mourtada2018minimax} did not consider the sparsity.  The result in \cite{mourtada2018minimax} achieved the minimax rate under a class of H\"{o}lder continuous functions with parameter $\beta$; their rate depends on dimensionality $p$ nontrivially  and becomes uninformative for a large $p$.  Our consistency result is the only result that allows for the original random forests algorithm, growing sparsity parameter, and growing ambient dimensionality so far. Furthermore, our rates consider explicitly the effect of column subsampling for random forests and hence $\gamma_{0}$ appears in the rates. Such differences make our consistency result unique and useful for understanding the original random forests algorithm.  However,  we do acknowledge that the rate of convergence given in Theorem~\ref{theorem1} is 
	not optimal due to the technical difficulties discussed in Section \ref{Sec3.2}.
	
	\begin{table}[h]
		\centering	
		\caption{Comparison of consistency rates}
		\begin{tabular}{ | m{2.8cm} | m{3.7cm}|m{3cm}|m{4cm}|} 
			\hline
			& Rate of convergence & Growing sparsity parameter &  Explicit dependence on dimensionality $p$ \\ 
			\hline 
			Our Theorem~\ref{theorem1} & $ \underbrace{n^{-\frac{c\gamma_{0}}{\alpha_{1}\alpha_{2}}} + n^{-\eta}}_{\textnormal{Squared bias}} + \underbrace{n^{-\eta}}_{\textnormal{Variance}}\Bigg.$ & Yes & No\\ 
			\hline
			Centered RF~\cite{klusowski2019sharp} & $(n(\sqrt{\log_{2}{n}})^{s-1} )^{-\frac{1}{s\log{2} + 1}}\Bigg.$ & No & No \\ [3ex]
			\hline
			Mondrian RF~\cite{mourtada2018minimax} & $n^{-\frac{2\beta}{p + 2\beta}}\Bigg.$ & No & Yes \\ 
			\hline
		\end{tabular}
		\label{table:rate}
	\end{table}


	\section{Approximation theory} \label{Sec4}
	
	In this section, we aim to build the approximation theory of random forests in two steps. We first derive in Theorem~\ref{theorem3} the decreasing rates of approximation error resulting from approximating $m(\boldsymbol{X})$. We approximate $m(\boldsymbol{X})$ by a class of theoretical forests estimates, each of which is associated with a tree growing rule from a class of tree growing rules denoted by $\mathcal{T}$ (see below for its definition). Each growing rule in $\mathcal{T}$ is then associated with a deterministic splitting criterion comparable to the theoretical CART-split criterion in terms of impurity decrease. Then we verify in Theorem~\ref{theorem4} that on a high probability event, a version of the sample tree growing rule conditional on the observed sample is an instance of $\mathcal{T}$. In other words, we will show that the sample CART-splits are comparable to the theoretical CART-splits in terms of impurity decrease. We start with presenting in Lemma~\ref{lemma1} below the bound on the approximation error of (\ref{decom1}), which plays a key role in establishing the consistency rate in Theorem \ref{theorem1}.
	
	\begin{lemma}\label{lemma1}
		Assume that Conditions \ref{P}--\ref{BO} hold and let $0 < \gamma_{0}\le 1$, $\alpha_{2} > 1$, $0 < \eta < \frac{1}{8}$, $\delta$ with $2\eta<\delta <\frac{1}{4}$, and $  c >0$ be given. Then, for all large $n$ and each $1\le k \le   c\log_{2}{n} $,
		\[\mathbb{E}  \left(  m(\boldsymbol{X}) - m_{\widehat{T} }^{*} (\boldsymbol{\Theta}_{1:k} , \boldsymbol{X})\right)^{2}  \le 8M_{0}^{2}n^{-\delta} 2^{k} +2\alpha_{1}\alpha_{2}n^{-\eta} + 2M_{0}^{2}(1 - \gamma_{0}(\alpha_{1}\alpha_{2})^{-1})^{k}  + 2n^{-1}.\]
	\end{lemma}	
	Roughly speaking,	the main idea for proving the desired upper bound in Lemma~\ref{lemma1} is to find a class of deterministic tree growing rules  $\mathcal{T}$ such that given an event $\boldsymbol{U}_n$ of asymptotic probability one, a slightly modified version of the sample rule $\hat{T}$ is an instance of $\mathcal{T}$.  Hence,  we can obtain that 
	\begin{equation}
		\begin{split}\label{decom3}
			& \mathbb{E}  [ ( m(\boldsymbol{X}) - m_{\widehat{T} }^{*} (\boldsymbol{\Theta}_{1:k} , \boldsymbol{X}) )^{2}   \boldsymbol{1}_{ \boldsymbol{U}_{n}} |  \mathcal{X}_{n} ] \lesssim \sup_{T \in \mathcal{T}} \mathbb{E} [ (  m(\boldsymbol{X}) - m_{T }^{*} (\boldsymbol{\Theta}_{1:k} , \boldsymbol{X}))^{2} \boldsymbol{1}_{ \boldsymbol{U}_{n}} |  \mathcal{X}_{n}  ]\\
			& \leq \sup_{T \in \mathcal{T}} \mathbb{E} (  m(\boldsymbol{X}) - m_{T }^{*} (\boldsymbol{\Theta}_{1:k} , \boldsymbol{X}))^{2},\end{split}
	\end{equation}
	where the expectation is with respect to $\boldsymbol{\Theta}_{1:k} $ and $\boldsymbol{X}$,  and we use the notation $\lesssim$ in the first step above to emphasize that a slightly modified version of 
	$\widehat T$ is involved in the rigorous derivation. 
	The last step holds because $\mathcal{T}$ consists of growing rules associated with deterministic splitting criteria.  
	With the above inequalities,  it remains to bound the very last term in (\ref{decom3}). Then, since $\mathbb{P}(\boldsymbol{U}_n^c)$ is sufficiently small for all large $n$, we obtain the desired result in Lemma~\ref{lemma1}. We will provide the definition of $\mathcal{T}$ in Section~\ref{Sec4.1} and discuss the slightly modified version of $\widehat T$ and how to bound the RHS of \eqref{decom3} in Section~\ref{Sec4.2}.

	\subsection{Main results} \label{Sec4.1}	
	
	Given parameters $\varepsilon$, $\alpha_{2}$, and $k$, all tree growing rules satisfying Condition~\ref{tree} below form a class of tree growing rules (i.e., $\mathcal{T}$), each of which is associated with an abstract deterministic splitting criterion. 
	\begin{condi} 
		\label{tree} 
		For the tree growing rule $T$, there exist some $\varepsilon \ge 0$, $ \alpha_{2}\ge 1$, and positive integer $k$ such that for any sets of available features $\Theta_{1}, \dots, \Theta_{k}$, 
		each $(\boldsymbol{t}_{1}, \dots, \boldsymbol{t}_{k}) \in T(\Theta_{1}, \dots, \Theta_{k})$,
		and each $1 \leq l\le k$,
		\begin{enumerate}
			\item[1)] if $(II)_{\boldsymbol{t}_{l-1}, \boldsymbol{t}_{l}} \le \varepsilon, \textnormal{then } \sup_{ (j \in \Theta_{l}, c)}(II)_{\boldsymbol{t}_{l-1}, \boldsymbol{t}_{l-1}(j,c )} \le \alpha_{2}\varepsilon$;
			\item[2)] if $(II)_{\boldsymbol{t}_{l-1}, \boldsymbol{t}_{l}} > \varepsilon$, \textnormal{then} $\sup_{(j\in\Theta_{l}, c)} (II)_{\boldsymbol{t}_{l-1}, \boldsymbol{t}_{l-1}(j, c)} \le \alpha_{2}(II)_{\boldsymbol{t}_{l-1}, \boldsymbol{t}_{l}}$,
		\end{enumerate}
		\noindent
		where we do not specify to which $\Theta_{l, 1}, \dots, \Theta_{l, 2^{l-1}}$ in $\Theta_{l}$ feature $j$ belongs in the supremum for simplicity.
	\end{condi}
	When $\alpha_{2} = 1$ and $\varepsilon = 0$, for each integer $k>0$, there is only one tree growing rule satisfying Condition~\ref{tree}, that is, the one associated with the
	theoretical CART-split criterion. The parameters $\alpha_2 > 1$ and $\varepsilon > 0$  are introduced to account for the statistical estimation error when using the sample CART-splits to estimate the theoretical CART-splits. We will show in Theorem~\ref{theorem4} in Section~\ref{Sec4.2} that with high probability, a slightly modified version of the sample tree growing rule satisfies Condition \ref{tree}. 

	\begin{theorem}\label{theorem3} 
		Assume that Condition \ref{P} holds with $\alpha_{1} \ge 1 $, $\textnormal{Var}(m(\boldsymbol{X})) < \infty$, and the tree growing rule $T$ satisfies Condition~\ref{tree} with some integer $k > 0$, $\varepsilon \ge 0$, and $\alpha_{2} \ge 1$. Then for each $0 < \gamma_{0}\le 1$, we have 
		\begin{equation*}\begin{split} 
				& \mathbb{E}\Big(m(\boldsymbol{X}) - {m}_{T}^{*}\left(\boldsymbol{\Theta}_{1}, \dots, \boldsymbol{\Theta}_{k}, \boldsymbol{X}\right) \Big)^{2} \le \alpha_{1} \alpha_{2} \varepsilon+ \left(1  -  \gamma_{0}(\alpha_{1}\alpha_{2})^{-1}  \right)^{k}\textnormal{Var}(m(\boldsymbol{X})). \end{split}\end{equation*}
	\end{theorem}
	\begin{remark}\label{good1}
		If we set $\alpha_{2} = 1$ and $\varepsilon = 0$, we see that Theorem~\ref{theorem3} applies only to the growing rule associated with
		the theoretical CART-split criterion. In this sense, Condition~\ref{tree} can be understood as a way for extending the applicability of Theorem~\ref{theorem3} to a wider class of tree growing rules in $\mathcal T$.  
		Indeed, we set $\varepsilon$ as $n^{-\eta}$, which accounts for the estimation error due to sample CART-splits, when proving Theorem \ref{theorem1} by exploiting Theorem \ref{theorem3}.
	\end{remark}

	Condition~\ref{tree} enables us to apply Theorem~\ref{theorem3} to an abstract tree growing rule obtained by slightly modifying the tree growing rule associated with the sample CART splitting criterion.  In Section~\ref{Sec4.2}, we will discuss this abstract tree growing rule in detail.  The exponential upper bound $\left(1  -  \gamma_{0}(\alpha_{1}\alpha_{2})^{-1}  \right)^{k}$ in Theorem \ref{theorem3} is obtained by a recursive analysis.  To appreciate the recursive analysis, let us consider the special case of $\varepsilon=0$. Then we can show that
	\begin{equation}\begin{split}
			\label{t.new.2}
			& \mathbb{E}\left(m(\boldsymbol{X}) - {m}_{T}^{*}\left(\boldsymbol{\Theta}_{1:k}, \boldsymbol{X}\right) \right)^{2}  \le \left(1  - \frac{ \gamma_{0}}{ \alpha_{1}\alpha_{2}} \right) \mathbb{E}\left(m(\boldsymbol{X}) - {m}_{T}^{*}\left(\boldsymbol{\Theta}_{1:k-1}, \boldsymbol{X}\right) \right)^{2}.
		\end{split}
	\end{equation}	
	In (\ref{t.new.2}), we clearly see the recursive structure for controlling the approximation error. See Section~\ref{SecA.4} of Supplementary Material for more details of the recursive inequality.  
	
	Theorem~\ref{theorem3} is the key result that makes our approximation error analysis unique and practical.  It differs sharply from the existing literature in the sense that our technical analysis is more specific to random forests and does not rely on general methods of data-independent partition such as Stone's theorem~\cite{Stone1977} or data-dependent partition such as 
	\cite{nobel1996histogram}. This is in contrast to most existing works~\cite{Biau2015, Wager2018, Biau2012, Biau2008}.

	\subsection{Sample tree growing rule}\label{Sec4.2}
	
	In Theorem~\ref{theorem4}, we will analyze a version of the sample tree growing rule defined below and show that conditional on the observed sample, on a high probability $\mathcal{X}_{n}$-measurable event, this rule satisfies Condition~\ref{tree} with $\varepsilon = \varepsilon_{n}$ decreasing to zero. Let sets of available features $\Theta_{1:k}$ be given for some positive integer $k$. Consider the following procedure of modifying a subtree with some $\zeta > 0$. For each $(\boldsymbol{t}_{1},\dots, \boldsymbol{t}_{k})$ $\in \widehat{T}(\Theta_{1:k})$ with $\mathop{{}\mathbb{P}}(\boldsymbol{X} \in \boldsymbol{t}_{k-1}) < \zeta$, let us fix  $l_{0} \coloneqq \min \{ l -1 : \mathbb{P}(\boldsymbol{X}\in\boldsymbol{t}_{l -1})<\zeta, 1 \le l\le k\}$. Then we trim the descendant cells of $\boldsymbol{t}_{l_{0}}$ off from $\widehat{T}(\Theta_{1:k})$ and grow new descendant cells back in such a way that each new descendant cell $\boldsymbol{t}^{'}$ and its parent cell $\boldsymbol{t}$ satisfy $\sup_{(j\in\Theta, c)} (II)_{\boldsymbol{t}, \boldsymbol{t}(j, c)} = (II)_{\boldsymbol{t}, \boldsymbol{t}^{'}}$, where sets of available features are those in $\Theta_{1:k}$ and we do not specify them in the supremum. That is, each new descendant cell of $\boldsymbol{t}_{l_{0}}$ is grown according to the theoretical CART-split criterion; a graphical illustration is given in Figure~\ref{fig:cut1}.  
	\begin{figure}[th]
		\subfloat[Trim the subtree after the specified cell.]{\includegraphics[width=6cm]{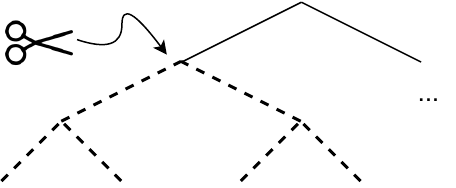}}
		\hspace{2cm}
		\subfloat[Split the cells $A$, $B$, and $C$ by the theoretical CART-split criterion with the corresponding sets of available features.]{\includegraphics[width=6cm]{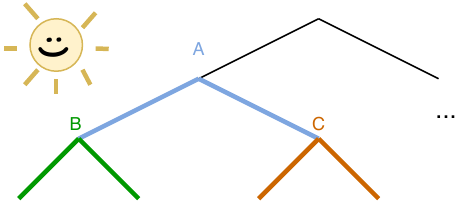}}
		\centering
		\caption{Trim a subtree and grow it back.}
		\label{fig:cut1}
	\end{figure}

	We next define the modified version of the sample tree as follows. For each cell path in $\hat{T}(\Theta_{1:k})$, there is at most one cell $\boldsymbol{t}_{l_{0}}$ as defined previously. We collect these cells in a subset, perform the previously described  procedure accordingly, and obtain the modified sample tree. We denote such a new tree as $\widehat{T}_{\zeta}(\Theta_{1:k})$ and refer to it as the semi-sample tree growing rule.
	
	\begin{theorem} \label{theorem4}
		Assume that Conditions \ref{abc}--\ref{BO} hold and let $\alpha_{2} > 1$, $0 < \eta < \frac{1}{8}$, $c > 0$, and $\delta$ with $2\eta < \delta < \frac{1}{4}$  be given. Then there exists an $\mathcal{X}_{n}$-measurable event $\boldsymbol{U}_{n}$ with $\mathop{{}\mathbb{P}} (\boldsymbol{U}_{n}^{c}) = o( n^{-1})$ such that conditional on $\mathcal{X}_{n}$, on event $\boldsymbol{U}_{n}$ and for all large $n$, $ \widehat{T}_{\zeta}$ with $\zeta=n^{-\delta}$ satisfies Condition~\ref{tree} with  $ k = \floor{c\log_{2}{n}}$, $\varepsilon = n^{-\eta}$, and  $\alpha_{2}$. 
	\end{theorem}
	\begin{remark}
		Since Theorem~\ref{theorem4} above is a result for the semi-sample tree growing rule instead of the sample tree growing rule, the tree height parameter $c>0$  is an arbitrary constant in Theorem~\ref{theorem4}. To use the result of Theorem~\ref{theorem4} for Lemma~\ref{lemma1}, we need to control the $\mathbb{L}^{2}$ difference between the population version random forests estimates using these two rules, which is the reason for the first term in Lemma~\ref{lemma1}. Such term is bounded by $O(n^{-\delta + c})$ and as a result, the value of $c$ is required to be limited when applying Lemma~\ref{lemma1} to obtain Theorem~\ref{theorem1}. There is a similar remark for Lemma~\ref{lemma3} in Section~\ref{Sec5}.
	\end{remark}
	\begin{remark}\label{r4}
		Thanks to Theorem \ref{theorem4}, we can apply Theorem \ref{theorem3} to $\hat{T}_{\zeta}$ with $\zeta=n^{-\delta}$ and obtain the following inequality 
		\begin{equation*}
			\mathbb{E}\left[\left(  m(\boldsymbol{X}) - m_{ \widehat{T}_{\zeta}}^{*} (\boldsymbol{\Theta}_{1:k} , \boldsymbol{X})\right)^{2}  \ \Big\vert \ \mathcal{X}_{n} \right] \boldsymbol{1}_{ \boldsymbol{U}_{n}} \le \alpha_{1}\alpha_{2}n^{-\eta} + (1 - \gamma_{0}(\alpha_{1}\alpha_{2})^{-1})^{k} \textnormal{Var}(m(\boldsymbol{X})),
		\end{equation*}
		where the estimate $m_{ \widehat{T}_{\zeta }}^{*}$ is defined similarly as $m_{ \widehat{T}}^{*}$. For more details, see the proof of Lemma~\ref{lemma1} in Section~\ref{SecC.2} in Supplementary Material. As a result, we do not need Condition~\ref{tree} in Lemma~\ref{lemma1} since it is a consequence of Theorem~\ref{theorem4} that the sample tree growing rule is an instance of Condition~\ref{tree} in a probabilistic sense. This is one of our main contributions to proving such a result in Theorem~\ref{theorem4} instead of assuming that the sample tree growing rule satisfies Condition~\ref{tree}.
	\end{remark}
	
	\section{Upper bounds for statistical estimation error}
	\label{Sec5}
	
	In this section, we aim to develop a general high-dimensional estimation foundation for analyzing random forests consistency and use it to derive the convergence rate for the estimation variance (i.e., the second term in (\ref{decom1})) in the lemma below.
	
	\begin{lemma}\label{con3}
		Assume that Conditions~\ref{abc}--\ref{BO} hold and let $0<  \eta < 1/4$, $0 < c< 1/4$, and $\nu >0$ be given. Then there exists some constant $C>0$ such that for all large $n$ and each $1\le k \le c\log_{2}{n}$,
		\begin{equation}\label{estimation1}\mathop{{}\mathbb{E}}      \Big(m_{ \widehat{T}}^{*}(\boldsymbol{\Theta}_{1}, \dots, \boldsymbol{\Theta}_{k} , \boldsymbol{X}) - \widehat{m}_{\widehat{T}}(\boldsymbol{\Theta}_{1}, \dots, \boldsymbol{\Theta}_{k} , \boldsymbol{X}, \mathcal{X}_{n} ) \Big)^{2} \le   n^{-\eta} + C2^{k}n^{-\frac{1}{2} + \nu}. \end{equation}
	\end{lemma}
	
	Let us gain some insights into the challenge associated with \eqref{estimation1}. To analyze the estimation variance, essentially we have to control the probabilistic difference between every conditional mean and average on each cell grown by the sample tree growing rule. This is a challenging task because these cells have random boundaries. A naive consideration of all possible cells in $[0, 1]^{p}$ takes every cell into account but results in an uncountably infinite set, which precludes the application of standard concentration inequalities such as Hoeffding's inequality. In what follows, we describe in detail our approach to overcoming such a challenge.

	Instead of considering estimation of conditional means on all possible cells with random boundaries in $[0, 1]^{p}$ directly, we estimate only conditional means on each of a set of deterministic cells from \emph{a predetermined grid}, which we will formally define next. The grid contains many cells such that for an arbitrary cell $\boldsymbol{t}$, there is a cell $\boldsymbol{t}^{\#}$ on the grid being so close to $\boldsymbol{t}$ that the values of their theoretical conditional means, $\mathbb{E}(m(\boldsymbol{X})| \boldsymbol{X} \in \boldsymbol{t})$ and $\mathbb{E}(m(\boldsymbol{X})| \boldsymbol{X} \in \boldsymbol{t}^{\#})$, are very close, and the values of their empirical conditional means, 
	$\frac{\sum_{\boldsymbol{x}_{i} \in \boldsymbol{t}} y_{i}}{ \#\{ i: \boldsymbol{x}_{i} \in \boldsymbol{t}\}}$  and $\frac{\sum_{\boldsymbol{x}_{i} \in \boldsymbol{t}^{\#}} y_{i}}{ \#\{ i: \boldsymbol{x}_{i} \in \boldsymbol{t}^{\#}\}}$,  are also close.  Since the number of cells on the grid is not too large,  we can show that the theoretical conditional mean $\mathbb{E}(m(\boldsymbol{X})| \boldsymbol{X} \in \boldsymbol{t}^{\#})$ and its empirical counterpart  $\frac{\sum_{\boldsymbol{x}_{i} \in \boldsymbol{t}^{\#}} y_{i}}{ \#\{ i: \boldsymbol{x}_{i} \in \boldsymbol{t}^{\#}\}}$ are uniformly close using Hoeffding's inequality. Combining all these results can yield Lemma \ref{con3}.

	The grid mentioned before is defined as follows. Let $ \rho_{1} $ be a given positive constant and consider a sequence of $b_{i} = \frac{i}{\lceil n^{1+\rho_{1}}\rceil}$ with $0 \leq  i  \leq \lceil n^{1+\rho_{1}}\rceil$. We construct hyperplanes such that along each $j$th coordinate, each point $b_{i}$ is crossed by one and only one of the hyperplanes and this hyperplane is perpendicular to the $j$th axis. The result is exactly $(\lceil n^{1+\rho_{1}}\rceil + 1)^{p}$ distinct hyperplanes and each boundary of the root cell $[0, 1]^{p}$ is also one of these hyperplanes.  Naturally these hyperplanes form a grid on $[0, 1]^{p}$ and we refer to each of these hyperplanes as a grid hyperplane or a grid line.  For a cell $\boldsymbol{t}$,  we define the cell $ \boldsymbol{t}^{\#}$ by moving all boundaries  of $\boldsymbol{t}$ to the corresponding nearest grid lines; see Figure \ref{fig:rf} for a graphical illustration. For a tree growing rule $T$, we define $T^{\#}$ such that for each $\Theta_{1:k}$, $(\boldsymbol{t}_{1}^{\#}, \dots, \boldsymbol{t}_{k}^{\#}) \in T^{\#}(\Theta_{1:k})$ if $(\boldsymbol{t}_{1}, \dots, \boldsymbol{t}_{k}) \in T(\Theta_{1:k})$. Let us observe two important properties of the sharp notation. First, for each cell $\boldsymbol{t}$,
	if $\boldsymbol{t}^{'}$ and $\boldsymbol{t}^{''}$ are its daughter  cells, then $(\boldsymbol{t}^{'})^{\#}$ and $(\boldsymbol{t}^{''})^{\#}$ are daughter cells of $\boldsymbol{t}^{\#}$. Second, for each integer $k > 0$, the collection of end cells $\boldsymbol{t}_{k}^{\#}$ at level $k$ is a partition of $[0, 1]^{p}$. As a result, $T^{\#}$ can be understood as a tree growing rule (induced by $T$). The same definition of the sharp notation goes for the sample tree growing rule.

	\begin{figure}[t]
		\includegraphics[width=15cm]{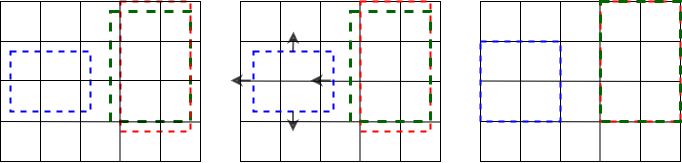}
		\centering
		\caption{From the left panel to the right, we move the original boundaries  to the nearest grid lines.}
		\label{fig:rf}
	\end{figure}
	
	We now demonstrate how to use the grid for obtaining the result in Lemma~\ref{con3}. To control the $\mathbb{L}^{2}$ loss between $m^{*}_{\hat{T}}$ and $\hat{m}_{\hat{T}}$ as in (\ref{estimation1}), we decompose the squared loss into three terms as 
	\begin{equation}
		\begin{split}\label{T4_11}
			& \underbrace{\textnormal{ $\mathbb{L}^{2}$ loss between } m^{*}_{\hat{T}}  \textnormal{ and  } m^{*}_{\hat{T}^{\#}}}_{\textnormal{Controlled by (\ref{consistency2})}} \leftarrow\rightarrow 
			\underbrace{\textnormal{ $\mathbb{L}^{2}$ loss between }  m^{*}_{\hat{T}^{\#}}  \textnormal{ and  } \hat{m}_{\hat{T}^{\#}}}_{\textnormal{Controlled by Theorem~\ref{theorem5}}} \\
			& \qquad \leftarrow\rightarrow 
			\underbrace{\textnormal{ $\mathbb{L}^{2}$ loss between }  \hat{m}_{\hat{T}^{\#}}  \textnormal{ and  } \hat{m}_{\hat{T}}}_{\textnormal{Controlled by (\ref{consistency1})}},
	\end{split}\end{equation}
	and establish bounds for each of them in Theorem~\ref{theorem5}, (\ref{consistency2}), and (\ref{consistency1}) below, respectively, using the grid.
	In particular, in Theorem~\ref{theorem5} we will see that the grid helps us bound the LHS of (\ref{new.eq.015}) below uniformly over all possible tree growing rules $T$'s. This approach provides a solution to a fundamental estimation problem in proving random forests consistency that involves infinitely many possible $T$'s. By (\ref{eq.new.2}) and (\ref{eq.new.1}), for any $T$ and $\Theta_{1:k}$, we can deduce that on $\cap_{i = 1}^{k}\{\boldsymbol{\Theta}_{i} = \Theta_{i}\}$,
	\begin{equation}
		\begin{split}
			\label{new.eq.015}
			&  \mathop{{}\mathbb{E}}    \left[     \Big(m_{ T^{\#}}^{*}(\boldsymbol{\Theta}_{1:k}, \boldsymbol{X}) - \widehat{m}_{ T^{\#}}(\boldsymbol{\Theta}_{1:k}, \boldsymbol{X}, \mathcal{X}_{n} ) \Big)^{2} \ \Big\vert \ \boldsymbol{\Theta}_{1:k}= \Theta_{1:k}, \mathcal{X}_{n} \right] \\
			& = \sum_{ (\boldsymbol{t}_{1}, \dots, \boldsymbol{t}_{k}) \in T^{\#} (\Theta_{1:k}) } \mathbb{P}(\boldsymbol{X}\in\boldsymbol{t}_{k})  \left(  \mathbb{E}(m(\boldsymbol{X}) \ \vert \ \boldsymbol{X} \in \boldsymbol{t}_{k}) - \frac{\sum_{i\in \{i : \boldsymbol{x}_{i} \in \boldsymbol{t}_{k }\}} y_{i}} {\#\{i : \boldsymbol{x}_{i} \in \boldsymbol{t}_{k }\}}  \right)^{2}.
	\end{split}\end{equation}
	From the expression on the RHS of (\ref{new.eq.015}) above, we see that with the grid, we only need to deal with estimation of conditional means on  each cell in the set 
	\[\{\boldsymbol{t}^{\#} : \boldsymbol{t} \in \{\textnormal{All the end cells grown by all possible growing rules given } \Theta_{1}, \dots, \Theta_{k}\} \}.\]
	Such a set contains only finitely many distinct cells given $k$, $n$, and $p$. This set can be further enlarged to consider all possible growing rules and sets of available features (i.e., the collection of end cells grown by all possible $T^{\#}(\Theta_{1:k})$'s). According to (\ref{numberOf}) in Supplementary Material, the number of distinct cells of the enlarged set is bounded by $2^{k} \left(p (\lceil n^{1+\rho_{1}}\rceil + 1)\right)^{k}$.  
	
	\begin{theorem}\label{theorem5} 
		Assume that Conditions \ref{ME}--\ref{BO} hold and let $0 < \eta < \frac{1}{4}$ and $0 < c < \frac{1}{4}$ be given. Then for all large $n$ and each $1\le k \le c\log_{2}{n}$, we have 
		\[\mathbb{E} \left\{\sup_{T}  \mathop{{}\mathbb{E}}    \left[     \Big(m_{ T^{\#}}^{*}(\boldsymbol{\Theta}_{1:k}, \boldsymbol{X}) - \widehat{m}_{T^{\#}}(\boldsymbol{\Theta}_{1:k}, \boldsymbol{X}, \mathcal{X}_{n} ) \Big)^{2} \ \Big\vert \ \boldsymbol{\Theta}_{1:k}, \mathcal{X}_{n} \right] \right\} \le n^{-\eta},\]
		where the supremum is over all possible tree growing rules. Note that due to the use of the grid, the supremum can be simplified to a max over a finitely many tree growing rules.
	\end{theorem}

	\begin{lemma} \label{lemma3}
		Assume that Conditions~\ref{abc}--\ref{BO} hold and let $1/2< \Delta < 1$ and $c > 0$ be given. Then there exists some constant $C > 0$ such that for all large $n$ and each $1\le k \le c\log_{2}{n}$,
		\begin{equation}\label{consistency2}
			\mathbb{E}\left(m_{ \widehat{T}^{\#}}^{*} (\boldsymbol{\Theta}_{1:k}, \boldsymbol{X}) - m_{\widehat{T}}^{*} (\boldsymbol{\Theta}_{1:k}, \boldsymbol{X})\right)^{2} \le C2^{k}n^{\Delta -1} \end{equation}
		and
		\begin{equation}\label{consistency1}
			\mathbb{E}\left(\widehat{m}_{\widehat{T}^{\#}} (\boldsymbol{\Theta}_{1:k}, \boldsymbol{X}, \mathcal{X}_{n}) - \widehat{m}_{\widehat{T}} (\boldsymbol{\Theta}_{1:k}, \boldsymbol{X}, \mathcal{X}_{n})\right)^{2} \le C2^{k}n^{\Delta -1} .\end{equation}
	\end{lemma}

	\section{Discussions} \label{Sec6}
	In this paper, we have investigated the asymptotic properties of the widely used method of random forests in a high-dimensional feature space. In contrast to existing theoretical results, our asymptotic analysis has considered the original version of the random forests algorithm in a general high-dimensional nonparametric regression setting in which the covariates can be dependent, and the underlying true regression function can be discontinuous.  Explicit rates of convergence have been established for the high-dimensional consistency of random forests, justifying its theoretical advantages as a flexible nonparametric learning tool in high dimensions. We provide a new technical analysis for polynomially growing dimensionality through natural regularity conditions that characterize the intrinsic learning behavior of random forests at the population level.
	
	Our technical analysis has been based on the bias-variance decomposition  of random forests prediction loss, where we have analyzed the bias and variance  separately. 
	Despite some limitations (see Section~\ref{Sec3.2}),  the current bias-variance analysis has revealed some great details on how random forests bias depends on the sample size, tree height, and the column subsampling parameter $\gamma_{0}$.  Our current results apply only to random forests with non-fully-grown trees.  It would be interesting to extend our results to the case of fully-grown trees. When the scale of the problem becomes very large in terms of the growth of dimensionality (e.g., of nonpolynomial order of sample size), it would be appealing to incorporate the ideas of two-scale learning and inference with feature screening \cite{FanLv2008,FanFan2008,FanLv2018}. In addition, it is important to provide the asymptotic distributions for different tasks of statistical inference with random forests. These problems are beyond the scope of the current paper and are interesting topics for future research.


	
	\bibliographystyle{imsart-number} 
	\bibliography{references}       

\begin{thebibliography}{39}

\bibitem{athey2019generalized}
\begin{barticle}[author]
\bauthor{\bsnm{Athey},~\bfnm{Susan}\binits{S.}},
  \bauthor{\bsnm{Tibshirani},~\bfnm{Julie}\binits{J.}} \AND
  \bauthor{\bsnm{Wager},~\bfnm{Stefan}\binits{S.}}
(\byear{2019}).
\btitle{Generalized random forests}.
\bjournal{The Annals of Statistics}
\bvolume{47}
\bpages{1148--1178}.
\end{barticle}
\endbibitem

\bibitem{bai2005maxima}
\begin{barticle}[author]
\bauthor{\bsnm{Bai},~\bfnm{Zhi-Dong}\binits{Z.-D.}},
  \bauthor{\bsnm{Devroye},~\bfnm{Luc}\binits{L.}},
  \bauthor{\bsnm{Hwang},~\bfnm{Hsien-Kuei}\binits{H.-K.}} \AND
  \bauthor{\bsnm{Tsai},~\bfnm{Tsung-Hsi}\binits{T.-H.}}
(\byear{2005}).
\btitle{Maxima in hypercubes}.
\bjournal{Random Structures \& Algorithms}
\bvolume{27}
\bpages{290--309}.
\end{barticle}
\endbibitem

\bibitem{bennett1962probability}
\begin{barticle}[author]
\bauthor{\bsnm{Bennett},~\bfnm{George}\binits{G.}}
(\byear{1962}).
\btitle{Probability inequalities for the sum of independent random variables}.
\bjournal{Journal of the American Statistical Association}
\bvolume{57}
\bpages{33--45}.
\end{barticle}
\endbibitem

\bibitem{bernstein1924modification}
\begin{barticle}[author]
\bauthor{\bsnm{Bernstein},~\bfnm{S}\binits{S.}}
(\byear{1924}).
\btitle{Sur une modification de l’in{\'e}qualit{\'e} de Tchebichef}.
\bjournal{Annal. Sci. Inst. Sav. Ukr. Sect. Math. I}
\bpages{38--49}.
\end{barticle}
\endbibitem

\bibitem{Biau2016}
\begin{barticle}[author]
\bauthor{\bsnm{Biau},~\bfnm{G{\'e}rard}\binits{G.}} \AND
  \bauthor{\bsnm{Scornet},~\bfnm{Erwan}\binits{E.}}
(\byear{2016}).
\btitle{A random forest guided tour}.
\bjournal{Test}
\bvolume{25}
\bpages{197--227}.
\end{barticle}
\endbibitem

\bibitem{Biau2012}
\begin{barticle}[author]
\bauthor{\bsnm{Biau},~\bfnm{G{\~A}{\v{S}}rard}\binits{G.}}
(\byear{2012}).
\btitle{Analysis of a random forests model}.
\bjournal{Journal of Machine Learning Research}
\bvolume{13}
\bpages{1063--1095}.
\end{barticle}
\endbibitem

\bibitem{Biau2008}
\begin{barticle}[author]
\bauthor{\bsnm{Biau},~\bfnm{G{\~A}{\v{S}}rard}\binits{G.}},
  \bauthor{\bsnm{Devroye},~\bfnm{Luc}\binits{L.}} \AND
  \bauthor{\bsnm{Lugosi},~\bfnm{G{\~A}{\k{A}}bor}\binits{G.}}
(\byear{2008}).
\btitle{Consistency of random forests and other averaging classifiers}.
\bjournal{Journal of Machine Learning Research}
\bvolume{9}
\bpages{2015--2033}.
\end{barticle}
\endbibitem

\bibitem{breiman1996bagging}
\begin{barticle}[author]
\bauthor{\bsnm{Breiman},~\bfnm{Leo}\binits{L.}}
(\byear{1996}).
\btitle{Bagging predictors}.
\bjournal{Machine learning}
\bvolume{24}
\bpages{123--140}.
\end{barticle}
\endbibitem

\bibitem{Breiman2001}
\begin{barticle}[author]
\bauthor{\bsnm{Breiman},~\bfnm{Leo}\binits{L.}}
(\byear{2001}).
\btitle{Random forests}.
\bjournal{Machine Learning}
\bvolume{45}
\bpages{5--32}.
\end{barticle}
\endbibitem

\bibitem{Breiman2002}
\begin{barticle}[author]
\bauthor{\bsnm{Breiman},~\bfnm{Leo}\binits{L.}}
(\byear{2002}).
\btitle{Manual on setting up, using, and understanding random forests v3. 1}.
\bjournal{Statistics Department University of California Berkeley, CA, USA}
\bvolume{1}
\bpages{58}.
\end{barticle}
\endbibitem

\bibitem{diaz2006gene}
\begin{barticle}[author]
\bauthor{\bsnm{D{\'\i}az-Uriarte},~\bfnm{Ram{\'o}n}\binits{R.}} \AND
  \bauthor{\bsnm{De~Andres},~\bfnm{Sara~Alvarez}\binits{S.~A.}}
(\byear{2006}).
\btitle{Gene selection and classification of microarray data using random
  forest}.
\bjournal{BMC Bioinformatics}
\bvolume{7}
\bpages{3}.
\end{barticle}
\endbibitem

\bibitem{FanFan2008}
\begin{barticle}[author]
\bauthor{\bsnm{Fan},~\bfnm{J.}\binits{J.}} \AND
  \bauthor{\bsnm{Fan},~\bfnm{Y.}\binits{Y.}}
(\byear{2008}).
\btitle{High-dimensional classification using features annealed independence
  rules}.
\bjournal{The Annals of Statistics}
\bvolume{36}
\bpages{2605--2637}.
\end{barticle}
\endbibitem

\bibitem{Fan2011}
\begin{barticle}[author]
\bauthor{\bsnm{Fan},~\bfnm{Jianqing}\binits{J.}},
  \bauthor{\bsnm{Feng},~\bfnm{Yang}\binits{Y.}} \AND
  \bauthor{\bsnm{Song},~\bfnm{Rui}\binits{R.}}
(\byear{2011}).
\btitle{Nonparametric independence screening in sparse ultra-high-dimensional
  additive models}.
\bjournal{Journal of the American Statistical Association}
\bvolume{106}
\bpages{544--557}.
\end{barticle}
\endbibitem

\bibitem{FanLv2008}
\begin{barticle}[author]
\bauthor{\bsnm{Fan},~\bfnm{J.}\binits{J.}} \AND
  \bauthor{\bsnm{Lv},~\bfnm{J.}\binits{J.}}
(\byear{2008}).
\btitle{Sure independence screening for ultrahigh dimensional feature space
  (with discussion)}.
\bjournal{Journal of the Royal Statistical Society Series B}
\bvolume{70}
\bpages{849--911}.
\end{barticle}
\endbibitem

\bibitem{FanLv2018}
\begin{barticle}[author]
\bauthor{\bsnm{Fan},~\bfnm{J.}\binits{J.}} \AND
  \bauthor{\bsnm{Lv},~\bfnm{J.}\binits{J.}}
(\byear{2018}).
\btitle{Sure independence screening (invited review article)}.
\bjournal{Wiley StatsRef: Statistics Reference Online}
\bpages{1--8}.
\end{barticle}
\endbibitem

\bibitem{genuer2012variance}
\begin{barticle}[author]
\bauthor{\bsnm{Genuer},~\bfnm{Robin}\binits{R.}}
(\byear{2012}).
\btitle{Variance reduction in purely random forests}.
\bjournal{Journal of Nonparametric Statistics}
\bvolume{24}
\bpages{543--562}.
\end{barticle}
\endbibitem

\bibitem{gislason2006random}
\begin{barticle}[author]
\bauthor{\bsnm{Gislason},~\bfnm{Pall~Oskar}\binits{P.~O.}},
  \bauthor{\bsnm{Benediktsson},~\bfnm{Jon~Atli}\binits{J.~A.}} \AND
  \bauthor{\bsnm{Sveinsson},~\bfnm{Johannes~R}\binits{J.~R.}}
(\byear{2006}).
\btitle{Random forests for land cover classification}.
\bjournal{Pattern Recognition Letters}
\bvolume{27}
\bpages{294--300}.
\end{barticle}
\endbibitem

\bibitem{Goldstein2011}
\begin{barticle}[author]
\bauthor{\bsnm{Goldstein},~\bfnm{Benjamin~A}\binits{B.~A.}},
  \bauthor{\bsnm{Polley},~\bfnm{Eric~C}\binits{E.~C.}} \AND
  \bauthor{\bsnm{Briggs},~\bfnm{Farren~BS}\binits{F.~B.}}
(\byear{2011}).
\btitle{Random forests for genetic association studies}.
\bjournal{Statistical Applications in Genetics and Molecular Biology}
\bvolume{10}
\bpages{32}.
\end{barticle}
\endbibitem

\bibitem{howard2012two}
\begin{binproceedings}[author]
\bauthor{\bsnm{Howard},~\bfnm{Jeremy}\binits{J.}} \AND
  \bauthor{\bsnm{Bowles},~\bfnm{Mike}\binits{M.}}
(\byear{2012}).
\btitle{The two most important algorithms in predictive modeling today}.
In \bbooktitle{Strata Conference presentation}
\bvolume{28}.
\end{binproceedings}
\endbibitem

\bibitem{ishwaran2010consistency}
\begin{barticle}[author]
\bauthor{\bsnm{Ishwaran},~\bfnm{Hemant}\binits{H.}} \AND
  \bauthor{\bsnm{Kogalur},~\bfnm{Udaya~B}\binits{U.~B.}}
(\byear{2010}).
\btitle{Consistency of random survival forests}.
\bjournal{Statistics \& Probability Letters}
\bvolume{80}
\bpages{1056--1064}.
\end{barticle}
\endbibitem

\bibitem{Ishwaran2008}
\begin{barticle}[author]
\bauthor{\bsnm{Ishwaran},~\bfnm{Hemant}\binits{H.}},
  \bauthor{\bsnm{Kogalur},~\bfnm{Udaya~B}\binits{U.~B.}},
  \bauthor{\bsnm{Blackstone},~\bfnm{Eugene~H}\binits{E.~H.}} \AND
  \bauthor{\bsnm{Lauer},~\bfnm{Michael~S}\binits{M.~S.}}
(\byear{2008}).
\btitle{Random survival forests}.
\bjournal{The Annals of Applied Statistics}
\bvolume{2}
\bpages{841--860}.
\end{barticle}
\endbibitem

\bibitem{khaidem2016predicting}
\begin{barticle}[author]
\bauthor{\bsnm{Khaidem},~\bfnm{Luckyson}\binits{L.}},
  \bauthor{\bsnm{Saha},~\bfnm{Snehanshu}\binits{S.}} \AND
  \bauthor{\bsnm{Dey},~\bfnm{Sudeepa~Roy}\binits{S.~R.}}
(\byear{2016}).
\btitle{Predicting the direction of stock market prices using random forest}.
\bjournal{arXiv preprint arXiv:1605.00003}.
\end{barticle}
\endbibitem

\bibitem{klusowski2019analyzing}
\begin{barticle}[author]
\bauthor{\bsnm{Klusowski},~\bfnm{Jason~M}\binits{J.~M.}}
(\byear{2019}).
\btitle{Analyzing {CART}}.
\bjournal{arXiv preprint arXiv:1906.10086}.
\end{barticle}
\endbibitem

\bibitem{klusowski2019sharp}
\begin{binproceedings}[author]
\bauthor{\bsnm{Klusowski},~\bfnm{Jason~M}\binits{J.~M.}}
(\byear{2021}).
\btitle{Sharp Analysis of a Simple Model for Random Forests}.
In \bbooktitle{Proceedings of The 24th International Conference on Artificial
  Intelligence and Statistics}
(\beditor{\bfnm{Arindam}\binits{A.}~\bsnm{Banerjee}} \AND
  \beditor{\bfnm{Kenji}\binits{K.}~\bsnm{Fukumizu}}, eds.).
\bseries{Proceedings of Machine Learning Research}
\bvolume{130}
\bpages{757--765}.
\end{binproceedings}
\endbibitem

\bibitem{rf}
\begin{barticle}[author]
\bauthor{\bsnm{Liaw},~\bfnm{Andy}\binits{A.}} \AND
  \bauthor{\bsnm{Wiener},~\bfnm{Matthew}\binits{M.}}
(\byear{2002}).
\btitle{Classification and Regression by random{F}orest}.
\bjournal{R News}
\bvolume{2}
\bpages{18-22}.
\end{barticle}
\endbibitem

\bibitem{lin2006random}
\begin{barticle}[author]
\bauthor{\bsnm{Lin},~\bfnm{Yi}\binits{Y.}} \AND
  \bauthor{\bsnm{Jeon},~\bfnm{Yongho}\binits{Y.}}
(\byear{2006}).
\btitle{Random forests and adaptive nearest neighbors}.
\bjournal{Journal of the American Statistical Association}
\bvolume{101}
\bpages{578--590}.
\end{barticle}
\endbibitem

\bibitem{louppe2013understanding}
\begin{barticle}[author]
\bauthor{\bsnm{Louppe},~\bfnm{Gilles}\binits{G.}},
  \bauthor{\bsnm{Wehenkel},~\bfnm{Louis}\binits{L.}},
  \bauthor{\bsnm{Sutera},~\bfnm{Antonio}\binits{A.}} \AND
  \bauthor{\bsnm{Geurts},~\bfnm{Pierre}\binits{P.}}
(\byear{2013}).
\btitle{Understanding variable importances in forests of randomized trees}.
\bjournal{Advances in neural information processing systems}
\bvolume{26}
\bpages{431--439}.
\end{barticle}
\endbibitem

\bibitem{Mentch2014}
\begin{barticle}[author]
\bauthor{\bsnm{Mentch},~\bfnm{Lucas}\binits{L.}} \AND
  \bauthor{\bsnm{Hooker},~\bfnm{Giles}\binits{G.}}
(\byear{2014}).
\btitle{Ensemble trees and {CLT}s: Statistical inference for supervised
  learning}.
\bjournal{arXiv preprint arXiv:1404.6473}.
\end{barticle}
\endbibitem

\bibitem{mourtada2018minimax}
\begin{barticle}[author]
\bauthor{\bsnm{Mourtada},~\bfnm{Jaouad}\binits{J.}},
  \bauthor{\bsnm{Ga\"{i}ffas},~\bfnm{St\'ephane}\binits{S.}} \AND
  \bauthor{\bsnm{Scornet},~\bfnm{Erwan}\binits{E.}}
(\byear{2020}).
\btitle{{Minimax optimal rates for Mondrian trees and forests}}.
\bjournal{The Annals of Statistics}
\bvolume{48}
\bpages{2253--2276}.
\bdoi{10.1214/19-AOS1886}
\end{barticle}
\endbibitem

\bibitem{nobel1996histogram}
\begin{barticle}[author]
\bauthor{\bsnm{Nobel},~\bfnm{Andrew}\binits{A.}}
(\byear{1996}).
\btitle{Histogram regression estimation using data-dependent partitions}.
\bjournal{The Annals of Statistics}
\bvolume{24}
\bpages{1084--1105}.
\end{barticle}
\endbibitem

\bibitem{qi2012random}
\begin{bincollection}[author]
\bauthor{\bsnm{Qi},~\bfnm{Yanjun}\binits{Y.}}
(\byear{2012}).
\btitle{Random forest for bioinformatics}.
In \bbooktitle{Ensemble Machine Learning}
\bpages{307--323}.
\bpublisher{Springer}.
\end{bincollection}
\endbibitem

\bibitem{scornet2020trees}
\begin{barticle}[author]
\bauthor{\bsnm{Scornet},~\bfnm{Erwan}\binits{E.}}
(\byear{2020}).
\btitle{Trees, forests, and impurity-based variable importance}.
\bjournal{arXiv preprint arXiv:2001.04295}.
\end{barticle}
\endbibitem

\bibitem{Biau2015}
\begin{barticle}[author]
\bauthor{\bsnm{Scornet},~\bfnm{Erwan}\binits{E.}},
  \bauthor{\bsnm{Biau},~\bfnm{G{\'e}rard}\binits{G.}} \AND
  \bauthor{\bsnm{Vert},~\bfnm{Jean-Philippe}\binits{J.-P.}}
(\byear{2015}).
\btitle{Consistency of random forests}.
\bjournal{The Annals of Statistics}
\bvolume{43}
\bpages{1716--1741}.
\end{barticle}
\endbibitem

\bibitem{Stone1977}
\begin{barticle}[author]
\bauthor{\bsnm{Stone},~\bfnm{Charles~J}\binits{C.~J.}}
(\byear{1977}).
\btitle{Consistent nonparametric regression}.
\bjournal{The Annals of Statistics}
\bvolume{5}
\bpages{595--620}.
\end{barticle}
\endbibitem

\bibitem{syrgkanis2020estimation}
\begin{binproceedings}[author]
\bauthor{\bsnm{Syrgkanis},~\bfnm{Vasilis}\binits{V.}} \AND
  \bauthor{\bsnm{Zampetakis},~\bfnm{Manolis}\binits{M.}}
(\byear{2020}).
\btitle{Estimation and inference with trees and forests in high dimensions}.
In \bbooktitle{Conference on Learning Theory}
\bpages{3453--3454}.
\bpublisher{PMLR}.
\end{binproceedings}
\endbibitem

\bibitem{varian2014big}
\begin{barticle}[author]
\bauthor{\bsnm{Varian},~\bfnm{Hal~R}\binits{H.~R.}}
(\byear{2014}).
\btitle{Big data: New tricks for econometrics}.
\bjournal{Journal of Economic Perspectives}
\bvolume{28}
\bpages{3--28}.
\end{barticle}
\endbibitem

\bibitem{Wager2018}
\begin{barticle}[author]
\bauthor{\bsnm{Wager},~\bfnm{Stefan}\binits{S.}} \AND
  \bauthor{\bsnm{Athey},~\bfnm{Susan}\binits{S.}}
(\byear{2018}).
\btitle{Estimation and inference of heterogeneous treatment effects using
  random forests}.
\bjournal{Journal of the American Statistical Association}
\bvolume{113}
\bpages{1228--1242}.
\end{barticle}
\endbibitem

\bibitem{wager2014confidence}
\begin{barticle}[author]
\bauthor{\bsnm{Wager},~\bfnm{Stefan}\binits{S.}},
  \bauthor{\bsnm{Hastie},~\bfnm{Trevor}\binits{T.}} \AND
  \bauthor{\bsnm{Efron},~\bfnm{Bradley}\binits{B.}}
(\byear{2014}).
\btitle{Confidence intervals for random forests: The jackknife and the
  infinitesimal jackknife}.
\bjournal{Journal of Machine Learning Research}
\bvolume{15}
\bpages{1625--1651}.
\end{barticle}
\endbibitem

\bibitem{zhu2015reinforcement}
\begin{barticle}[author]
\bauthor{\bsnm{Zhu},~\bfnm{Ruoqing}\binits{R.}},
  \bauthor{\bsnm{Zeng},~\bfnm{Donglin}\binits{D.}} \AND
  \bauthor{\bsnm{Kosorok},~\bfnm{Michael~R}\binits{M.~R.}}
(\byear{2015}).
\btitle{Reinforcement learning trees}.
\bjournal{Journal of the American Statistical Association}
\bvolume{110}
\bpages{1770--1784}.
\end{barticle}
\endbibitem

\end{thebibliography}
	

	
	
	
	
	

\newpage
	
	\appendix
	\setcounter{page}{1}
	\setcounter{section}{0}
	\renewcommand{\theequation}{A.\arabic{equation}}
	\setcounter{equation}{0}

	\begin{center}{\bf \large Supplement to  ``Asymptotic Properties of High-Dimensional Random Forests"}
		
		\bigskip
		
		Chien-Ming Chi, Patrick Vossler, Yingying Fan and Jinchi Lv
		\medskip
	\end{center}
	
	\noindent This Supplementary Material contains the proofs of all main results and technical lemmas, and some additional technical details. All the notation is the same as defined in the main body of the paper. We use $C$ to denote a generic positive constant whose value may change from line to line.
	
	
	\section{Proofs of main results} \label{SecA}

	\subsection{Technical preparation} \label{SecA.1}

	We now describe further details of the grid introduced in Section \ref{Sec5} for the analysis purpose. For any $s\in \{1, \dots, p\}$ and integer $0 \le q<\lceil n^{1+\rho_{1}}\rceil$, define\footnote{We can let the interval $[b_{\lceil n^{1+\rho_{1}} \rceil- 1}, b_{\lceil n^{1+\rho_{1}} \rceil}]$ have a closed right end. Since we assume that the density of the distribution of $\boldsymbol{X}$ exists, it does not affect our technical analysis.}
	\[\boldsymbol{t}(s, q)\coloneqq  [0, 1]^{s-1} \times [b_{q}, b_{q+1}) \times [0, 1]^{p - s},\]
	where we recall that $b_i= \frac{i}{\ceil*{n^{1+\rho_1}} }$'s are the grid points defined in Section \ref{Sec5}; the parameter $\rho_{1}$ is defined in the same section. Let us assume that Condition~\ref{abc} is satisfied with $f(\cdot)$ denoting the density function of the distribution of $\boldsymbol{X}$. Then it holds that for each $n \ge 1$,
	\begin{equation}\label{grid2}
		\sup_{s, \, q}\mathop{{}\mathbb{P}}(\boldsymbol{X} \in \boldsymbol{t}(s, q)) \le \frac{\sup f}{ \lceil n^{1+\rho_{1}}\rceil}.
	\end{equation}
	Thus, for each  $n\ge 1$, positive integer $k$, and each cell $\boldsymbol{t}$ with 
at most $k$ boundaries not on the grid hyperplanes (e.g., for the left plot in Figure~\ref{fig:rf}, the blue cell has $4$ boundaries not on the grid hyperplanes, whereas the red one has only $2$), we have 
	\begin{equation}
		\label{grid3}
		\sup_{\boldsymbol{t}} \mathbb{P}(\boldsymbol{X} \in \boldsymbol{t}\Delta\boldsymbol{t}^{\#}) \le k\times\frac{\sup f}{ \lceil n^{1+\rho_{1}}\rceil},
	\end{equation}
	where the supremum is over all possible such $\boldsymbol{t}$ and $A \Delta B \coloneqq (A \cap B^{c} )\cup (A^{c} \cap B)$ for any two sets $A$ and $B$. Observe that (\ref{grid3}) applies to all cells constructed by at most $k$ cuts.
	
	Let $p$-dimensional random vectors $\boldsymbol{x}_{i}, i = 1, \dots, n$, be independent and identically distributed (i.i.d.) with the same distribution as $\boldsymbol{X}$. Let $\rho_{2} > 0$ be given. We next show that if Condition~\ref{abc} holds, it follows from (\ref{grid2}) that for each  $p \ge 1$ and all large $n$,	
	\begin{equation}
		\label{eventA}
		\mathbb{P}  \Big(\cup_{s, \, q} \Big\{\#\{i : \boldsymbol{x}_{i} \in \boldsymbol{t}(s, q) \} \ge \lceil(\log{n})^{1+\rho_{2}}\rceil \Big\}\Big) \le  p \lceil n^{1+\rho_{1}} \rceil  \left(\frac{\sup f}{  n^{\rho_{1}} }\right)^{(\log{n})^{1+\rho_{2}}},
	\end{equation}
where the union is over all possible $s\in\{1, \dots, p\}$ and $0 \le q<\lceil n^{1+\rho_{1}}\rceil$. To develop some intuition for \eqref{eventA}, note that if $\mathbb{P}(\boldsymbol{x}_{i} \in \boldsymbol{t}(s, q)) = c n^{-1}$ for some constant $c > 0$, then $\#\{i : \boldsymbol{x}_{i} \in \boldsymbol{t}(s, q) \}$ has an asymptotic Poisson distribution with mean $c$. Moreover, the probability upper bound in \eqref{grid2} is in fact much smaller than $n^{-1}$ asymptotically. To establish \eqref{eventA}, a direct calculation shows that 
	\begin{equation}\begin{split}\label{eventA.3}
			\mathbb{P} & \Big(\cup_{s, \, q} \Big\{\#\{i : \boldsymbol{x}_{i} \in \boldsymbol{t}(s, q) \} \ge \lceil(\log{n})^{1+\rho_{2}}\rceil \Big\}\Big) \\
			& \quad \le p \lceil n^{1+\rho_{1}}\rceil \sup_{s, \, q} \mathbb{P}\Big( \#\{i : \boldsymbol{x}_{i} \in \boldsymbol{t}(s, q) \} \ge \lceil(\log{n})^{1+\rho_{2}}\rceil \Big) \\
			& \quad = p \lceil n^{1+\rho_{1}}\rceil \times \sup_{s,q }\left(\sum_{l\ge l_{0}}^{n}\binom{n}{l} \Big(\mathbb{P}(\boldsymbol{x}_{i} \in \boldsymbol{t}(s, q)) \Big)^{l}\Big(1 -   \mathbb{P}(\boldsymbol{x}_{i} \in \boldsymbol{t}(s, q))   \Big)^{n-l} \right),
		\end{split}
	\end{equation}
	where $l_{0} = \lceil(\log{n})^{1+\rho_{2}}\rceil$. Since the cumulative probability inside the parentheses on the RHS of \eqref{eventA.3} is an increasing function of $\mathbb{P}(\boldsymbol{x}_{i} \in \boldsymbol{t}(s, q))$, it follows from \eqref{grid2} and Condition~\ref{abc} that for all large $n$,
	\begin{equation}	
		\begin{split}\label{eventA.2}
			\textnormal{RHS of \eqref{eventA.3} } &  \le p \lceil n^{1+\rho_{1}}\rceil \times \left( \sum_{l\ge l_{0}}^{n}\binom{n}{l} \left(\frac{\sup f}{ \lceil n^{1+\rho_{1}}\rceil} \right)^{l}\left(1 - \frac{\sup f}{ \lceil n^{1+\rho_{1}}\rceil}\right)^{n-l} \right)\\ 
			& \le p \lceil n^{1+\rho_{1}}\rceil \left( (l_{0}!)^{-1}\sum_{l \ge l_{0}}^{n} \left(  \frac{\sup f}{n^{\rho_{1}}}\right)^{l} \right) \\
			&\le p \lceil n^{1+\rho_{1}}\rceil  \left( \frac{\sup f}{n^{\rho_{1}}}\right)^{l_{0}} \left(1 - \frac{\sup f}{n^{\rho_{1}}}\right)^{-1} (l_{0}!)^{-1}\\
			&\le p \lceil n^{1+\rho_{1}}\rceil  \left( \frac{\sup f}{n^{\rho_{1}}}\right)^{l_{0}},
	\end{split}\end{equation}
	where $l! \coloneqq 1\times \dots \times l$. This completes the proof of \eqref{eventA}.
	
	We denote the event on the LHS of \eqref{eventA} by $\mathcal{A}$ with $\rho_{1}, \rho_{2} >0$ as follows.
	\begin{equation}
		\begin{split}					
		\label{eventA.1}
	\mathcal{A} &\coloneqq \Big(\cup_{s\in\{1, \dots, p\}, \, 0 \le q<\lceil n^{1+\rho_{1}}\rceil} \Big\{\#\{i : \boldsymbol{x}_{i} \in \boldsymbol{t}(s, q) \} \ge \lceil(\log{n})^{1+\rho_{2}}\rceil \Big\}\Big)^c\\
	& = \cap_{s\in\{1, \dots, p\}, \, 0 \le q<\lceil n^{1+\rho_{1}}\rceil} \Big\{\#\{i : \boldsymbol{x}_{i} \in \boldsymbol{t}(s, q) \}< (\log{n})^{1+\rho_{2}}\Big\}.
\end{split}\end{equation}
	On event $\mathcal{A}$, it holds that for each cell $\boldsymbol{t}$ constructed using at most $k$ cuts,
	\begin{equation}\label{grid1}
		\#\{i: \boldsymbol{x}_{i}\in \boldsymbol{t}\Delta\boldsymbol{t}^{\#} \} < k(\log{n})^{1 + \rho_{2}}.
	\end{equation}
	We next provide an upper bound on the number of conditional means required to be estimated. Define $G_{n,k}$ as the set containing all cells constructed by at most $k$ cuts with cuts all on the grid hyperplanes. We can see that there are at most $(p(\ceil{n^{1+\rho_{1}}} + 1))^{k}$ distinct choices of $k$ cuts on the grid hyperplanes. Furthermore, each of these $k$ cuts results in at most $2^{k}$ cells, which are all possible cells grown by the given $k$ cuts. Thus, we can obtain that 
	\begin{equation}\begin{split}\label{numberOf}
			\#G_{n,k}\le 2^k\left(p (\lceil n^{1+\rho_{1}}\rceil + 1)\right)^{k}.
		\end{split}
	\end{equation}

	\subsection{Additional examples for SID}\label{SecA.7}
	We provide three additional examples for showing the flexibility of SID. In particular, Example~\ref{new.example6} below is an example of Example~\ref{new.example3}, and Example~\ref{new.example7} considers regression function $m(\boldsymbol{X})$ that is not monotonic. Example~\ref{new.example8} is a non-additive model with a linear combination of intercepts. The proofs for these examples are respectively in Sections~\ref{SecE.8}--\ref{SecE.10}.
	
	\begin{exmp}\label{new.example6}
		Assume that $\boldsymbol{X}$ is uniformly distributed in $[0, 1]^p$, and let $S^*$ be some subset of $\{1, \dots, p\}$.
		\begin{itemize}
			\item[1)] Let $m(\boldsymbol{X}) =  \frac{\exp(\sum_{j\in S^*}\beta_{j} X_{j}) }{1 + \exp(\sum_{j\in S^*}\beta_{j} X_{j} ) }$ be given with $|\beta_{j}|\not = 0$. Then, $$m(\boldsymbol{X})\in \textnormal{SID}\Big( 4\left(\frac{\#S^* \max_{j \in S^*}|\beta_{j}|}{\min_{j \in S^*}|\beta_{j}|} \right)^2
			\times \exp{\Big(2\sum_{j\in S^*}|\beta_{j}|\Big)}\Big) .$$
			\item [2)] Let $m(\boldsymbol{X}) = \sum_{k=1}^{k_{1}} \beta_{kk}\Pi_{j\in T_{k}} X_{j}^{r_{jk}} + \sum_{j\in S^*}\beta_{j}X_{j}$  be given
			with $r_{jk}$'s being positive integers, $\cup_{k=1}^{k_{1}} T_{k} \subset S^*$, and all positive (or all negative) $\beta_{kk}$'s and $\beta_{j}$'s. Then 
			$$m(\boldsymbol{X})\in \textnormal{SID} \Big(4(\# S^*)^2\Big(\frac{ (\max_{j,k}r_{jk})\sum_{k=1}^{k_{1}}|\beta_{kk}| + \max_{j\in S^*}|\beta_{j}| }{ \min_{j\in S^*}|\beta_{j}| } \Big)^2\Big).$$
		\end{itemize}
	\end{exmp} 
	
	\begin{exmp} 
		\label{new.example7} 
			Assume that $\boldsymbol{X}$ is uniformly distributed in $[0, 1]^{p}$ with $p\ge s^*$ for some positive integer $s^*$. The regression function is defined as  
			$ m(\boldsymbol{X}) \coloneqq \sum_{j=1}^{s^*} m_{j}(X_{j})$, where for each $j\le s^*$ and $x\in[0, 1]$,
			\[m_{j}(x) \coloneqq h_{j, K}(x)\boldsymbol{1}_{[b_{j, K-1}, b_{j, K}] } + \sum_{k= 1}^{K-1} h_{j, k}(x)\boldsymbol{1}_{[b_{j, k-1}, b_{j, k}) } \]
			with some integer $K > 0$, linear functions $h_{j, 1}, \dots, h_{j, K}$ such that $m_{j}(x)$ is continuous and that $r \le |\frac{d h_{j, k}(x)}{d x}| \le R$ for all $j, k$ with some $R\ge r>0$, and constants $b_{j, k}$'s such that $0 = b_{j, 0}  < \dots< b_{j, K} = 1$.  Then, $m(\boldsymbol{X})\in \textnormal{SID} \Big(s^*\left(\frac{1024R^5}{(b^*)^3 r^5}\right)\Big)$, where
			$b^*\coloneqq\min_{j\le s^*, 1 \le k\le K} (b_{j, k} - b_{j,k-1})$. 
	\end{exmp}

\begin{exmp}\label{new.example8}
	Assume that $\boldsymbol{X}$ is uniformly distributed in $[0, 1]^p$ with $p\ge s^*$ for some positive integer $s^*$. Let positive integer $k_{j}$ be given and $0 = c_{0}^{(j)}<  \dots < c_{k_{j}}^{(j)} = 1$ be real numbers for each $j = 1\dots, s^*$. Let $\beta(i_{1}, \dots, i_{s^*})$ with $1\le i_{j}\le k_{j}$ and $1\le j\le s^*$ be real coefficients such that for some $\iota>0$, it holds that for each $j$, either \textnormal{1)} for every $(i_{1}, \dots, i_{s^*})$ with $i_{j}\ge 2$, $\Delta\beta\coloneqq \beta(i_{1}, \dots, i_{j},\dots,  i_{s^*}) - \beta(i_{1}, \dots, i_{j}-1,\dots, i_{s^*})\ge\iota$, or \textnormal{2)} for every $(i_{1}, \dots, i_{s^*})$ with $i_{j}\ge 2$, $\Delta\beta\le -\iota$. The regression function is defined to be 
	$$m(\boldsymbol{X}) = \sum_{i_{1}=1}^{k_{1}}\dots \sum_{i_{s^*}=1}^{k_{s^*}} \beta(i_{1}, \dots, i_{s^*})\Pi_{j=1}^{s^*}\boldsymbol{1}_{X_{j} \in [c_{i_{j} - 1}^{(j)}, c_{i_{j} }^{(j)})}.$$
	In addition, assume that $\sup_{\boldsymbol{c} \in [0, 1]^{p}}|m(\boldsymbol{c})|\le M_{0}$. Then, $m(\boldsymbol{X})\in\textnormal{SID}\Big(  \frac{s^*}{c^\dagger(1-c^\dagger)}  \left(\frac{2M_{0}}{\iota}\right)^2\Big),$ where $c^\dagger\coloneqq \min\{\frac{1}{4}, \min_{j\le s^*, 1\le i\le k_{j}}\{c_{i}^{(j)} - c_{i-1}^{(j)}\}\}$.
\end{exmp}

	\subsection{Proof of Theorem~\ref{theorem1}} \label{SecA.2}
	We begin with considering the case when $a$ contains the full sample. We will apply standard inequalities to separate the $\mathbb{L}^{2}$ loss into two terms that can be dealt with by Lemmas~\ref{lemma1} and~\ref{con3}, respectively, to obtain the conclusion in (\ref{main1}). Observe that the results in these lemmas are applicable to the case of any $a$ with $\# a = \ceil{bn}$ and without replacement. The other case with sample subsampling can be dealt with similarly by an application of Jensen's inequality. 
	
	Let us first examine the case without sample subsampling. By Jensen's inequality and the triangle inequality, we can deduce that 
	\begin{equation}\begin{split}\label{t.new.eq.2}
			& \mathop{{}\mathbb{E}} \Big(m(\boldsymbol{X}) - \mathop{{}\mathbb{E}}\Big(\widehat{m}_{ \widehat{T}}(\boldsymbol{\Theta}_{1}, \dots, \boldsymbol{\Theta}_{k}, \boldsymbol{X}, \mathcal{X}_{n} )\ \Big\vert \ \boldsymbol{X}, \mathcal{X}_{n} \Big) \Big)^{2} \\
			& \le \mathop{{}\mathbb{E}} \Big(m(\boldsymbol{X}) - \widehat{m}_{ \widehat{T}}(\boldsymbol{\Theta}_{1}, \dots, \boldsymbol{\Theta}_{k} , \boldsymbol{X}, \mathcal{X}_{n} ) \Big)^{2}\\
			& = \mathop{{}\mathbb{E}} \Big(m(\boldsymbol{X})  -m_{\widehat{T} }^{*} (\boldsymbol{\Theta}_{1}, \dots, \boldsymbol{\Theta}_{k} , \boldsymbol{X}) \\
			&\quad+ m_{ \widehat{T} }^{*} (\boldsymbol{\Theta}_{1}, \dots, \boldsymbol{\Theta}_{k} , \boldsymbol{X})- \widehat{m}_{ \widehat{T}}(\boldsymbol{\Theta}_{1}, \dots, \boldsymbol{\Theta}_{k} , \boldsymbol{X}, \mathcal{X}_{n} ) \Big)^{2}  \\
			& \le 2\Big( \mathbb{E} \Big(m(\boldsymbol{X})  -m_{ \widehat{T} }^{*} (\boldsymbol{\Theta}_{1}, \dots, \boldsymbol{\Theta}_{k} , \boldsymbol{X})\Big)^{2} \\
			& \qquad+  \mathbb{E}\Big(m_{ \widehat{T} }^{*} (\boldsymbol{\Theta}_{1}, \dots, \boldsymbol{\Theta}_{k} , \boldsymbol{X})- \widehat{m}_{ \widehat{T}}(\boldsymbol{\Theta}_{1}, \dots, \boldsymbol{\Theta}_{k} , \boldsymbol{X}, \mathcal{X}_{n} ) \Big)^{2}  \Big).
	\end{split}\end{equation}
	This result is also shown in (\ref{decom1}).
	
	By Lemma~\ref{lemma1}, it holds that for all large $n$ and each $1\le k \le c\log{n}$,
	\begin{equation}
		\begin{split}
			\label{t.new.eq.1}
			& \mathbb{E}  \left(  m(\boldsymbol{X}) - m_{\widehat{T} }^{*} (\boldsymbol{\Theta}_{1}, \dots, \boldsymbol{\Theta}_{k} , \boldsymbol{X})\right)^{2} \\
			& \qquad \le 8M_{0}^{2}n^{-\delta} 2^{k} +2\alpha_{1}\alpha_{2}n^{-\eta} + 2M_{0}^{2}(1 - \gamma_{0}(\alpha_{1}\alpha_{2})^{-1})^{k}  + 2n^{-1}.
	\end{split}\end{equation}
	Let $\nu > 0$ be sufficiently small. Then by Lemma~\ref{con3}, there exists some constant $C>0$ such that for all large $n$ and each $1\le k \le c\log_{2}{n}$,
	\begin{equation}\label{t.new.eq.3}
		\mathbb{E} \Big(m_{ \widehat{T}}^{*}(\boldsymbol{\Theta}_{1}, \dots, \boldsymbol{\Theta}_{k} , \boldsymbol{X}) - \widehat{m}_{\widehat{T}}(\boldsymbol{\Theta}_{1}, \dots, \boldsymbol{\Theta}_{k} , \boldsymbol{X}, \mathcal{X}_{n} ) \Big)^{2} \le   n^{-\eta} + C2^{k}n^{-\frac{1}{2} + \nu}. \end{equation}
	In view of (\ref{t.new.eq.2})--(\ref{t.new.eq.3}), we can conclude that there exists some constant $C > 0$ such that for all large $n$ and each $1\le k \le c\log_{2}{n}$,
	\[ \mathop{{}\mathbb{E}} \Big(m(\boldsymbol{X}) - \mathop{{}\mathbb{E}}\Big(\widehat{m}_{ \widehat{T}}(\boldsymbol{\Theta}_{1}, \dots, \boldsymbol{\Theta}_{k}, \boldsymbol{X}, \mathcal{X}_{n} )\ \Big\vert \ \boldsymbol{X}, \mathcal{X}_{n} \Big) \Big)^{2}  \le C\Big( \alpha_{1}\alpha_{2}n^{-\eta}  + (1 - \gamma_{0} (\alpha_{1} \alpha_{2})^{-1})^{k} + n^{-\delta  +c}  \Big). \]
	The above result uses the fact that $n^{-\frac{1}{2} + \nu + c} = o(n^{-\delta + c})$ due to a small $\nu$. Thus, replacing $n$ with $\ceil{bn}$ leads to (\ref{main1}).

	To show the second assertion, we use Jensen's inequality to obtain that 
	\begin{equation*}
		\begin{split}
			&  \mathbb{E} \Big(m(\boldsymbol{X}) - B^{-1}\sum_{a \in A}\mathbb{E} \Big( \widehat{m}_{ \widehat{T}_{a}, a}(\boldsymbol{\Theta}_{1}, \dots, \boldsymbol{\Theta}_{k}, \boldsymbol{X}, \mathcal{X}_{n} )\ \Big\vert \ \boldsymbol{X}, \mathcal{X}_{n} \Big) \Big)^{2} \\
			&  =  \mathbb{E} \Big[ B^{-1}\sum_{a \in A} \Big(m(\boldsymbol{X}) - \mathbb{E}\Big( \widehat{m}_{ \widehat{T}_{a}, a}(\boldsymbol{\Theta}_{1}, \dots, \boldsymbol{\Theta}_{k}, \boldsymbol{X}, \mathcal{X}_{n} )\ \Big\vert \ \boldsymbol{X}, \mathcal{X}_{n} \Big) \Big) \Big]^{2} \\
			& \le B^{-1}\sum_{a \in A}  \mathbb{E}  \Big(m(\boldsymbol{X}) - \mathbb{E}\Big( \widehat{m}_{ \widehat{T}_{a}, a}(\boldsymbol{\Theta}_{1}, \dots, \boldsymbol{\Theta}_{k}, \boldsymbol{X}, \mathcal{X}_{n} )\ \Big\vert \ \boldsymbol{X}, \mathcal{X}_{n} \Big) \Big)^{2}.
		\end{split}
	\end{equation*}
	Combining this result and the first assertion completes the proof of Theorem~\ref{theorem1}.

	
	
	\subsection{Proof of Theorem \ref{theorem2}} \label{SecA.3}

	Recall that $\mathcal{X}_{n}= \{\boldsymbol{x}_{i}, y_{i}\}_{i=1}^{n}$ are i.i.d. training data, and $\boldsymbol{X}$ is the independent copy of $\boldsymbol{x}_{1} = (x_{11},  \dots, x_{1p})^{\top}$. When the $j$th feature is not involved in the random forests model training procedure, the random forests estimate (\ref{estimate1}) is trained on $\{(y_{i}, x_{i1}, \dots, x_{i(j-1)}, x_{i(j+1)}, \dots, x_{ip})\}_{i=1}^{n}$. We first show that such a random forests estimate is $(\boldsymbol{X}_{-j}, \mathcal{X}_{n})$-measurable, where $\boldsymbol{X}_{-j} \coloneqq (X_{1}, \dots, X_{j-1}, X_{j+1}, \dots, X_{p})^{\top}$. Then by the independence between $\boldsymbol{X}$ and $\mathcal{X}_{n}$, we can resort to the projection theorem to obtain the desired conclusion. Let us begin the formal proof.
	
	We denote such a random forest estimate by
	\begin{equation}
		\label{re1}
		\frac{1}{B}\sum_{a\in A}\mathbb{E}\Big(\widehat{m}_{ \widehat{T}_{a}, a}(\boldsymbol{\Theta}_{1}, \dots, \boldsymbol{\Theta}_{k}, \boldsymbol{X}, \mathcal{X}_{n} )\ \Big\vert \ \boldsymbol{X}, \mathcal{X}_{n} \Big).
	\end{equation}
	In order that the conditional expectation in (\ref{re1}) is well defined, we use Conditions~\ref{ME}--\ref{BO} to ensure the existence of the first moment of the integrand in (\ref{re1}). Specifically, by Conditions~\ref{ME}--\ref{BO}, it holds that for each $a\in A$,
	\begin{equation*}
		\begin{split}
			\mathbb{E}\Big(\Big|\widehat{m}_{ \widehat{T}_{a}, a}(\boldsymbol{\Theta}_{1}, \dots, \boldsymbol{\Theta}_{k}, \boldsymbol{X}, \mathcal{X}_{n} ) \Big|\Big)   & < \infty,
		\end{split}
	\end{equation*}
	and hence (\ref{re1}) is well defined.

	By assumption, during the training phase, the $j$th feature is not involved, which entails that $\widehat{m}_{ \widehat{T}_{a}, a}(\boldsymbol{\Theta}_{1}, \dots, \boldsymbol{\Theta}_{k}, \boldsymbol{c}_{1}, \mathcal{X}_{n} ) = \widehat{m}_{ \widehat{T}_{a}, a}(\boldsymbol{\Theta}_{1}, \dots, \boldsymbol{\Theta}_{k}, \boldsymbol{c}_{2}, \mathcal{X}_{n} )$ for each $\boldsymbol{c}_{i} \coloneqq (c_{i1}, \dots, c_{ip})^{\top}\in [0, 1]^{p}, i = 1, 2$, with $c_{1l} = c_{2l}$ for $l\not=j$. Then it follows that 
	\begin{equation*}\begin{split} & \frac{1}{B}\sum_{a\in A}\mathbb{E}\Big( \widehat{m}_{ \widehat{T}_{a}, a}(\boldsymbol{\Theta}_{1:k}, \boldsymbol{X}, \mathcal{X}_{n} )\ \Big\vert \ \boldsymbol{\Theta}_{1:k}, \boldsymbol{X}, \mathcal{X}_{n} \Big) \\
			& = \frac{1}{B}\sum_{a\in A}\mathbb{E}\Big(\widehat{m}_{ \widehat{T}_{ a}, a}(\boldsymbol{\Theta}_{1:k}, \boldsymbol{X}, \mathcal{X}_{n} )\ \Big\vert \ \boldsymbol{\Theta}_{1:k}, \boldsymbol{X}_{-j}, \mathcal{X}_{n} \Big). \end{split}\end{equation*}
	In view of this result, we can see that 
	\begin{equation}\label{project1} \frac{1}{B}\sum_{a\in A}\mathbb{E}\Big(\widehat{m}_{ \widehat{T}_{a}, a}(\boldsymbol{\Theta}_{1}, \dots, \boldsymbol{\Theta}_{k}, \boldsymbol{X}, \mathcal{X}_{n} )\ \Big\vert \ \boldsymbol{X}, \mathcal{X}_{n} \Big) \textnormal{ is } (\boldsymbol{X}_{-j}, \mathcal{X}_{n})\textnormal{-measurable}.\end{equation}
	Since $\mathcal{X}_{n}$ is independent of $\boldsymbol{X}$, we have 
	\begin{equation}\label{project2}
		\textnormal{Var} \Big(m(\boldsymbol{X}) \ \vert \ \boldsymbol{X}_{-j}, \mathcal{X}_{n} \Big)  = \textnormal{Var} \Big(m(\boldsymbol{X}) \ \vert \ \boldsymbol{X}_{-j} \Big).\end{equation}

	By the definition of relevant features, we can deduce that 
	\begin{equation*}
		\begin{split} & \mathbb{E} \Big(m(\boldsymbol{X}) - \frac{1}{B}\sum_{a\in A}\mathbb{E} \Big(\widehat{m}_{\widehat{T}_{a}, a}(\boldsymbol{\Theta}_{1}, \dots, \boldsymbol{\Theta}_{k}, \boldsymbol{X}, \mathcal{X}_{n} )\ \Big\vert \ \boldsymbol{X}, \mathcal{X}_{n} \Big) \Big)^{2} \\
			&= \mathbb{E}\Big\{ \mathbb{E} \Big[  \Big(m(\boldsymbol{X}) - \frac{1}{B}\sum_{a\in A}\mathbb{E} \Big(\widehat{m}_{\widehat{T}_{a}, a}(\boldsymbol{\Theta}_{1}, \dots, \boldsymbol{\Theta}_{k}, \boldsymbol{X}, \mathcal{X}_{n} )\ \Big\vert \ \boldsymbol{X}, \mathcal{X}_{n} \Big) \Big)^{2} \ \Big\vert \ \boldsymbol{X}_{-j}, \mathcal{X}_{n} \Big]\Big\}\\
			& \ge  \mathbb{E} \Big(\textnormal{Var} \Big(m(\boldsymbol{X}) \ \vert \  \boldsymbol{X}_{-j}, \mathcal{X}_{n}\Big)\Big) \\
			& = \mathbb{E} \Big(\textnormal{Var} \Big(m(\boldsymbol{X}) \ \vert \ X_{s}, s \in \{1, \dots, p\} \backslash \{j\}\Big)\Big)
			\\&  \ge \iota.
	\end{split}\end{equation*}
	Here, in the first inequality, we apply (\ref{project1}) and the  projection theorem. For the second equality, we resort to (\ref{project2}). This concludes the proof of Theorem \ref{theorem2}.

	\subsection{Proof of Theorem~\ref{theorem3}} \label{SecA.4}
	Let $0 < \gamma_{0}\le 1$ be given. We deal with the case where there are no random splits first (see the end of this proof for details). Let us begin with a closed-form expression for the $\mathbb{L}^{2}$ approximation error in (\ref{t6.1}) below obtained using (\ref{t.new.5}). We argue that in the expression
	\begin{equation}
		\begin{split}
			\label{t6.1}
			& \mathbb{E}\Big(m(\boldsymbol{X}) - m_{T}^{*}\left(\boldsymbol{\Theta}_{1:k}, \boldsymbol{X}\right)\Big)^{2} \\
			& = \sum_{\Theta_{1:k}}  \mathop{{}\mathbb{P}}(\boldsymbol{\Theta}_{1:k} = \Theta_{1:k})\sum_{(\boldsymbol{t}_{1}, \dots, \boldsymbol{t}_{k}) \in T(\Theta_{1:k})} \mathbb{P}(\boldsymbol{X}\in \boldsymbol{t}_{k})  \textnormal{Var} (m(\boldsymbol{X})\ \vert\ \boldsymbol{X}\in\boldsymbol{t}_{k}),\\
	\end{split} \end{equation}
     the (average) conditional variance on the end cells at the last level can be bounded by the (average) conditional variance at the one to the last level multiplied by a factor $(1 - \gamma_{0}(\alpha_{1}\alpha_{2})^{-1})$, and hence we have the recursive argument as in (\ref{t.new.2}). However, according to Condition~\ref{tree}, for the cell with too small probabilities, we need to use a different approach to deal with the case, which results in an additional term $\varepsilon\alpha_{1}\alpha_{2}$ in Theorem~\ref{theorem3}.

	In what follows, $T(\Theta_{1:k})$ is categorized into two groups, where upper bounds are constructed accordingly.  Let $\varepsilon \ge 0$ be given.  Then we introduce a set of tuples denoted as $T_{\varepsilon}$.  For each $\Theta_{1}, \dots, \Theta_{k}$, define a set of $k$-dimensional tuples $T_{\varepsilon}(\Theta_{1}, \dots, \Theta_{k})$ such that  if the following two properties hold:
	\begin{enumerate}
		\item[1)] $(\boldsymbol{t}_{1}, \dots, \boldsymbol{t}_{k})\in T(\Theta_{1}, \dots, \Theta_{k})$,
		\item[2)] There exists some positive integer $l\le k$  such that  $\sup_{j, c}(II)_{\boldsymbol{t}_{l-1}, \boldsymbol{t}_{l -1}(j, c) } \le \alpha_{2}\varepsilon$ with the supremum over all possible $(j, c)$'s,
	\end{enumerate}
	then $(\boldsymbol{t}_{1}, \dots, \boldsymbol{t}_{k})\in T_{\varepsilon}(\Theta_{1}, \dots, \Theta_{k})$. In view of the definition of $T_{\varepsilon}$, we can deduce that
	\begin{equation}\begin{split}\label{t6}
			&\textnormal{RHS of (\ref{t6.1})}
			\\&= \sum_{\Theta_{1:k}}  \mathop{{}\mathbb{P}}(\boldsymbol{\Theta}_{1:k} = \Theta_{1:k}) \Bigg[ \sum_{(\boldsymbol{t}_{1}, \dots, \boldsymbol{t}_{k}) \in T_{\varepsilon}(\Theta_{1:k})} \mathbb{P}(\boldsymbol{X}\in \boldsymbol{t}_{k})  \textnormal{Var} (m(\boldsymbol{X})\ \vert\ \boldsymbol{X}\in\boldsymbol{t}_{k})\\
			& \hspace{1em} + \sum_{(\boldsymbol{t}_{1}, \dots, \boldsymbol{t}_{k}) \in \big(T(\Theta_{1:k}) \big\backslash T_{\varepsilon}(\Theta_{1:k})\big) } \mathbb{P}(\boldsymbol{X}\in \boldsymbol{t}_{k})  \textnormal{Var} (m(\boldsymbol{X})\ \vert\ \boldsymbol{X}\in\boldsymbol{t}_{k})  \Bigg],
	\end{split}\end{equation}
	where the summation is over all possible $\Theta_{1:k}$.

	 For simplicity, define $T^{\dagger}(\Theta_{1:k}) \coloneqq T(\Theta_{1:k}) \big\backslash T_{\varepsilon}(\Theta_{1:k})$ and $\textnormal{V}(\boldsymbol{t}) \coloneqq \textnormal{Var} (m(\boldsymbol{X}) \vert \boldsymbol{X}\in\boldsymbol{t})$. We can observe two properties of $T^{\dagger}$. First, if $(\boldsymbol{t}_{1}, \dots, \boldsymbol{t}_{k}) \in T^{\dagger}(\Theta_{1:k})$, then we have $(\boldsymbol{t}_{1}, \dots, \boldsymbol{t}_{l}) \in T^{\dagger}(\Theta_{1:l})$ for each $1 \le l< k$, but not the other way around. Second, if $(\boldsymbol{t}_{1}, \dots, \boldsymbol{t}_{k}) \in T^{\dagger}(\Theta_{1:k})$ and $\boldsymbol{t}_{k}^{'}$ is the other daughter cell of $\boldsymbol{t}_{k-1}$ in $T(\Theta_{1:k})$, then it holds that $(\boldsymbol{t}_{1}, \dots, \boldsymbol{t}_{k-1}, \boldsymbol{t}_{k}^{'}) \in T^{\dagger}(\Theta_{1:k})$; that is, daughter cells are included in $T^{\dagger}(\Theta_{1:k})$ as a pair. It is worth emphasizing that $T^{\dagger}(\Theta_{1:k})$ and  $T_{\varepsilon}(\Theta_{1:k})$ are two sets of tuples such that  $\{\boldsymbol{t}_{k}: (\boldsymbol{t}_{1}, \dots, \boldsymbol{t}_{k})\in T^{\dagger}(\Theta_{1:k})\}$ and $\{\boldsymbol{t}_{k}: (\boldsymbol{t}_{1}, \dots, \boldsymbol{t}_{k})\in T_{\varepsilon}(\Theta_{1:k})\}$ are mutually exclusive, and collectively they are a partition of feature space.
	 
	With these notations, let us deal with the second term on the RHS of (\ref{t6}) first. Simple calculations show that  
	{\small\begin{equation}\begin{split}\label{t7.1}
			&\textnormal{The second term on the RHS of (\ref{t6})}\\
			& = \sum_{\Theta_{1:k}}  \mathbb{P}(\boldsymbol{\Theta}_{1:k} = \Theta_{1:k}) \sum_{(\boldsymbol{t}_{1:k}) \in T^{\dagger}(\Theta_{1:k}) }\mathbb{P}( \boldsymbol{X} \in \boldsymbol{t}_{k-1})\mathbb{P}(  \boldsymbol{X} \in \boldsymbol{t}_{k} | \boldsymbol{X}\in \boldsymbol{t}_{k-1} )  \textnormal{V} (\boldsymbol{t}_{k})\\
			& = \sum_{\Theta_{1:k}}  \mathop{{}\mathbb{P}}(\boldsymbol{\Theta}_{1:k} = \Theta_{1:k})\sum_{(\boldsymbol{t}_{1}, \dots, \boldsymbol{t}_{k}) \in T^{\dagger}(\Theta_{1:k}) }\mathbb{P}( \boldsymbol{X} \in \boldsymbol{t}_{k-1}) \frac{(I)_{\boldsymbol{t}_{k-1}, \boldsymbol{t}_{k}}}{2}\\
			& = \sum_{\Theta_{1:k-1}} \mathop{{}\mathbb{P}}(\boldsymbol{\Theta}_{1:k-1} = \Theta_{1:k-1}) \sum_{\Theta_{k}} \mathbb{P}(\boldsymbol{\Theta}_{k} = \Theta_{k}) \sum_{(\boldsymbol{t}_{1}, \dots, \boldsymbol{t}_{k}) \in T^{\dagger}(\Theta_{1:k}) }\mathbb{P}( \boldsymbol{X} \in \boldsymbol{t}_{k-1}) \frac{(I)_{\boldsymbol{t}_{k-1}, \boldsymbol{t}_{k}}}{2},
	\end{split}\end{equation}}
	
	\noindent where  the second equality is due to the definition of $(I)_{\boldsymbol{t}_{k-1}, \boldsymbol{t}_{k}}$ in \eqref{I1} and the second property of $T^{\dagger}$, and the third equality is due to the independence of random parameters.

	\begin{figure}[t]		
		\centering
		\includegraphics[width=8cm]{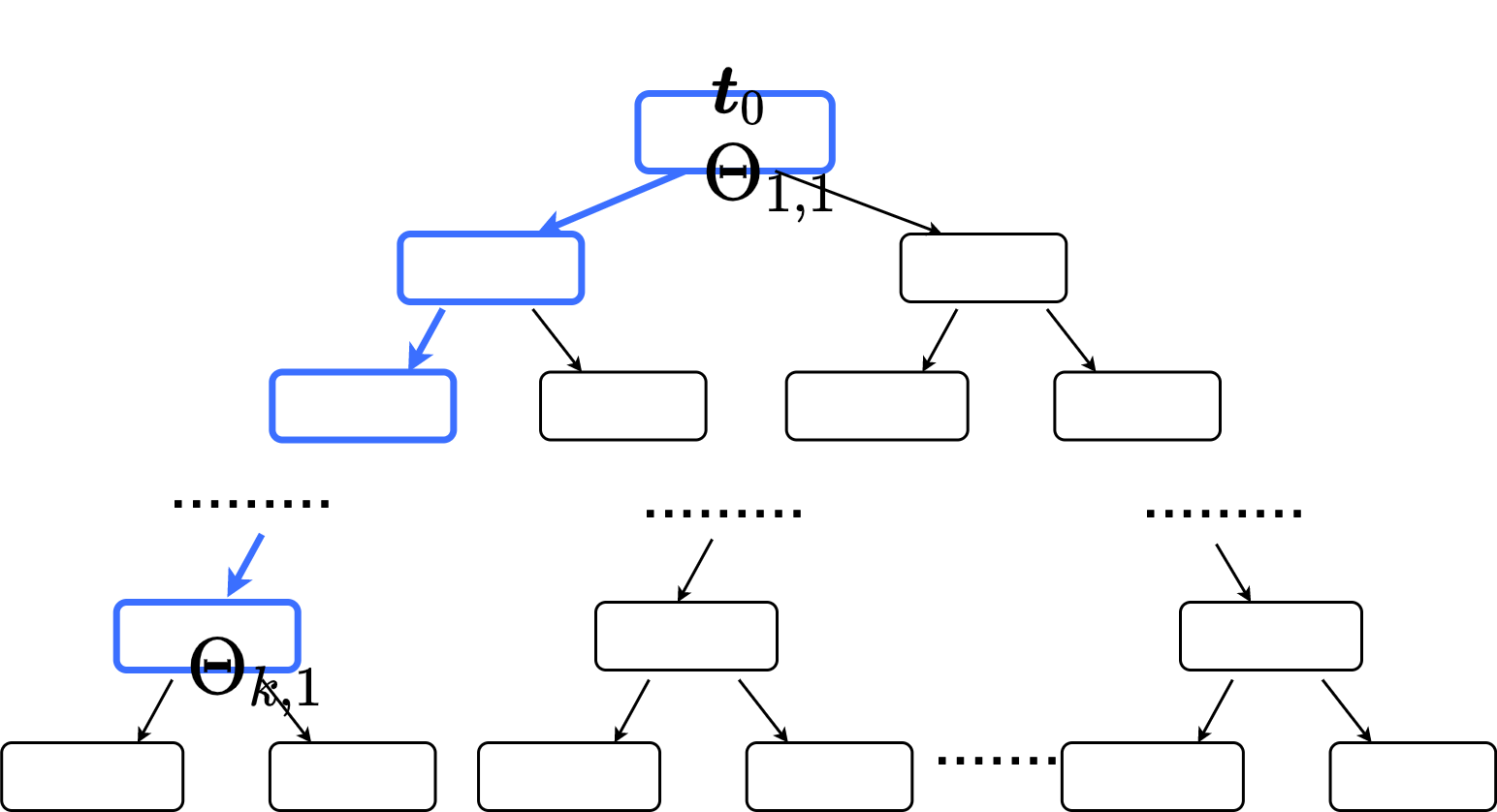}
		\caption{The thick blue tree branch is the first tree branch of $T(\Theta_{1:k})$.}
		\label{fig:firstTreeBranch}
	\end{figure}
	
	To deal with the RHS of \eqref{t7.1}, we consider tree branches of the tree $T(\Theta_{1:k})$ (not $T^{\dagger}(\Theta_{1:k})$) as follows. There are $2^k$ distinct tree branches $\boldsymbol{t}_{1:k}$ in $T(\Theta_{1:k})$, and we call the first two of these tree branches ``the first tree branch of $T(\Theta_{1:k})$,'' whose corresponding last column set restriction is $\Theta_{k, 1}$ (recall that $\Theta_{k} = \{\Theta_{k, 1}, \dots, \Theta_{k, 2^{k-1}}\}$). See Figure~\ref{fig:firstTreeBranch} for a graphical illustration. Note that there are two daughter cells of the first tree branch. In addition, note that it is possible that some tree branches of $T(\Theta_{1:k})$ are not included in $T^{\dagger}(\Theta_{1:k})$; in such cases, the corresponding summations (e.g., see \eqref{new.bias.5} below) ignore these tree branches since we have defined that summations  over empty sets are zeros.

	 Now, with the definition of tree branches, we write the inner term on the RHS of \eqref{t7.1} as follows.
	 {\small\begin{equation}
	 		\begin{split}\label{new.bias.5}
	 			&\sum_{\Theta_{k}} \mathbb{P}(\boldsymbol{\Theta}_{k} = \Theta_{k}) \sum_{ (\boldsymbol{t}_{1}, \dots, \boldsymbol{t}_{k}) \in T^{\dagger}(\Theta_{1:k}) }\mathbb{P}( \boldsymbol{X} \in \boldsymbol{t}_{k-1}) \frac{(I)_{\boldsymbol{t}_{k-1}, \boldsymbol{t}_{k}}}{2}\\
	 			& =  \sum_{\Theta_{k}}\mathbb{P}(\boldsymbol{\Theta}_{k, 1} = \Theta_{k, 1}, \dots, \boldsymbol{\Theta}_{k, 2^{k-1}} = \Theta_{k, 2^{k-1}})  \\
	 			&\qquad\times\bigg[\bigg(\sum_{\substack{ (\boldsymbol{t}_{1},\dots,\boldsymbol{t}_{k} )\in T^{\dagger}(\Theta_{1:k}) \textnormal{ where}\\ (\boldsymbol{t}_{1},\dots,\boldsymbol{t}_{k-1} ) \textnormal{ is the first tree branch in } T(\Theta_{1:k}) } }\mathbb{P}( \boldsymbol{X} \in \boldsymbol{t}_{k-1}) \frac{(I)_{\boldsymbol{t}_{k-1}, \boldsymbol{t}_{k}}}{2} \bigg)\\
	 			&\qquad\qquad+  \bigg( \sum_{\substack{ (\boldsymbol{t}_{1},\dots,\boldsymbol{t}_{k} )\in T^{\dagger}(\Theta_{1:k}) \textnormal{ where}\\ (\boldsymbol{t}_{1},\dots,\boldsymbol{t}_{k-1} ) \textnormal{ is not the first tree branch in } T(\Theta_{1:k}) } }\mathbb{P}( \boldsymbol{X} \in \boldsymbol{t}_{k-1}) \frac{(I)_{\boldsymbol{t}_{k-1}, \boldsymbol{t}_{k}}}{2}\bigg)\bigg],
	 		\end{split}
	 \end{equation}}
 
	 \noindent and then, since the first tree branch is related to only the feature restriction $\boldsymbol{\Theta}_{k, 1}$ and the other tree branches are only subject to $\boldsymbol{\Theta}_{k, 2}, \dots, \boldsymbol{\Theta}_{k, 2^{k-1}}$, and that $\boldsymbol{\Theta}_{k ,l}$'s are independent,
	{\small\begin{equation}
		\begin{split}\label{new.bias.3}
			&\textnormal{RHS of \eqref{new.bias.5} }\\
			& =  \sum_{\Theta_{k, 1}}\mathbb{P}(\boldsymbol{\Theta}_{k, 1} = \Theta_{k, 1}) \\
			&\qquad\qquad\times\bigg(\sum_{\substack{ (\boldsymbol{t}_{1},\dots,\boldsymbol{t}_{k} )\in T^{\dagger}(\Theta_{1:k}) \textnormal{ where}\\ (\boldsymbol{t}_{1},\dots,\boldsymbol{t}_{k-1} ) \textnormal{ is the first tree branch in } T(\Theta_{1:k}) } }\mathbb{P}( \boldsymbol{X} \in \boldsymbol{t}_{k-1}) \frac{(I)_{\boldsymbol{t}_{k-1}, \boldsymbol{t}_{k}}}{2} \bigg)\\
			&\qquad+ \sum_{\Theta_{k, 2}, \dots, \Theta_{k, 2^{k-1}}}\mathbb{P}(\boldsymbol{\Theta}_{k, 2} = \Theta_{k, 2}, \dots, \boldsymbol{\Theta}_{k, 2^{k-1}} = \Theta_{k, 2^{k-1}})\\
			& \qquad\qquad\times \bigg( \sum_{\substack{ (\boldsymbol{t}_{1},\dots,\boldsymbol{t}_{k} )\in T^{\dagger}(\Theta_{1:k}) \textnormal{ where}\\ (\boldsymbol{t}_{1},\dots,\boldsymbol{t}_{k-1} ) \textnormal{ is not the first tree branch in } T(\Theta_{1:k}) } }\mathbb{P}( \boldsymbol{X} \in \boldsymbol{t}_{k-1}) \frac{(I)_{\boldsymbol{t}_{k-1}, \boldsymbol{t}_{k}}}{2}\bigg).
		\end{split}
	\end{equation}}
	
	With \eqref{new.bias.3}, we can focus on the summation with the first tree branch in the following; without loss of generality, we suppose the first tree branch of $T(\Theta_{1:k})$ is also included in $T^{\dagger}(\Theta_{1:k})$ (otherwise, we can consider some other tree branch). Observe that there exists at least one optimal feature $j^{*}$ such that  $\sup_{c} (II)_{\boldsymbol{t}_{k-1}, \boldsymbol{t}_{k-1}(j^{*}, c)} = \sup_{j , c} (II)_{\boldsymbol{t}_{k-1}, \boldsymbol{t}_{k-1}(j, c)}$, where the the supremum on the RHS is the unconstrained supremum.  It is not difficult to see that the probability that $\boldsymbol{\Theta}_{k, 1}$ includes one of these optimal features is at least $\gamma_{0}$; that is, 
	\begin{equation}
		\begin{split}\label{t7.6}
			&\mathbb{P}(\boldsymbol{\Theta}_{k, 1} \textnormal{ includes one of the optimal features}) \ge \gamma_{0} \quad (\textnormal{the good state}),\\
			& \mathbb{P}(\{\boldsymbol{\Theta}_{k, 1} \textnormal{ includes one of the optimal features}\}^{c}) < 1 - \gamma_{0} \quad (\textnormal{the bad state}).
	\end{split}\end{equation}
	It follows from the definition of $T^{\dagger}$ that $\sup_{j, c} (II)_{\boldsymbol{t}_{k-1}, \boldsymbol{t}_{k-1}(j, c) } > \alpha_{2}\varepsilon$. Then if $\Theta_{k, 1}$ is in the good state, by the  first item  of Condition~\ref{tree} it holds that $(II)_{\boldsymbol{t}_{k-1}, \boldsymbol{t}_{k}} > \varepsilon$. By this, the second item of Condition~\ref{tree}, and the fact that $\Theta_{k, 1}$ is in the good state, we have
	\begin{equation}
		\begin{split}\label{t.new.61}
			(I)_{\boldsymbol{t}_{k-1}, \boldsymbol{t}_{k}} &  = \text{Var}(m(\boldsymbol{X} \ \vert \ \boldsymbol{X} \in \boldsymbol{t}_{k-1}) - (II)_{\boldsymbol{t}_{k-1}, \boldsymbol{t}_{k}}\\
			& \le \text{Var}(m(\boldsymbol{X}) \ \vert \ \boldsymbol{X} \in \boldsymbol{t}_{k-1}) - \alpha_{2}^{-1}\sup_{j, c}(II)_{\boldsymbol{t}_{k-1}, \boldsymbol{t}_{k-1}(j, c)}.
		\end{split}	
	\end{equation}
	Moreover, it follows from Condition~\ref{P} that
	\[\textnormal{RHS of (\ref{t.new.61})}  \le \text{Var}(m(\boldsymbol{X})  \vert  \boldsymbol{X} \in \boldsymbol{t}_{k-1}) (1 - (\alpha_{1}\alpha_{2})^{-1}). \]
	On the other hand, if $\Theta_{k, 1}$ is in the bad state, it holds that 
	$$(I)_{\boldsymbol{t}_{k-1}, \boldsymbol{t}_{k}} \le  \textnormal{Var}(m(\boldsymbol{X})  \vert  \boldsymbol{X} \in \boldsymbol{t}_{k-1}).$$ 

	By the above observation, for the first tree branch,
	{\small\begin{equation}
			\begin{split}\label{new.bias.1}
				& \sum_{\Theta_{k, 1}}\mathbb{P}(\boldsymbol{\Theta}_{k, 1} = \Theta_{k, 1}) \\
				&\qquad\times\bigg(\sum_{\substack{ (\boldsymbol{t}_{1},\dots,\boldsymbol{t}_{k} )\in T^{\dagger}(\Theta_{1:k}) \textnormal{ where}\\ (\boldsymbol{t}_{1},\dots,\boldsymbol{t}_{k-1} ) \textnormal{ is  the first tree branch in } T(\Theta_{1:k}) }}\mathbb{P}( \boldsymbol{X} \in \boldsymbol{t}_{k-1}) \frac{(I)_{\boldsymbol{t}_{k-1}, \boldsymbol{t}_{k}}}{2} \bigg) \\
				& \le \sum_{\textnormal{good }\Theta_{k, 1}}\mathbb{P}(\boldsymbol{\Theta}_{k, 1} = \Theta_{k, 1}) \\
				&\qquad\times\bigg(\sum_{\substack{ (\boldsymbol{t}_{1},\dots,\boldsymbol{t}_{k} )\in T^{\dagger}(\Theta_{1:k}) \textnormal{ where}\\ (\boldsymbol{t}_{1},\dots,\boldsymbol{t}_{k-1} ) \textnormal{ is  the first tree branch in } T(\Theta_{1:k}) }}\mathbb{P}( \boldsymbol{X} \in \boldsymbol{t}_{k-1}) \frac{\text{V}( \boldsymbol{t}_{k-1}) (1 - (\alpha_{1}\alpha_{2})^{-1})}{2} \bigg) \\
				&+ \sum_{\textnormal{bad }\Theta_{k, 1}}\mathbb{P}(\boldsymbol{\Theta}_{k, 1} = \Theta_{k, 1}) \\
				&\qquad\times\bigg(\sum_{\substack{ (\boldsymbol{t}_{1},\dots,\boldsymbol{t}_{k} )\in T^{\dagger}(\Theta_{1:k}) \textnormal{ where}\\ (\boldsymbol{t}_{1},\dots,\boldsymbol{t}_{k-1} ) \textnormal{ is  the first tree branch in } T(\Theta_{1:k}) }}\mathbb{P}( \boldsymbol{X} \in \boldsymbol{t}_{k-1}) \frac{\text{V}( \boldsymbol{t}_{k-1})}{2} \bigg).
			\end{split}
	\end{equation}}

If $(\boldsymbol{t}_{1}, \dots, \boldsymbol{t}_{k-1})$ is the first tree branch in $T(\Theta_{1:k})$, due to the facts that there are two daughter cells of $\boldsymbol{t}_{k-1}$ and that the terms in the summations on the RHS of \eqref{new.bias.1} does not depends on $\boldsymbol{t}_{k}$, and \eqref{t7.6},
{\small\begin{equation}
		\begin{split}\label{new.bias.4}
			& \textnormal{RHS of \eqref{new.bias.1}} \\
			& \le \sum_{\textnormal{good }\Theta_{k, 1}}\mathbb{P}(\boldsymbol{\Theta}_{k, 1} = \Theta_{k, 1}) \mathbb{P}( \boldsymbol{X} \in \boldsymbol{t}_{k-1}) \text{V}( \boldsymbol{t}_{k-1}) (1 - (\alpha_{1}\alpha_{2})^{-1})  \\
			&\qquad + \sum_{\textnormal{bad }\Theta_{k, 1}}\mathbb{P}(\boldsymbol{\Theta}_{k, 1} = \Theta_{k, 1}) \mathbb{P}( \boldsymbol{X} \in \boldsymbol{t}_{k-1}) \text{V}( \boldsymbol{t}_{k-1})\\
			& \le \sup_{\gamma \ge \gamma_{0}}\bigg(\gamma\mathbb{P}( \boldsymbol{X} \in \boldsymbol{t}_{k-1}) \text{V}( \boldsymbol{t}_{k-1}) (1 - (\alpha_{1}\alpha_{2})^{-1})    + (1-\gamma) \mathbb{P}( \boldsymbol{X} \in \boldsymbol{t}_{k-1}) \text{V}( \boldsymbol{t}_{k-1})\bigg)\\
			&\le (1-\gamma_{0}(\alpha_{1}\alpha_{2})^{-1})\mathbb{P}(\boldsymbol{X}\in\boldsymbol{t}_{k-1})\textnormal{V}(\boldsymbol{t}_{k-1}).
		\end{split}
\end{equation}}

\noindent We notice that \eqref{new.bias.4} holds if the first tree branch is not included in $T^{\dagger}(\Theta_{1:k})$ since the summation would be zero.

We can apply the arguments for \eqref{new.bias.5}--\eqref{new.bias.4} to each tree branch in $T(\Theta_{1:k})$ to get
\begin{equation*}
	\begin{split}
	& \textnormal{RHS of \eqref{new.bias.3}} \le \sum_{ (\boldsymbol{t}_{1}, \dots, \boldsymbol{t}_{k-1}) \in T^{\dagger}(\Theta_{1:k-1}) }(1-\gamma_{0}(\alpha_{1}\alpha_{2})^{-1})\mathbb{P}(\boldsymbol{X}\in\boldsymbol{t}_{k-1})\textnormal{V}(\boldsymbol{t}_{k-1}).
\end{split}
\end{equation*}

Thus,
	{\small\begin{equation}
			\begin{split}\label{t7.2}
			&\textnormal{RHS of (\ref{t7.1})}\\
			& \le (1-\gamma_{0}(\alpha_{1}\alpha_{2})^{-1}) \sum_{\Theta_{1:k-1}} \mathop{{}\mathbb{P}}(\boldsymbol{\Theta}_{1:k-1} = \Theta_{1:k-1})  \sum_{(\boldsymbol{t}_{1:k-1}) \in T^{\dagger}(\Theta_{1:k-1}) }\mathbb{P}( \boldsymbol{X} \in \boldsymbol{t}_{k-1}) \textnormal{V} ( \boldsymbol{t}_{k-1}).
	\end{split}
\end{equation}}
	We can repeat the calculation in (\ref{t7.2}) $k$ times to conclude that  
	\begin{equation}\begin{split}\label{t7.3}
			&\textnormal{RHS of (\ref{t7.2})}  \le (1 - \gamma_{0}(\alpha_{1}\alpha_{2})^{-1} )^{k} \textnormal{Var} (m(\boldsymbol{X})).
	\end{split}\end{equation}

	Next, we bound the first term in (\ref{t6}). Let $(\boldsymbol{t}_{1}, \dots, \boldsymbol{t}_{k}) \in T_{\varepsilon} (\Theta_{1},\dots, \Theta_{k} )$ be a given tuple. By the second property in the definition of $T_{\varepsilon}$, there exists a smallest integer  $1 \le l\le k$ such that $\sup_{j, c} (II)_{\boldsymbol{t}_{l-1}, \boldsymbol{t}_{l-1}(j, c) } \le \alpha_{2}\varepsilon$. By Condition~\ref{P}, we have 
	\begin{equation}\label{E14}\textnormal{Var} (m(\boldsymbol{X})\ \vert \ \boldsymbol{X} \in \boldsymbol{t}_{l-1})\le \alpha_{1} \alpha_{2} \varepsilon.
	\end{equation} 
	Denote by $S$ the set of tuples in $T_{\varepsilon} (\Theta_{1},\dots, \Theta_{k})$ such that the first $l -1$ cells are $\boldsymbol{t}_{1}, \dots, \boldsymbol{t}_{l-1}$. For each $q \in \{l-1, \dots, k-1\}$, let $S_{q}$ be the set of distinct tuples in $\{( \boldsymbol{t}_{1}, \dots, \boldsymbol{t}_{q}) : (\boldsymbol{t}_{1}, \dots, \boldsymbol{t}_{k}) \in S\}$. Then we can deduce that 
	\begin{equation}
		\begin{split}
			\label{E13}
			& \sum_{(\boldsymbol{t}_{1}, \dots, \boldsymbol{t}_{k}) \in S} \mathbb{P} (\boldsymbol{X}\in \boldsymbol{t}_{k})   \textnormal{Var} (m(\boldsymbol{X})\ \vert\ \boldsymbol{X}\in\boldsymbol{t}_{k}) \\
			& = \sum_{(\boldsymbol{t}_{1}, \dots, \boldsymbol{t}_{k}) \in S} \mathbb{P}( \boldsymbol{X}\in \boldsymbol{t}_{k-1})\mathbb{P} (\boldsymbol{X}\in \boldsymbol{t}_{k} | \boldsymbol{X}\in \boldsymbol{t}_{k-1})   \textnormal{Var} (m(\boldsymbol{X})\ \vert\ \boldsymbol{X}\in\boldsymbol{t}_{k}) \\
			& \le \sum_{(\boldsymbol{t}_{1}, \dots, \boldsymbol{t}_{k-1}) \in S_{k-1}} \mathbb{P}( \boldsymbol{X}\in \boldsymbol{t}_{k-1})\textnormal{Var} (m(\boldsymbol{X})\ \vert\ \boldsymbol{X}\in\boldsymbol{t}_{k-1}) \\
			& \le  \sum_{(\boldsymbol{t}_{1}, \dots, \boldsymbol{t}_{l-1}) \in S_{l-1}}\mathbb{P} (\boldsymbol{X}\in \boldsymbol{t}_{l-1})  \textnormal{Var}(m(\boldsymbol{X})\ \vert \ \boldsymbol{X}\in \boldsymbol{t}_{l-1})\\
			& = \mathbb{P} (\boldsymbol{X}\in \boldsymbol{t}_{l-1})  \textnormal{Var}(m(\boldsymbol{X})\ \vert \ \boldsymbol{X}\in \boldsymbol{t}_{l-1})\\
			& \le \mathbb{P} (\boldsymbol{X}\in \boldsymbol{t}_{l-1})  \alpha_{1} \alpha_{2} \varepsilon.
	\end{split}\end{equation}
	Here, the first inequality in (\ref{E13}) follows from the fact that $\textnormal{Var}(m(\boldsymbol{X}) | \boldsymbol{X} \in \boldsymbol{t}_{k-1}) \ge (I)_{\boldsymbol{t}_{k-1}, \boldsymbol{t}_{k}} $. The second inequality is obtained by repeating the same argument for the first inequality. Moreover, the second equality is because $S_{l-1}$ contains exactly one tuple, while the last inequality follows from \eqref{E14}.

	Given $\Theta_{1}, \dots, \Theta_{k}$ and $\varepsilon$, it is seen that the summation summing the LHS of \eqref{E13} over all possible (and mutually exclusive) tuple sets $S$ is bounded by the summation over the probabilities of exclusive events multiplied by $\alpha_{1} \alpha_{2} \varepsilon$. Thus, it holds that 
	\[\sum_{(\boldsymbol{t}_{1}, \dots, \boldsymbol{t}_{k}) \in T_{\varepsilon}(\Theta_{1:k})}\mathbb{P} (\boldsymbol{X}\in \boldsymbol{t}_{k})  \textnormal{Var} (m(\boldsymbol{X})\ \vert\ \boldsymbol{X}\in\boldsymbol{t}_{k})\le \alpha_{1} \alpha_{2} \varepsilon. \]
	Since summing over the probabilities of $\Theta_{1:k}$ gives one, we have
	\begin{equation*}
		\sum_{\Theta_{1:k}}  \mathop{{}\mathbb{P}}(\boldsymbol{\Theta}_{1:k} = \Theta_{1:k})  \sum_{(\boldsymbol{t}_{1}, \dots, \boldsymbol{t}_{k}) \in T_{\varepsilon}(\Theta_{1:k} )} \mathbb{P} (\boldsymbol{X}\in \boldsymbol{t}_{k})  \textnormal{Var} (m(\boldsymbol{X})\ \vert\ \boldsymbol{X}\in\boldsymbol{t}_{k})\le \alpha_{1} \alpha_{2} \varepsilon.
	\end{equation*}
	Therefore, combining this inequality, and (\ref{t6.1})--(\ref{t7.1}), (\ref{t7.2})--(\ref{t7.3}) yields the desired conclusion of Theorem~\ref{theorem3} for the case without random splits.
	
	For the case where there are random splits, we can conditional on these random splits and apply the previous arguments to get the same conclusion. Specifically, let ``Random splits'' denote the random parameter of these random splits, we have
	\begin{equation*}
	    \begin{split}
	    & \mathbb{E}\Big(m(\boldsymbol{X}) - m_{T}^{*}\left(\boldsymbol{\Theta}_{1:k}, \boldsymbol{X}\right)\Big)^{2}\\
	        & = \mathbb{E}\Big(\mathbb{E}(\textnormal{RHS of } \eqref{t6.1} \ | \textnormal{ Random splits})\Big) \\
	        & \le \mathbb{E}\big(\alpha_{1} \alpha_{2} \varepsilon+ \left(1  -  \gamma_{0}(\alpha_{1}\alpha_{2})^{-1}  \right)^{k}\textnormal{Var}(m(\boldsymbol{X})) \big)\\
	        & = \alpha_{1} \alpha_{2} \varepsilon+ \left(1  -  \gamma_{0}(\alpha_{1}\alpha_{2})^{-1}  \right)^{k}\textnormal{Var}(m(\boldsymbol{X})),
	    \end{split}
	\end{equation*}
	where the inequality is due to the previous arguments. This completes the proof of Theorem~\ref{theorem3}.
	
	Finally, we note that the previous arguments also lead to the desired bound in  \eqref{t.new.2}.
	
	\subsection{Proof of Theorem~\ref{theorem4}} \label{SecA.5}
	
	Let us first briefly outline the proof idea for the main assertion of Theorem~\ref{theorem4}.  We argue that for each cell $\boldsymbol{t}$, the (sample) CART-split criterion in (\ref{new.eq.004}) gives results that are very close to those of the theoretical CART-split introduced in Section~\ref{Sec4.1}. More precisely, let $\hat{\boldsymbol{t}}$ be one of the daughter cells of $\boldsymbol{t}$ after the CART-split given a set of available features $\Theta$, and we argue that the value of $(II)_{\boldsymbol{t}, \hat{\boldsymbol{t}}}$  is very close to that of $\sup_{j\in \Theta, c }(II)_{\boldsymbol{t}, \boldsymbol{t}(j, c)}$. Since this argument involves quantities $(II)$'s, to obtain the desired result we need to control the differences between the theoretical and sample (conditional) moments. Thus, we rely on the grid introduced in Sections~\ref{Sec5} and \ref{SecA.1}, where we also introduce the ideas for the grid.

	The formal proof starts with constructing the $\mathcal{X}_{n}$-measurable  event $\boldsymbol{U}_{n}$ described in Theorem~\ref{theorem4}. Define for some $\Delta>0, s >0$ (further requirements on $\Delta, s$ will be specified shortly in \eqref{T4_0} below),
		\[\boldsymbol{U}_{n} \coloneqq \boldsymbol{C}_{n}\cap \mathcal{A}_{1}(k, \Delta)\cap\mathcal{A}_{2}(k , \Delta)\cap\mathcal{A}_{3}( k + 1, \Delta)\cap\mathcal{A}, \]
	where $k = \floor{ c\log{(n)}}$, $\boldsymbol{C}_{n}= \cap_{i=1}^{n} \{ |\varepsilon_{i}| \le n^{s} \}$, the event $\mathcal{A}$ is defined in \eqref{eventA.1}, and $\mathcal{A}_{i}(k, \Delta), i \in \{1, 2, 3\}$ are defined in Lemma~\ref{CI1} in Section \ref{SecE.1}; note that we let $k = \floor{ c\log{(n)}}$ in the proof of Theorem~\ref{theorem4} for simplicity. Briefly, the events $\mathcal{A}_{i}(k, \Delta), i \in \{1, 2, 3\}$ control the conditional moments, which include the conditional means and probabilities, and the numbers of 
	observations on each of the sufficiently large cells on the grid hyperplanes.  Since we have assumed Condition~\ref{abc} and Condition~\ref{ME} with sufficiently large $q$, it follows from Lemma~\ref{CI1} and \eqref{eventA} that for all large $n$, 
	\[ \mathbb{P}( \boldsymbol{U}_{n}^{c}) = o(n^{-1}), \] 
	which concludes the first assertion of Theorem~\ref{theorem4} regarding the event $\boldsymbol{U}_{n}$.

	It remains to show the second assertion of Theorem~\ref{theorem4}. Let us introduce some needed notation and parameter restrictions as follows. It is required that $\frac{1}{2} < \Delta < 1 -2\delta$, which is possible because $\delta <\frac{1}{4}$ ($\delta$ and $\eta$ in \eqref{T4_0} below are given by Theorem~\ref{theorem4}). In addition, we let  $\Delta^{'}$  and (a sufficiently small) $s > 0$ be such that $\frac{1}{2} < \Delta^{'} < \Delta$ and
	\begin{equation}\label{T4_0} 
		\eta < \min\{ \frac{\Delta^{'}}{4} - 2s, \delta - 2s, \frac{\delta}{2}\}.
	\end{equation}

	To better understand the technical arguments, we provide some useful intuitions first. For each cell $\boldsymbol{t}= \times_{j=1}^{p}t_{j}$  and a set of available features $\Theta$, let us fix a best cut $(j^{*}(\boldsymbol{t}), c^{*}(\boldsymbol{t})) \coloneqq \arg\sup_{j \in \Theta, c} (II)_{\boldsymbol{t}, \boldsymbol{t}(j, c) }$ and for simplicity, we do not specify the dependence of the cut on $\Theta$. Let $\boldsymbol{t}^{*}$ be one of the daughter cells of $\boldsymbol{t}$ after $(j^{*}(\boldsymbol{t}), c^{*}(\boldsymbol{t}))$. Our goal is to find the lower bound of  $(II)_{\boldsymbol{t}, \hat{\boldsymbol{t}}} \ -(II)_{\boldsymbol{t}, \boldsymbol{t}^{*}}$ in terms of the sample size $n$. The main idea of the proof is to find a semi-sample daughter cell of $\boldsymbol{t}$ denoted as $\boldsymbol{t}^{\dagger}$ such that $\boldsymbol{t}^{\dagger}$ is grown by a cut $(j^{*}(\boldsymbol{t}), c^{\dagger}(\boldsymbol{t}))$ with $c^{\dagger}(\boldsymbol{t}) = x_{i, j^{*}(\boldsymbol{t})}$ for some $i \in \{1,\dots, n\}$ and the value of $c^{\dagger}(\boldsymbol{t})$ is very close to $c^{*}(\boldsymbol{t})$ (recall that $\boldsymbol{x}_{i} = (x_{i, 1}, \dots, x_{i, p})^{\top}$'s are the observations in the sample). Intuitively, on one hand,  $(II)_{\boldsymbol{t}, \hat{\boldsymbol{t}}} - (II)_{\boldsymbol{t}, \boldsymbol{t}^{\dagger}}$ should be bounded from below because $\hat{\boldsymbol{t}}$  maximizes the sample counterpart of $(II)_{\boldsymbol{t}, \hat{\boldsymbol{t}}}$ and hence $\widehat{(II)}_{\boldsymbol{t}, \hat{\boldsymbol{t}}} \ge \widehat{(II)}_{\boldsymbol{t}, \boldsymbol{t}^{\dagger}}$ (the sample conditional bias decrease; a formal definition is in \eqref{hatI2} below), and the values of these sample counterparts are close to themselves, respectively, in a probabilistic sense. On the other hand, the difference between $c^{\dagger}(\boldsymbol{t})$ and $c^{*}(\boldsymbol{t})$ is very small and thus  $|(II)_{\boldsymbol{t}, \boldsymbol{t}^{\dagger}}- (II)_{\boldsymbol{t}, \boldsymbol{t}^{*}}|$ is controlled. Then by the use of the semi-sample daughter cell, we can complete the technical analysis.  
	
	We now introduce some necessary notation for the remaining proof. Let us fix an interval $I^{*}(\boldsymbol{t})$ such that $c^{*}(\boldsymbol{t})\in I^{*}(\boldsymbol{t})\subset t_{j^{*}(\boldsymbol{t})}$ and 
	\[ \mathbb{P} ( X_{j^{*}(\boldsymbol{t})} \in I^{*}(\boldsymbol{t}) \ \vert \ \boldsymbol{X} \in \boldsymbol{t})= n^{-\delta}. \]
	In view of Condition~\ref{abc}, such $I^{*}(\boldsymbol{t})$ is well defined. In addition, for the cell $\boldsymbol{t}$ and $\Theta$, we fix another cut  $(j^{*}(\boldsymbol{t}) , c^{\dagger}(\boldsymbol{t}))$ such that $ c^{\dagger}(\boldsymbol{t})$ is an element of the set
	\begin{equation}
		\label{cdagger1}
		\left\{ x_{i, j^{*}(\boldsymbol{t})} : \boldsymbol{x}_{i}\in\boldsymbol{t}, \  x_{i, j^{*}(\boldsymbol{t}) } \in I^{*}(\boldsymbol{t}) \right\}
	\end{equation}
	when the set is not empty, and otherwise $ c^{\dagger}(\boldsymbol{t})$ is a random value in $t_{j^{*}(\boldsymbol{t})}$. 
	
	
	Recall that $\widehat{\boldsymbol{t}}$,  $\boldsymbol{t}^{\dagger}$, and $\boldsymbol{t}^{*}$ denote, respectively, one of the daughter cells constructed by the CART-split (\ref{new.eq.004}), the cut $(j^{*}(\boldsymbol{t}), c^{\dagger}(\boldsymbol{t}))$, and the cut $(j^{*}(\boldsymbol{t}), c^{*}(\boldsymbol{t}))$. Particularly, due to the definition of $c^{\dagger}(\boldsymbol{t})$, we have ensured that
	\begin{equation}\label{control1}|  \mathbb{P}( \boldsymbol{X} \in \boldsymbol{t}^{\dagger}  \ \vert \ \boldsymbol{X} \in \boldsymbol{t}) -  \mathbb{P}( \boldsymbol{X} \in \boldsymbol{t}^{*}  \ \vert \ \boldsymbol{X} \in \boldsymbol{t}) | \le n^{-\delta}.
	\end{equation}
	Given the cell $\boldsymbol{t}$ and an arbitrary partition of $\boldsymbol{t}^{'}$ and $\boldsymbol{t}^{''}$, we can define the sample version of (\ref{I2}) as 
	\begin{equation}\begin{split}\label{hatI2}
			\widehat{(II)}_{\boldsymbol{t}, \boldsymbol{t}^{'}} & \coloneqq \frac{\# \{ i: \boldsymbol{x}_{i} \in\boldsymbol{t}^{'}\} }{\# \{ i: \boldsymbol{x}_{i} \in\boldsymbol{t}\}} \left( \sum_{\boldsymbol{x}_{i}\in \boldsymbol{t}^{'}} \frac{y_{i}}{\#\{   i: \boldsymbol{x}_{i} \in \boldsymbol{t}^{'}\} }  -  \sum_{\boldsymbol{x}_{i}\in \boldsymbol{t}} \frac{y_{i}}{\#\{   i: \boldsymbol{x}_{i} \in \boldsymbol{t}\} }    \right)^{2} \\
			&\quad+  \frac{\# \{ i: \boldsymbol{x}_{i} \in\boldsymbol{t}^{''}\} }{\# \{ i: \boldsymbol{x}_{i} \in\boldsymbol{t}\}} \left( \sum_{\boldsymbol{x}_{i}\in \boldsymbol{t}^{''}} \frac{y_{i}}{\#\{   i: \boldsymbol{x}_{i} \in \boldsymbol{t}^{''}\} }  -  \sum_{\boldsymbol{x}_{i}\in \boldsymbol{t}} \frac{y_{i}}{\#\{   i: \boldsymbol{x}_{i} \in \boldsymbol{t}\} }    \right)^{2}.
		\end{split}
	\end{equation}
	Here, we define a summation over an empty set as zero. In particular, $\widehat{(II)}_{\boldsymbol{t}, \boldsymbol{t}^{'}}$ is zero if $\boldsymbol{t}$ contains only one observation.
	
	To complete the proof for the second conclusion, we need Lemmas \ref{T4_1} and \ref{T4_2} in Sections \ref{SecD.1} and \ref{SecD.2}, respectively. Let a constant $c_{1}>0$ with $\alpha_{2} \ge 1 + c_{1}$ be given. It follows from Lemmas \ref{T4_1} and \ref{T4_2} and the definition of sample tree growing rule $\widehat{T}$ that on the event $\boldsymbol{U}_{n}$, there exists some  constant $C>0$ such that for all large $n$, each sequence of sets of available features $\Theta_{1}, \dots, \Theta_{k}$, each $(\boldsymbol{t}_{1}, \dots, \boldsymbol{t}_{k}) \in \widehat{T} (\Theta_{1}, \dots, \Theta_{k})$, and each $1\le l\le k $ with $\mathbb{P}(\boldsymbol{X}\in\boldsymbol{t}_{l-1}) \ge n^{-\delta}$, we have 
	\begin{equation}\begin{split} \label{T4_3}
			& (II)_{\boldsymbol{t}_{l-1}, \boldsymbol{t}_{l} } - (II)_{\boldsymbol{t}_{l-1}, \boldsymbol{t}^{*}} = \underbrace{  (II)_{\boldsymbol{t}_{l-1}, \boldsymbol{t}_{l} } - (II)_{\boldsymbol{t}_{l-1}^{\#}, \boldsymbol{t}_{l}^{\#}}  }_{(i)} +   \underbrace{  (II)_{\boldsymbol{t}_{l-1}^{\#}, \boldsymbol{t}_{l}^{\#}}-  \widehat{(II)}_{\boldsymbol{t}_{l-1}^{\#}, \boldsymbol{t}_{l}^{\#}}  }_{(ii)}  \\
			&\quad\quad+  \underbrace{ \widehat{(II)}_{\boldsymbol{t}_{l-1}^{\#}, \boldsymbol{t}_{l}^{\#}}  -  \widehat{(II)}_{\boldsymbol{t}_{l-1}, \boldsymbol{t}_{l}}  }_{(iii)}+  \underbrace{    \widehat{(II)}_{\boldsymbol{t}_{l-1}, \boldsymbol{t}_{l}}  -  \widehat{(II)}_{\boldsymbol{t}_{l-1}, \boldsymbol{t}_{l}^{\dagger} }  }_{(iv)}  \\
			& \quad\quad +  \underbrace{     \widehat{(II)}_{\boldsymbol{t}_{l-1}, \boldsymbol{t}_{l}^{\dagger} }  -  \widehat{(II)}_{(\boldsymbol{t}_{l-1})^{\#},( \boldsymbol{t}_{l}^{\dagger})^{\#} }  }_{(v)} 
			+ \underbrace{     \widehat{(II)}_{(\boldsymbol{t}_{l-1})^{\#},( \boldsymbol{t}_{l}^{\dagger})^{\#} }  -  (II)_{(\boldsymbol{t}_{l-1})^{\#},( \boldsymbol{t}_{l}^{\dagger})^{\#} }  }_{(vi)} \\
			& \quad\quad+ \underbrace{   (II)_{(\boldsymbol{t}_{l-1})^{\#},( \boldsymbol{t}_{l}^{\dagger})^{\#} }      -  (II)_{\boldsymbol{t}_{l-1}, \boldsymbol{t}_{l}^{\dagger} }  }_{(vii)} + \underbrace{   (II)_{\boldsymbol{t}_{l-1}, \boldsymbol{t}_{l}^{\dagger} }      -  (II)_{\boldsymbol{t}_{l-1}, \boldsymbol{t}_{l}^{*} }  }_{(viii)}\\
			& \quad \ge -C(n^{-\frac{\delta}{2}} + n^{-\frac{\Delta^{'}}{4} + 2s} + n^{-\delta + 2s} )\\
			& \quad \ge -c_{1}n^{-\eta},
		\end{split}
	\end{equation}
	where we suppress the dependence of all the daughter cells on the set of available features. In (\ref{T4_3}), terms (i)--(iii) and (v)--(vii) are bounded in Lemma \ref{T4_2}, while terms (iv) and (viii) are analyzed in Lemma \ref{T4_1}. To apply Lemma~\ref{T4_1}, notice that $\boldsymbol{t}_{l}$ in term (iv) is grown by the (sample) CART-split given $\boldsymbol{t}_{l-1}$ and the available features. The last inequality above is due to all large $n$ and (\ref{T4_0}).

	In view of (\ref{T4_3}) and the definition of $\hat{T}_{\zeta}$, on event $\boldsymbol{U}_{n}$ it holds that for all large $n$, each sequence of sets of available features $\Theta_{1}, \dots, \Theta_{k}$, and each $(\boldsymbol{t}_{1}, \dots, \boldsymbol{t}_{k}) \in \widehat{T}_{\zeta} (\Theta_{1}, \dots, \Theta_{k})$ with $\zeta = n^{-\delta}$, we have for $1\le l \le k$,
	\begin{equation}
		\label{T4_14}\sup_{j \in \Theta_{l}, c}(II)_{\boldsymbol{t}_{l-1}, \boldsymbol{t}_{l-1}(j, c)} \le (II)_{\boldsymbol{t}_{l-1}, \boldsymbol{t}_{l}} + c_{1}n^{-\eta} ,\end{equation}
	where we do not specify to which set of available features $j$ belongs in the supremum for simplicity as in Condition~\ref{tree}. Observe that because of the construction of the semi-sample tree growing rule, we do not require the condition of  $\mathbb{P}(\boldsymbol{X} \in \boldsymbol{t}_{l-1}) \ge n^{-\delta}$ in the statement of (\ref{T4_14}) as we do in (\ref{T4_3}).

	Finally, by the same conditions as for (\ref{T4_14}) and the choices of $\alpha_{1}$ and $c_{1}$, we have that for each $1\le l\le k $, if $(II)_{\boldsymbol{t}_{l-1}, \boldsymbol{t}_{l}} > n^{-\eta}$, it holds that
	\[  \sup_{j \in \Theta_{l}, c}(II)_{\boldsymbol{t}_{l-1}, \boldsymbol{t}_{l-1}(j, c)} \le \alpha_{2} (II)_{\boldsymbol{t} _{l-1}, \boldsymbol{t}_{l}},  \]
	and if $(II)_{\boldsymbol{t}_{l-1}, \boldsymbol{t}_{l}} \le  n^{-\eta}$, it holds that 
	\[  \sup_{j\in \Theta_{l}, c}(II)_{\boldsymbol{t}_{l-1}, \boldsymbol{t}_{l-1}(j, c)} \le \alpha_{2} n^{-\eta},  \]
	which concludes the proof of Theorem~\ref{theorem4}.

	\subsection{Proof of Theorem~\ref{theorem5}} \label{SecA.6}
	To outline the proof idea, let us rewrite the expectation and obtain a closed-form expression below. From (\ref{new.eq.015}), we can see that 
	{\small\begin{equation}\begin{split}\label{E3}
			& \mathbb{E} \left\{\sup_{T}  \mathop{{}\mathbb{E}}    \left[     \Big(m_{ T^{\#}}^{*}(\boldsymbol{\Theta}_{1}, \dots, \boldsymbol{\Theta}_{k}, \boldsymbol{X}) - \widehat{m}_{T^{\#}}(\boldsymbol{\Theta}_{1}, \dots, \boldsymbol{\Theta}_{k}, \boldsymbol{X}, \mathcal{X}_{n} ) \Big)^{2} \ \Big\vert \ \boldsymbol{\Theta}_{1:k}, \mathcal{X}_{n} \right] \right\}\\
			& = \mathbb{E}\Big[\sup_{T}\sum_{ (\boldsymbol{t}_{1}, \dots, \boldsymbol{t}_{k}) \in T^{\#} (\Theta_{1}, \dots, \Theta_{k}) }  \mathbb{P}(\boldsymbol{X}\in\boldsymbol{t}_{k}) \Big(  \mathbb{E}(m(\boldsymbol{X}) \ \vert \ \boldsymbol{X} \in \boldsymbol{t}_{k}) - \frac{\sum_{i\in \{i : \boldsymbol{x}_{i} \in \boldsymbol{t}_{k }\}} y_{i}} {\#\{i : \boldsymbol{x}_{i} \in \boldsymbol{t}_{k }\}}  \Big)^{2}\Big].
		\end{split}
	\end{equation}}
	To have a clearer picture of how to apply Hoeffding's inequality to our case, we utilize an even larger upper bound to get rid of $\Theta_{1}, \dots, \Theta_{k}$ (feature restrictions). Observe that the summation on the RHS of (\ref{E3}) is  over the partition $\{\boldsymbol{t}_{k} : (\boldsymbol{t}_{1}, \dots, \boldsymbol{t}_{k}) \in T^{\#}(\Theta_{1}, \dots, \Theta_{k})\}$. Thus, the RHS of (\ref{E3}) can be further bounded by considering the supremum of partitions over all $T$ and $\Theta_{1}, \dots, \Theta_{k}$; we use $T_{k}$ to denote a level $k$ tree such that $\{\boldsymbol{t}_{k}: (\boldsymbol{t}_{1}, \dots, \boldsymbol{t}_{k}) \in T_{k}^{\#}\}$ is an instance of such a partition to simplify the notation.
	
	 Then it follows that 
	\begin{equation}
		\begin{split}\label{E7}
			& \textnormal{The RHS of (\ref{E3})}\\ 
			& \le \mathbb{E}\left[\sup_{T_{k}^{\#} }\sum_{ (\boldsymbol{t}_{1}, \dots, \boldsymbol{t}_{k}) \in T_{k}^{\#} } \mathbb{P}(\boldsymbol{X}\in\boldsymbol{t}_{k}) \left(  \mathbb{E}(m(\boldsymbol{X}) \ \vert \ \boldsymbol{X} \in \boldsymbol{t}_{k}) - \frac{\sum_{i\in \{i : \boldsymbol{x}_{i} \in \boldsymbol{t}_{k }\}} y_{i}} {\#\{i : \boldsymbol{x}_{i} \in \boldsymbol{t}_{k }\}}  \right)^{2}\right].
		\end{split}
	\end{equation}
	Notice that there is no $\Theta$ on the RHS of (\ref{E7}), and hence the outer expectation is only over the sample.  In what follows, we bound the RHS of (\ref{E7}). The argument is based on the event $\mathcal{A}_{1}(k, \Delta)$ introduced in Lemma~\ref{CI1}  in Section \ref{SecE.1}, which in turn relies on Hoeffding's inequality. On such event, for each cell $\boldsymbol{t}$ on the grid constructed with at most $k$ cuts and satisfying $\mathbb{P}(\boldsymbol{X} \in \boldsymbol{t}) \ge n^{\Delta -1}$, the deviation between its sample and 
	population conditional means can be controlled. 
	
	Let $\Delta>0$, $\Delta^{'}>0$, and sufficiently small $0 < s <\frac{1}{4}$ be given such that $\frac{1}{2} <\Delta^{'}< \Delta < 1$ and
	\begin{equation}\label{eta1}
		\eta < \min\{1 - c - \Delta - 2s , \frac{\Delta^{'}}{2} \},
	\end{equation}
	where $\eta, c$ are given by Theorem~\ref{theorem5}. Moreover, let $\delta$ be such that $\delta -2s >  \frac{\Delta^{'}}{2}$. Assume that the moment condition parameter $q$ in Condition~\ref{ME} is sufficiently large with $q > \frac{5 + 2\delta}{s}$ and define 
	\[\mathcal{E}_{n,k}\coloneqq\sup_{T_{k}^{\#}}\sum_{(\boldsymbol{t}_{1}, \dots, \boldsymbol{t}_{k}) \in T_{k}^{\#}} \mathop{{}\mathbb{P}}(\boldsymbol{X} \in \boldsymbol{t}_{k})   \left(\mathop{{}\mathbb{E}}(m(\boldsymbol{X})\ \vert \ \boldsymbol{X}\in\boldsymbol{t}_{k}) - \sum_{\boldsymbol{x}_{i}\in\boldsymbol{t}_{k}} \frac{y_{i}}{ \#\{ i:\boldsymbol{x}_{i} \in \boldsymbol{t}_{k}\}}\right)^{2}.\]
	Then the RHS of (\ref{E7})  can be rewritten as 
	\begin{equation}\begin{split}\label{E4}
			\mathbb{E}\left(\mathcal{E}_{n,k} \boldsymbol{1}_{\cup_{i= 1}^{n} \{|\varepsilon_{i}| > n^{s} \}} \right) + \mathop{{}\mathbb{E}}\left(\mathcal{E}_{n,k} \boldsymbol{1}_{\cap_{i= 1}^{n} \{|\varepsilon_{i}| \le n^{s} \}} \right).
		\end{split}
	\end{equation}

	Let us bound the first term in (\ref{E4}). By Condition~\ref{BO}, which requires $\sup_{\boldsymbol{c} \in [0, 1]^{p}}|m(\boldsymbol{c})|\le M_{0}$, a simple upper bound for $\mathcal{E}_{n, k}$ is given by 
	\begin{equation}\label{u1}
		\mathcal{E}_{n, k} \le \Big(M_{0} +\sum_{i=1}^{n} |y_{i}|\Big)^{2}
	\end{equation}
	for each $n\ge 1$ and $k\ge 1$. 
	It follows from the Cauchy--Schwarz inequality, (\ref{u1}), Minkowski's inequality, Conditions~\ref{ME}--\ref{BO}, and the definition of $\delta$ that there exists some constant $C > 0$ such that for each $n \ge 1$ and $k\ge 1$,
	\begin{equation}\label{upper1}
		\begin{split}
			\mathop{{}\mathbb{E}}\left(\mathcal{E}_{n,k} \boldsymbol{1}_{\cup_{i= 1}^{n} \{|\varepsilon_{i}| > n^{s} \}} \right) & \le \sqrt{\mathop{{}\mathbb{E}} (\mathcal{E}_{n,k}^{2}) } \sqrt{\mathop{{}\mathbb{P}} \Big(\cup_{i= 1}^{n} \{|\varepsilon_{i}| > n^{s} \}\Big) } \\
			&  \le \sqrt{\mathop{{}\mathbb{E}} \Big( M_{0} + \sum_{i=1}^{n}|y_{i} |\Big)^{4} } \sqrt{\sum_{i=1}^{n}\mathop{{}\mathbb{P}} \Big( |\varepsilon_{i}| > n^{s} \Big) }\\
			& \le \Big(M_{0} + \sum_{i=1}^{n} \Big(\mathop{{}\mathbb{E}}|y_{i}|^{4}\Big)^{1/4}   \Big)^{2} \sqrt{\sum_{i=1}^{n}\mathop{{}\mathbb{P}} \Big( |\varepsilon_{i}| > n^{s} \Big) }\\
			& \le \Big((n + 1) M_{0} + n  (\mathbb{E}|\varepsilon_{1}|^{4})^{1/4}   \Big)^{2} \sqrt{\sum_{i=1}^{n}\mathop{{}\mathbb{P}} \Big( |\varepsilon_{i}| > n^{s} \Big) }\\
			& \le C n^{-\delta}.
		\end{split}
	\end{equation}
	
	We next deal with the second term in (\ref{E4}).  Let us define for each $\boldsymbol{t}$,
	\begin{equation*}\begin{split} 
			E_{\boldsymbol{t}, n}& \coloneqq \mathop{{}\mathbb{E}}(m(\boldsymbol{X})\ \vert \ \boldsymbol{X}\in\boldsymbol{t}) - \sum_{\boldsymbol{x}_{i}\in\boldsymbol{t}} \frac{y_{i}}{ \#\{ i:\boldsymbol{x}_{i} \in \boldsymbol{t}\}},\\			
			\mathcal{E}_{n,k}^{\dagger} & \coloneqq \sup_{T_{k}^{\#}}\sum_{ \substack{(\boldsymbol{t}_{1}, \dots, \boldsymbol{t}_{k}) \in T_{k}^{\#}, \\ \mathop{{}\mathbb{P}}(\boldsymbol{X} \in \boldsymbol{t}_{k}) \ge   n^{\Delta-1}} } \mathbb{P}(\boldsymbol{X} \in \boldsymbol{t}_{k})   \Big( E_{\boldsymbol{t}_{k}, n}  \Big)^{2}.
	\end{split}\end{equation*}
	We can make three useful observations:
	\begin{enumerate}\setlength{\itemsep}{-10pt} %
		\item[1)] Under Condition~\ref{BO} and on the event $\cap_{i= 1}^{n} \{|\varepsilon_{i}| \le n^{s} \}$, it holds that for each $\boldsymbol{t}$, all large $n$, and each $k \ge 1$, 
		$$\left(E_{\boldsymbol{t}, n}\right)^{2} \le 2n^{2s}.$$
		
		\item[2)] On the event $\mathcal{A}_{1}(k, \Delta)$, it holds that for all large $n$ and each $1 \le k \le c\log{n}$,
		\[\mathcal{E}_{n, k}^{\dagger} \le n^{-\frac{\Delta^{'}}{2}}.\]
		\item[3)] For each $n\ge 1$ and $k\ge 1$,
		\[\mathcal{E}_{n, k}^{\dagger} \le \sup_{\boldsymbol{t}} (E_{\boldsymbol{t}, n})^{2},\]
		where the supremum is over all possible cells.
	\end{enumerate}
	By observation 1) above and the definition of $\mathcal{E}_{n, k}^{\dagger}$, we have that for all large $n$ and each $1\le k \le c\log_{2}{n}$, 
	\begin{equation}
		\begin{split}\label{E5}
			& \mathbb{E} \left(\mathcal{E}_{n, k} \boldsymbol{1}_{\cap_{i= 1}^{n} \{|\varepsilon_{i}| \le n^{s} \} } \right)  \\
			& = \mathbb{E}\bigg(
			\sup_{T_{k}^{\#}}\bigg( \sum_{ \substack{(\boldsymbol{t}_{1}, \dots, \boldsymbol{t}_{k}) \in T_{k}^{\#}, \\ \mathop{{}\mathbb{P}}(\boldsymbol{X} \in \boldsymbol{t}_{k}) < n^{\Delta-1}} } \mathbb{P}(\boldsymbol{X} \in \boldsymbol{t}_{k})   \Big( E_{\boldsymbol{t}_{k}, n}  \Big)^{2} \\
			&\qquad\qquad+  \sum_{ \substack{(\boldsymbol{t}_{1}, \dots, \boldsymbol{t}_{k}) \in T_{k}^{\#}, \\ \mathbb{P}(\boldsymbol{X} \in \boldsymbol{t}_{k}) \ge   n^{\Delta-1}} } \mathbb{P}(\boldsymbol{X} \in \boldsymbol{t}_{k})   \Big( E_{\boldsymbol{t}_{k}, n}  \Big)^{2} \bigg)
			\boldsymbol{1}_{\cap_{i= 1}^{n} \{|\varepsilon_{i}| \le n^{s} \} } \bigg) \\
			& \le 2n^{c + \Delta + 2s -1 } + \mathbb{E}\left(\mathcal{E}_{n, k}^{\dagger}    \boldsymbol{1}_{\cap_{i= 1}^{n} \{|\varepsilon_{i}| \le n^{s} \} }\right),
		\end{split}
	\end{equation}
	where  the first term on the RHS of the inequality follows from the fact that the summation is over at most $2^{c\log_{2}{n}}$ cells.
	
	From the three observations above and Lemma~\ref{CI1} (with $\kappa$ in Lemma~\ref{CI1} set to $\delta$), we can deduce that for all large $n$ and each $1\le k\le c\log{n}$,
	\begin{equation}\begin{split}\label{E6}
			&\mathbb{E}  \left(\mathcal{E}_{n, k}^{\dagger}    \boldsymbol{1}_{\cap_{i= 1}^{n} \{|\varepsilon_{i}| \le n^{s} \} }\right) \\& =  \mathbb{E}\left(\mathcal{E}_{n, k}^{\dagger}\boldsymbol{1}_{\cap_{i= 1}^{n} \{|\varepsilon_{i}| \le n^{s} \} }  \boldsymbol{1}_{\big(\mathcal{A}_{1}(k, \Delta)\big)^{c} } \right)
			+ \mathbb{E} \left(\mathcal{E}_{n, k}^{\dagger} \boldsymbol{1}_{\cap_{i= 1}^{n} \{|\varepsilon_{i}| \le n^{s} \} } \boldsymbol{1}_{\mathcal{A}_{1}(k, \Delta)} \right) \\
			& \le  \mathbb{E} \left(\sup_{\boldsymbol{t}}(E_{\boldsymbol{t}, n})^{2} \boldsymbol{1}_{\cap_{i= 1}^{n} \{|\varepsilon_{i}| \le n^{s} \} }  \boldsymbol{1}_{\big(\mathcal{A}_{1}(k, \Delta)\big)^{c} } \right)
			+ \mathop{{}\mathbb{E}}\left(\mathcal{E}_{n, k}^{\dagger} \boldsymbol{1}_{\mathcal{A}_{1}(k, \Delta)} \right) \\
			& \le 2n^{2s} \mathop{{}\mathbb{P}}\left(\Big(\mathcal{A}_{1}(k, \Delta)\Big)^{c}\right)
			+ \mathop{{}\mathbb{E}}\left(\mathcal{E}_{n, k}^{\dagger}  \boldsymbol{1}_{\mathcal{A}_{1}(k, \Delta)} \right) \\
			& \le 3n^{-\frac{\Delta^{'}}{2}},
	\end{split}\end{equation}
where for the last inequality, recall that $\delta - 2s > \frac{\Delta^{'}}{2}$.

	Then in light of (\ref{eta1}) and (\ref{E5})--(\ref{E6}), it holds that for all large $n$ and each $1 \le k \le c\log{n}$, 
	\begin{equation}\label{upper2}
		\mathop{{}\mathbb{E}}\left(\mathcal{E}_{n, k} \boldsymbol{1}_{\cap_{i= 1}^{n} \{|\varepsilon_{i}| \le n^{s} \}} \right) \le  \frac{n^{-\eta}}{2}. \end{equation}
	Therefore, combining (\ref{E3})--(\ref{E7}), (\ref{E4}), (\ref{upper1}), (\ref{upper2}), and that $\delta>\eta$ completes the proof of Theorem~\ref{theorem5}.
	

	\section{Proofs of Corollaries~\ref{corollary1}--\ref{corollary2}, Proposition~\ref{proposition2}, and some key lemmas}
	\label{SecC}
	
	\subsection{Proof of Corollary~\ref{corollary1}}
	\label{SecC.1}
	The arguments for showing \eqref{main4} and \eqref{main5} in Corollary~\ref{corollary1} can be found in \eqref{t.new.eq.1} and \eqref{t.new.eq.3} in Section~\ref{SecA.2}, respectively.
	
	\subsection{Proof of Corollary~\ref{corollary2}}
	\label{SecC.5}

	First, we set $\eta = \frac{1}{8} - \epsilon$, $\delta = \frac{1}{4} - \epsilon$, $c = \frac{1}{8}$, and $k= \floor{\frac{1}{8}\log_{2}(n)}$ in Theorem~\ref{theorem1}.    Since $e^x\ge (1-\frac{x}{n})^{n}$ for $0\le x\le n$,  it holds that  for all large $n$,	$$(1-\gamma_{0}(\alpha_{1}\alpha_{2})^{-1})^{k}\le e^{-\frac{k\gamma_{0}}{\alpha_{1}\alpha_{2}} }\le  2^{ -\frac{\frac{1}{8}\log_{2}(e)\gamma_{0}}{\alpha_{1}\alpha_{2}}\log_{2}(n) }\times e^{\frac{\gamma_{0}}{\alpha_{1}\alpha_{2}} } = n^{-\frac{\log_{2}(e)}{8}\times\frac{\gamma_{0}}{\alpha_{1}\alpha_{2}}}\times e^{\frac{\gamma_{0}}{\alpha_{1}\alpha_{2}} }.$$
	
	By this, Theorem~\ref{theorem1}, we can show that there exist $N>0$ and $C>0$ such that for any $m(\boldsymbol{X})$ satisfies Condition~\ref{P} with $\alpha_{1}$ and all $n\ge N$,
	\begin{equation}\begin{split}
			\label{corollary2.1}
			& \mathbb{E} \Big(m(\boldsymbol{X}) - \mathbb{E}\Big(\widehat{m}_{ \widehat{T}}(\boldsymbol{\Theta}_{1:k}, \boldsymbol{X}, \mathcal{X}_{n} )\ \Big\vert \ \boldsymbol{X}, \mathcal{X}_{n} \Big) \Big)^{2} \le C\left(n^{-\frac{1}{8} +  \epsilon} + n^{-\frac{\log_{2}(e)}{8}\times\frac{\gamma_{0}}{\alpha_{1}\alpha_{2}}}\right).
	\end{split}\end{equation}
	
	To obtain \eqref{corollary2.1}, we note that the results in Lemmas~\ref{lemma1}--\ref{lemma3} and Theorems~\ref{theorem4}--\ref{theorem5} can be shown to be uniform over all $m(\boldsymbol{X})$ satisfying the respective requirements of these results. Particularly, the result in Theorem~\ref{theorem3} is already for all $m(\boldsymbol{X}) \in \textnormal{SID}(\alpha_{1})$. For simplicity, we omit the detailed analysis for \eqref{corollary2.1}. 
	
	By \eqref{corollary2.1} and the definition of SID($\alpha_{1}$), we conclude the desired result.

	\subsection{Proof of Proposition~\ref{proposition2}}
	\label{SecC.6}

	Let us deal with the first assertion first. A direct calculation shows that for every $\boldsymbol{t}=t_{1} \times \dots \times t_{p}$,
	$$	\begin{cases}
		\sup_{j\in \{1, \dots, p\}, c\in t_{j}}(II)_{\boldsymbol{t}, \boldsymbol{t}(j,c)} = \frac{\beta^2}{4},  &\textnormal{ if } \  \ \textnormal{Var}(m(\boldsymbol{X})|\boldsymbol{X} \in \boldsymbol{t})  >0,\\
		\sup_{j\in \{1, \dots, p\}, c\in t_{j}}(II)_{\boldsymbol{t}, \boldsymbol{t}(j,c)} = 0,   &\textnormal{ if } \ \ \textnormal{Var}(m(\boldsymbol{X})|\boldsymbol{X} \in \boldsymbol{t})  =0,
	\end{cases} $$
	and that $\textnormal{Var}(m(\boldsymbol{X})|\boldsymbol{X} \in \boldsymbol{t}) \le  s^* \frac{\beta^2}{4}$, which concludes that $m(\boldsymbol{X})\in  \textnormal{SID}(s^*)$.

	Next, we proceed to deal with the second assertion, and we begin with the bias-variance decomposition upper bound in \eqref{upper.new.1} and some details for  CART in \eqref{new.cart.split.1} below. By Jensen's inequality and triangular inequality,
	\begin{equation}\begin{split}\label{upper.new.1}
			& \mathop{{}\mathbb{E}} \Big(m(\boldsymbol{X}) - \mathop{{}\mathbb{E}}\Big(\widehat{m}_{ \widehat{T}}(\boldsymbol{\Theta}_{1}, \dots, \boldsymbol{\Theta}_{k}, \boldsymbol{X}, \mathcal{X}_{n} )\ \Big\vert \ \boldsymbol{X}, \mathcal{X}_{n} \Big) \Big)^{2} \\
			& \le 2\Big( \mathbb{E} \Big(m(\boldsymbol{X})  -m_{ \widehat{T} }^{*} (\boldsymbol{\Theta}_{1}, \dots, \boldsymbol{\Theta}_{k} , \boldsymbol{X})\Big)^{2} \\
			& \qquad+  \mathbb{E}\Big(m_{ \widehat{T} }^{*} (\boldsymbol{\Theta}_{1}, \dots, \boldsymbol{\Theta}_{k} , \boldsymbol{X})- \widehat{m}_{ \widehat{T}}(\boldsymbol{\Theta}_{1}, \dots, \boldsymbol{\Theta}_{k} , \boldsymbol{X}, \mathcal{X}_{n} ) \Big)^{2}  \Big).
	\end{split}\end{equation}
	
	As have mentioned in the remark before Proposition~\ref{proposition2}, for each feature restriction $\Theta$ and cell $\boldsymbol{t}$, the sample CART split in the case with binary features is $(\widehat{j}, 1)$ with
	\begin{equation}
	    \label{new.cart.split.1}
	    \widehat{j} \coloneqq \arg\max_{j\in \Theta} \widehat{(II)}_{\boldsymbol{t}, \boldsymbol{t}(j,1)},
	\end{equation} 
	where for $\boldsymbol{t}$ and its two daughter cells $\boldsymbol{t}_{1}, \boldsymbol{t}_{2}$,
 	\begin{equation*}
			\begin{split}
				\widehat{(II)}_{\boldsymbol{t}, \boldsymbol{t}_{1}} &= \left( \frac{\sum_{i=1}^{n}\boldsymbol{1}_{\boldsymbol{x}_{i} \in \boldsymbol{t}_{1}}}{ \sum_{i=1}^{n}\boldsymbol{1}_{\boldsymbol{x}_{i} \in \boldsymbol{t}} }\right) \left(\frac{\sum_{i=1}^{n}\boldsymbol{1}_{\boldsymbol{x}_{i} \in \boldsymbol{t}_{1}} (m(\boldsymbol{x}_{i}) + \varepsilon_{i}) }{\sum_{i=1}^{n}\boldsymbol{1}_{\boldsymbol{x}_{i} \in \boldsymbol{t}_{1}}} - \frac{\sum_{i=1}^{n}\boldsymbol{1}_{\boldsymbol{x}_{i} \in \boldsymbol{t}} (m(\boldsymbol{x}_{i}) + \varepsilon_{i}) }{\sum_{i=1}^{n}\boldsymbol{1}_{\boldsymbol{x}_{i} \in \boldsymbol{t}}}  \right)^2 \\
				& +\left( \frac{\sum_{i=1}^{n}\boldsymbol{1}_{\boldsymbol{x}_{i} \in \boldsymbol{t}_{2}}}{ \sum_{i=1}^{n}\boldsymbol{1}_{\boldsymbol{x}_{i} \in \boldsymbol{t}} }\right) \left(\frac{\sum_{i=1}^{n}\boldsymbol{1}_{\boldsymbol{x}_{i} \in \boldsymbol{t}_{2}} (m(\boldsymbol{x}_{i}) + \varepsilon_{i}) }{\sum_{i=1}^{n}\boldsymbol{1}_{\boldsymbol{x}_{i} \in \boldsymbol{t}_{2}}} - \frac{\sum_{i=1}^{n}\boldsymbol{1}_{\boldsymbol{x}_{i} \in \boldsymbol{t}} (m(\boldsymbol{x}_{i}) + \varepsilon_{i}) }{\sum_{i=1}^{n}\boldsymbol{1}_{\boldsymbol{x}_{i} \in \boldsymbol{t}}}  \right)^2,
			\end{split}
		\end{equation*}
	and the ties are broken randomly; the definition of splits here is the same as the one given in the remark before Proposition~\ref{proposition2}. 
	
	Additional remarks for splitting in this case with binary features are as follows. Due to the definition that $\frac{0}{0}=0$, for any trivial split $(j, c)$, which is a split gives a daughter cell $\boldsymbol{t}^{'}$ with $\mathbb{P}(\boldsymbol{X} \in \boldsymbol{t}^{'}) = 0$, it holds that  $\widehat{(II)}_{\boldsymbol{t}, \boldsymbol{t}(j,c)} = 0$ and $(II)_{\boldsymbol{t}, \boldsymbol{t}(j,c)} = 0$. If all coordinates in some $\Theta$ have already been split, the CART stops splitting; to have well-defined level $k$ trees, we allow CART to make trivial splits that give empty sets as daughter cells, and we define daughter cells of an empty set to be two empty sets. As a result, $\widehat{T}(\Theta_{1:k})$ may contain empty end cells. 
	
	To bound the two terms on the RHS of \eqref{upper.new.1}, our first step is to show that the the sample CART split $(\widehat{j}, 1)$ for each cell $\boldsymbol{t}$ is ``very close to'' its theoretical CART split counterpart $(j^*, c^*) = \arg\sup_{j\in \Theta, c\in t_{j}} (II)_{\boldsymbol{t}, \boldsymbol{t}(j,c)}$, in the sense as in iii) of Lemma~\ref{upper.lemma.1} below. To get Lemma~\ref{upper.lemma.1}, we need an event $U_{n}$ defined as follows.

	Denote by $G_{n}$ the collection of all end cells of trees of level lower than $\log_{2}(n)$, and the cells are formed by using the splits $(1, 1), \dots (p,1)$. A direct calculation shows that 
	\begin{equation}\label{upper.new.2}
		\# G_{n}\le \sum_{k=0}^{\floor{\log_{2}(n)}} {p\choose k} 2^k\le 1 + p^{\log_{2}(n)} n\log_{2}(n).
	\end{equation}

	Define events
	\begin{equation*}
		\begin{split}
			Q_{1}(\boldsymbol{t}) &= \Big\{\Big|\mathbb{E} ( m(\boldsymbol{X})\boldsymbol{1}_{\boldsymbol{X} \in \boldsymbol{t}}) - n^{-1}\sum_{i=1}^{n} \boldsymbol{1}_{\boldsymbol{x}_{i} \in \boldsymbol{t}} (m(\boldsymbol{x}_{i}) + \varepsilon_{i})\Big|\\
			& \qquad\qquad\le \big(\log_{e}({\max\{n, p\}})\big)^{\frac{2+\epsilon}{2}}\sqrt{2M_{0}^2+M_{\varepsilon}^2}\sqrt{\frac{\mathbb{P}(\boldsymbol{X} \in \boldsymbol{t})}{n}}\Big\},\\
			Q_{2}(\boldsymbol{t}) &= \Big\{\Big|\mathbb{P}(\boldsymbol{X}\in\boldsymbol{t}) - n^{-1}\sum_{i=1}^{n} \boldsymbol{1}_{\boldsymbol{x}_{i} \in \boldsymbol{t}}  \Big| \le \big(\log_{e}({\max\{n, p\}})\big)^{\frac{2+\epsilon}{2}}\sqrt{\frac{\mathbb{P}(\boldsymbol{X} \in \boldsymbol{t})}{n}}\Big\},\\
			U_{n} &= \big(\cap_{\boldsymbol{t}\in G_{n}} Q_{1}(\boldsymbol{t}) \big)\cap \big(\cap_{\boldsymbol{t}\in G_{n}} Q_{2}(\boldsymbol{t})\big).
		\end{split}
	\end{equation*}

	Note that $U_{n}$ depends only on the training data $\mathcal{X}_{n}$ and is independent of $\boldsymbol{X}$, which is the independent copy of $\boldsymbol{x}_{1}$. It holds that 
	\begin{equation}
		\begin{split}			
			\label{upper.new.4}
			\mathbb{P}(U_{n}^{c}) = o(n^{-1}),
		\end{split}
	\end{equation}
	whose proof is deferred to the end of the proof of Proposition~\ref{proposition2}.

	With event $U_{n}$, we introduce Lemma~\ref{upper.lemma.1} below, whose proof is also deferred to the end of the proof of Proposition~\ref{proposition2}. 
	\begin{lemma}\label{upper.lemma.1}
		\textnormal{i)} For every $\boldsymbol{t}$ and every split $(j, c)$, it is either $(II)_{\boldsymbol{t}, \boldsymbol{t}(j, c)} = \frac{\beta^2}{4}$ or $(II)_{\boldsymbol{t}, \boldsymbol{t}(j, c)} = 0$. In addition, on $U_{n}$, for all large $n$, it holds that for every end cell $\boldsymbol{t}$ of trees of level $k\le \eta\log_{2}(n)-1$,
		\begin{enumerate}
			\item[\textnormal{ii)}]  For every $1\le j\le p$, 
			\begin{equation*}
			    \begin{split}
					& \left|\left(\frac{\sum_{i=1}^{n}\boldsymbol{1}_{\boldsymbol{x}_{i} \in \boldsymbol{t}}}{n}\right)\widehat{(II)}_{\boldsymbol{t}, \boldsymbol{t}(j,1)} - \mathbb{P}(\boldsymbol{X} \in \boldsymbol{t}) (II)_{\boldsymbol{t}, \boldsymbol{t}(j, 1)}\right| \\
					& \le 18(M_{\varepsilon} + 2M_{0} )^2 \big(\log_{e}({\max\{n, p\}})\big)^{\frac{2+\epsilon}{2}}\sqrt{\frac{\mathbb{P}(\boldsymbol{X} \in \boldsymbol{t})}{n}}.
				\end{split}
			\end{equation*}
			\item[\textnormal{iii)}] For each feature restriction $\Theta$ and the sample CART split $(\widehat{j}, 1)$ given in \eqref{new.cart.split.1},
			$$(II)_{\boldsymbol{t}, \boldsymbol{t}(\widehat{j}, 1)}= \sup_{j\in \Theta, c\in t_{j}}(II)_{\boldsymbol{t}, \boldsymbol{t}(j, c)}.$$
		\end{enumerate}
	\end{lemma}

	Now, we deal with the two terms on the RHS of \eqref{upper.new.1}, and begin with the first term. By the specific model setting assumed here, we have
	{\small\begin{equation}
			\begin{split}\label{upper.new.p.7}
				& \mathbb{E} \Big(m(\boldsymbol{X})  -m_{ \widehat{T} }^{*} (\boldsymbol{\Theta}_{1}, \dots, \boldsymbol{\Theta}_{k} , \boldsymbol{X})\Big)^{2}  \\
				& \le \mathbb{E} \Big((m(\boldsymbol{X})  -m_{ \widehat{T} }^{*} (\boldsymbol{\Theta}_{1}, \dots, \boldsymbol{\Theta}_{k} , \boldsymbol{X}))^2\boldsymbol{1}_{U_{n}} \Big) + \mathbb{E} \Big((m(\boldsymbol{X})  -m_{ \widehat{T} }^{*} (\boldsymbol{\Theta}_{1}, \dots, \boldsymbol{\Theta}_{k} , \boldsymbol{X}))^2\boldsymbol{1}_{U_{n}^c}\Big)\\
				& \le \mathbb{E} \Big((m(\boldsymbol{X})  -m_{ \widehat{T} }^{*} (\boldsymbol{\Theta}_{1}, \dots, \boldsymbol{\Theta}_{k} , \boldsymbol{X}))^2\boldsymbol{1}_{U_{n}} \Big) + 4M_{0}^2\mathbb{P} (U_{n}^{c}),
			\end{split}
	\end{equation}}
	
		\noindent where $m_{ T }^{*} (\boldsymbol{\Theta}_{1}, \dots, \boldsymbol{\Theta}_{k} , \boldsymbol{X})$ is defined  to be $\mathbb{E}(m(\boldsymbol{X}))$ when $k=0$, which is a trivial theoretical random forests model.

		To further deal with the first term on the RHS of \eqref{upper.new.p.7}, let us define $T^*(\Theta_{1:k})$ to be a tree of level $k$ grown by theoretical CART splits with sets of available features specified as in $\Theta_{1:k}$. We want to make a connection between $\widehat{T}$ and $T^*$ as in \eqref{upper.new.p.8} below. However, because theoretical CART and sample CART split the cells differently, it is unclear whether the equality in \eqref{upper.new.p.8} holds if ties are broken randomly. To ensure such an equality, we additionally require that for all large $n$ and $0\le k \le \eta\log_{2}(n)$, the theoretical CART breaks ties such that 
		$$(m(\boldsymbol{X})  -m_{ T^* }^{*} (\boldsymbol{\Theta}_{1}, \dots, \boldsymbol{\Theta}_{k} , \boldsymbol{X}))^2\boldsymbol{1}_{U_{n}} = (m(\boldsymbol{X})  -m_{ \widehat{T} }^{*} (\boldsymbol{\Theta}_{1}, \dots, \boldsymbol{\Theta}_{k} , \boldsymbol{X}))^2\boldsymbol{1}_{U_{n}},$$
		which is possible because of iii) of Lemma~\ref{upper.lemma.1}. Therefore,
	\begin{equation}
	\begin{split}\label{upper.new.p.8}
	&\mathbb{E} \Big((m(\boldsymbol{X})  -m_{ \widehat{T} }^{*} (\boldsymbol{\Theta}_{1}, \dots, \boldsymbol{\Theta}_{k} , \boldsymbol{X}))^2\boldsymbol{1}_{U_{n}} \Big) \\
	& = \mathbb{E} \Big((m(\boldsymbol{X})  -m_{ T^* }^{*} (\boldsymbol{\Theta}_{1}, \dots, \boldsymbol{\Theta}_{k} , \boldsymbol{X}))^2\boldsymbol{1}_{U_{n}}\Big)\\
	& \le \mathbb{E} \Big(m(\boldsymbol{X})  -m_{ T^* }^{*} (\boldsymbol{\Theta}_{1}, \dots, \boldsymbol{\Theta}_{k} , \boldsymbol{X})\Big)^{2}.
	\end{split}
	\end{equation}

	To deal with the first term on the RHS of \eqref{upper.new.p.8}, we need \eqref{upper.new.p.1} below. For each $k\ge 0$,
	\begin{equation}\label{upper.new.p.1}
		\mathbb{E} \Big(m(\boldsymbol{X})  -m_{ T^* }^{*} (\boldsymbol{\Theta}_{1}, \dots, \boldsymbol{\Theta}_{k} , \boldsymbol{X})\Big)^{2}\le (1-\gamma_{0}(s^*)^{-1})^k \textnormal{Var} (m(\boldsymbol{X})).
	\end{equation}
	In addition, if it is known that $\gamma_0 = 1$, a sharp squared bias upper bound can be obtained by
	\begin{equation}\label{upper.new.p.10}
		\mathbb{E} \Big(m(\boldsymbol{X})  -m_{ T^* }^{*} (\boldsymbol{\Theta}_{1}, \dots, \boldsymbol{\Theta}_{k} , \boldsymbol{X})\Big)^{2}\le \max\left\{(s^*-k)\frac{\beta^2}{4}, 0\right\}
	\end{equation}
	for each $k\ge 0$. On the other hand, the second term on the RHS of \eqref{upper.new.1} is bounded by \eqref{upper.new.p.6} below. For each $0\le k\le \eta\log_{2}(n)$,
	\begin{equation}
		\begin{split}\label{upper.new.p.6}
			&\mathbb{E}\Big(m_{ \widehat{T} }^{*} (\boldsymbol{\Theta}_{1:k} , \boldsymbol{X})- \widehat{m}_{ \widehat{T}}(\boldsymbol{\Theta}_{1:k} , \boldsymbol{X}, \mathcal{X}_{n} ) \Big)^{2} \\
			&\le  ( 3M_{0} + 2M_{\varepsilon})^2\frac{2^{k}\big(\log_{e}({\max\{n, p\}})\big)^{2+\epsilon}}{n} + (2M_{0} + M_{\varepsilon})^2\mathbb{P}(U_{n}^c).
		\end{split}
	\end{equation}
	The proofs of \eqref{upper.new.p.1}--\eqref{upper.new.p.6} are deferred to the end of the proof of Proposition~\ref{proposition2}. We also give some intuition in Remark~\ref{gamma.intuition} in the proof of \eqref{upper.new.p.6} for improving the estimation upper bound. The results of \eqref{upper.new.1} and \eqref{upper.new.4}--\eqref{upper.new.p.6} lead to the desired result, and hence we have finished the proof.

	Let us give a closing remark. It is seen that both estimation variance and squared bias analyses rely on the event $U_{n}$. A general version of such an event is used for random forests analysis for general cases, and the technique is called ``the grid.'' A brief introduction of the grid for general cases, which is far more complicated than the simple case here, can be found in Section~\ref{Sec5}. In addition, the bias analysis \eqref{upper.new.p.7}--\eqref{upper.new.p.1} is a simple version of the one in Section~\ref{Sec4}. As remarked after Proposition~\ref{proposition2}, the general bias analysis depends on the sample size since the optimal split on each coordinate is unknown and features are dependent.

	\begin{proof}[Proof of \eqref{upper.new.4}]
	    
	To bound the probabilities of the complements of the events $Q_{1}(\boldsymbol{t}), Q_{2}(\boldsymbol{t}), U_{n}$ with concentration inequalities, we need the variance upper and lower bounds and \eqref{binary.6} below. For each cell $\boldsymbol{t}$,
	\begin{equation}
		\begin{split}\label{binary.2}
			\textnormal{Var}(\boldsymbol{1}_{\boldsymbol{x}_{1} \in \boldsymbol{t}}) &= \mathbb{P}(\boldsymbol{X} \in \boldsymbol{t})(1-\mathbb{P}(\boldsymbol{X} \in \boldsymbol{t})),\\
			&\\
			\textnormal{Var}(\varepsilon_{1})\mathbb{P}(\boldsymbol{X} \in \boldsymbol{t})&\le \textnormal{Var}(\boldsymbol{1}_{\boldsymbol{x}_{1} \in \boldsymbol{t}} (m(\boldsymbol{x}_{1}) + \varepsilon_{1})) \\
			&= \textnormal{Var}(\boldsymbol{1}_{\boldsymbol{x}_{1} \in \boldsymbol{t}} m(\boldsymbol{x}_{1})) +  \textnormal{Var}(\varepsilon_{1})\mathbb{P}(\boldsymbol{X} \in \boldsymbol{t})\\
			&\le (2M_{0}^2 + M_{\varepsilon}^2 )\mathbb{P}(\boldsymbol{X} \in \boldsymbol{t}).
		\end{split}
	\end{equation}
	Since $t$ is constructed by at most $k\le \eta \log_{2}(n)$ cuts,
    \begin{equation}
        \label{binary.6}\mathbb{P}(\boldsymbol{X} \in \boldsymbol{t}) \ge n^{-\eta}.
    \end{equation}

	By Bernstein's inequality~\cite{bennett1962probability, bernstein1924modification}, \eqref{binary.2}, the assumptions of i.i.d. observations and bounded regression function and model errors, \eqref{binary.6} and that $(\log_{e}p)^{2+\epsilon} = o(n^{1-\eta})$, it holds that for all large $n$ and every $\boldsymbol{t}$,
	\begin{equation}\label{binary.3}
		\mathbb{P}((Q_{1}(\boldsymbol{t}))^c) \le 2\exp{\left(\frac{-(\log_{e}({\max\{n, p\}}))^{2+\epsilon}}{3}\right)},
	\end{equation}
	and if $\mathbb{P}(\boldsymbol{X} \in \boldsymbol{t}) <1$, for all large $n$,
	\begin{equation}\label{binary.4}
		\mathbb{P}((Q_{2}(\boldsymbol{t}))^c) \le 2\exp{\left(\frac{-(\log_{e}({\max\{n, p\}}))^{2+\epsilon}}{3}\right)},
	\end{equation}
	and if $\mathbb{P}(\boldsymbol{X} \in \boldsymbol{t}) =1$, for all $n\ge 1$,
	\begin{equation}\label{binary.5}
		\mathbb{P}((Q_{2}(\boldsymbol{t}))^c) =0,
	\end{equation}
	since $\mathbb{P}(\boldsymbol{X}\in\boldsymbol{t}) = n^{-1}\sum_{i=1}^{n} \boldsymbol{1}_{\boldsymbol{x}_{i} \in \boldsymbol{t}} = 1$.

	By \eqref{upper.new.2} and \eqref{binary.3}--\eqref{binary.5} and the assumptions of i.i.d. observations, a bounded regression function, and bounded model errors, it holds that
	\begin{equation}
		\begin{split}			
			\mathbb{P}(U_{n}^{c}) & \le  2\times \big(1 + p^{\log_{2}(n)} n\log_{2}(n)\big)\times 2\exp{\left(-\frac{1}{3}(\log_{e}({\max\{n, p\}}))^{2+\epsilon}\right)} \\
			&= o(n^{-1}),
		\end{split}
	\end{equation}
	which finishes the proof.
	
	\end{proof}

	\begin{proof}[Proof of Lemma~\ref{upper.lemma.1}]
	The first assertion can be shown by a direct calculation and hence we omit the detail. Let a feature restriction $\Theta$ be given. The third assertion is a result of the first two assertions and that 
		\begin{equation}
			\label{upper.new.19}
			\mathbb{P}(\boldsymbol{X} \in \boldsymbol{t}) \ge n^{-\eta},
		\end{equation}
		which is due to that $\boldsymbol{t}$ is constructed by $k\le \eta\log_{2}(n)-1$ cuts.
		Specifically, suppose the first two assertions hold and let $(j^*, c^*) = \arg\sup_{j\in \Theta, c\in t_{j}}(II)_{\boldsymbol{t}, \boldsymbol{t}(j, c)}$ and $j^{\dagger}$ such that $(II)_{\boldsymbol{t}, \boldsymbol{t}(j^{\dagger}, 1)} = 0$ be given. If $\sup_{j\in \Theta, c\in t_{j}} (II)_{\boldsymbol{t}, \boldsymbol{t}(j, c)} = 0$, the desired result is obviously true. Suppose otherwise $\sup_{j\in \Theta, c\in t_{j}} (II)_{\boldsymbol{t}, \boldsymbol{t}(j, c)} = \frac{\beta^2}{4}$ (by the first assertion). By the fact that features are binary, 
		\begin{equation*}
			\begin{split}
			(II)_{\boldsymbol{t}, \boldsymbol{t}(j^*,1)} =(II)_{\boldsymbol{t}, \boldsymbol{t}(j^*,c^*)},
			\end{split}
		\end{equation*}
		which in combination with $\sup_{j\in \Theta, c\in t_{j}} (II)_{\boldsymbol{t}, \boldsymbol{t}(j, c)} = \frac{\beta^2}{4}$ and the second assertion leads to
		\begin{equation*}
			\begin{split}
			&\left(\frac{\sum_{i=1}^{n}\boldsymbol{1}_{\boldsymbol{x}_{i} \in \boldsymbol{t}}}{n}\right)\widehat{(II)}_{\boldsymbol{t}, \boldsymbol{t}(j^*,1)} \\
				&\ge \frac{\beta^2}{4}\mathbb{P}(\boldsymbol{X}\in\boldsymbol{t}) - 18(M_{\varepsilon} + 2M_{0} )^2 \big(\log_{e}({\max\{n, p\}})\big)^{\frac{2+\epsilon}{2}}\sqrt{\frac{\mathbb{P}(\boldsymbol{X}\in\boldsymbol{t})}{n}}.
			\end{split}
		\end{equation*}
		
		Meanwhile, by $(II)_{\boldsymbol{t}, \boldsymbol{t}(j^{\dagger}, 1)} = 0$ and the second assertion,
		\begin{equation*}
			\begin{split}
			&18(M_{\varepsilon} + 2M_{0} )^2 \big(\log_{e}({\max\{n, p\}})\big)^{\frac{2+\epsilon}{2}}\sqrt{\frac{\mathbb{P}(\boldsymbol{X}\in\boldsymbol{t})}{n}} \ge \left(\frac{\sum_{i=1}^{n}\boldsymbol{1}_{\boldsymbol{x}_{i} \in \boldsymbol{t}}}{n}\right)\widehat{(II)}_{\boldsymbol{t}, \boldsymbol{t}(j^{\dagger},1)}.
			\end{split}
		\end{equation*}
		
		Combining these and \eqref{upper.new.19}, $(\log_{e}p)^{2+\epsilon} = o(n^{1-\eta})$, it holds that for all large $n$,
		\begin{equation*}
			\begin{split}
				& \left(\frac{\sum_{i=1}^{n}\boldsymbol{1}_{\boldsymbol{x}_{i} \in \boldsymbol{t}}}{n}\right)\widehat{(II)}_{\boldsymbol{t}, \boldsymbol{t}(j^*,1)} > \left(\frac{\sum_{i=1}^{n}\boldsymbol{1}_{\boldsymbol{x}_{i} \in \boldsymbol{t}}}{n}\right)\widehat{(II)}_{\boldsymbol{t}, \boldsymbol{t}(j^{\dagger},1)},
			\end{split}
		\end{equation*}
		which in combination with the fact that
		$$ \left(\frac{\sum_{i=1}^{n}\boldsymbol{1}_{\boldsymbol{x}_{i} \in \boldsymbol{t}}}{n}\right)\widehat{(II)}_{\boldsymbol{t}, \boldsymbol{t}(\widehat{j},1)}\ge \left(\frac{\sum_{i=1}^{n}\boldsymbol{1}_{\boldsymbol{x}_{i} \in \boldsymbol{t}}}{n}\right)\widehat{(II)}_{\boldsymbol{t}, \boldsymbol{t}(j^*,1)},$$ 
		which is due to $j^* \in \Theta$, implies that $\widehat{(II)}_{\boldsymbol{t}, \boldsymbol{t}(\widehat{j},1)} > \widehat{(II)}_{\boldsymbol{t}, \boldsymbol{t}(j^{\dagger}, 1)}$ for every such $(j^{\dagger}, 1)$. Therefore, we have $(II)_{\boldsymbol{t}, \boldsymbol{t}(\widehat{j},1)} > 0$ in this scenario. This together with the first assertion concludes the third assertion.

		In the following, we prove the second assertion. Let us consider a cell $\boldsymbol{t}$ and a split $(j, 1)$. If the $j$th coordinate has already been split, the desired result is obviously true. Therefore, we suppose the $j$th coordinate of $\boldsymbol{t}$ has not been split on. Let $\boldsymbol{t}_{1}$ and $\boldsymbol{t}_{2}$ denote the two daughter cells, respectively. Our goal is to deal with the difference
		\begin{equation}\label{upper.new.12}
			 \Big|\left(\frac{\sum_{i=1}^{n}\boldsymbol{1}_{\boldsymbol{x}_{i} \in \boldsymbol{t}}}{n}\right)\widehat{(II)}_{\boldsymbol{t}, \boldsymbol{t}(j,1)} - \mathbb{P}(\boldsymbol{X} \in \boldsymbol{t}) (II)_{\boldsymbol{t}, \boldsymbol{t}(j, 1)}\Big|,
		\end{equation}
		where 
		{\small\begin{equation}
			\begin{split}\label{upper.new.13}
				& \left(\frac{\sum_{i=1}^{n}\boldsymbol{1}_{\boldsymbol{x}_{i} \in \boldsymbol{t}}}{n}\right)\widehat{(II)}_{\boldsymbol{t}, \boldsymbol{t}(j,1)} \\
				& = \left( \frac{\sum_{i=1}^{n}\boldsymbol{1}_{\boldsymbol{x}_{i} \in \boldsymbol{t}_{1}}}{n}\right) \left(\frac{\sum_{i=1}^{n}\boldsymbol{1}_{\boldsymbol{x}_{i} \in \boldsymbol{t}_{1}} (m(\boldsymbol{x}_{i}) + \varepsilon_{i}) }{\sum_{i=1}^{n}\boldsymbol{1}_{\boldsymbol{x}_{i} \in \boldsymbol{t}_{1}}} - \frac{\sum_{i=1}^{n}\boldsymbol{1}_{\boldsymbol{x}_{i} \in \boldsymbol{t}} (m(\boldsymbol{x}_{i}) + \varepsilon_{i}) }{\sum_{i=1}^{n}\boldsymbol{1}_{\boldsymbol{x}_{i} \in \boldsymbol{t}}}  \right)^2 \\
				& \qquad+\left( \frac{\sum_{i=1}^{n}\boldsymbol{1}_{\boldsymbol{x}_{i} \in \boldsymbol{t}_{2}}}{n}\right) \left(\frac{\sum_{i=1}^{n}\boldsymbol{1}_{\boldsymbol{x}_{i} \in \boldsymbol{t}_{2}} (m(\boldsymbol{x}_{i}) + \varepsilon_{i}) }{\sum_{i=1}^{n}\boldsymbol{1}_{\boldsymbol{x}_{i} \in \boldsymbol{t}_{2}}} - \frac{\sum_{i=1}^{n}\boldsymbol{1}_{\boldsymbol{x}_{i} \in \boldsymbol{t}} (m(\boldsymbol{x}_{i}) + \varepsilon_{i}) }{\sum_{i=1}^{n}\boldsymbol{1}_{\boldsymbol{x}_{i} \in \boldsymbol{t}}}  \right)^2,
			\end{split}
		\end{equation}}
		and
		\begin{equation}
			\begin{split}\label{upper.new.14}
				\mathbb{P}(\boldsymbol{X} \in \boldsymbol{t}) (II)_{\boldsymbol{t}, \boldsymbol{t}(j, 1)} &= \mathbb{P}(\boldsymbol{X} \in \boldsymbol{t}_{1}) (\mathbb{E} (m(\boldsymbol{X})| \boldsymbol{X} \in \boldsymbol{t}_{1}) - \mathbb{E} (m(\boldsymbol{X})| \boldsymbol{X} \in \boldsymbol{t}))^2\\
				& \qquad+  \mathbb{P}(\boldsymbol{X} \in \boldsymbol{t}_{2}) (\mathbb{E} (m(\boldsymbol{X})| \boldsymbol{X} \in \boldsymbol{t}_{2}) - \mathbb{E} (m(\boldsymbol{X})| \boldsymbol{X} \in \boldsymbol{t}))^2.
			\end{split}
		\end{equation}
		We begin with the difference between the respective first terms of the RHS of  \eqref{upper.new.13}--\eqref{upper.new.14} as follows.
		\begin{equation}
			\begin{split}\label{upper.new.17}
				& \Big|\left( \frac{\sum_{i=1}^{n}\boldsymbol{1}_{\boldsymbol{x}_{i} \in \boldsymbol{t}_{1}}}{n}\right) \left(\frac{\sum_{i=1}^{n}\boldsymbol{1}_{\boldsymbol{x}_{i} \in \boldsymbol{t}_{1}} (m(\boldsymbol{x}_{i}) + \varepsilon_{i}) }{\sum_{i=1}^{n}\boldsymbol{1}_{\boldsymbol{x}_{i} \in \boldsymbol{t}_{1}}} - \frac{\sum_{i=1}^{n}\boldsymbol{1}_{\boldsymbol{x}_{i} \in \boldsymbol{t}} (m(\boldsymbol{x}_{i}) + \varepsilon_{i}) }{\sum_{i=1}^{n}\boldsymbol{1}_{\boldsymbol{x}_{i} \in \boldsymbol{t}}}  \right)^2\\
				& \qquad- \mathbb{P}(\boldsymbol{X} \in \boldsymbol{t}_{1}) (\mathbb{E} (m(\boldsymbol{X})| \boldsymbol{X} \in \boldsymbol{t}_{1}) - \mathbb{E} (m(\boldsymbol{X})| \boldsymbol{X} \in \boldsymbol{t}))^2\Big|\\
				& \le \Big|\left( \frac{\sum_{i=1}^{n}\boldsymbol{1}_{\boldsymbol{x}_{i} \in \boldsymbol{t}_{1}}}{n}\right) \left(\frac{\sum_{i=1}^{n}\boldsymbol{1}_{\boldsymbol{x}_{i} \in \boldsymbol{t}_{1}} (m(\boldsymbol{x}_{i}) + \varepsilon_{i}) }{\sum_{i=1}^{n}\boldsymbol{1}_{\boldsymbol{x}_{i} \in \boldsymbol{t}_{1}}} - \frac{\sum_{i=1}^{n}\boldsymbol{1}_{\boldsymbol{x}_{i} \in \boldsymbol{t}} (m(\boldsymbol{x}_{i}) + \varepsilon_{i}) }{\sum_{i=1}^{n}\boldsymbol{1}_{\boldsymbol{x}_{i} \in \boldsymbol{t}}}  \right)^2\\
				&\qquad - \left( \frac{\sum_{i=1}^{n}\boldsymbol{1}_{\boldsymbol{x}_{i} \in \boldsymbol{t}_{1}}}{n}\right) (\mathbb{E} (m(\boldsymbol{X})| \boldsymbol{X} \in \boldsymbol{t}_{1}) - \mathbb{E} (m(\boldsymbol{X})| \boldsymbol{X} \in \boldsymbol{t}))^2\\
				& \qquad + \left( \frac{\sum_{i=1}^{n}\boldsymbol{1}_{\boldsymbol{x}_{i} \in \boldsymbol{t}_{1}}}{n}\right) (\mathbb{E} (m(\boldsymbol{X})| \boldsymbol{X} \in \boldsymbol{t}_{1}) - \mathbb{E} (m(\boldsymbol{X})| \boldsymbol{X} \in \boldsymbol{t}))^2\\
				& \qquad - \mathbb{P}(\boldsymbol{X} \in \boldsymbol{t}_{1}) (\mathbb{E} (m(\boldsymbol{X})| \boldsymbol{X} \in \boldsymbol{t}_{1}) - \mathbb{E} (m(\boldsymbol{X})| \boldsymbol{X} \in \boldsymbol{t}))^2\Big|,
			\end{split}
		\end{equation}
		and that the difference between the first two terms on the RHS of \eqref{upper.new.17} is bounded by
		\begin{equation}
			\begin{split}\label{upper.new.15}
				& \Big|\left( \frac{\sum_{i=1}^{n}\boldsymbol{1}_{\boldsymbol{x}_{i} \in \boldsymbol{t}_{1}}}{n}\right) \left(\frac{\sum_{i=1}^{n}\boldsymbol{1}_{\boldsymbol{x}_{i} \in \boldsymbol{t}_{1}} (m(\boldsymbol{x}_{i}) + \varepsilon_{i}) }{\sum_{i=1}^{n}\boldsymbol{1}_{\boldsymbol{x}_{i} \in \boldsymbol{t}_{1}}} - \frac{\sum_{i=1}^{n}\boldsymbol{1}_{\boldsymbol{x}_{i} \in \boldsymbol{t}} (m(\boldsymbol{x}_{i}) + \varepsilon_{i}) }{\sum_{i=1}^{n}\boldsymbol{1}_{\boldsymbol{x}_{i} \in \boldsymbol{t}}}  \right)^2\\
				&\qquad - \left( \frac{\sum_{i=1}^{n}\boldsymbol{1}_{\boldsymbol{x}_{i} \in \boldsymbol{t}_{1}}}{n}\right) (\mathbb{E} (m(\boldsymbol{X})| \boldsymbol{X} \in \boldsymbol{t}_{1}) - \mathbb{E} (m(\boldsymbol{X})| \boldsymbol{X} \in \boldsymbol{t}))^2 \Big|\\
				&\le \left( \frac{\sum_{i=1}^{n}\boldsymbol{1}_{\boldsymbol{x}_{i} \in \boldsymbol{t}_{1}}}{n}\right) \Big(\Big|\frac{\sum_{i=1}^{n}\boldsymbol{1}_{\boldsymbol{x}_{i} \in \boldsymbol{t}_{1}} (m(\boldsymbol{x}_{i}) + \varepsilon_{i}) }{\sum_{i=1}^{n}\boldsymbol{1}_{\boldsymbol{x}_{i} \in \boldsymbol{t}_{1}}} -\mathbb{E} (m(\boldsymbol{X})| \boldsymbol{X} \in \boldsymbol{t}_{1}) \Big| \\
				&\qquad +   \Big| \mathbb{E} (m(\boldsymbol{X})| \boldsymbol{X} \in \boldsymbol{t}) -  \frac{\sum_{i=1}^{n}\boldsymbol{1}_{\boldsymbol{x}_{i} \in \boldsymbol{t}} (m(\boldsymbol{x}_{i}) + \varepsilon_{i}) }{\sum_{i=1}^{n}\boldsymbol{1}_{\boldsymbol{x}_{i} \in \boldsymbol{t}}} \Big| \Big) (2M_{\varepsilon} + 4M_{0}),
			\end{split}
		\end{equation}
		where we use the identity $a^2 - b^2 = (a-b)(a+b)$ and the assumptions of a bounded regression function and model errors. 
		
		Two terms of differences on the RHS of \eqref{upper.new.15} can be further bounded respectively as follows.
		\begin{equation}
			\begin{split}\label{upper.new.16}
				& \Big|\frac{\sum_{i=1}^{n}\boldsymbol{1}_{\boldsymbol{x}_{i} \in \boldsymbol{t}_{1}} (m(\boldsymbol{x}_{i}) + \varepsilon_{i}) }{\sum_{i=1}^{n}\boldsymbol{1}_{\boldsymbol{x}_{i} \in \boldsymbol{t}_{1}}} -\mathbb{E} (m(\boldsymbol{X})| \boldsymbol{X} \in \boldsymbol{t}_{1}) \Big| \\
				&=\Big| \frac{n^{-1}\sum_{i=1}^{n}\boldsymbol{1}_{\boldsymbol{x}_{i} \in \boldsymbol{t}_{1}} (m(\boldsymbol{x}_{i}) + \varepsilon_{i}) }{n^{-1}\sum_{i=1}^{n}\boldsymbol{1}_{\boldsymbol{x}_{i} \in \boldsymbol{t}_{1}}} - \frac{\mathbb{E} (m(\boldsymbol{X}) \boldsymbol{1}_{\boldsymbol{X} \in \boldsymbol{t}_{1}})
				}{ n^{-1}\sum_{i=1}^{n}\boldsymbol{1}_{\boldsymbol{x}_{i} \in \boldsymbol{t}_{1}}} \\
				&\qquad\qquad+ \frac{\mathbb{E} (m(\boldsymbol{X}) \boldsymbol{1}_{\boldsymbol{X} \in \boldsymbol{t}_{1}})
				}{ n^{-1}\sum_{i=1}^{n}\boldsymbol{1}_{\boldsymbol{x}_{i} \in \boldsymbol{t}_{1}}}  -  \frac{\mathbb{E} (m(\boldsymbol{X}) \boldsymbol{1}_{\boldsymbol{X} \in \boldsymbol{t}_{1}})
				}{ \mathbb{P}(\boldsymbol{X} \in \boldsymbol{t}_{1})}\Big|\\
				&\le \left(\frac{\sum_{i=1}^{n}\boldsymbol{1}_{\boldsymbol{x}_{i} \in \boldsymbol{t}_{1}}}{n}\right)^{-1}\Big( \Big|\frac{\sum_{i=1}^{n}\boldsymbol{1}_{\boldsymbol{x}_{i} \in \boldsymbol{t}_{1}} (m(\boldsymbol{x}_{i}) + \varepsilon_{i}) }{n} - \mathbb{E} (m(\boldsymbol{X}) \boldsymbol{1}_{\boldsymbol{X} \in \boldsymbol{t}_{1}})
				\Big| \\
				&\qquad + \frac{\mathbb{E} (m(\boldsymbol{X}) \boldsymbol{1}_{\boldsymbol{X} \in \boldsymbol{t}_{1}})
				}{ \mathbb{P} (\boldsymbol{X} \in \boldsymbol{t}_{1})}\Big|\mathbb{P} (\boldsymbol{X} \in \boldsymbol{t}_{1}) - \frac{\sum_{i=1}^{n}\boldsymbol{1}_{\boldsymbol{x}_{i} \in \boldsymbol{t}_{1}}  }{n}\Big|  \Big),
			\end{split}
		\end{equation}
		and similarly,
		\begin{equation}
			\begin{split}\label{upper.new.20}
				& \Big|\frac{\sum_{i=1}^{n}\boldsymbol{1}_{\boldsymbol{x}_{i} \in \boldsymbol{t}} (m(\boldsymbol{x}_{i}) + \varepsilon_{i}) }{\sum_{i=1}^{n}\boldsymbol{1}_{\boldsymbol{x}_{i} \in \boldsymbol{t}}} -\mathbb{E} (m(\boldsymbol{X})| \boldsymbol{X} \in \boldsymbol{t}) \Big| \\
				&\le \left(\frac{\sum_{i=1}^{n}\boldsymbol{1}_{\boldsymbol{x}_{i} \in \boldsymbol{t}}}{n}\right)^{-1}\Big( \Big|\frac{\sum_{i=1}^{n}\boldsymbol{1}_{\boldsymbol{x}_{i} \in \boldsymbol{t}} (m(\boldsymbol{x}_{i}) + \varepsilon_{i}) }{n} - \mathbb{E} (m(\boldsymbol{X}) \boldsymbol{1}_{\boldsymbol{X} \in \boldsymbol{t}})
				\Big| \\
				&\qquad + \frac{\mathbb{E} (m(\boldsymbol{X}) \boldsymbol{1}_{\boldsymbol{X} \in \boldsymbol{t}})
				}{ \mathbb{P} (\boldsymbol{X} \in \boldsymbol{t})}\Big| \mathbb{P} (\boldsymbol{X} \in \boldsymbol{t})- \frac{\sum_{i=1}^{n}\boldsymbol{1}_{\boldsymbol{x}_{i} \in \boldsymbol{t}}  }{n}\Big|  \Big).
			\end{split}
		\end{equation}
		
		On the other hand, by the assumption of a bounded $m(\cdot)$, the last two terms on the RHS of \eqref{upper.new.17} is bounded by
		\begin{equation}
			\begin{split}\label{upper.new.18}
				& \Big|\left( \frac{\sum_{i=1}^{n}\boldsymbol{1}_{\boldsymbol{x}_{i} \in \boldsymbol{t}_{1}}}{n}\right) (\mathbb{E} (m(\boldsymbol{X})| \boldsymbol{X} \in \boldsymbol{t}_{1}) - \mathbb{E} (m(\boldsymbol{X})| \boldsymbol{X} \in \boldsymbol{t}))^2\\
				& \qquad - \mathbb{P}(\boldsymbol{X} \in \boldsymbol{t}_{1}) (\mathbb{E} (m(\boldsymbol{X})| \boldsymbol{X} \in \boldsymbol{t}_{1}) - \mathbb{E} (m(\boldsymbol{X})| \boldsymbol{X} \in \boldsymbol{t}))^2\Big|\\
				& \le 4M_{0}^2\Big|\frac{\sum_{i=1}^{n}\boldsymbol{1}_{\boldsymbol{x}_{i} \in \boldsymbol{t}_{1}}}{n}- \mathbb{P}(\boldsymbol{X} \in \boldsymbol{t}_{1})\Big|.
			\end{split}
		\end{equation}
		
		By \eqref{upper.new.13}--\eqref{upper.new.18}, it holds that for all large $n$ and every end cell $\boldsymbol{t}$ of trees of level $k\le \eta\log_{2}(n)-1$, on $U_{n}$,
		\begin{equation*}
			\begin{split}			
				&\max_{j} \Big|\left(\frac{\sum_{i=1}^{n}\boldsymbol{1}_{\boldsymbol{x}_{i} \in \boldsymbol{t}}}{n}\right)\widehat{(II)}_{\boldsymbol{t}, \boldsymbol{t}(j,1)} - \mathbb{P}(\boldsymbol{X} \in \boldsymbol{t}) (II)_{\boldsymbol{t}, \boldsymbol{t}(j, 1)}\Big| \\
				&\le 2\Big((2M_{\varepsilon} + 4M_{0})\times 2(\sqrt{2M_{0}^2 + M_{\varepsilon}^2}+M_{0}) + 4M_{0}^2\Big) \big(\log_{e}({\max\{n, p\}})\big)^{\frac{2+\epsilon}{2}}\sqrt{\frac{\mathbb{P}(\boldsymbol{X} \in \boldsymbol{t})}{n}}\\
				&
				\le 18(M_{\varepsilon} + 2M_{0} )^2 \big(\log_{e}({\max\{n, p\}})\big)^{\frac{2+\epsilon}{2}}\sqrt{\frac{\mathbb{P}(\boldsymbol{X} \in \boldsymbol{t})}{n}}.
			\end{split}
		\end{equation*}
		
		This and the argument before \eqref{upper.new.12} conclude the second assertion, and hence we have finished the proof.

	\end{proof}

	\begin{proof}[Proof of \eqref{upper.new.p.1}]

	The proof idea follows that for proof of Theorem~\ref{theorem3}, but is much simplified as the theoretical tree growing rule $T^*$ is considered here. In what follows, we deal with the case where there are no random splits first (see the end of this proof for details). Let us start with an expression for the LHS of \eqref{upper.new.p.1} in \eqref{upper.new.5} below, which can be obtained by direct calculations when there are no random splits. It holds that
		\begin{equation}
			\begin{split}\label{upper.new.5}
				&\mathbb{E} \Big(m(\boldsymbol{X})  -m_{ T^* }^{*} (\boldsymbol{\Theta}_{1:k} , \boldsymbol{X})\Big)^{2} \\
				&= \sum_{\Theta_{1:k}} \mathbb{P} (\boldsymbol{\Theta}_{1:k} = \Theta_{1:k})\sum_{(\boldsymbol{t}_{1:k})\in T^*(\Theta_{1:k})}\mathbb{P}(\boldsymbol{X}\in\boldsymbol{t}_{k})V(\boldsymbol{t}_{k})\\
				&= \sum_{\Theta_{1:k}} \mathbb{P} (\boldsymbol{\Theta}_{1:k} = \Theta_{1:k})\sum_{(\boldsymbol{t}_{1:k})\in T^*(\Theta_{1:k})}\mathbb{P}(\boldsymbol{X}\in\boldsymbol{t}_{k-1}) \mathbb{P}(\boldsymbol{X}\in\boldsymbol{t}_{k}|\boldsymbol{X}\in \boldsymbol{t}_{k-1})V(\boldsymbol{t}_{k}) \\
				&= \sum_{\Theta_{1:k}} \mathbb{P} (\boldsymbol{\Theta}_{1:k} = \Theta_{1:k})\sum_{(\boldsymbol{t}_{1:k})\in T^*(\Theta_{1:k})}\mathbb{P}(\boldsymbol{X}\in\boldsymbol{t}_{k-1}) \frac{(I)_{\boldsymbol{t}_{k-1}, \boldsymbol{t}_{k}}}{2}\\
				&=  \sum_{\Theta_{1:k}} \mathbb{P} (\boldsymbol{\Theta}_{1:k} = \Theta_{1:k})\sum_{(\boldsymbol{t}_{1:k})\in T^*(\Theta_{1:k})}\mathbb{P}(\boldsymbol{X}\in\boldsymbol{t}_{k-1}) \frac{V(\boldsymbol{t}_{k-1}) - (II)_{\boldsymbol{t}_{k-1}, \boldsymbol{t}_{k}}}{2}\\
				&=  \sum_{\Theta_{1:k-1}} \mathbb{P} (\boldsymbol{\Theta}_{1:k-1} = \Theta_{1:k-1}) \\
				&\qquad\times \sum_{\Theta_{k}} \mathbb{P} (\boldsymbol{\Theta}_{k} = \Theta_{k})\sum_{(\boldsymbol{t}_{1:k})\in T^*(\Theta_{1:k})}\mathbb{P}(\boldsymbol{X}\in\boldsymbol{t}_{k-1}) \frac{V(\boldsymbol{t}_{k-1}) - (II)_{\boldsymbol{t}_{k-1}, \boldsymbol{t}_{k}}}{2},
			\end{split}
		\end{equation}
		\noindent where $V(\boldsymbol{t}) \coloneqq\textnormal{Var}(m(\boldsymbol{X})|\boldsymbol{X} \in \boldsymbol{t})$, the third equality follows from the definition of $(I)_{\boldsymbol{t}_{k-1}, \boldsymbol{t}_{k}}$ and the fact that  there are exactly two daughter cells after each $\boldsymbol{t}_{k-1}$, the fourth equality is due to the identity  $\textnormal{Var}(m(\boldsymbol{X}) | \boldsymbol{X} \in \boldsymbol{t}) = (I)_{\boldsymbol{t}, \boldsymbol{t}(j, c)} +(II)_{\boldsymbol{t}, \boldsymbol{t}(j, c)}$ for every $\boldsymbol{t}$ and $j, c\in t_{j}$, and the last equality is from the independence between column sets. In addition, we let $\textnormal{Var}(m(\boldsymbol{X})|\boldsymbol{X} \in \boldsymbol{t}) \coloneqq 0$, $(I)_{\boldsymbol{t}, \boldsymbol{t}^{'}} \coloneqq 0$, and $(II)_{\boldsymbol{t}, \boldsymbol{t}^{'}} \coloneqq 0$ if $\boldsymbol{t}=\boldsymbol{t}^{'} = \emptyset$.

		To proceed, we separately deal with tree branches in the tree $T^*(\Theta_{1:k})$ as follows. There are $2^k$ distinct tree branches $\boldsymbol{t}_{1:k}$ in $T^*(\Theta_{1:k})$, and we call the first two of these tree branches ``the first tree branch of $T^*(\Theta_{1:k})$,'' whose corresponding last column set restriction is $\Theta_{k, 1}$ (recall that $\Theta_{k} = \{\Theta_{k, 1}, \dots, \Theta_{k, 2^{k-1}}\}$). See Figure~\ref{fig:firstTreeBranch} for a graphical illustration. Since column sets are independent,
		{\small\begin{equation}
				\begin{split}\label{upper.new.8}
					&\textnormal{RHS of \eqref{upper.new.5}}\\
					& =  \sum_{\Theta_{1:k-1}} \mathbb{P} (\boldsymbol{\Theta}_{1:k-1} = \Theta_{1:k-1}) \\
					& \times\bigg(\sum_{ \Theta_{k}} \mathbb{P} (\boldsymbol{\Theta}_{k} = \Theta_{k})\sum_{ \textnormal{ The first tree branch of } T^*(\Theta_{1:k})}\mathbb{P}(\boldsymbol{X}\in\boldsymbol{t}_{k-1}) \frac{ V(\boldsymbol{t}_{k-1}) - (II)_{\boldsymbol{t}_{k-1}, \boldsymbol{t}_{k}} }{2}\\
					& + \sum_{\Theta_{k}} \mathbb{P} (\boldsymbol{\Theta}_{k} = \Theta_{k})  \sum_{ \textnormal{  Other tree branches of } T^*(\Theta_{1:k})}\mathbb{P}(\boldsymbol{X}\in\boldsymbol{t}_{k-1}) \frac{V(\boldsymbol{t}_{k-1}) - (II)_{\boldsymbol{t}_{k-1}, \boldsymbol{t}_{k}}}{2}	 \bigg) \\
					& =  \sum_{\Theta_{1:k-1}} \mathbb{P} (\boldsymbol{\Theta}_{1:k-1} = \Theta_{1:k-1}) \\
					& \times\bigg(\sum_{ \Theta_{k, 1}} \mathbb{P} (\boldsymbol{\Theta}_{k,1} = \Theta_{k,1})\sum_{ \textnormal{ The first tree branch of } T^*(\Theta_{1:k})}\mathbb{P}(\boldsymbol{X}\in\boldsymbol{t}_{k-1}) \frac{ V(\boldsymbol{t}_{k-1}) - (II)_{\boldsymbol{t}_{k-1}, \boldsymbol{t}_{k}} }{2}\\
					& + \sum_{\Theta_{k, 2}, \dots, \Theta_{k, 2^{k-1}}} \mathbb{P} ((\boldsymbol{\Theta}_{k, 2} , \dots, \boldsymbol{\Theta}_{k, 2^{k-1}})= (\Theta_{k, 2},\dots, \Theta_{k, 2^{k-1}}) ) \\
					&\qquad \qquad\times \bigg( \sum_{\textnormal{ Other tree branches of } T^*(\Theta_{1:k})}\mathbb{P}(\boldsymbol{X}\in\boldsymbol{t}_{k-1}) \frac{V(\boldsymbol{t}_{k-1}) - (II)_{\boldsymbol{t}_{k-1}, \boldsymbol{t}_{k}}}{2}	 \bigg)\bigg).
				\end{split}
		\end{equation}}
		
		Now, let us say a good column set restriction $\Theta$ w.r.t. a cell $\boldsymbol{t}$ is such that  $\sup_{j\in \Theta, c\in t_{j}}(II)_{\boldsymbol{t}, \boldsymbol{t}(j, c)} = \sup_{j\in\{1, \dots, p\}, c\in t_{j}}(II)_{\boldsymbol{t}, \boldsymbol{t}(j, c)}$. Because $m(\boldsymbol{X}) \in \textnormal{SID} (s^*)$, it holds that for a cell $\boldsymbol{t}$,
		\begin{equation}
			\begin{cases}\label{upper.new.p.9}
				V(\boldsymbol{t}) - \sup_{j\in \Theta, c\in t_{j}}(II)_{\boldsymbol{t}, \boldsymbol{t}(j, c)}  \le (1-(s^*)^{-1})V(\boldsymbol{t}), & \textnormal{ if } \Theta \textnormal{ is good,}\\V(\boldsymbol{t}) - \sup_{j\in \Theta, c\in t_{j}}(II)_{\boldsymbol{t}, \boldsymbol{t}(j, c)} \le V(\boldsymbol{t}), & \textnormal{ o.w.}
			\end{cases}
		\end{equation}
		
		By \eqref{upper.new.p.9}, we deal with the first tree branch as follows.
		{\small\begin{equation}
				\begin{split}\label{upper.new.6}
					& \sum_{ \Theta_{k, 1}} \mathbb{P} (\boldsymbol{\Theta}_{k,1} = \Theta_{k,1})\sum_{ \textnormal{ The first tree branch of } T^*(\Theta_{1:k})}\mathbb{P}(\boldsymbol{X}\in\boldsymbol{t}_{k-1}) \frac{ V(\boldsymbol{t}_{k-1}) - (II)_{\boldsymbol{t}_{k-1}, \boldsymbol{t}_{k}} }{2}\\
					& \le 
					 \bigg(\sum_{\textnormal{Good } \Theta_{k, 1}} \mathbb{P} (\boldsymbol{\Theta}_{k,1} = \Theta_{k,1})\sum_{ \textnormal{ The first tree branch of } T^*(\Theta_{1:k})}\mathbb{P}(\boldsymbol{X}\in\boldsymbol{t}_{k-1}) \frac{(1-(s^*)^{-1})V(\boldsymbol{t}_{k-1}) }{2}	\\
					& \qquad+ 	\sum_{\textnormal{Bad } \Theta_{k,1}} \mathbb{P} (\boldsymbol{\Theta}_{k,1} = \Theta_{k,1})\sum_{ \textnormal{ The first tree branch of } T^*(\Theta_{1:k})}\mathbb{P}(\boldsymbol{X}\in\boldsymbol{t}_{k-1}) \frac{V(\boldsymbol{t}_{k-1})}{2} \bigg).
				\end{split}
		\end{equation}}
		
		Notice that the end cell $\boldsymbol{t}_{k}$ is not needed on the RHS of \eqref{upper.new.6}, and that the first tree branch consists of exactly two daughter cells. Let $\boldsymbol{t}_{1},\dots, \boldsymbol{t}_{k-1}$ denote the first tree branch. Then,
		{\small\begin{equation}
				\begin{split}\label{upper.new.11}
					&\textnormal{RHS of \eqref{upper.new.6}}\\
					& = 
					\bigg(\sum_{\textnormal{Good } \Theta_{k, 1}} \mathbb{P} (\boldsymbol{\Theta}_{k,1} = \Theta_{k,1}) \mathbb{P}(\boldsymbol{X}\in\boldsymbol{t}_{k-1}) (1-(s^*)^{-1})V(\boldsymbol{t}_{k-1}) 	\\
					& \qquad+ 	\sum_{\textnormal{Bad } \Theta_{k,1}} \mathbb{P} (\boldsymbol{\Theta}_{k,1} = \Theta_{k,1})
					\mathbb{P}(\boldsymbol{X}\in\boldsymbol{t}_{k-1}) V(\boldsymbol{t}_{k-1}) \bigg).
				\end{split}
		\end{equation}}
		
		Furthermore, recall that the probability of having a good column set is at least $\gamma_{0}$ according to our model assumption. Specifically,
		the probability of having any active $j$ in $\boldsymbol{\Theta}$ is $\gamma_{0}$; if no active feature is left for $\boldsymbol{t}$, then $\boldsymbol{\Theta}$ is a good column set with probability one. By this and the fact that $1-(s^*)^{-1}\le 1$,
		\begin{equation}
				\begin{split}\label{upper.new.7}
					&\textnormal{RHS of \eqref{upper.new.11}}\\
					& \le 
					\bigg(\gamma_{0}\times \mathbb{P}(\boldsymbol{X}\in\boldsymbol{t}_{k-1}) (1-(s^*)^{-1})V(\boldsymbol{t}_{k-1}) + (1-\gamma_{0})\times \mathbb{P}(\boldsymbol{X}\in\boldsymbol{t}_{k-1}) V(\boldsymbol{t}_{k-1}) \bigg)\\
					& \le \bigg(	(1-\gamma_{0}(s^*)^{-1})\times \mathbb{P}(\boldsymbol{X}\in\boldsymbol{t}_{k-1}) V(\boldsymbol{t}_{k-1}) \bigg).
				\end{split}
		\end{equation}
		Next, we apply the argument for \eqref{upper.new.6}--\eqref{upper.new.7} to other tree branches in \eqref{upper.new.8} and obtain
		{\small\begin{equation}
			\begin{split}\label{upper.new.9}
				&\textnormal{RHS of \eqref{upper.new.8}}\\
				& \le (1-\gamma_{0}(s^*)^{-1})\sum_{\Theta_{1:k-1}} \mathbb{P} (\boldsymbol{\Theta}_{1:k-1} = \Theta_{1:k-1})  \sum_{(\boldsymbol{t}_{1:k-1}) \in T^*(\Theta_{1:k-1})}\mathbb{P}(\boldsymbol{X}\in\boldsymbol{t}_{k-1}) V(\boldsymbol{t}_{k-1}).
			\end{split}
		\end{equation}}
		
		\noindent To conclude, we recursively apply these arguments to show that
		\begin{equation*}
			\begin{split}
				&\textnormal{RHS of \eqref{upper.new.9}} \le (1-\gamma_{0}(s^*)^{-1})^k V(\boldsymbol{t}_{0}),
			\end{split}
		\end{equation*}
		which leads to the desired result. 
		
		Lastly, to consider random splits, we let ``Random splits'' denote the random parameter of these random splits, and hence
	\begin{equation*}
	    \begin{split}
	    & \mathbb{E}\Big(m(\boldsymbol{X}) - m_{T^*}^{*}\left(\boldsymbol{\Theta}_{1:k}, \boldsymbol{X}\right)\Big)^{2}\\
	        & =\mathbb{E}\Big(\mathbb{E}\Big(\big(m(\boldsymbol{X}) - m_{T^*}^{*}\left(\boldsymbol{\Theta}_{1:k}, \boldsymbol{X}\right)\big)^2 | \textnormal{ Random splits}\Big)\Big) \\
	        & \le \mathbb{E}\big( (1-\gamma_{0}(s^*)^{-1})^k V(\boldsymbol{t}_{0}) \big)\\
	        & = (1-\gamma_{0}(s^*)^{-1})^k V(\boldsymbol{t}_{0}),
	    \end{split}
	\end{equation*}
	where the inequality is due to the previous arguments. This concludes the proof.

	\end{proof}

	\begin{proof}[Proof of \eqref{upper.new.p.10}]
		Recall that $\gamma_{0} = 1$ means that the all column sets contain all features, and that the forest model is essentially a decision tree model. In addition, recall that the total variance is 
		\begin{equation}
			\label{upper.new.p.11}
			\textnormal{Var}(m(\boldsymbol{X})) = s^*\frac{\beta^2}{4}.
		\end{equation}
		
		Since the tree model is grown by using theoretical CART, by i) of Lemma~\ref{upper.lemma.1}, the first split is on one of the first $s^*$ coordinates; the total bias decrease for the first split is $\frac{\beta^2}{4}$.
		
		Next, we split the resulting two daughter cells by using theoretical CART. Each split results in conditional bias decrease of an amount of $\frac{\beta^2}{4}$, and that each daughter cell $\boldsymbol{t}$ is such that $\mathbb{P}(\boldsymbol{X}\in\boldsymbol{t}) = \frac{1}{2}.$ Hence, the total bias decrease for the second split is  $\frac{1}{2}\frac{\beta^2}{4} + \frac{1}{2}\frac{\beta^2}{4} = \frac{\beta^2}{4}$.

		These steps repeat until there are no active features to be split; we see that at each level $k\le s^*$, the total bias decrease is $\frac{\beta^2}{4}$. By \eqref{upper.new.p.11} and this argument, we conclude the proof.
	\end{proof}

	\begin{proof}[Proof of \eqref{upper.new.p.6}]
		By \eqref{eq.new.2}, \eqref{eq.new.1}, and the assumptions of a bounded regression function and model errors,
		{\small \begin{equation}
				\begin{split}		\label{upper.new.p.3}
					&\mathbb{E}\Big(m_{ \widehat{T} }^{*} (\boldsymbol{\Theta}_{1:k} , \boldsymbol{X})- \widehat{m}_{ \widehat{T}}(\boldsymbol{\Theta}_{1:k} , \boldsymbol{X}, \mathcal{X}_{n} ) \Big)^{2} \\
					&\le  \mathbb{E}\bigg(\boldsymbol{1}_{U_{n}}\sum_{(\boldsymbol{t}_{1:k})\in\widehat{T}(\boldsymbol{\Theta}_{1:k})}\boldsymbol{1}_{\boldsymbol{X} \in\boldsymbol{t}_{k}} \mathbb{E}(m(\boldsymbol{X}) |\boldsymbol{X} \in\boldsymbol{t}_{k}) \\
					& \qquad\qquad- \sum_{(\boldsymbol{t}_{1:k})\in\widehat{T}(\boldsymbol{\Theta}_{1:k})} \boldsymbol{1}_{\boldsymbol{X} \in\boldsymbol{t}_{k}} \frac{\sum_{i=1}^{n} \boldsymbol{1} _{\boldsymbol{x}_{i} \in\boldsymbol{t}_{k}} (m(\boldsymbol{x}_{i}) + \varepsilon_{i})}{\sum_{i=1}^{n} \boldsymbol{1} _{\boldsymbol{x}_{i} \in\boldsymbol{t}_{k}} } \bigg)^2+ (2M_{0}+M_{\varepsilon})^2\mathbb{P}(U_{n}^{c})\\
					& = \mathbb{E}\bigg(\boldsymbol{1}_{U_{n}}\sum_{(\boldsymbol{t}_{1:k})\in\widehat{T}(\boldsymbol{\Theta}_{1:k})}\boldsymbol{1}_{\boldsymbol{X} \in\boldsymbol{t}_{k}} \bigg(\mathbb{E}(m(\boldsymbol{X}) |\boldsymbol{X} \in\boldsymbol{t}_{k})
					-  \frac{\sum_{i=1}^{n} \boldsymbol{1} _{\boldsymbol{x}_{i} \in\boldsymbol{t}_{k}} (m(\boldsymbol{x}_{i}) + \varepsilon_{i})}{\sum_{i=1}^{n} \boldsymbol{1} _{\boldsymbol{x}_{i} \in\boldsymbol{t}_{k}}} \bigg)^2\bigg) \\
					&\qquad \qquad+ (2M_{0}+M_{\varepsilon})^2\mathbb{P}(U_{n}^{c}),
				\end{split}
		\end{equation}}
		where the second equality is due to the fact that $\boldsymbol{1}_{\boldsymbol{X} \in \boldsymbol{t}} \times \boldsymbol{1}_{\boldsymbol{X} \in \boldsymbol{t}^{'}} = 0$ if $\boldsymbol{t}\cap \boldsymbol{t}^{'} = \emptyset$.
		
		The RHS of \eqref{upper.new.p.3} can be further dealt with as follows. By the model assumptions, for every end cell $\boldsymbol{t}$ of trees of level $k$, 
		\begin{equation}\label{upper.new.3}
			\mathbb{P}(\boldsymbol{X} \in \boldsymbol{t})\ge 2^{-k}.
		\end{equation}
		
		For every end cell $\boldsymbol{t}_{k}$ in \eqref{upper.new.p.3} with $0\le k \le \eta\log_{2}(n)$, it holds that either $\boldsymbol{t}_{k}\in G_{n}$ or $\boldsymbol{t}_{k}$ is an empty set. If $\boldsymbol{t}_{k}$ is an empty set, by the definition that $\frac{0}{0} = 0$, 
		\begin{equation}
			\begin{split}\label{upper.new.21}
			\mathbb{E}(m(\boldsymbol{X}) |\boldsymbol{X} \in\boldsymbol{t}_{k}) -  \frac{\sum_{i=1}^{n} \boldsymbol{1} _{\boldsymbol{x}_{i} \in\boldsymbol{t}_{k}} (m(\boldsymbol{x}_{i}) + \varepsilon_{i})}{\sum_{i=1}^{n} \boldsymbol{1} _{\boldsymbol{x}_{i} \in\boldsymbol{t}_{k}}}=0.
			\end{split}
		\end{equation}	
		On the other hand, on $U_{n}$, for each cell $\boldsymbol{t}_{k}\in G_{n}$,
		\begin{equation}
			\begin{split}\label{upper.new.22}
			    &\left|\mathbb{E}(m(\boldsymbol{X}) |\boldsymbol{X} \in\boldsymbol{t}_{k}) -  \frac{\sum_{i=1}^{n} \boldsymbol{1} _{\boldsymbol{x}_{i} \in\boldsymbol{t}_{k}} (m(\boldsymbol{x}_{i}) + \varepsilon_{i})}{\sum_{i=1}^{n} \boldsymbol{1} _{\boldsymbol{x}_{i} \in\boldsymbol{t}_{k}}} \right|\\
				& \le \Big|\frac{\mathbb{E}(m(\boldsymbol{X}) \boldsymbol{1}_{\boldsymbol{X} \in \boldsymbol{t}_{k}} )}{\mathbb{P}(\boldsymbol{X} \in \boldsymbol{t}_{k})} - \frac{n^{-1}\sum_{i=1}^{n} \boldsymbol{1} _{\boldsymbol{x}_{i} \in\boldsymbol{t}_{k}} (m(\boldsymbol{x}_{i}) + \varepsilon_{i})}{\mathbb{P}(\boldsymbol{X} \in \boldsymbol{t}_{k})} \Big| \\
				&\qquad+ \Big| \frac{n^{-1}\sum_{i=1}^{n} \boldsymbol{1} _{\boldsymbol{x}_{i} \in\boldsymbol{t}_{k}} (m(\boldsymbol{x}_{i}) + \varepsilon_{i})}{\mathbb{P}(\boldsymbol{X} \in \boldsymbol{t}_{k})} - \frac{\sum_{i=1}^{n} \boldsymbol{1} _{\boldsymbol{x}_{i} \in\boldsymbol{t}_{k}} (m(\boldsymbol{x}_{i}) + \varepsilon_{i})}{\sum_{i=1}^{n} \boldsymbol{1} _{\boldsymbol{x}_{i} \in\boldsymbol{t}_{k}}} \Big|\\
				& \le \frac{1}{\mathbb{P}(\boldsymbol{X} \in \boldsymbol{t}_{k} )}\Big|\mathbb{E}(m(\boldsymbol{X}) \boldsymbol{1}_{\boldsymbol{X} \in \boldsymbol{t}_{k}} ) - \frac{\sum_{i=1}^{n} \boldsymbol{1} _{\boldsymbol{x}_{i} \in\boldsymbol{t}_{k}} (m(\boldsymbol{x}_{i}) + \varepsilon_{i})}{n} \Big| \\
				&\qquad+ \frac{1}{\mathbb{P}(\boldsymbol{X} \in \boldsymbol{t}_{k})}\Big(  \frac{\sum_{i=1}^{n} \boldsymbol{1} _{\boldsymbol{x}_{i} \in\boldsymbol{t}_{k}} (m(\boldsymbol{x}_{i}) + \varepsilon_{i})}{\sum_{i=1}^{n} \boldsymbol{1} _{\boldsymbol{x}_{i} \in\boldsymbol{t}_{k}}} \Big)\Big| \frac{\sum_{i=1}^{n} \boldsymbol{1} _{\boldsymbol{x}_{i} \in\boldsymbol{t}_{k}}}{n} - \mathbb{P}(\boldsymbol{X} \in \boldsymbol{t}_{k}) \Big|\\
				&\le \frac{1}{\mathbb{P}(\boldsymbol{X} \in \boldsymbol{t}_{k})} (\sqrt{2M_{0}^2 + M_{\varepsilon}^2} + M_{0} + M_{\varepsilon})\big(\log_{e}({\max\{n, p\}})\big)^{\frac{2+\epsilon}{2}}\sqrt{\frac{\mathbb{P}(\boldsymbol{X} \in \boldsymbol{t}_{k})}{n}} \\
				&\le 2^{\frac{k}{2}}(3M_{0} + 2M_{\varepsilon})\frac{\big(\log_{e}({\max\{n, p\}})\big)^{\frac{2+\epsilon}{2}}}{\sqrt{n}},
			\end{split}
		\end{equation}
		where the third inequality is due to event $U_{n}$ and the assumption that $\boldsymbol{t}_{k} \in G_{n}$, and the last equality is due to \eqref{upper.new.3} and the subadditivity inequality.

		By \eqref{upper.new.p.3}, \eqref{upper.new.21}--\eqref{upper.new.22}, 
		\begin{equation*}
			\begin{split}	
				&\mathbb{E}\Big(m_{ \widehat{T} }^{*} (\boldsymbol{\Theta}_{1:k} , \boldsymbol{X})- \widehat{m}_{ \widehat{T}}(\boldsymbol{\Theta}_{1:k} , \boldsymbol{X}, \mathcal{X}_{n} ) \Big)^{2} \\
				&\le  ( 3M_{0} + 2M_{\varepsilon})^2\frac{2^{k}\big(\log_{e}({\max\{n, p\}})\big)^{2+\epsilon} }{n} + (2M_{0} + M_{\varepsilon})^2\mathbb{P}(U_{n}^c),
			\end{split}
		\end{equation*}
		which concludes the desired result.
		
		\begin{remark}\label{gamma.intuition}
			Let us give some intuition for how to establish a sharper estimation upper bound that depends on $\gamma_{0}$.  The way we deal with the estimation variance here is to bound the squared differences in \eqref{upper.new.p.3} directly; essentially, we establish the estimation variance upper bound for each tree. Notice that the end cells for each tree $\widehat{T}(\Theta_{1:k})$ are exclusive, but end cells of distinct trees are not exclusive. Our intuition is that it may be possible to aggregate distinct trees and sharpen the estimation upper bound. 
			The new upper bound should depend on $\gamma_{0}$ since column aggregation (i.e., the expectation over $\boldsymbol{\Theta}_{1:k}$) depends on $\gamma_{0}$. An example of utilising column aggregation for analysis can be seen in our bias analysis in \eqref{upper.new.p.1} and Section~\ref{Sec4}. There, we argue that the overall squared bias of a forest is controlled instead of arguing that each tree's bias is controlled, which is not right since there are always trees with high bias and trees with low bias in a forest with $\gamma_{0}<1$  and all possible trees.				 
		\end{remark}
	\end{proof}

	\subsection{Proof of Lemma~\ref{lemma1}} \label{SecC.2}

	Recall that $\mathcal{X}_{n}$ denotes the $n$ i.i.d. observations and $\boldsymbol{X}$ is the independent copy of $\boldsymbol{x}_{1}$.
	Let $\zeta = n^{-\delta}$, and $\hat{T}_{\zeta}$ and $\boldsymbol{U}_{n}$ be as defined in Theorem~\ref{theorem4}. See Sections~\ref{Sec4.2} and~\ref{SecA.5} for the definitions of $\hat{T}_{\zeta}$ and event $\boldsymbol{U}_{n}$, respectively. The main idea of the proof is the same as that described in (\ref{decom3}), but in the formal proof, it is $\hat{T}_{\zeta}$ instead of $\hat{T}$ that satisfies Condition~\ref{tree}. For details, see Theorem~\ref{theorem4} and Remark~\ref{r4}. An application of the triangle inequality leads to 
	\begin{equation}\begin{split}\label{con1}					
			& \mathop{{}\mathbb{E}}  \left(  m(\boldsymbol{X}) - m_{\widehat{T} }^{*} (\boldsymbol{\Theta}_{1:k} , \boldsymbol{X})\right)^{2} \\
			& =  \mathop{{}\mathbb{E}}\left(  \Big(m(\boldsymbol{X}) - m_{\widehat{T}_{\zeta}}^{*} (\boldsymbol{\Theta}_{1:k} , \boldsymbol{X})\Big)  -  \Big(  m_{\widehat{T} }^{*} (\boldsymbol{\Theta}_{1:k} , \boldsymbol{X}) - m_{ \widehat{T}_{\zeta}}^{*} (\boldsymbol{\Theta}_{1:k} , \boldsymbol{X})\Big)\right)^{2}\\
			& \le 2 \left( \mathop{{}\mathbb{E}}\left(  m(\boldsymbol{X}) - m_{\widehat{T}_{\zeta}}^{*} (\boldsymbol{\Theta}_{1:k} , \boldsymbol{X})\right)^{2}  +   \mathop{{}\mathbb{E}}\left(  m_{\widehat{T} }^{*} (\boldsymbol{\Theta}_{1:k} , \boldsymbol{X}) - m_{ \widehat{T}_{\zeta}}^{*} (\boldsymbol{\Theta}_{1:k} , \boldsymbol{X})\right)^{2}\right).
	\end{split}\end{equation}
	
	From Theorem~\ref{theorem4}, we see that for all large $n$, on event $\boldsymbol{U}_{n}$, $ \widehat{T}_{\zeta}$ with $\zeta=n^{-\delta}$ satisfies Condition~\ref{tree} with  $ k = \floor{c\log_{2}{n}}$, $\varepsilon = n^{-\eta}$, and  $\alpha_{2}$. Observe that if a tree growing rule satisfies Condition~\ref{tree} with $k$, it satisfies Condition~\ref{tree} with each positive integer no larger than $k$. By that $\mathcal{X}_{n}$ is independent of $\boldsymbol{X}$ and $\boldsymbol{\Theta}$, that $\boldsymbol{U}_{n}$ is $\mathcal{X}_{n}$-measurable, Condition~\ref{BO} (which states that $\sup_{\boldsymbol{c} \in [0, 1]^{p}}|m(\boldsymbol{c})| \le M_{0}$), and Theorem~\ref{theorem3} with $\varepsilon = n^{-\eta}$, it holds that for each $1 \le k \le c\log{n}$,
	\begin{equation}
		\label{eq.new.11} \mathbb{E}\left[\left(  m(\boldsymbol{X}) - m_{ \widehat{T}_{\zeta}}^{*} (\boldsymbol{\Theta}_{1:k} , \boldsymbol{X})\right)^{2}  \ \Big\vert \ \mathcal{X}_{n} \right] \boldsymbol{1}_{ \boldsymbol{U}_{n}} \le \alpha_{1}\alpha_{2}n^{-\eta} + (1 - \gamma_{0}(\alpha_{1}\alpha_{2})^{-1})^{k} M_{0}^{2}.
	\end{equation}
	Then by (\ref{eq.new.11}) and Condition~\ref{BO}, it holds that for all large $n$ and each integer $1\le k \le c\log{n}$,
	\begin{equation}\begin{split} \label{new.eq.013}
			\mathbb{E}\left(  m(\boldsymbol{X}) - m_{ \widehat{T}_{\zeta}}^{*} (\boldsymbol{\Theta}_{1:k} , \boldsymbol{X})\right)^{2} & = \mathbb{E}\left[\left(  m(\boldsymbol{X}) - m_{ \widehat{T}_{\zeta}}^{*} (\boldsymbol{\Theta}_{1:k} , \boldsymbol{X})\right)^{2}  \boldsymbol{1}_{ \boldsymbol{U}_{n}}\right] \\
			& \qquad+ \mathbb{E}\left[\left(  m(\boldsymbol{X}) - m_{ \widehat{T}_{\zeta}}^{*} (\boldsymbol{\Theta}_{1:k} , \boldsymbol{X})\right)^{2}  \boldsymbol{1}_{ \boldsymbol{U}_{n}^{c}}\right]\\
			& \le \alpha_{1}\alpha_{2}n^{-\eta} + (1 - \gamma_{0}(\alpha_{1}\alpha_{2})^{-1})^{k} M_{0}^{2} + n^{-1}.
	\end{split}\end{equation}
	Here, to bound the second term on the RHS of the equality above, we utilize Condition~\ref{BO}, standard inequalities, and that $\mathbb{P}(\boldsymbol{U}_{n}^c) = o(n^{-1})$.

	On the other hand, by Condition~\ref{BO} we have $\sup_{\boldsymbol{c} \in [0, 1]^{p}}|m_{\hat{T}}^{*}(\boldsymbol{c}) - m_{\hat{T}_{\zeta}}^{*}(\boldsymbol{c})|\le 2M_{0}$. By this and the fact that there are at most  $2^{k}$ cells at level $k$, it holds for each $\Theta_{1}, \dots, \Theta_{k}$ that on $\cap_{l=1}^{k}\{\boldsymbol{\Theta}_{l} = \Theta_{l}\}$,
	{\small\begin{equation}\label{nke1}
		\mathop{{}\mathbb{E}}\Big[ \Big(m_{ \widehat{T}}^{*}(\boldsymbol{\Theta}_{1}, \dots, \boldsymbol{\Theta}_{k}, \boldsymbol{X})  -   m_{\widehat{T}_{\zeta} }^{*}(\boldsymbol{\Theta}_{1}, \dots, \boldsymbol{\Theta}_{k}, \boldsymbol{X})\Big)^{2}\ \Big\vert \ \boldsymbol{\Theta}_{1}, \dots, \boldsymbol{\Theta}_{k}, \mathcal{X}_{n} \Big] \le \zeta2^{k}(2M_{0})^{2}. 
	\end{equation}}
	Hence, we can conclude that  for each $1 \le k \le c\log{n}$,
	\begin{equation}
		\label{new.eq.014}\mathop{{}\mathbb{E}}\left(  m_{\widehat{T} }^{*} (\boldsymbol{\Theta}_{1:k} , \boldsymbol{X}) - m_{ \widehat{T}_{\zeta }}^{*} (\boldsymbol{\Theta}_{1:k} , \boldsymbol{X})\right)^{2} \le \zeta2^{k+2}M_{0}^{2}.
	\end{equation}
	Therefore, in view of  (\ref{con1}), (\ref{new.eq.013}), and (\ref{new.eq.014}), it holds that for all large $n$ and each integer $1\le k \le c\log{n}$,
	\begin{equation}
		\begin{split}
			\label{new.eq.007}
			& \mathbb{E}  \left(  m(\boldsymbol{X}) - m_{\widehat{T} }^{*} (\boldsymbol{\Theta}_{1:k} , \boldsymbol{X})\right)^{2} \\
			&  \le 2\left( 4M_{0}^{2}n^{-\delta} 2^{k} +\alpha_{1}\alpha_{2}n^{-\eta} + M_{0}^{2}(1 - \gamma_{0}(\alpha_{1}\alpha_{2})^{-1})^{k} \right) + 2n^{-1},
	\end{split}\end{equation}
	which concludes the proof of Lemma~\ref{lemma1}.

	\subsection{Proof of Lemma~\ref{con3}} \label{SecC.3}
	The main idea of the proof for this lemma is based on the grid and has been discussed in Section~\ref{Sec5}. Let the grid be defined with positive parameters $\rho_{1}$ and  $\rho_{2}$; see Section~\ref{SecA.1} for details of these parameters.  With the grid and by some simple calculations, we can write 
	\begin{equation}\begin{split}	
			\label{con4}
			&\mathop{{}\mathbb{E}}      \Big(m_{\widehat{T}}^{*}(\boldsymbol{\Theta}_{1:k} , \boldsymbol{X}) - \widehat{m}_{\widehat{T}}(\boldsymbol{\Theta}_{1:k} , \boldsymbol{X}, \mathcal{X}_{n} ) \Big)^{2} \\ 
			&=  \mathbb{E}\Big(m_{ \widehat{T}}^{*}(\boldsymbol{\Theta}_{1:k} , \boldsymbol{X})- m_{ \widehat{T}^{\#}}^{*} (\boldsymbol{\Theta}_{1:k}, \boldsymbol{X}) + m_{ \widehat{T}^{\#}}^{*} (\boldsymbol{\Theta}_{1:k}, \boldsymbol{X}) - \widehat{m}_{ \widehat{T}^{\#}} (\boldsymbol{\Theta}_{1:k}, \boldsymbol{X}, \mathcal{X}_{n}) \\
			& \hspace{1em}+ \widehat{m}_{ \widehat{T}^{\#}} (\boldsymbol{\Theta}_{1:k}, \boldsymbol{X}, \mathcal{X}_{n}) - \widehat{m}_{ \widehat{T}} (\boldsymbol{\Theta}_{1:k}, \boldsymbol{X}, \mathcal{X}_{n})\Big)^{2}\\
			& \le 3\Big( \mathbb{E}\Big(m_{ \widehat{T}}^{*}(\boldsymbol{\Theta}_{1:k} , \boldsymbol{X})- m_{ \widehat{T}^{\#}}^{*} (\boldsymbol{\Theta}_{1:k}, \boldsymbol{X})   \Big)^{2} \\
			&\qquad+ \mathbb{E}\Big(  m_{ \widehat{T}^{\#}}^{*} (\boldsymbol{\Theta}_{1:k}, \boldsymbol{X}) - \widehat{m}_{ \widehat{T}^{\#}} (\boldsymbol{\Theta}_{1:k}, \boldsymbol{X}, \mathcal{X}_{n})  \Big)^{2} \\
			& \qquad + \mathbb{E}\Big( \widehat{m}_{ \widehat{T}^{\#}} (\boldsymbol{\Theta}_{1:k}, \boldsymbol{X}, \mathcal{X}_{n}) - \widehat{m}_{ \widehat{T}} (\boldsymbol{\Theta}_{1:k}, \boldsymbol{X}, \mathcal{X}_{n})\Big)^{2}  \Big).
	\end{split}\end{equation}
	
	Let us choose $ \min\{1, \nu + 1/2\} > \Delta > 1/2$. Then it follows from Lemma~\ref{lemma3} that there exists some constant $C > 0$ such that
	{\small\begin{equation}
		\begin{split}
			\label{con41}
			&\mathbb{E}\Big(m_{ \widehat{T}}^{*}(\boldsymbol{\Theta}_{1:k} , \boldsymbol{X})- m_{ \widehat{T}^{\#}}^{*} (\boldsymbol{\Theta}_{1:k}, \boldsymbol{X})   \Big)^{2} + \mathbb{E}\Big(  \widehat{m}_{ \widehat{T}^{\#}} (\boldsymbol{\Theta}_{1:k}, \boldsymbol{X}) - \widehat{m}_{ \widehat{T}} (\boldsymbol{\Theta}_{1:k}, \boldsymbol{X}, \mathcal{X}_{n})  \Big)^{2} \\
			& \le C2^{k}n^{-\frac{1}{2} + \nu}.
	\end{split}\end{equation}}
	
	\noindent Recall that conditional on $\mathcal{X}_{n}$, $\hat{T}$ is essentially associated with a deterministic splitting criterion. By this fact and Theorem \ref{theorem5}, we can deduce that \begin{equation}\begin{split}\label{con42}
			&\mathop{{}\mathbb{E}}      \Big(m_{\widehat{T}^{\#}}^{*}(\boldsymbol{\Theta}_{1:k} , \boldsymbol{X}) - \widehat{m}_{\widehat{T}^{\#}}(\boldsymbol{\Theta}_{1:k} , \boldsymbol{X}, \mathcal{X}_{n} ) \Big)^{2} \\
			&= \mathbb{E}      \left[ \mathbb{E} \left(\Big(m_{ \widehat{T}^{\#}}^{*}(\boldsymbol{\Theta}_{1:k} , \boldsymbol{X}) - \widehat{m}_{\widehat{T}^{\#}}(\boldsymbol{\Theta}_{1:k} , \boldsymbol{X}, \mathcal{X}_{n} ) \Big)^{2}\ \ \Big\vert \ \boldsymbol{\Theta}_{1:k}, \mathcal{X}_{n}\right) \right]\\
			& \le\mathop{{}\mathbb{E}}\left\{\sup_{T}  \mathop{{}\mathbb{E}}    \left[     \Big(m_{ T^{\#}}^{*}(\boldsymbol{\Theta}_{1:k} , \boldsymbol{X}) - \widehat{m}_{ T^{\#}}(\boldsymbol{\Theta}_{1:k} , \boldsymbol{X}, \mathcal{X}_{n} ) \Big)^{2} \ \Big\vert \ \boldsymbol{\Theta}_{1:k}, \mathcal{X}_{n} \right] \right\} \\
			& \le n^{-\eta},
	\end{split}\end{equation}
	where the supremum is over all possible tree growing rules. Therefore, combining (\ref{con4})--(\ref{con42}) completes the proof of Lemma~\ref{con3}.

	\subsection{Proof of Lemma \ref{lemma3}} \label{SecC.4}
	Recall that $\mathcal{X}_{n}$ denotes the $n$ i.i.d. observations and $\boldsymbol{X}$ is the independent copy of $\boldsymbol{x}_{1}$. The parameters $\Delta, c, C$ are given by Lemma~\ref{SecC.4}. Essentially Lemma~\ref{lemma3} shows that the population means conditional on an arbitrary cell $\boldsymbol{t}$ is very close to those conditional on cell $\boldsymbol{t}^{\#}$ in terms of the $\mathbb{L}^{2}$ distance. In addition to the population means, Lemma~\ref{lemma3} also considers the deviations of the sample means. To control those deviations, we exploit the results in \eqref{grid3} and \eqref{grid1}, the moment bounds of the model errors, and Condition~\ref{BO} (which states that $\sup_{\boldsymbol{c} \in [0, 1]^{p}}|m(\boldsymbol{c})| \le M_{0}$).  In what follows, we first establish the bound in (\ref{consistency2}). By Condition~\ref{BO}, we have that for each $n\ge1$ and $k \ge 1$,
	{\small\begin{equation}\begin{split}\label{h21}
			& \mathop{{}\mathbb{E}} \left(m_{\widehat{T}^{\#}}^{*} (\boldsymbol{\Theta}_{1:k}, \boldsymbol{X}) - m_{\widehat{T}}^{*} (\boldsymbol{\Theta}_{1:k}, \boldsymbol{X}) \right)^{2} \\
			&  = \mathbb{E}\left[   \mathbb{E} \left( \sum_{(\boldsymbol{t}_{1}, \dots, \boldsymbol{t}_{k}) \in \widehat{T}(\boldsymbol{\Theta}_{1:k})}  \Big(  m_{\widehat{T}^{\#}}^{*} (\boldsymbol{\Theta}_{1:k}, \boldsymbol{X}) - m_{\widehat{T}}^{*} (\boldsymbol{\Theta}_{1:k}, \boldsymbol{X})  \Big)^{2}  \boldsymbol{1}_{\boldsymbol{X} \in \boldsymbol{t}_{k}}  \ \Bigg\vert  \ \boldsymbol{\Theta}_{1:k}, \mathcal{X}_{n} \right) \right] \\
			& \le \mathbb{E}\Bigg[   \mathbb{E} \Bigg( \sum_{ \substack{(\boldsymbol{t}_{1}, \dots, \boldsymbol{t}_{k}) \in \widehat{T}(\boldsymbol{\Theta}_{1:k})  \\ \mathbb{P}(\boldsymbol{X} \in \boldsymbol{t}_{k}) \ge n^{\Delta - 1} }}  \Big(  m_{\widehat{T}^{\#}}^{*} (\boldsymbol{\Theta}_{1:k}, \boldsymbol{X}) - m_{ \widehat{T}}^{*} (\boldsymbol{\Theta}_{1:k}, \boldsymbol{X})  \Big)^{2}  \boldsymbol{1}_{\boldsymbol{X} \in \boldsymbol{t}_{k}}  \ \Bigg\vert  \ \boldsymbol{\Theta}_{1:k}, \mathcal{X}_{n} \Bigg)  \\
			&\quad+ 2^{k}(n^{\Delta - 1})(2M_{0})^{2}\Bigg].
	\end{split}\end{equation}}
	
	It follows from Condition~\ref{BO} and the definitions of the sharp notation and the population tree model that for each $n\ge 1$ and $ k \ge 1$,
	{\small\begin{equation}\begin{split}\label{h22}	
			&\textnormal{RHS of } (\ref{h21}) \\
			&  \le \mathbb{E}\Bigg[   \mathbb{E} \Bigg( \sum_{ \substack{(\boldsymbol{t}_{1}, \dots, \boldsymbol{t}_{k}) \in \widehat{T}(\boldsymbol{\Theta}_{1:k})  \\ \mathbb{P}(\boldsymbol{X} \in \boldsymbol{t}_{k}) \ge n^{\Delta - 1} }}  \Big( \mathbb{E}(m(\boldsymbol{X}) \ \vert \ \boldsymbol{X} \in\boldsymbol{t}_{k}) - \mathbb{E}(m(\boldsymbol{X}) \ \vert \ \boldsymbol{X} \in\boldsymbol{t}_{k}^{\#}) \Big)^{2}  \boldsymbol{1}_{\boldsymbol{X} \in \boldsymbol{t}_{k} \cap \boldsymbol{t}_{k}^{\#}} \\
			& \qquad +   (2M_{0})^{2}\boldsymbol{1}_{\boldsymbol{X} \in \boldsymbol{t}_{k}\Delta\boldsymbol{t}_{k}^{\#}}  \ \Bigg\vert  \ \boldsymbol{\Theta}_{1:k}, \mathcal{X}_{n} \Bigg)\Bigg]  + 2^{k}(n^{\Delta - 1})(2M_{0})^{2}\\
			&  \le \mathbb{E}\Bigg[   \mathbb{E} \Bigg( \sum_{ \substack{(\boldsymbol{t}_{1}, \dots, \boldsymbol{t}_{k}) \in \widehat{T}(\boldsymbol{\Theta}_{1:k})  \\ \mathbb{P}(\boldsymbol{X} \in \boldsymbol{t}_{k}) \ge n^{\Delta - 1} }}  \Big( \mathbb{E}(m(\boldsymbol{X}) \ \vert \ \boldsymbol{X} \in\boldsymbol{t}_{k}) - \mathbb{E}(m(\boldsymbol{X}) \ \vert \ \boldsymbol{X} \in\boldsymbol{t}_{k}^{\#}) \Big)^{2}  \boldsymbol{1}_{\boldsymbol{X} \in \boldsymbol{t}_{k}} \\
			& \qquad +   (2M_{0})^{2}\boldsymbol{1}_{\boldsymbol{X} \in \boldsymbol{t}_{k}\Delta\boldsymbol{t}_{k}^{\#}}  \ \Bigg\vert  \ \boldsymbol{\Theta}_{1:k}, \mathcal{X}_{n} \Bigg)\Bigg]  + 2^{k}(n^{\Delta - 1})(2M_{0})^{2}.
		\end{split}
	\end{equation}}
	
	To deal with the RHS of \eqref{h22}, we need to establish an upper bound for $ \mathbb{E}(m(\boldsymbol{X}) \ \vert \ \boldsymbol{X} \in\boldsymbol{t}_{k}) - \mathbb{E}(m(\boldsymbol{X}) \ \vert \ \boldsymbol{X} \in\boldsymbol{t}_{k}^{\#})$. In light of Condition~\ref{abc} ($f(\cdot)$ is the density of the distribution of $\boldsymbol{X}$), (\ref{grid3}), and Condition~\ref{BO}, it holds that for $\boldsymbol{t}_{k}$ in \eqref{h22}  with $\mathbb{P}(\boldsymbol{X} \in \boldsymbol{t}_{k}) \ge n^{\Delta - 1}$ and $1 \le k \le c\log{n}$, 
	{\small\begin{equation}
		\begin{split}
			\label{h24}
			& \left|\mathbb{E}(m(\boldsymbol{X}) \ \vert \ \boldsymbol{X} \in\boldsymbol{t}_{k}) - \mathbb{E}(m(\boldsymbol{X}) \ \vert \ \boldsymbol{X} \in\boldsymbol{t}_{k}^{\#}) \right|\\
			&  = \left|\frac{\mathbb{E}(m(\boldsymbol{X}) \boldsymbol{1}_{\boldsymbol{X} \in\boldsymbol{t}_{k}})}{\mathbb{P}(\boldsymbol{X} \in\boldsymbol{t}_{k})} - \frac{\mathbb{E}(m(\boldsymbol{X}) \boldsymbol{1}_{\boldsymbol{X} \in\boldsymbol{t}_{k}^{\#} })}{\mathbb{P}(\boldsymbol{X} \in\boldsymbol{t}_{k})} + \frac{\mathbb{E}(m(\boldsymbol{X}) \boldsymbol{1}_{\boldsymbol{X} \in\boldsymbol{t}_{k}^{\#}})}{\mathbb{P}(\boldsymbol{X} \in\boldsymbol{t}_{k})} - \frac{\mathbb{E}(m(\boldsymbol{X}) \boldsymbol{1}_{\boldsymbol{X} \in\boldsymbol{t}_{k}^{\#}})}{\mathbb{P}(\boldsymbol{X} \in\boldsymbol{t}_{k}^{\#})} \right|\\
			& \le \frac{\mathbb{E}(|m(\boldsymbol{X})| \boldsymbol{1}_{\boldsymbol{X} \in\boldsymbol{t}_{k}\Delta \boldsymbol{t}^{\#} }) }{\mathbb{P}(\boldsymbol{X} \in\boldsymbol{t}_{k})} +  \frac{\mathbb{E}(|m(\boldsymbol{X})| \boldsymbol{1}_{\boldsymbol{X} \in\boldsymbol{t}_{k}^{\#}})}{\mathbb{P}(\boldsymbol{X} \in\boldsymbol{t}_{k}^{\#})}\left| 1 - \frac{\mathbb{P}(\boldsymbol{X} \in\boldsymbol{t}_{k}^{\#})}{\mathbb{P}(\boldsymbol{X} \in\boldsymbol{t}_{k})} \right|	\\
			&\le \frac{M_{0}\mathbb{P}( \boldsymbol{X} \in\boldsymbol{t}_{k}\Delta \boldsymbol{t}^{\#} ) }{\mathbb{P}(\boldsymbol{X} \in\boldsymbol{t}_{k})} +  M_{0}\left| 1 - \frac{\mathbb{P}(\boldsymbol{X} \in\boldsymbol{t}_{k}^{\#})}{\mathbb{P}(\boldsymbol{X} \in\boldsymbol{t}_{k})} \right| \\
			& \le 2\frac{M_{0}\mathbb{P}( \boldsymbol{X} \in\boldsymbol{t}_{k}\Delta \boldsymbol{t}_{k}^{\#} ) }{\mathbb{P}(\boldsymbol{X} \in\boldsymbol{t}_{k})}\\
			&\le 2M_{0} (\sup f)(c\log{n})\frac{n^{1-\Delta}}{n^{1+\rho_{1}}},
	\end{split}\end{equation}}

\noindent where the third inequality follows from $|\mathbb{P}(A) - \mathbb{P}(B)| \le \mathbb{P}(A\Delta B)$ for two events $A, B$.

	Then by \eqref{h22}--\eqref{h24},  we have that for all large $n$ and each $1 \le k \le c\log{n}$,  
	\begin{equation}
		\begin{split}
			\label{h23}
			&\textnormal{RHS of } (\ref{h22}) \le \mathbb{E}\Bigg[   \mathbb{E} \Bigg( \sum_{ \substack{(\boldsymbol{t}_{1}, \dots, \boldsymbol{t}_{k}) \in \widehat{T}(\boldsymbol{\Theta}_{1:k}) }}  \Big( 2M_{0} (\sup f) (c\log{n}) \frac{n^{ 1 - \Delta}}{ n^{1 + \rho_{1}}}  \Big)^{2}  \boldsymbol{1}_{\boldsymbol{X} \in \boldsymbol{t}_{k}} \\
			& \hspace{10em}+   (2M_{0})^{2}\boldsymbol{1}_{\boldsymbol{X} \in \boldsymbol{t}_{k}\Delta\boldsymbol{t}_{k}^{\#}}  \ \Bigg\vert  \ \boldsymbol{\Theta}_{1:k}, \mathcal{X}_{n} \Bigg)\Bigg]  + 2^{k}(n^{\Delta - 1})(2M_{0})^{2}\\
			& \quad \le 2^{k}\Bigg(\Big( 2M_{0} (\sup f) (c\log{n}) \frac{n^{1 - \Delta}}{ n^{1 + \rho_{1}}}  \Big)^{2} + \Big((2M_{0})^{2} (\sup f) (c\log{n}) \frac{1}{ n^{1 + \rho_{1}}}  \Big)  \Bigg) \\
			&\quad \quad+ 2^{k}(n^{\Delta - 1})(2M_{0})^{2},
		\end{split}
	\end{equation}
	which leads to (\ref{consistency2}).
	
	We next proceed to show the bound in (\ref{consistency1}). Let $\bar{\Delta}$ with $1/2 < \bar{\Delta} <\Delta$ and sufficiently small $s>0$ be given such that $\bar{\Delta} = \Delta - 2s$. Let us define $\boldsymbol{C}_{n} \coloneqq \cap_{i=1}^{n} \{ |\varepsilon_{i}| \le n^{s} \}$. Observe that for each $n\ge 1$, conditional on $\mathcal{X}_{n}$ we have 
	\begin{enumerate}
		\item[1)] For each $\Theta_{1}, \dots, \Theta_{k}$ and tree growing rule, $ \sup_{\boldsymbol{c}\in [0, 1]^{p}}  |\widehat{m}_{ T}(\Theta_{1:k} , \boldsymbol{c}, \mathcal{X}_{n} ) | \le \sum_{i=1}^{n} |y_{i}|$;
		\item[2)] For each $\Theta_{1}, \dots, \Theta_{k}$ and tree growing rule, $ \sup_{\boldsymbol{c}\in [0, 1]^{p}} |\widehat{m}_{ T}(\Theta_{1:k} , \boldsymbol{c}, \mathcal{X}_{n} ) |\boldsymbol{1}_{\boldsymbol{C}_{n}} \le M_{0} + n^{s}$.
	\end{enumerate}
	We further define an $\mathcal{X}_{n}$-measurable event
	\[\boldsymbol{E}_{n} \coloneqq \boldsymbol{C}_{n} \cap \mathcal{A}_{3} (\floor{c\log{n}}, \bar{\Delta}) \cap \mathcal{A},\]
	where the event $\mathcal{A}_{3} (\floor{c\log{n}}, \bar{\Delta})$ is given in Lemma~\ref{CI1} in Section \ref{SecE.1} and $\mathcal{A}$ is defined in \eqref{eventA.1}. In particular, event $\mathcal{A}_{3} (\floor{c\log{n}}, \bar{\Delta})$ says that the number of observations on each cell $\boldsymbol{t}$ on the grid constructed by at most $\floor{c\log{n}}$ cuts and with $\mathbb{P}(\boldsymbol{X} \in \boldsymbol{t}) > n^{\bar{\Delta} - 1}$  is no less than $n^{1/2}$.

	Then by property 1) above, the Cauchy--Schwarz inequality, and  Minkowski's inequality, it holds that for each $n\ge 1$ and $k \ge 1$,
	\begin{equation}\begin{split}\label{h31}
			& \mathop{{}\mathbb{E}}      \Big(\widehat{m}_{\widehat{T}^{\#}} (\boldsymbol{\Theta}_{1:k}, \boldsymbol{X}, \mathcal{X}_{n}) - \widehat{m}_{\widehat{T}} (\boldsymbol{\Theta}_{1:k}, \boldsymbol{X}, \mathcal{X}_{n}) \Big)^{2} \\
			& \le \mathop{{}\mathbb{E}}    \left(  (2\sum_{i=1}^{n} |y_{i}|)^{2}\boldsymbol{1}_{\boldsymbol{E}_{n}^{c}}  +  \Big(\widehat{m}_{ \widehat{T}^{\#}} (\boldsymbol{\Theta}_{1:k}, \boldsymbol{X}, \mathcal{X}_{n}) - \widehat{m}_{\widehat{T}} (\boldsymbol{\Theta}_{1:k}, \boldsymbol{X}, \mathcal{X}_{n}) \Big)^{2} \boldsymbol{1}_{\boldsymbol{E}_{n}}\right) \\
			& \le   4 \left(\sum_{i=1}^{n}\Big( \mathbb{E} |y_{i}|^{4} \Big)^{1/4}\right)^{2} \Big(\mathbb{P}(\boldsymbol{E}_{n}^{c})\Big)^{\frac{1}{2}}  \\
			&\quad +  \mathop{{}\mathbb{E}} \left(\Big(\widehat{m}_{\widehat{T}^{\#}} (\boldsymbol{\Theta}_{1:k}, \boldsymbol{X}, \mathcal{X}_{n}) - \widehat{m}_{\widehat{T}} (\boldsymbol{\Theta}_{1:k}, \boldsymbol{X}, \mathcal{X}_{n}) \Big)^{2} \boldsymbol{1}_{\boldsymbol{E}_{n}}\right).
		\end{split}
	\end{equation}
	To bound the second term above, from the aforementioned property 2) and some basic calculations, we can obtain that for each $n$ and $1\le k \le c\log{n}$,
	\begin{equation}
		\begin{split}\label{h33}
			& \mathbb{E} \left(\Big(\widehat{m}_{ \widehat{T}^{\#}} (\boldsymbol{\Theta}_{1:k}, \boldsymbol{X}, \mathcal{X}_{n}) - \widehat{m}_{ \widehat{T}} (\boldsymbol{\Theta}_{1:k}, \boldsymbol{X}, \mathcal{X}_{n}) \Big)^{2} \boldsymbol{1}_{\boldsymbol{E}_{n}} \right)\\
			&\le Q_{1}+ 2^{k}(2n^{\bar{\Delta} - 1}) (2(M_{0} + n^{s}))^{2},
		\end{split}
	\end{equation}
	where
	\[Q_{1} \coloneqq \mathbb{E}\Bigg[ \mathbb{E}\Bigg( \sum_{\substack{(\boldsymbol{t}_{1:k}) \in \widehat{T} (\boldsymbol{\Theta}_{1:k} ) \\  \mathbb{P} (\boldsymbol{X} \in \boldsymbol{t}_{k}) \ge 2 n^{\bar{\Delta}-1}}}  \left( \widehat{m}_{ \widehat{T}^{\#}} (\boldsymbol{\Theta}_{1:k}, \boldsymbol{X}, \mathcal{X}_{n}) - \widehat{m}_{ \widehat{T}} (\boldsymbol{\Theta}_{1:k}, \boldsymbol{X}, \mathcal{X}_{n}) \right)^{2}\boldsymbol{1}_{{\boldsymbol{X} \in \boldsymbol{t}_{k}}}  \boldsymbol{1}_{\boldsymbol{E}_{n}} \ \Bigg\vert \ \mathcal{X}_{n}, \boldsymbol{\Theta}_{1:k}\Bigg)\Bigg] \]
	and the second term on the right-hand side of 	\eqref{h33} is the upper bound for a term similar to $Q_1$ but summing over $\{\boldsymbol{t}_{k}:\boldsymbol{t}_{1:k} \in \widehat{T} (\boldsymbol{\Theta}_{1:k} )\}$ with $  \mathbb{P} (\boldsymbol{X} \in \boldsymbol{t}_{k}) < 2 n^{\bar{\Delta}-1}$.

	 To further deal with $Q_{1}$, we need the following results \eqref{new.h.1}--\eqref{new.h.2}. Due to Condition~\ref{abc}, \eqref{grid3}, and the fact that $k\le c\log_{2}(n)$, for all large $n$, it holds that if $\mathbb{P} (\boldsymbol{X} \in \boldsymbol{t}_{k}) \ge 2 n^{\bar{\Delta}-1}$, then	
	\begin{equation}
		\label{new.h.1}\mathbb{P} (\boldsymbol{X} \in \boldsymbol{t}_{k}^{\#}) \ge n^{\bar{\Delta}-1}.
	\end{equation} 	
	In addition, it follows from the definition of sharp notation and property 2) above that for each $(\boldsymbol{t}_{1}, \dots, \boldsymbol{t}_{k}) \in \widehat{T}(\Theta_{1:k})$,
	\begin{equation}
		\begin{split}\label{new.h.2}
			&\left( \widehat{m}_{ \widehat{T}^{\#}} (\Theta_{1:k}, \boldsymbol{X}, \mathcal{X}_{n}) - \widehat{m}_{ \widehat{T}} (\Theta_{1:k}, \boldsymbol{X}, \mathcal{X}_{n}) \right)^{2}\boldsymbol{1}_{{\boldsymbol{X} \in \boldsymbol{t}_{k}}}\boldsymbol{1}_{\boldsymbol{E}_{n}} \\
			& \le
			\left(\left(  \bar{y}(\boldsymbol{t}_{k}^{\#})- \bar{y}(\boldsymbol{t}_{k}) \right)^{2}\boldsymbol{1}_{\boldsymbol{X} \in \boldsymbol{t}_{k}\cap\boldsymbol{t}_{k}^{\#}} + (2(M_{0} + n^{s}))^{2}\boldsymbol{1}_{\boldsymbol{X} \in \boldsymbol{t}_{k} \backslash \boldsymbol{t}_{k}^{\#}}\right)\boldsymbol{1}_{\boldsymbol{E}_{n}}\\
			& \le
			\left(\left(  \bar{y}(\boldsymbol{t}_{k}^{\#})- \bar{y}(\boldsymbol{t}_{k}) \right)^{2}\boldsymbol{1}_{\boldsymbol{X} \in \boldsymbol{t}_{k}}  + (2(M_{0} + n^{s}))^{2}\boldsymbol{1}_{\boldsymbol{X} \in \boldsymbol{t}_{k} \Delta \boldsymbol{t}_{k}^{\#}}\right)\boldsymbol{1}_{\boldsymbol{E}_{n}},
		\end{split}
	\end{equation}
	where for each cell $\boldsymbol{t}$,
	\[\bar{y}(\boldsymbol{t}) \coloneqq \frac{\sum_{\boldsymbol{x}_{i} \in \boldsymbol{t}} y_{i}}{ \#\{ i: \boldsymbol{x}_{i} \in \boldsymbol{t}\}}, \]
	and $\bar{y}(\boldsymbol{t})$ is defined as zero if the denominator is zero.
	
	By \eqref{new.h.1}--\eqref{new.h.2}, for all large $n$ and $1\le k \le c\log{n}$,
	\begin{equation}
		\begin{split}\label{h34}
			& \textnormal{RHS of (\ref{h33}) } \le Q_{2}+ 2^{k}(2n^{\bar{\Delta} - 1}) (2(M_{0} + n^{s}))^{2}
		\end{split}
	\end{equation}
	with
	\begin{equation*}
		\begin{split}
			&Q_{2} \coloneqq \\
			&\mathbb{E}\Bigg[ \mathbb{E}\Bigg( \sum_{\substack{(\boldsymbol{t}_{1:k}) \in \widehat{T} (\boldsymbol{\Theta}_{1:k} ) \\  \mathbb{P} (\boldsymbol{X} \in \boldsymbol{t}_{k}) \ge 2 n^{\bar{\Delta}-1} \\ \mathbb{P} (\boldsymbol{X} \in \boldsymbol{t}_{k}^{\#}) \ge n^{\bar{\Delta}-1} }   }  \left(\left(  \bar{y}(\boldsymbol{t}_{k}^{\#}) - \bar{y}(\boldsymbol{t}_{k}) \right)^{2} \boldsymbol{1}_{\boldsymbol{X} \in \boldsymbol{t}_{k}}  + (2(M_{0} + n^{s}))^{2}\boldsymbol{1}_{\boldsymbol{X} \in \boldsymbol{t}_{k} \Delta \boldsymbol{t}_{k}^{\#}}
			\right)
			\boldsymbol{1}_{\boldsymbol{E}_{n}} \ \Bigg\vert \ \mathcal{X}_{n}, \boldsymbol{\Theta}_{1:k}\Bigg)\Bigg].
	\end{split}\end{equation*}
	
	In what follows, we deal with $  \bar{y}(\boldsymbol{t}_{k}^{\#}) - \bar{y}(\boldsymbol{t}_{k}) $. By simple calculations, for every $\boldsymbol{t}_{k}$,
	\begin{equation}
		\begin{split}\label{new.h.3}
			&\bar{y}(\boldsymbol{t}_{k}^{\#}) - \bar{y}(\boldsymbol{t}_{k}) \\
			& = \frac{\sum_{\boldsymbol{x}_{i} \in \boldsymbol{t}_{k}^{\#}} y_{i}}{ \#\{ i: \boldsymbol{x}_{i} \in \boldsymbol{t}_{k}^{\#}\}} - \frac{\sum_{\boldsymbol{x}_{i} \in \boldsymbol{t}_{k}} y_{i}}{ \#\{ i: \boldsymbol{x}_{i} \in \boldsymbol{t}_{k}\}}\\
			& = \frac{(\#\{ i: \boldsymbol{x}_{i} \in \boldsymbol{t}_{k}\}) \left(\sum_{\boldsymbol{x}_{i} \in \boldsymbol{t}_{k}^{\#}} y_{i}\right) - (\#\{ i: \boldsymbol{x}_{i} \in \boldsymbol{t}_{k}^{\#}\}) \left(\sum_{\boldsymbol{x}_{i} \in \boldsymbol{t}_{k}} y_{i} \right)}{ (\#\{ i: \boldsymbol{x}_{i} \in \boldsymbol{t}_{k}^{\#}\}) \times (\#\{ i: \boldsymbol{x}_{i} \in \boldsymbol{t}_{k}\} ) } \\
			& = \frac{ \left(\sum_{\boldsymbol{x}_{i} \in \boldsymbol{t}_{k}^{\#}} y_{i}\right) -  \left(\sum_{\boldsymbol{x}_{i} \in \boldsymbol{t}_{k}} y_{i} \right)}{ \#\{ i: \boldsymbol{x}_{i} \in \boldsymbol{t}_{k}^{\#}\}  } \\
			& \qquad\qquad+ \left(\frac{ \#\{ i: \boldsymbol{x}_{i} \in \boldsymbol{t}_{k}\} - \#\{ i: \boldsymbol{x}_{i} \in \boldsymbol{t}_{k}^{\#}\}  }{ \#\{ i: \boldsymbol{x}_{i} \in \boldsymbol{t}_{k}^{\#}\} }\right)  \frac{  \sum_{\boldsymbol{x}_{i} \in \boldsymbol{t}_{k}} y_{i} }{ \#\{ i: \boldsymbol{x}_{i} \in \boldsymbol{t}_{k}\}  }\\
			& \le \frac{ \sum_{\boldsymbol{x}_{i} \in \boldsymbol{t}_{k}^{\#}\Delta \boldsymbol{t}} (|m(\boldsymbol{x}_{i})| + |\varepsilon_{i}|) }{ \#\{ i: \boldsymbol{x}_{i} \in \boldsymbol{t}_{k}^{\#}\}  } + \frac{\#\{i: \boldsymbol{x}_{i} \in \boldsymbol{t}_{k}\Delta \boldsymbol{t}_{k}^{\#}\}}{ \#\{ i: \boldsymbol{x}_{i} \in \boldsymbol{t}_{k}^{\#}\}  } \frac{  \sum_{\boldsymbol{x}_{i} \in \boldsymbol{t}_{k}} (|m(\boldsymbol{x}_{i})| + |\varepsilon_{i}|) }{ \#\{ i: \boldsymbol{x}_{i} \in \boldsymbol{t}_{k}\} }.
		\end{split}
	\end{equation}

By the definition of $\mathcal{A}_{3}(k, \bar{\Delta})$ (in Lemma~\ref{CI1}; recall that $\bar{\Delta} > \frac{1}{2}$), for each $\boldsymbol{t}_{k}$ satisfying the conditions specified in $Q_{2}$,
$$\#\{ i: \boldsymbol{x}_{i} \in \boldsymbol{t}_{k}^{\#}\}\ge n^{\frac{1}{2}}.$$

By this, \eqref{grid1},  Condition~\ref{BO}, and the fact that $k\le c\log{n}$, it holds that on $\boldsymbol{E}_{n}$, for each $\boldsymbol{t}_{k}$ satisfying the conditions specified in $Q_{2}$,
\begin{equation}
	\begin{split}\label{new.h.4}
			\textnormal{RHS of \eqref{new.h.3}}&\le \frac{2(M_{0} + n^{s})c(\log{n})^{2+\rho_{2}}}{n^{\frac{1}{2}}},
\end{split}
\end{equation}
	where we recall that $\rho_{2}>0$ is defined in \eqref{eventA.1}.

	By \eqref{new.h.3}--\eqref{new.h.4}, \eqref{grid3}, and  $\Delta = \bar{\Delta}+2s$, there exists some constant $ C>0$ such that for all large $n$ and $1\le k \le c\log{n}$,
	{\small\begin{equation}
		\begin{split}
			\label{h32}
			& \textnormal{RHS of (\ref{h34}) }\\
			& \le 2^{k}\left(  \left( \frac{2(M_{0}+n^{s}) c(\log{n})^{2+\rho_{2}}  }{n^{\frac{1}{2}}   }\right)^{2} + \frac{(2(M_{0}+n^{s}))^{2}}{n}\right) + 2^{k}(2n^{\bar{\Delta} - 1}) (2(M_{0} + n^{s}))^{2}\\
			& \le 2^{k}  \times (2(M_{0} + n^{s}))^{2} \times \left(\frac{2c^{2}(\log{n})^{4+2\rho_{2}} }{n}+ 2n^{\bar{\Delta} - 1}\right)\\
			& \le C2^{k}n^{\Delta -1},
		\end{split}
	\end{equation}}
	
	\noindent which gives the bound for the second term on the RHS of \eqref{h31}. For the first term on the RHS of  \eqref{h31}, in view of Conditions \ref{ME}--\ref{BO}, we have 
	\[ \frac{(\sum_{i=1}^{n} (\mathbb{E}|y_{i}|^{4} )^{1/4})^{2} }{ n^{2}} = O(1),\]
	and by Lemma~\ref{CI1}, \eqref{eventA}, and Condition~\ref{ME} with sufficiently large $q$, it holds that 
	\[ \mathbb{P}(\boldsymbol{E}_{n}^{c}) = o(n^{-6}). \]
	Therefore, combining these results, (\ref{h31}), and (\ref{h32}) yields (\ref{consistency1}), which concludes the proof of Lemma \ref{lemma3}.

	
	\subsection{Lemma \ref{T4_1} and its proof} \label{SecD.1}
	
	All the assumptions and notation in Lemma~\ref{T4_1} below follow those in the proof of Theorem~\ref{theorem4}. In particular, we set $k = \floor{ c\log{(n)}}$.
	
	\begin{lemma} \label{T4_1}
		There exists some constant $C > 0$ such that on event $\mathcal{A}_{3}( k + 1, \Delta)\cap \mathcal{A}$, it holds that for all large $n$ and each set of available features, each $\boldsymbol{t}$ constructed by less than $k$ cuts with $\mathbb{P}(\boldsymbol{X}\in \boldsymbol{t}) \ge n^{-\delta}$ and its daughter cells $\widehat{\boldsymbol{t}}$, $\boldsymbol{t}^{\dagger}$, and $\boldsymbol{t}^{*}$ satisfy the following properties:
		\begin{enumerate}
			\item[1)] $c^{\dagger}(\boldsymbol{t})$ is not random (i.e., $c^{\dagger}(\boldsymbol{t})$ is an element in \eqref{cdagger1}).
			
			\item[2)] $|(II)_{\boldsymbol{t}, \boldsymbol{t}^{\dagger} } - (II)_{\boldsymbol{t}, \boldsymbol{t}^{*}}| \le Cn^{-\frac{\delta}{2}}$.
			\item[3)] $\widehat{(II)}_{\boldsymbol{t}, \widehat{\boldsymbol{t}} } \ge \widehat{(II)}_{\boldsymbol{t}, \boldsymbol{t}^{\dagger} }$.
		\end{enumerate}
	\end{lemma}
	
	\noindent \textit{Proof}. Let us assume that property 1) holds for the moment. 
	Then by the definition of $\widehat{(II)}$ and the definitions of $\hat{\boldsymbol{t}}$ and $\boldsymbol{t}^{\dagger}$ (they are both daughter cells of $\boldsymbol{t}$ and the corresponding cuts are along directions subject to the same set of available features), we have that if $c^{\dagger}(\boldsymbol{t})$ is not random, then it holds that  $\widehat{(II)}_{\boldsymbol{t}, \widehat{\boldsymbol{t}} } \ge \widehat{(II)}_{\boldsymbol{t}, \boldsymbol{t}^{\dagger} }$, which establishes property 3). 
	From (\ref{control1}), the assumptions of Theorem~\ref{theorem4}, and some simple calculations, we can see that there exists some constant $C >0$ such that for all large $n$ and each set of available features,
	\[|(II)_{\boldsymbol{t}, \boldsymbol{t}^{\dagger} } - (II)_{\boldsymbol{t}, \boldsymbol{t}^{*}}| \le  Cn^{-\frac{\delta}{2}}, \]
	which proves property 2). Note that (\ref{control1}) holds for each feature restriction $\Theta$.

	Now it remains to establish property 1), which means we have to show that the set \eqref{cdagger1} is not empty. For each cell $\boldsymbol{t} = \times_{j=1}^{p}t_{j}$, define a cell
	\[ \boldsymbol{I}(\boldsymbol{t}, h, I) \coloneqq t_{1} \times \dots \times t_{h-1}\times I\times t_{h+1} \times \dots \times t_{p} \]
	for $h\in \{1, \dots, p\}$ and an interval $I \subset [0, 1]$. Denote by $R(\boldsymbol{t}, h, \delta)$ a set containing all the intervals $J$ such that $J \subset t_{h}$ and $\mathbb{P}( X_{h} \in J | \boldsymbol{X} \in \boldsymbol{t}) = n^{-\delta}$. Observe that if cell $\boldsymbol{t}$ is constructed by less than $k$ cuts, then $\boldsymbol{I}(\boldsymbol{t}, h, I)$ with $I \in R(\boldsymbol{t}, h, \delta)$ is constructed by at most $k+1$ cuts. For each integer $k$, define $H_{k}$ as the set containing all cells constructed by at most $k$ arbitrary cuts (these cuts are not necessarily on the gridlines). Let us define an event 
	\[\boldsymbol{B}(k) \coloneqq \left\{ \inf_{\substack{   \mathbb{P}(\boldsymbol{X} \in \boldsymbol{t} ) \ge  n^{-\delta}, \ \boldsymbol{t}\in H_{k-1}, \\ h \in \{1, \dots, p\}, \ I \in R (\boldsymbol{t}, h, \delta)}  } \ \sum_{i=1}^{n} \boldsymbol{1}_{\boldsymbol{x}_{i} \in  \boldsymbol{I}(\boldsymbol{t}, h, I)} < 1  \right\},\]
	where the infimum is over all $\boldsymbol{t}$, $h$, and $I$ such that the conditions hold. Then we can see that on event $(\boldsymbol{B}(k))^{c}$, property 1) holds, where the superscript $c$ denotes the set complement. 
	
	Next, recall some notation related to the grid defined in Sections~\ref{Sec5} and~\ref{SecA.1}, including $\boldsymbol{t}^{\#}$, the event $\mathcal{A}$, and parameters $\rho_{1}>0, \rho_{2}>0$. By $\Delta < 1 - 2\delta$ (see \eqref{T4_0} for the definitions of $\delta, \Delta$) and Condition~\ref{abc}, for all large $n$ we have 
	\begin{equation}\label{T4_4}
		\boldsymbol{B}(k)\subset   \left\{ \inf_{\substack{   \mathbb{P}(\boldsymbol{X} \in \boldsymbol{t} ) \ge  \big(n^{\Delta - 1} +   (k+1) \frac{\sup f}{ \lceil n^{1+\rho_{1}} \rceil } \big)n^{\delta}, \\ \boldsymbol{t}\in H_{k-1}, \  h \in\{1, \dots, p\}, \ I \in R(\boldsymbol{t}, h, \delta) }  } \ \sum_{i=1}^{n} \boldsymbol{1}_{\boldsymbol{x}_{i} \in  \boldsymbol{I}(\boldsymbol{t}, h, I)} < 1  \right\},
	\end{equation}
where we use $n^{-\delta}\ge \Big(n^{\Delta - 1} +   (k+1) \frac{\sup f}{ \lceil n^{1+\rho_{1}} \rceil } \Big)n^{\delta}$ for all large $n$ because of $\Delta < 1 - 2\delta$ and $k = \floor{c\log{n}}$ (recall that $k$ is defined to be $\floor{c\log{n}}$ in this proof). Notice that the infimum on the RHS of \eqref{T4_4} is over the cells in 
{\small$$W\coloneqq \left\{\boldsymbol{I}(\boldsymbol{t}, h, I):\mathbb{P}(\boldsymbol{X} \in \boldsymbol{t} ) \ge  \Big(n^{\Delta - 1} +   (k +1) \frac{\sup f}{ \lceil n^{1+\rho_{1}} \rceil }\Big)n^{\delta}, \boldsymbol{t} \in H_{k-1}, h\in \{1, \dots, p\}, I \in R(\boldsymbol{t},h, \delta)\right\}.$$}

	Next, it follows from the definitions of $R(\boldsymbol{t}, h, \delta)$ and $H_{k}$ that for  every $n\ge1$,
	\begin{equation*}
		\begin{split}
		&W \subset  \left\{\boldsymbol{t}: \boldsymbol{t}\in H_{k+1}, \mathbb{P}(\boldsymbol{X} \in \boldsymbol{t} ) \ge   n^{\Delta - 1} +   (k +1) \frac{\sup f}{ \lceil n^{1+\rho_{1}} \rceil }\right\},
	\end{split}
\end{equation*}
	and hence for each $n\ge 1$,
	\begin{equation}
		\begin{split}\label{T4_5}					
		&\textnormal{RHS of } (\ref{T4_4}) \\
		& \subset \left\{ \inf_{\boldsymbol{t}:\  \boldsymbol{t}\in H_{k+1}, \  \mathbb{P}(\boldsymbol{X} \in \boldsymbol{t} ) \ge  n^{\Delta - 1} +   (k +1) \frac{\sup f}{ \lceil n^{1+\rho_{1}} \rceil }} \sum_{i=1}^{n} \boldsymbol{1}_{\boldsymbol{x}_{i} \in \boldsymbol{t}} < 1  \right\}.
		\end{split}
	\end{equation}

Moreover, from (\ref{grid3}) we can obtain that for each $n\ge 1$,
	\begin{equation}\begin{split}			
		\label{T4_6}
		\textnormal{RHS of } (\ref{T4_5}) &\subset   \left\{ \inf_{\substack{   \mathbb{P}(\boldsymbol{X} \in \boldsymbol{t}^{\#} ) \ge  n^{\Delta - 1} \\ \boldsymbol{t}\in H_{k+1} }  } \sum_{i=1}^{n} \boldsymbol{1}_{\boldsymbol{x}_{i} \in \boldsymbol{t}} < 1  \right\},
		\end{split}
	\end{equation}
	and by simple calculations,
	\begin{equation}\begin{split}			
		\label{T4_16}
		\textnormal{RHS of } (\ref{T4_6}) & = \left\{ \inf_{\substack{   \mathbb{P}(\boldsymbol{X} \in \boldsymbol{t}^{\#} ) \ge  n^{\Delta - 1} \\ \boldsymbol{t}\in H_{k+1} }  } \sum_{i=1}^{n} \Big(\boldsymbol{1}_{\boldsymbol{x}_{i}\in \boldsymbol{t}^{\#}} + \boldsymbol{1}_{\boldsymbol{x}_{i}\in \boldsymbol{t}\backslash\boldsymbol{t}^{\#}} - \boldsymbol{1}_{\boldsymbol{x}_{i}\in \boldsymbol{t}^{\#}\backslash\boldsymbol{t}}\Big) < 1  \right\}\\
		& \subset \left\{ \inf_{\substack{   \mathbb{P}(\boldsymbol{X} \in \boldsymbol{t}^{\#} ) \ge  n^{\Delta - 1} \\ \boldsymbol{t}\in H_{k+1} }  } \left(\sum_{i=1}^{n} \boldsymbol{1}_{\boldsymbol{x}_{i}\in \boldsymbol{t}^{\#}} - \sum_{i=1}^{n} \boldsymbol{1}_{\boldsymbol{x}_{i}\in \boldsymbol{t}^{\#}\Delta \boldsymbol{t}}\right) < 1  \right\}.
	\end{split}\end{equation}

	Then by \eqref{grid1}, which says $\sum_{i=1}^{n} \boldsymbol{1}_{\boldsymbol{x}_{i}\in \boldsymbol{t}^{\#}\Delta \boldsymbol{t}} < (k+1)(\log{n})^{1 + \rho_{2}}$ on $\mathcal{A}$, it holds that for each $n \ge 1$,
	\begin{equation}
		\begin{split}	
			\label{T4_7}
			& \textnormal{RHS of } (\ref{T4_6}) \\
			& \subset   \left(\left\{ \inf_{\substack{   \mathbb{P}(\boldsymbol{X} \in \boldsymbol{t}^{\#} ) \ge  n^{\Delta - 1} \\ \boldsymbol{t}\in H_{k+1} }  } \sum_{i=1}^{n} \boldsymbol{1}_{\boldsymbol{x}_{i} \in \boldsymbol{t}^{\#}} < 1  + (k + 1) (\log{n})^{1+\rho_{2}} \right\}\cap \mathcal{A}\right)\cup\mathcal{A}^{c}.
	\end{split}\end{equation}

	By $k = \floor{c\log{(n)}}$, for all large $n$, 
	\begin{equation}\begin{split}\label{T4_8}
			\textnormal{RHS of } (\ref{T4_7}) & \subset   \left\{ \inf_{\substack{   \mathbb{P}(\boldsymbol{X} \in \boldsymbol{t}^{\#} ) \ge  n^{\Delta - 1} \\ \boldsymbol{t}\in H_{k+1} }  } \sum_{i=1}^{n} \boldsymbol{1}_{\boldsymbol{x}_{i} \in \boldsymbol{t}^{\#}} < n^{\frac{1}{2}} \right\}\cup\mathcal{A}^{c},
	\end{split}\end{equation}
where we also remove the intersection of the event $\mathcal{A}$. By the definitions of $G_{n, k+1}(\Delta)$ and $\mathcal{A}_{3}(k + 1, \Delta)$ in Lemma~\ref{CI1} of Section~\ref{SecE.1},
\begin{equation*}
	\begin{split}
		\left\{ \inf_{\substack{   \mathbb{P}(\boldsymbol{X} \in \boldsymbol{t}^{\#} ) \ge  n^{\Delta - 1} \\ \boldsymbol{t}\in H_{k+1} }  } \sum_{i=1}^{n} \boldsymbol{1}_{\boldsymbol{x}_{i} \in \boldsymbol{t}^{\#}} < n^{\frac{1}{2}} \right\} &= \left\{ \inf_{ \boldsymbol{t}\in G_{n, k+1}(\Delta) }   \sum_{i=1}^{n} \boldsymbol{1}_{\boldsymbol{x}_{i} \in \boldsymbol{t}} < n^{\frac{1}{2}} \right\}\\
		&=\cup_{\boldsymbol{t}\in G_{n, k+1}(\Delta) } \{    \#\{i: \boldsymbol{x}_{i} \in \boldsymbol{t}\} < n^{\frac{1}{2}} \}\\
		& = (\mathcal{A}_{3}(k + 1, \Delta))^c.
\end{split}
\end{equation*}

By this,
\begin{equation}\begin{split}\label{T4_15}
			\textnormal{RHS of } (\ref{T4_8})  & \subset ( \mathcal{A}_{3}(k+1, \Delta) )^{c}\cup\mathcal{A}^{c}.
				\end{split}\end{equation}
	Therefore, in view of (\ref{T4_4})--(\ref{T4_15}), we can conclude that $\mathcal{A}_{3}(k + 1, \Delta)\cap \mathcal{A} \subset (\boldsymbol{B}(k))^{c}$, which leads to property 1). This completes the proof of Lemma \ref{T4_1}.

	\subsection{Lemma \ref{T4_2} and its proof} \label{SecD.2}
	All the assumptions and notation in Lemma~\ref{T4_2} below follow those in the proof of Theorem~\ref{theorem4}. In particular, we set $k = \floor{ c\log{(n)}}$.
	
	\begin{lemma} \label{T4_2}
		There exists some constant $C > 0$ such that on event $\boldsymbol{U}_{n}$, it holds that for all large $n$, each cell $\boldsymbol{t}$ constructed by less than $k$ cuts with $\mathbb{P}(\boldsymbol{X}\in\boldsymbol{t}) \ge n^{-\delta}$ and each daughter cell $\boldsymbol{t}^{'}$ of  $\boldsymbol{t}$ satisfy the following properties:
		\begin{enumerate}
			\item[1)] $|(II)_{\boldsymbol{t}, \boldsymbol{t}^{'} } - (II)_{\boldsymbol{t}^{\#}, (\boldsymbol{t}^{'})^{\#}}| \le Cn^{-\delta}$.
			\item[2)] $| (II)_{\boldsymbol{t}^{\#}, (\boldsymbol{t}^{'})^{\#}} -\widehat{ (II)}_{\boldsymbol{t}^{\#}, (\boldsymbol{t}^{'})^{\#}}| \le C(  n^{-\delta + 2s}+  n^{-\frac{\Delta^{'}}{4} + 2s} )$.
			\item[3)] $|\widehat{ (II)}_{\boldsymbol{t}^{\#}, (\boldsymbol{t}^{'})^{\#}}  - \widehat{(II)}_{\boldsymbol{t}, \boldsymbol{t}^{'} }| \le C(n^{-\delta + 2s} + n^{-\frac{\Delta^{'}}{4} + 2s} )$.
		\end{enumerate}
	\end{lemma}
	
	\noindent \textit{Proof}.  We prove property 2) and the other two properties can be shown using similar arguments. Let $\boldsymbol{t}^{''}$ be the other daughter cell of $\boldsymbol{t}$. From the definition, we can deduce that 
	{\small\begin{equation}
		\begin{split}\label{T4_10}
			& | (II)_{\boldsymbol{t}^{\#}, (\boldsymbol{t}^{'})^{\#}} -\widehat{ (II)}_{\boldsymbol{t}^{\#}, (\boldsymbol{t}^{'})^{\#}}|\\
			& \le  \Bigg|   \frac{\# \{ i: \boldsymbol{x}_{i} \in (\boldsymbol{t}^{'})^{\#}\} }{\# \{ i: \boldsymbol{x}_{i} \in\boldsymbol{t}^{\#} \}} \left( \sum_{\boldsymbol{x}_{i}\in( \boldsymbol{t}^{'})^{\#}} \frac{y_{i}}{\#\{   i: \boldsymbol{x}_{i} \in (\boldsymbol{t}^{'})^{\#}\} }  -  \sum_{\boldsymbol{x}_{i}\in \boldsymbol{t}^{\#}  } \frac{y_{i}}{\#\{   i: \boldsymbol{x}_{i} \in \boldsymbol{t}^{\#}\} }    \right)^{2}    \\
			& \hspace{1em}-     \mathop{{}\mathbb{P}}(\boldsymbol{X}\in(\boldsymbol{t}^{'})^{\#}\ \vert \ \boldsymbol{X} \in \boldsymbol{t}^{\#} )\Big(\mathop{{}\mathbb{E}}(m(\boldsymbol{X}) \ \vert \ \boldsymbol{X}\in (\boldsymbol{t}^{'})^{\#}) - \mathop{{}\mathbb{E}}(m(\boldsymbol{X}) \ \vert \ \boldsymbol{X} \in \boldsymbol{t}^{\#} ) \Big)^{2} \Bigg|\\
			& \hspace{1em}+ \Bigg|   \frac{\# \{ i: \boldsymbol{x}_{i} \in (\boldsymbol{t}^{''})^{\#}\} }{\# \{ i: \boldsymbol{x}_{i} \in\boldsymbol{t}^{\#} \}} \left( \sum_{\boldsymbol{x}_{i}\in( \boldsymbol{t}^{''})^{\#}} \frac{y_{i}}{\#\{   i: \boldsymbol{x}_{i} \in (\boldsymbol{t}^{''})^{\#}\} }  -  \sum_{\boldsymbol{x}_{i}\in \boldsymbol{t}^{\#}  } \frac{y_{i}}{\#\{   i: \boldsymbol{x}_{i} \in \boldsymbol{t}^{\#}\} }    \right)^{2}    \\
			& \hspace{1em}-     \mathop{{}\mathbb{P}}(\boldsymbol{X}\in(\boldsymbol{t}^{''})^{\#}\ \vert \ \boldsymbol{X} \in \boldsymbol{t}^{\#} )\Big(\mathop{{}\mathbb{E}}(m(\boldsymbol{X}) \ \vert \ \boldsymbol{X}\in (\boldsymbol{t}^{''})^{\#}) - \mathop{{}\mathbb{E}}(m(\boldsymbol{X}) \ \vert \ \boldsymbol{X} \in \boldsymbol{t}^{\#} ) \Big)^{2} \Bigg|.
		\end{split}
	\end{equation}}
	Without loss of generality, we need only to deal with the first term on the RHS of (\ref{T4_10}). 
	
	By Condition~\ref{BO}, we have that for each $n\ge 1$, 
	\begin{equation}\label{T4_9}
		\left|     \frac{\sum_{\boldsymbol{x}_{i} \in \boldsymbol{t}} y_{i}}{  \#\{  i : \boldsymbol{x}_{i}\in\boldsymbol{t}\}} \right| \boldsymbol{1}_{\boldsymbol{U}_{n}} \le M_{0}+ n^{s}.
	\end{equation}
	Let us make use of three useful claims below, where we omit the presumption that $\boldsymbol{t}$ is constructed by less than $k$ cuts, and that $\boldsymbol{t}^{'}$ is the daughter cell of $\boldsymbol{t}$. By \eqref{grid3}, Condition~\ref{abc} ($f(\cdot)$ is the density of the distribution of $\boldsymbol{X}$, which is the independent copy of $\boldsymbol{x}_{1}$), and the choices of $\Delta$ and $\delta$ ($\Delta, \delta$ are defined in Theorem~\ref{theorem4}), it holds that 
	\begin{enumerate}
		\item[a)] For all large $n$ and each $\boldsymbol{t}$ with $\mathbb{P}(\boldsymbol{X}\in\boldsymbol{t}) \ge n^{-\delta}$,
		\[ \mathbb{P}(\boldsymbol{X}\in\boldsymbol{t}^{\#}) \ge n^{\Delta - 1};   \] 
		\item[b)] For all large $n$ and each $\boldsymbol{t}^{'}$ and $\boldsymbol{t}$ with $\mathbb{P} (\boldsymbol{X} \in \boldsymbol{t}^{'} \ \vert \ \boldsymbol{X} \in \boldsymbol{t}) \ge n^{-\delta}$ and $\mathbb{P}(\boldsymbol{X}\in\boldsymbol{t}) \ge n^{-\delta}$, 
		\[   \mathbb{P}(\boldsymbol{X}\in(\boldsymbol{t}^{'} )^{\#}) \ge n^{\Delta - 1}.  \]
		
		To show this result, note that by \eqref{grid3}, that $k = \floor{c\log{n}}$, and the assumption that $\boldsymbol{t}$ is constructed by less than $k$ cuts,
		\begin{equation*}
		    \label{new.explain.b1}
		    |\mathbb{P} (\boldsymbol{X} \in \boldsymbol{t}^{'}) - \mathbb{P} (\boldsymbol{X} \in (\boldsymbol{t}^{'})^{\#})| \le \floor{c\log{n}}\frac{\sup f}{ \ceil{n^{1+\rho_{1}}}},
		\end{equation*}
		and hence by the assumptions and that $\Delta  -1<  -2\delta$, it holds that for all large $n$,
		$$\mathbb{P} (\boldsymbol{X} \in (\boldsymbol{t}^{'})^{\#})\ge \mathbb{P} (\boldsymbol{X} \in \boldsymbol{t}^{'}) - \floor{c\log{n}}\frac{\sup f}{ \ceil{n^{1+\rho_{1}}}}\ge n^{-2\delta}- \floor{c\log{n}}\frac{\sup f}{ \ceil{n^{1+\rho_{1}}}} \ge n^{\Delta -1}.$$
		Recall that $\rho_{1} >0$ is defined for the gird in Section~\ref{Sec5}.
		\item[c)] For all large $n$ and each $\boldsymbol{t}^{'}$ and $\boldsymbol{t}$ with $\mathbb{P} (\boldsymbol{X} \in \boldsymbol{t}^{'} \ \vert \ \boldsymbol{X} \in \boldsymbol{t}) < n^{-\delta}$ and $\mathbb{P}(\boldsymbol{X}\in\boldsymbol{t}) \ge n^{-\delta}$, 
		\[   \mathbb{P}(\boldsymbol{X}\in(\boldsymbol{t}^{'} )^{\#} \ \vert \ \boldsymbol{X} \in \boldsymbol{t}^{\#})< 2 n^{-\delta}.\]
		To show this result, notice that
		\begin{equation*}
		    \begin{split}
		         &\Big| \frac{\mathbb{P}(\boldsymbol{X}\in(\boldsymbol{t}^{'} )^{\#})}{\mathbb{P}( \boldsymbol{X} \in \boldsymbol{t}^{\#})} - \frac{\mathbb{P}(\boldsymbol{X}\in(\boldsymbol{t}^{'} ))}{\mathbb{P}( \boldsymbol{X} \in \boldsymbol{t})} \Big|\\
		         & = \Big| \frac{\mathbb{P}(\boldsymbol{X}\in(\boldsymbol{t}^{'} )^{\#})}{\mathbb{P}( \boldsymbol{X} \in \boldsymbol{t}^{\#})} - \frac{\mathbb{P}(\boldsymbol{X}\in(\boldsymbol{t}^{'} )^{\#})}{\mathbb{P}( \boldsymbol{X} \in \boldsymbol{t})} + \frac{\mathbb{P}(\boldsymbol{X}\in(\boldsymbol{t}^{'} )^{\#})}{\mathbb{P}( \boldsymbol{X} \in \boldsymbol{t})} - \frac{\mathbb{P}(\boldsymbol{X}\in(\boldsymbol{t}^{'} ))}{\mathbb{P}( \boldsymbol{X} \in \boldsymbol{t})} \Big|\\
		         & \le  \Big|\frac{(\mathbb{P}( \boldsymbol{X} \in \boldsymbol{t}) - \mathbb{P}( \boldsymbol{X} \in \boldsymbol{t}^{\#})) \mathbb{P}(\boldsymbol{X}\in(\boldsymbol{t}^{'} )^{\#})}{\mathbb{P}( \boldsymbol{X} \in \boldsymbol{t}^{\#}) \mathbb{P}( \boldsymbol{X} \in \boldsymbol{t})}\Big| + n^{\delta}\mathbb{P}(\boldsymbol{X}\in(\boldsymbol{t}^{'})^{\#}\Delta\boldsymbol{t}^{'})\\
		         & \le 2n^{\delta}\mathbb{P}(\boldsymbol{X}\in(\boldsymbol{t}^{'})^{\#}\Delta\boldsymbol{t}^{'})\\
		         & = o(n^{-\delta}),
		    \end{split}
		\end{equation*}
		where the inequalities uses the assumption $\mathbb{P}(\boldsymbol{X}\in\boldsymbol{t}) \ge n^{-\delta}$ and the last equality follows from \eqref{grid3}. The desired result follows from this and the assumption $\mathbb{P} (\boldsymbol{X} \in \boldsymbol{t}^{'} \ \vert \ \boldsymbol{X} \in \boldsymbol{t}) < n^{-\delta}$.
	\end{enumerate}

	We first consider the case of $\boldsymbol{t}$ and $\boldsymbol{t}^{'}$ with $\mathbb{P} (\boldsymbol{X} \in \boldsymbol{t}^{'} \ \vert \ \boldsymbol{X} \in \boldsymbol{t}) \ge n^{-\delta}$.  In light of (\ref{T4_9}) and the above claims  a) and b), there exists some constant $C > 0$ such that on event $\boldsymbol{U}_{n}$ (specifically, on $\mathcal{A}_{1}(\floor{c\log{(n)}}, \Delta)\cap\mathcal{A}_{2}(\floor{c\log{(n)}}, \Delta)$), for all large $n$ and each such $\boldsymbol{t}$ and $\boldsymbol{t}^{'}$ we have 
	\begin{equation}\begin{split}\label{T4_12}
			&  \Bigg|   \frac{\# \{ i: \boldsymbol{x}_{i} \in (\boldsymbol{t}^{'})^{\#}\} }{\# \{ i: \boldsymbol{x}_{i} \in\boldsymbol{t}^{\#} \}} \left( \sum_{\boldsymbol{x}_{i}\in( \boldsymbol{t}^{'})^{\#}} \frac{y_{i}}{\#\{   i: \boldsymbol{x}_{i} \in (\boldsymbol{t}^{'})^{\#}\} }  -  \sum_{\boldsymbol{x}_{i}\in \boldsymbol{t}^{\#}  } \frac{y_{i}}{\#\{   i: \boldsymbol{x}_{i} \in \boldsymbol{t}^{\#}\} }    \right)^{2}    \\
			& \hspace{1em}-     \mathop{{}\mathbb{P}}(\boldsymbol{X}\in(\boldsymbol{t}^{'})^{\#}\ \vert \ \boldsymbol{X} \in \boldsymbol{t}^{\#} )\Big(\mathop{{}\mathbb{E}}(m(\boldsymbol{X}) \ \vert \ \boldsymbol{X}\in (\boldsymbol{t}^{'})^{\#}) - \mathop{{}\mathbb{E}}(m(\boldsymbol{X}) \ \vert \ \boldsymbol{X} \in \boldsymbol{t}^{\#} ) \Big)^{2} \Bigg| \\
			& \le Cn^{\frac{-\Delta^{'}}{4} + s}, 
		\end{split}
	\end{equation}
	where $\Delta^{'}$ is defined in Theorem~\ref{theorem4}.
	
	For the other case of $\boldsymbol{t}$ and $\boldsymbol{t}^{'}$ with $\mathbb{P} (\boldsymbol{X} \in \boldsymbol{t}^{'} \ \vert \ \boldsymbol{X} \in \boldsymbol{t}) < n^{-\delta}$, it follows from (\ref{T4_9}) and the above claim c) that there exists some constant $C>0$ such that on event $\boldsymbol{U}_{n}$, for all large $n$ and each such $\boldsymbol{t}$ and $\boldsymbol{t}^{'}$ we have 
	\begin{equation}\label{T4_13}
		\textnormal{LHS of } (\ref{T4_12}) \le  C(  n^{-\delta + 2s}+  n^{-\frac{\Delta^{'}}{4} + 2s} ). 
	\end{equation}
	Therefore, combining (\ref{T4_10}) and (\ref{T4_12})--(\ref{T4_13}), we can establish property 2), which concludes the proof of Lemma \ref{T4_2}.
	

	\section{Additional lemmas and technical details} \label{SecE}

	\subsection{Lemma \ref{CI1} and its proof} \label{SecE.1}
	Let the sample size $n$ and  tree level $k$ be given. Let $G_{n, k}$ be as defined in Section~\ref{SecA.1}; for the reader's convenience, $G_{n,k}$ is the set containing all cells on the grid constructed by at most $k$ cuts with cuts all on the grid hyperplanes defined in Section \ref{Sec5}. For $\Delta > 0$, we also define $G_{n, k}(\Delta)$ as the subset of $G_{n,k}$ such that if $\boldsymbol{t} \in G_{n, k}$ and
	$\mathbb{P}(\boldsymbol{X} \in \boldsymbol{t}) \ge n^{\Delta - 1}$, then $\boldsymbol{t} \in G_{n, k}(\Delta)$. To simplify the notation, the complement of an event that depends on some parameters such as  $\mathcal{A}(\cdot)$ is denoted as $\mathcal{A}^{c}(\cdot)$.
	
	\begin{lemma}\label{CI1}		
		Let $\frac{1}{2} < \Delta < 1$, $c>0$, $\kappa > 0$, and $0 <\Delta^{'} < \Delta$ be given and assume Condition~\ref{ME} with $q > \frac{4 + 4\kappa}{\Delta^{'}}$ and Condition~\ref{BO}. We define 
		\begin{equation*}
			\begin{split}
				\mathcal{A}_{1}^{c}(k, \Delta)& \coloneqq \cup_{\boldsymbol{t} \in G_{n, k}(\Delta) } \left\{\left|  \frac{\sum_{\boldsymbol{x}_{i} \in \boldsymbol{t}}  y_{i}  }{ \#\{ i: \boldsymbol{x}_{i} \in \boldsymbol{t} \}}  - \mathop{{}\mathbb{E}}(m(\boldsymbol{X}) \ \vert \ \boldsymbol{X} \in \boldsymbol{t})\right|  \ge n^{-\frac{\Delta^{'}}{4}} \right\},\\
				\mathcal{A}_{2}^{c}(k, \Delta)&\coloneqq  \cup_{\substack{  \boldsymbol{t} \in G_{n, k-1}(\Delta) \\ \boldsymbol{t}^{'} \in G_{n, k}\\ \boldsymbol{t}^{'} \subset \boldsymbol{t} } } \left\{\left|  \frac{  \#\{ i: \boldsymbol{x}_{i} \in \boldsymbol{t}^{'} \}   }{ \#\{ i: \boldsymbol{x}_{i} \in \boldsymbol{t} \}} - \mathop{{}\mathbb{P}}( \boldsymbol{X} \in \boldsymbol{t}^{'}\ \vert \ \boldsymbol{X} \in \boldsymbol{t}) \right| \ge n^{-\frac{\Delta^{'}}{4}} \right\},\\
				\mathcal{A}_{3}^{c}(k, \Delta)& \coloneqq \cup_{\boldsymbol{t} \in G_{n, k}(\Delta) } \Big\{\#\{ i: \boldsymbol{x}_{i} \in \boldsymbol{t} \}  < n^{\frac{1}{2}} \Big\}.
			\end{split}
		\end{equation*}
		Then, it holds that  for all large $n$ and $0\le k\le c\log{(n)} + 1$,
		\begin{equation}\label{3}
			\begin{split}
				\mathop{{}\mathbb{P}}\left( \mathcal{A}_{1}^{c}( k, \Delta)\right)& \le  n^{-\kappa}, \\
				\mathop{{}\mathbb{P}}\left(\mathcal{A}_{2}^{c}( k, \Delta)\right) &\le  n^{-\kappa},\\
				\mathop{{}\mathbb{P}}\left(\mathcal{A}_{3}^{c}( k, \Delta)\right) &\le  n^{-\kappa}.
		\end{split}\end{equation}
	\end{lemma}

	\noindent \textit{Proof}. The arguments for the three inequalities in \eqref{3} are similar, and we begin with showing the first one. The main idea of the proof is based on Hoeffding's inequality. Since 
	Hoeffding's inequality is for bounded random variables, we will consider the truncated model errors in order to apply this inequality. For each $n\ge 1$ and $k\ge 0$, we can deduce that 
	\begin{equation}
		\begin{split}\label{EE1}
			&\mathop{{}\mathbb{P}} (\mathcal{A}_{1}^{c}(k, \Delta)) \\
			& = 	\mathop{{}\mathbb{P}}\Big(\mathcal{A}_{1}^{c}(k, \Delta) \cap \Big( \cap_{i= 1}^{n}\{ |\varepsilon_{i} | \le n^{\frac{\Delta^{'}}{4}}   \} \Big)\Big) \\
			&\qquad+ \mathop{{}\mathbb{P}}\Big(\mathcal{A}_{1}^{c}(k,  \Delta) \cap \Big( \cup_{i= 1}^{n}\{ |\varepsilon_{i} | > n^{\frac{\Delta^{'}}{4}}   \} \Big)\Big) \\
			& = \mathbb{P} \left(  \left(  \cup_{\boldsymbol{t} \in G_{n, k}(\Delta)} \boldsymbol{E}(\boldsymbol{t}) \right) \cap \Big( \cap_{i= 1}^{n}\{ |\varepsilon_{i} | \le n^{\frac{\Delta^{'}}{4}}   \} \Big)\right) \\
			&\quad+\mathop{{}\mathbb{P}}\Big(\mathcal{A}_{1}^{c}(k,  \Delta) \cap \Big( \cup_{i= 1}^{n}\{ |\varepsilon_{i} | > n^{\frac{\Delta^{'}}{4}}   \} \Big)\Big) \\
			& \le \sum_{\boldsymbol{t} \in G_{n, k}(\Delta) } \mathbb{P}(\boldsymbol{E}(\boldsymbol{t})) + \sum_{i = 1}^{n}\mathop{{}\mathbb{P}}\Big( |\varepsilon_{i} | > n^{\frac{\Delta^{'}}{4}}   \Big)\\
			& = \sum_{\boldsymbol{t} \in G_{n, k}(\Delta) } \sum_{\mathcal{B}}\mathop{{}\mathbb{P}}\Big(\{  i : \boldsymbol{x}_{i} \in \boldsymbol{t}\} = \mathcal{B}\Big) \mathop{{}\mathbb{P}}\left(  \boldsymbol{E}(\boldsymbol{t}) \ \Big\vert \ \{  i : \boldsymbol{x}_{i} \in \boldsymbol{t}\} = \mathcal{B} \right) \\
			&\quad+ \sum_{i = 1}^{n}\mathop{{}\mathbb{P}}\Big( |\varepsilon_{i} | > n^{\frac{\Delta^{'}}{4}}   \Big),
		\end{split}
	\end{equation}
	where $\sum_{\mathcal{B}}$ represents the summation over all possible subsets of $\{1, \dots, n\}$ and
	\begin{equation*}\begin{split}
			\boldsymbol{E}(\boldsymbol{t})\coloneqq \left\{\left|  \frac{\sum_{\boldsymbol{x}_{i} \in \boldsymbol{t}} \Big(  m(\boldsymbol{x}_{i}) + \varepsilon_{i}\boldsymbol{1}_{|\varepsilon_{i}| \le n^{\frac{\Delta^{'} } {4}}}   \Big)  }{ \#\{ i: \boldsymbol{x}_{i} \in \boldsymbol{t} \}} - \mathop{{}\mathbb{E}}(m(\boldsymbol{X}) \ \vert \ \boldsymbol{X} \in \boldsymbol{t}) \right|  \ge n^{-\frac{\Delta^{'}}{4}} \right\}.
	\end{split}\end{equation*}
	
	Note that the summation for the first term on the RHS of (\ref{EE1}) can be further decomposed into two terms as 
	\begin{equation}\begin{split}
			\label{EE2}
			& \textnormal{RHS of (\ref{EE1})} \\
			&= \sum_{\boldsymbol{t} \in G_{n, k}(\Delta) } \sum_{\# \mathcal{B} \ge \frac{n^{\Delta} }{ 2 }} \mathbb{P}\Big(\{  i : \boldsymbol{x}_{i} \in \boldsymbol{t}\} = \mathcal{B}\Big) \mathop{{}\mathbb{P}}\left(  \boldsymbol{E}(\boldsymbol{t}) \ \Big\vert \ \{  i : \boldsymbol{x}_{i} \in \boldsymbol{t}\} = \mathcal{B} \right) \\
			& \qquad+ \sum_{\boldsymbol{t} \in G_{n, k}(\Delta) } \sum_{\# \mathcal{B} < \frac{n^{\Delta} }{ 2 } }\mathbb{P}\Big(\{  i : \boldsymbol{x}_{i} \in \boldsymbol{t}\} = \mathcal{B}\Big) \mathop{{}\mathbb{P}}\left(  \boldsymbol{E}(\boldsymbol{t}) \ \Big\vert \ \{  i : \boldsymbol{x}_{i} \in \boldsymbol{t}\} = \mathcal{B} \right) \\
			&\qquad+ \sum_{i = 1}^{n}\mathop{{}\mathbb{P}}\Big( |\varepsilon_{i} | > n^{\frac{\Delta^{'}}{4}}   \Big)\\
			& \le \sum_{\boldsymbol{t} \in G_{n, k}(\Delta) } \sum_{\# \mathcal{B} \ge \frac{n^{\Delta} }{ 2 }}  \mathbb{P}\Big(\{  i : \boldsymbol{x}_{i} \in \boldsymbol{t}\} = \mathcal{B}\Big) \mathbb{P}\left(  \boldsymbol{E}(\boldsymbol{t}) \ \Big\vert \ \{  i : \boldsymbol{x}_{i} \in \boldsymbol{t}\} = \mathcal{B} \right) \\
			& \qquad+ \sum_{\boldsymbol{t} \in G_{n, k}(\Delta) } \mathbb{P}\Big( \sum_{i=1}^{n} \boldsymbol{1}_{\boldsymbol{x}_{i} \in \boldsymbol{t}} <\frac{n^{\Delta}}{2}\Big) + \sum_{i = 1}^{n}\mathop{{}\mathbb{P}}\Big( |\varepsilon_{i} | > n^{\frac{\Delta^{'}}{4}}   \Big).
		\end{split}
	\end{equation}
	The last inequality above is due to
	\begin{equation*}\begin{split}
			\sum_{\# \mathcal{B} < \frac{n^{\Delta} }{ 2 }}  & \mathbb{P}\Big(\{  i : \boldsymbol{x}_{i} \in \boldsymbol{t}\} = \mathcal{B}\Big)\mathbb{P}\left(  \boldsymbol{E}(\boldsymbol{t}) \ \Big\vert \ \{  i : \boldsymbol{x}_{i} \in \boldsymbol{t}\} = \mathcal{B} \right) 
			\\
			&= \sum_{\# \mathcal{B} < \frac{n^{\Delta} }{ 2 }}  \mathbb{P}\left(  \boldsymbol{E}(\boldsymbol{t})\cap \{\{  i : \boldsymbol{x}_{i} \in \boldsymbol{t}\} = \mathcal{B}\} \right) \\
			& \le \sum_{\# \mathcal{B} < \frac{n^{\Delta} }{ 2 }}  \mathbb{P}(  \{ i : \boldsymbol{x}_{i} \in \boldsymbol{t}\} = \mathcal{B}) \\
			& = \mathbb{P}\Big( \sum_{i=1}^{n} \boldsymbol{1}_{\boldsymbol{x}_{i} \in \boldsymbol{t}} <\frac{n^{\Delta}}{2}\Big).
	\end{split}\end{equation*}
	
	Then by the definition of $G_{n, k}(\Delta)$, $\Delta > 1/2$, and Conditions~\ref{ME}--\ref{BO}, an application of Lemma~\ref{CI4} in Section \ref{SecE.2} shows that for all large $n$ and each $k \ge 0$,
	\begin{equation}
		\begin{split}	
			\label{EE3}
			\textnormal{RHS of (\ref{EE2})} &\le \sum_{\boldsymbol{t} \in G_{n, k}(\Delta) } \left( 2\exp{\left(\frac{-n^{\Delta - \Delta^{'}}}{8}\right)}  + 2\exp{\left( \frac{-(\log{n})^{2+ \Delta} }{2}\right)} \right) \\
			& \qquad + \sum_{i = 1}^{n}\mathop{{}\mathbb{P}}\Big( |\varepsilon_{i} | > n^{\frac{\Delta^{'}}{4}}   \Big).
		\end{split}
	\end{equation}
	Thus, it follows from $\Delta > \Delta^{'}$, $p = O( n^{K_{0}})$  in Condition~\ref{ME} with $q > \frac{4 + 4\kappa}{\Delta^{'}}$, and \eqref{numberOf} that for all large $n$ and each $0 \le k \le c\log{n} +1$,
	\[\textnormal{RHS of (\ref{EE3})} \le n^{-\kappa},\]
	which establishes the first inequality in \eqref{3}. The other two inequalities in \eqref{3} can be shown in a similar fashion, which completes the proof of Lemma \ref{CI1}.

	\subsection{Lemma \ref{CI4} and its proof} \label{SecE.2}
	
	\begin{lemma} \label{CI4}
		Assume that $\boldsymbol{x}_{i}$'s are independent copies of $\boldsymbol{X}$. Then for each $n \ge 1$, $\Delta>0$, and $\boldsymbol{t}$ such that $\mathop{{}\mathbb{P}}(\boldsymbol{X}\in\boldsymbol{t}) \ge n^{\Delta - 1}$, we have 
		\begin{equation}\label{CI2}
			\mathop{{}\mathbb{P}} \left(\sum_{i=1}^{n} \boldsymbol{1}_{\boldsymbol{x}_{i} \in \boldsymbol{t}} \le n^{\Delta }  - \sqrt{n}(\log{n})^{1 + \frac{\Delta}{2} } \right) \le 2\exp{\left( \frac{-(\log{n})^{2+ \Delta} }{2}\right)}.
		\end{equation}
		Assume further that $\sup_{\boldsymbol{c} \in [0, 1]^{p}}|m(\boldsymbol{c})| < \infty$, $\boldsymbol{x}_{i}$ and  $\varepsilon_{i}$'s are independent, $\varepsilon_{i}$'s are identically distributed, and $\varepsilon_{1}$ has a symmetric distribution around zero. Then for each $\mathcal{B} \subset\{1,\dots, n\}$, $\boldsymbol{t}$ and $\boldsymbol{t}^{'}$ with $\boldsymbol{t}^{'}\subset\boldsymbol{t}$, $\Delta^{''} > 0$, and $t>0$, it holds for all large $n$ that 
		\begin{equation}\begin{split}\label{CI3}
				\mathop{{}\mathbb{P}} & \left( \left| \frac{\sum_{i\in\mathcal{B}}\left( \mathop{{}\mathbb{E}}(m(\boldsymbol{X}) \ \vert \ \boldsymbol{X} \in \boldsymbol{t})  - \Big(m(\boldsymbol{x}_{i}) + \varepsilon_{i}\boldsymbol{1}_{|\varepsilon_{i}| \le n^{\Delta^{''}}}\Big) \right)} {\#\mathcal{B}}  \right| \ge t \ \bigg\vert\  \{i: \boldsymbol{x}_{i} \in \boldsymbol{t}\} = \mathcal{B}\right) \\
				&\quad \le  2\exp{\left( \frac{-t^{2} \#\mathcal{B}} {4n^{2\Delta^{''}}}\right)},\\
				\mathop{{}\mathbb{P}} & \left( \left| \frac{\sum_{i\in\mathcal{B}}\left( \mathop{{}\mathbb{P}}( \boldsymbol{X} \in \boldsymbol{t}^{'}\ \vert \ \boldsymbol{X} \in \boldsymbol{t})  -  \boldsymbol{1}_{\boldsymbol{x}_{i}\in\boldsymbol{t}^{'}}   \right)} {\#\mathcal{B}}  \right| \ge t\ \bigg\vert\  \{i: \boldsymbol{x}_{i} \in \boldsymbol{t}\} = \mathcal{B}  \right) \\
				& \quad \le 2\exp{\left( \frac{-t^{2} \#\mathcal{B}} {2}\right)}.
		\end{split}\end{equation}
	\end{lemma}
	
	\noindent \textit{Proof}. Observe that by Hoeffding's inequality, we have that for each $\Delta>0$ and $n\ge 1$, 
	\[\mathop{{}\mathbb{P}} \left( \left|n^{-1}\sum_{i=1}^{n} \Big(\boldsymbol{1}_{\boldsymbol{x}_{i} \in \boldsymbol{t}}  - \mathop{{}\mathbb{P}}(\boldsymbol{X} \in\boldsymbol{t})  \Big) \right|\ge  \frac{ (\log{n})^{1 + \frac{\Delta}{2}  }}{\sqrt{n}}  \right)\le 2\exp{\left( \frac{-(\log{n})^{2+ \Delta} }{2}\right)}. \]
	Then by some algebraic calculations, we can establish the desired probability bound in (\ref{CI2}).

	On the other hand, it follows from the assumptions that for each $i\in \mathcal{B}$, $\Delta^{''} > 0$, and $\boldsymbol{t}\subset \boldsymbol{t}_{0}$, we have 
	\begin{align*}
		& \mathop{{}\mathbb{E}}  \Big(\mathop{{}\mathbb{E}}(m(\boldsymbol{X}) \ \vert \ \boldsymbol{X} \in \boldsymbol{t})  - \Big(m(\boldsymbol{x}_{i}) + \varepsilon_{i}\boldsymbol{1}_{|\varepsilon_{i}| \le n^{\Delta^{''}}}\Big) \ \Big\vert \  \{s: \boldsymbol{x}_{s} \in \boldsymbol{t} \} = \mathcal{B}\Big)  \\
		& =  \mathop{{}\mathbb{E}}  \Big( \mathop{{}\mathbb{E}}(m(\boldsymbol{X}) \ \vert \ \boldsymbol{X} \in \boldsymbol{t})  - \Big(m(\boldsymbol{x}_{i}) + \varepsilon_{i}\boldsymbol{1}_{|\varepsilon_{i}| \le n^{\Delta^{''}}}\Big) \ \Big\vert \  \boldsymbol{x}_{i}\in\boldsymbol{t}\Big)= 0, 
	\end{align*}
	which along with conditional Hoeffding's inequality leads to the first probability bound in (\ref{CI3}).  The second probability bound in (\ref{CI3}) can also be shown using similar arguments, which concludes the proof of Lemma \ref{CI4}.

	\subsection{Verifying Condition \ref{P} for Example~\ref{new.example1}} \label{SecE.3}

	For a cell $\boldsymbol{t} = \times_{j=1}^{p}t_{j}$, if $b\not \in t_{1}$, then $\textnormal{Var}( m(\boldsymbol{X})  \vert \boldsymbol{X} \in \boldsymbol{t}) = 0$. In this case, the desired result clearly holds. As for the other case with $b\in t_{1}$, a split at $b$ on the first coordinate leads to the desired result.

	\subsection{Verifying Condition \ref{P} for Example~\ref{new.example2}} \label{SecE.4}

	In this proof, we show that for each $\boldsymbol{t} = t_{1}\times \dots \times t_{p}$, $\sup_{j\in \{1, \dots, p\}, c\in t_{j}}(II)_{\boldsymbol{t}, \boldsymbol{t}(j,c)}$ is proportional to $\textnormal{Var}( m(\boldsymbol{X})  \vert \boldsymbol{X} \in \boldsymbol{t})$. A simple calculation shows that the conditional bias decrease $(II)_{\boldsymbol{t}, \boldsymbol{t}^{'}}$ given a split $(j, x)$ with $x\in t_{j}$ is bounded from below such that 
	\begin{equation}\label{new.example.4.1}
		(II)_{\boldsymbol{t}, \boldsymbol{t}^{'}} \ge \mathbb{P} (\boldsymbol{X}\in\boldsymbol{t}^{'}| \boldsymbol{X} \in \boldsymbol{t})\mathbb{P} (\boldsymbol{X}\in\boldsymbol{t}^{''}| \boldsymbol{X} \in \boldsymbol{t}) \big(H_{j}(x)\big)^2,
	\end{equation}
	where
	$$H_{j}(x) \coloneqq \mathbb{E} (m(\boldsymbol{X}) | \boldsymbol{X} \in \boldsymbol{t}^{''}) - \mathbb{E} (m(\boldsymbol{X}) | \boldsymbol{X} \in \boldsymbol{t}^{'}),$$
	and that  $\boldsymbol{t}^{'} = t_{1}\times \dots \times t_{j}^{'}\times \dots \times t_{p}$ and $\boldsymbol{t}^{''} = t_{1}\times \dots \times t_{j}^{''}\times \dots \times t_{p}$ with $t_{j}^{'} = [\inf t_{j}, x)$ and $t_{j}^{''} = [x, \sup t_{j} ]\cap t_{j}$. For the reader's convenience, recall that
	\begin{equation*}\begin{split}
			& (II)_{\boldsymbol{t}, \boldsymbol{t}^{'}}= \mathbb{P} (\boldsymbol{X}\in\boldsymbol{t}^{'}   \vert  \boldsymbol{X} \in \boldsymbol{t} )
			\Big(\mathbb{E} (m(\boldsymbol{X})  \vert  \boldsymbol{X} \in \boldsymbol{t}^{'}) - \mathbb{E} (m(\boldsymbol{X})  \vert  \boldsymbol{X} \in \boldsymbol{t} ) \Big)^{2} \\
			& +  \mathbb{P} (\boldsymbol{X}\in\boldsymbol{t}^{''}   \vert  \boldsymbol{X} \in \boldsymbol{t} )
			\Big(\mathbb{E} (m(\boldsymbol{X})  \vert  \boldsymbol{X} \in \boldsymbol{t}^{''}) - \mathbb{E} (m(\boldsymbol{X})  \vert  \boldsymbol{X} \in \boldsymbol{t} ) \Big)^{2}.
	\end{split}\end{equation*}

	In the following, we show that the RHS of \eqref{new.example.4.1} given the optimal split among  
	$$\underbrace{(j, \frac{1}{2}\inf t_{j} + \frac{1}{2}\sup t_{j})}_{\textnormal{Mid point on the jth coordinate}} \ \ \ \textnormal{ and } \ \ \ \underbrace{(j, \frac{1}{4}\inf t_{j} + \frac{3}{4}\sup t_{j})}_{\textnormal{Third quarter of the jth coordinate}},\ \ j = 1, \dots, s^*$$
	is at least a proportion of $\textnormal{Var}(m(\boldsymbol{X}) | \boldsymbol{X}\in \boldsymbol{t})$. Notice that for the splits $(j, \frac{1}{2}\inf t_{j} + \frac{1}{2}\sup t_{j})$ and $(j, \frac{1}{4}\inf t_{j} + \frac{3}{4}\sup t_{j})$, it holds that $\frac{\mathbb{P} (\boldsymbol{X}\in\boldsymbol{t}^{'})}{\mathbb{P}(   \boldsymbol{X} \in \boldsymbol{t} )}\frac{\mathbb{P} (\boldsymbol{X}\in\boldsymbol{t}^{''})}{\mathbb{P}(   \boldsymbol{X} \in \boldsymbol{t} )} = \frac{1}{4}$ and 
	$\frac{\mathbb{P} (\boldsymbol{X}\in\boldsymbol{t}^{'})}{\mathbb{P}(   \boldsymbol{X} \in \boldsymbol{t} )}\frac{\mathbb{P} (\boldsymbol{X}\in\boldsymbol{t}^{''})}{\mathbb{P}(   \boldsymbol{X} \in \boldsymbol{t} )} = \frac{3}{16}$ respectively for each $j$ due to the uniform distribution assumption on $\boldsymbol{X}$. We start with the following results \eqref{new.example.4.2}--\eqref{new.example.4.4}. Recall that $\sum_{l>s^*}^{s^*} (\cdots) = 0$ since it is a summation over an empty set.
	
	It holds that
	\begin{equation}\begin{split}\label{new.example.4.2}	
			&\textnormal{Var}(m(\boldsymbol{X}) | \boldsymbol{X}\in \boldsymbol{t}) \\
			&= \sum_{j=1}^{s^*}\Big(\frac{1}{3} \Big(H_{j} \Big(\frac{\inf t_{j} + \sup t_{j}}{2}\Big)\Big)^2 + \frac{\beta_{jj}^2}{180} (\sup t_{j} - \inf t_{j})^4 \\
			&\qquad\qquad + \sum_{l>j}^{s^*}\frac{\beta_{lj}^2}{144}(\sup t_{j} - \inf t_{j})^2(\sup t_{l}- \inf t_{l})^2\Big).
		\end{split}
	\end{equation}
	If $\big(H_{j}(\frac{\inf t_{j} + \sup t_{j}}{2})\big)^2 \le \frac{1}{1296}\beta_{jj}^2(\sup t_{j} - \inf t_{j})^4$, then
	\begin{equation}\label{new.example.4.3}
		\Big(H_{j}\Big(\frac{1}{4}\inf t_{j} + \frac{3}{4}\sup t_{j}\Big)\Big)^2 \ge \frac{\beta_{jj}^2(\sup t_{j} - \inf t_{j})^4}{648}.
	\end{equation}
	For each $j$,
	\begin{equation}\label{new.example.4.4}
		\sum_{l>j}^{s^*}\frac{\beta_{lj}^2}{144}(\sup t_{j} - \inf t_{j})^2(\sup t_{l}- \inf t_{l})^2 \le \frac{1}{3} \Big(H_{j} \Big(\frac{\inf t_{j} + \sup t_{j}}{2}\Big)\Big)^2.
	\end{equation}
	
	The results of \eqref{new.example.4.2}--\eqref{new.example.4.4} are proven after the proof for Example~\ref{new.example2}; the coefficient assumption of Example~\ref{new.example2} is used for deriving  \eqref{new.example.4.4}. Define 
	\begin{equation*}
		\begin{split}
			&T_{j}\coloneqq \underbrace{\frac{1}{3} \Big(H_{j} \Big(\frac{\inf t_{j} + \sup t_{j}}{2}\Big)\Big)^2}_{\textnormal{the first term}} + \underbrace{\frac{\beta_{jj}^2}{180} (\sup t_{j} - \inf t_{j})^4}_{\textnormal{the second term}} \\
			&\qquad\qquad + \underbrace{\sum_{l>j}^{s^*}\frac{\beta_{lj}^2}{144}(\sup t_{j} - \inf t_{j})^2(\sup t_{l}- \inf t_{l})^2}_{\textnormal{the third term}}.
		\end{split}
	\end{equation*}
	
	By \eqref{new.example.4.2}, we suppose without loss of generality that for some $j\le s^*$,
	\begin{equation}\label{new.example.4.8}
		T_{j} \ge \frac{\textnormal{Var} (m(\boldsymbol{X})|\boldsymbol{X} \in \boldsymbol{t})}{s^*}.
	\end{equation}
	By \eqref{new.example.4.4}, one of the first two terms of $T_{j}$ is the largest term in $T_{j}$ (ties are possible). If the largest term of $T_{j}$ is its first term, then by \eqref{new.example.4.8},
	$$\frac{1}{3}\Big(H_{j} \Big(\frac{\inf t_{j} + \sup t_{j}}{2}\Big)\Big)^2 \ge \frac{\textnormal{Var} (m(\boldsymbol{X})|\boldsymbol{X} \in \boldsymbol{t})}{3s^*}.$$
	Therefore, by \eqref{new.example.4.1}, the split $(j, \frac{\inf t_{j} + \sup t_{j}}{2})$ is such that 
	\begin{equation}\label{new.example.4.5}
		(II)_{\boldsymbol{t}, \boldsymbol{t}^{'}} \ge \frac{1}{4} \frac{\textnormal{Var} (m(\boldsymbol{X})|\boldsymbol{X} \in \boldsymbol{t})}{s^*}.
	\end{equation}
	If the second term of $T_{j}$ is the largest term of $T_{j}$ and that $\big(H_{j}(\frac{\inf t_{j} + \sup t_{j}}{2})\big)^2 > \frac{1}{1296}\beta_{jj}^2(\sup t_{j} - \inf t_{j})^4$, then by \eqref{new.example.4.8},
	$$\Big(H_{j} \Big(\frac{\inf t_{j} + \sup t_{j}}{2}\Big)\Big)^2 \ge \frac{180}{1296}\times \frac{\textnormal{Var} (m(\boldsymbol{X})|\boldsymbol{X} \in \boldsymbol{t})}{3s^*},$$
	which in combination with \eqref{new.example.4.1} concludes that the split $(j, \frac{\inf t_{j} + \sup t_{j}}{2})$ is such that 
	\begin{equation}\label{new.example.4.6}
		(II)_{\boldsymbol{t}, \boldsymbol{t}^{'}} \ge \frac{5}{432} \frac{\textnormal{Var} (m(\boldsymbol{X})|\boldsymbol{X} \in \boldsymbol{t})}{s^*}.
	\end{equation}
	Otherwise, if the second term of $T_{j}$ is the largest term of $T_{j}$ and that $\big(H_{j}(\frac{\inf t_{j} + \sup t_{j}}{2})\big)^2 \le \frac{1}{1296}\beta_{jj}^2(\sup t_{j} - \inf t_{j})^4$, then by \eqref{new.example.4.3} and \eqref{new.example.4.8},
	$$\Big(H_{j}\Big(\frac{1}{4}\inf t_{j} + \frac{3}{4}\sup t_{j}\Big)\Big)^2 \ge \frac{180}{648}\times \frac{\textnormal{Var} (m(\boldsymbol{X})|\boldsymbol{X} \in \boldsymbol{t})}{3s^*},$$
	which in combination with \eqref{new.example.4.1} concludes that the split $(j, \frac{1}{4}\inf t_{j} + \frac{3}{4}\sup t_{j})$ is such that 
	\begin{equation}\label{new.example.4.7}
		(II)_{\boldsymbol{t}, \boldsymbol{t}^{'}} \ge \frac{5}{288} \frac{\textnormal{Var} (m(\boldsymbol{X})|\boldsymbol{X} \in \boldsymbol{t})}{s^*}.
	\end{equation}
	
	The results of \eqref{new.example.4.5}--\eqref{new.example.4.7} concludes that $\sup_{j\in \{1, \dots, p\}, c\in t_{j}}(II)_{\boldsymbol{t}, \boldsymbol{t}(j,c)}$ is proportional to $\textnormal{Var}( m(\boldsymbol{X})  \vert \boldsymbol{X} \in \boldsymbol{t})$ for each $\boldsymbol{t}$, and that
	$$m(\boldsymbol{X}) \in \textnormal{SID}(86.4\times s^*),$$
	which is the desired result.
	
	\begin{proof}[Proofs of \eqref{new.example.4.2}--\eqref{new.example.4.4}]
		We start with proving \eqref{new.example.4.4}. Define 
		$$g_{j}(\boldsymbol{X}) = \beta_{jj}X_{j}^2 + \beta_{j}X_{j} + \sum_{\substack{l=1\\ l\not = j}}^{s^*}\beta_{lj}X_{l}X_{j}$$
		with $\beta_{lj}=\beta_{jl}$. Consider a cell $\boldsymbol{t}$, a split $(j, x)$, and two daughter cells $\boldsymbol{t}^{'}$ and $\boldsymbol{t}^{''}$ as in \eqref{new.example.4.1}. By the uniform distribution assumption on $\boldsymbol{X}$, 
		$$H_{j}(x) = \mathbb{E}(g_{j}(\boldsymbol{X})|\boldsymbol{X} \in \boldsymbol{t}^{''}) - \mathbb{E}(g_{j}(\boldsymbol{X})|\boldsymbol{X} \in \boldsymbol{t}^{'}),$$
		whose detailed derivation is omitted for simplicity. Recall that $\boldsymbol{t}^{'} = t_{1}\times \dots \times t_{j}^{'}\times \dots \times t_{p}$ and $\boldsymbol{t}^{''} = t_{1}\times \dots \times t_{j}^{''}\times \dots \times t_{p}$ with $t_{j}^{'} = [\inf t_{j}, x)$ and $t_{j}^{''} = [x, \sup t_{j} ]\cap t_{j}$.
		
		In what follows, we derive a closed-form expression for $H_{j}(x)$ in \eqref{new.example.4.13} below. By the uniform distribution assumption on the feature vector and the change of variables formula,
		\begin{equation}
			\begin{split}\label{new.example.4.15}
				&\mathbb{E}(g_{j}(\boldsymbol{X})|\boldsymbol{X} \in \boldsymbol{t}^{''})\\
				& = \frac{1}{|t_{j}^{''}|} \int_{z\in t_{j}^{''}} \beta_{jj}z^2 + \beta_{j}z + \sum_{l=1; l\not = j}^{s^*} \beta_{lj} \frac{R_{l}}{2} z dz\\
				&= \frac{1}{|t_{j}^{'}|} \int_{z\in t_{j}^{'}} \beta_{jj}(f(z))^2 + \left(\beta_{j} + \sum_{l=1; l\not = j}^{s^*} \beta_{lj} \frac{R_{l}}{2}\right) f(z) dz,
			\end{split}
		\end{equation}
		where $f(z) = \frac{(\sup t_{j} - x)(z - \inf t_{j})}{x-\inf t_{j}} + x$, $R_{j} \coloneqq \sup t_{j} + \inf t_{j}$, and $r_{j} \coloneqq \sup t_{j} - \inf t_{j}$. Since
		$$\frac{1}{|t_{j}^{'}|} \int_{z\in t_{j}^{'}} (f(z))^k dz = \frac{(f(x))^{k+1}}{(k+1) (\sup t_{j} - x)} -\frac{(f(\inf t_{j}))^{k+1}}{(k+1) (\sup t_{j} - x)} = \frac{ (\sup t_{j})^{k+1}- x^{k+1}}{(k+1) (\sup t_{j} - x)}, $$
		we have
		\begin{equation*}
			\begin{split}
				\textnormal{RHS of \eqref{new.example.4.15}}& = \frac{\beta_{jj}}{3} \big((\sup t_{j})^2 + (\sup t_{j}) x + x^2 \big) + \left(\beta_{j} + \sum_{l=1; l\not = j}^{s^*} \beta_{lj} \frac{R_{l}}{2}\right) \left(\frac{\sup t_{j} + x}{2}\right).
			\end{split}
		\end{equation*}

		 By this and similar calculations, it holds that 
		\begin{equation}\label{new.example.4.13}
			H_{j}(x) = \frac{\beta_{jj}r_{j}}{3} \big(R_{j} + x\big) + \left(\beta_{j} + \sum_{l=1; l\not = j}^{s^*} \frac{\beta_{lj} R_{l}}{2}\right) \frac{r_{j}}{2}.
		\end{equation}

		Hence, 
		\begin{equation}
			\begin{split}\label{new.example.4.11}
				H_{j}(\frac{1}{2}\inf t_{j} + \frac{1}{2}\sup t_{j}) & = \frac{r_{j}}{2}\left(\beta_{jj} R_{j} + \beta_{j} + \sum_{l=1; l\not = j}^{s^*} \frac{\beta_{lj} R_{l}}{2}\right), \\
				H_{j}(\frac{1}{4}\inf t_{j} + \frac{3}{4}\sup t_{j}) & = \frac{r_{j}}{2}\left(\beta_{jj} \left( \frac{5}{6}\inf t_{j} + \frac{7}{6}\sup t_{j}\right) + \beta_{j} + \sum_{l=1; l\not = j}^{s^*} \frac{\beta_{lj} R_{l}}{2}\right),
			\end{split}
		\end{equation}
		where the first equality concludes \eqref{new.example.4.4} given the coefficient assumption.

		Next, we proceed to show \eqref{new.example.4.2}. We have
		\begin{equation}\label{new.example.4.9}
			\textnormal{Var}(m(\boldsymbol{X})|\boldsymbol{X}\in\boldsymbol{t}) = \sum_{j=1}^{s^*} \textnormal{Var}(g_{j}(\boldsymbol{X})|\boldsymbol{X}\in\boldsymbol{t}) - \sum_{j=1}^{s^*-1}\sum_{l>j}^{s^*} \beta_{lj}^2\textnormal{Var}(X_{l}X_{j}|\boldsymbol{X}\in\boldsymbol{t}),
		\end{equation}
		where to justify the equality, we notice that when expanding $\textnormal{Var}(m(\boldsymbol{X})|\boldsymbol{X}\in\boldsymbol{t})$, all terms appearing are i) $2\mathbb{E} \big( (\beta_{i}X_{i} - \mathbb{E} (\beta_{i}X_{i} | \boldsymbol{X} \in \boldsymbol{t}) ) (\beta_{j}X_{j} - \mathbb{E} (\beta_{j}X_{j} | \boldsymbol{X} \in \boldsymbol{t}) ) |\boldsymbol{X} \in\boldsymbol{t}\big)$,  
		
		ii) $2\mathbb{E} \big( (\beta_{ij}X_{i}X_{j} - \mathbb{E} (\beta_{ij}X_{i}X_{j} | \boldsymbol{X} \in \boldsymbol{t}) ) (\beta_{kl}X_{k}X_{l} - \mathbb{E} (\beta_{kl}X_{k}X_{l} | \boldsymbol{X} \in \boldsymbol{t}) ) |\boldsymbol{X} \in\boldsymbol{t}\big)$, 
		
		iii) $2\mathbb{E} \big( (\beta_{j}X_{j} - \mathbb{E} (\beta_{j}X_{j} | \boldsymbol{X} \in \boldsymbol{t}) ) (\beta_{kl}X_{k}X_{l} - \mathbb{E} (\beta_{kl}X_{k}X_{l} | \boldsymbol{X} \in \boldsymbol{t}) ) |\boldsymbol{X} \in\boldsymbol{t}\big)$,

		iv) $2\mathbb{E} \big( (\beta_{jj}X_{j}^2 - \mathbb{E} (\beta_{jj}X_{j}^2 | \boldsymbol{X} \in \boldsymbol{t}) ) (\beta_{k}X_{k} - \mathbb{E} (\beta_{k}X_{k} | \boldsymbol{X} \in \boldsymbol{t}) ) |\boldsymbol{X} \in\boldsymbol{t}\big)$,

		v) $2\mathbb{E} \big( (\beta_{jj}X_{j}^2 - \mathbb{E} (\beta_{jj}X_{j}^2 | \boldsymbol{X} \in \boldsymbol{t}) ) (\beta_{kl}X_{k}X_{l} - \mathbb{E} (\beta_{kl}X_{k}X_{l} | \boldsymbol{X} \in \boldsymbol{t}) ) |\boldsymbol{X} \in\boldsymbol{t}\big)$,
		
		vi) $2\mathbb{E} \big( (\beta_{jj}X_{j}^2 - \mathbb{E} (\beta_{jj}X_{j}^2 | \boldsymbol{X} \in \boldsymbol{t}) ) (\beta_{kk}X_{k}^2 - \mathbb{E} (\beta_{kk}X_{k}^2 | \boldsymbol{X} \in \boldsymbol{t}) ) |\boldsymbol{X} \in\boldsymbol{t}\big)$,
		
		vii) $\textnormal{Var}(\beta_{lj}X_{l}X_{j}|\boldsymbol{X}\in\boldsymbol{t})$, 
		
		viii) $\textnormal{Var}(\beta_{jj}X_{j}^2| \boldsymbol{X} \in \boldsymbol{t})$, 
		
		ix)  $\textnormal{Var}(\beta_{j}X_{j}|\boldsymbol{X}\in\boldsymbol{t})$, 
		
		x) $2\mathbb{E} \big( (\beta_{j}X_{j} - \mathbb{E} (\beta_{j}X_{j} | \boldsymbol{X} \in \boldsymbol{t}) ) (\beta_{lj}X_{l}X_{j} - \mathbb{E} (\beta_{lj}X_{l}X_{j} | \boldsymbol{X} \in \boldsymbol{t}) ) |\boldsymbol{X} \in\boldsymbol{t}\big)$,
		
		xi) $2\mathbb{E} \big( (\beta_{ij}X_{i}X_{j} - \mathbb{E} (\beta_{ij}X_{i}X_{j} | \boldsymbol{X} \in \boldsymbol{t}) ) (\beta_{lj}X_{l}X_{j} - \mathbb{E} (\beta_{lj}X_{l}X_{j} | \boldsymbol{X} \in \boldsymbol{t}) ) |\boldsymbol{X} \in\boldsymbol{t}\big)$, 
		
		xii) $2\mathbb{E}\big( (\beta_{jj} X_{j}^2 - \mathbb{E}(\beta_{jj}X_{j}^2|\boldsymbol{X} \in \boldsymbol{t}) ) (\beta_{j} X_{j} - \mathbb{E}(\beta_{j}X_{j}|\boldsymbol{X} \in \boldsymbol{t}) )| \boldsymbol{X} \in \boldsymbol{t}\big)$, and 
		
		xiii) $2\mathbb{E}\big( (\beta_{jj} X_{j}^2 - \mathbb{E}(\beta_{jj}X_{j}^2|\boldsymbol{X} \in \boldsymbol{t}) ) (\beta_{lj} X_{l}X_{j} - \mathbb{E}(\beta_{lj}X_{l}X_{j}|\boldsymbol{X} \in \boldsymbol{t}) )| \boldsymbol{X} \in \boldsymbol{t}\big)$, where $i, j, k, l$ are distinct indices. Among them, i)--vi) are zeros since features are independent, viii)--xiii) are considered in $\sum_{j=1}^{s^*} \textnormal{Var}(g_{j}(\boldsymbol{X})|\boldsymbol{X}\in\boldsymbol{t})$, and vii) is considered twice (for each $l,j$) in  $\sum_{j=1}^{s^*} \textnormal{Var}(g_{j}(\boldsymbol{X})|\boldsymbol{X}\in\boldsymbol{t})$, which concludes  \eqref{new.example.4.9}. In addition, by the uniform distribution assumption on $\boldsymbol{X}$ and a direct calculation, 
		\begin{equation*}\begin{split}			
				\label{new.example.4.10}
				& \textnormal{Var}(g_{j}(\boldsymbol{X})|\boldsymbol{X}\in\boldsymbol{t}) = \beta_{jj}^2 \textnormal{Var}(X_{j}^2 | X_{j}\in t_{j}) + \beta_{j}^2 \textnormal{Var}(X_{j} | X_{j}\in t_{j}) \\
				&+ \sum_{l=1; l\not = j}^{s^*} \beta_{lj}^2 \textnormal{Var}(X_{l}X_{j} | \boldsymbol{X} \in \boldsymbol{t}) \\
				& + \sum_{\substack{k=1 \\ k\not= l\\k\not= j}}^{s^*}\sum_{\substack{l=1\\ l\not= j}}^{ s^*}\beta_{kj}\beta_{lj}\big(\mathbb{E}( X_{k}X_{l}X_{j}^2|\boldsymbol{X} \in \boldsymbol{t}) - \mathbb{E}(X_{k}X_{j}|\boldsymbol{X}\in\boldsymbol{t})\mathbb{E}(X_{l}X_{j}|\boldsymbol{X}\in\boldsymbol{t})\big) \\
				& +  2\beta_{jj}\beta_{j}\big(\mathbb{E}( X_{j}^3| X_{j} \in t_{j}) - \mathbb{E}(X_{j}^2|X_{j}\in t_{j})\mathbb{E}(X_{j}|X_{j}\in t_{j})\big)\\
				& + 2\sum_{\substack{l=1\\ l\not = j}}^{s^{*}}\beta_{jj}\beta_{lj}\big( \mathbb{E}(X_{j}^3X_{l}|\boldsymbol{X} \in \boldsymbol{t}) - \mathbb{E}(X_{j}^2|X_{j}\in t_{j})\mathbb{E}(X_{l}X_{j}| \boldsymbol{X} \in \boldsymbol{t})\big) \\
				& + 2\sum_{\substack{l=1\\ l\not = j}}^{s^{*}}\beta_{j}\beta_{lj}\big( \mathbb{E}(X_{j}^2X_{l}|\boldsymbol{X} \in \boldsymbol{t}) - \mathbb{E}(X_{j}|X_{j}\in t_{j})\mathbb{E}(X_{l}X_{j}| \boldsymbol{X} \in \boldsymbol{t})\big)\\
				& \eqqcolon (A) + (B) + (C) + (D) + (E) + (F) + (G),
			\end{split}
		\end{equation*}
		where  $(A) =  \frac{\beta_{jj}^2}{45} r_{j}^{2} \big(4(\sup t_{j})^2 + 7(\sup t_{j}) (\inf t_{j}) + 4(\inf t_{j})^2\big)$, $(B) = \beta_{j}^2 \frac{r_{j}^2}{12}$, 
		$$(C) = \sum_{\substack{l=1\\ l\not = j}}^{s^*}  \frac{\beta_{lj}^2}{144} \big( 3R_{j}^2 r_{l}^2 + 3R_{l}^2 r_{j}^2 + r_{j}^2 r_{l}^2\big), \qquad(D)= \sum_{\substack{k=1 \\ k\not= l\\k\not= j}}^{s^*}\sum_{\substack{l=1\\ l\not= j}}^{ s^*} \frac{\beta_{kj}\beta_{lj}}{48} R_{l} R_{k} r_{j}^2,\qquad (E) =  \frac{\beta_{jj}\beta_{j}}{6}r_{j}^2 R_{j},$$ 
		$$(F) = \sum_{\substack{l=1\\ l\not = j}}^{s^{*}}\frac{\beta_{jj}\beta_{lj}}{12} R_{l}r_{j}^2R_{j}, \textnormal{ and } (G) = \sum_{\substack{l=1\\ l\not = j}}^{s^{*}}\frac{\beta_{j}\beta_{lj}}{12} R_{l}r_{j}^2.$$

		By this, \eqref{new.example.4.11}, and a direct calculation, 
		\begin{equation}
			\begin{split}\label{new.example.4.12}
				&\textnormal{Var}(g_{j}(\boldsymbol{X})|\boldsymbol{X}\in\boldsymbol{t}) - \frac{1}{3}\left(H_{j}(\frac{1}{2}\inf t_{j} + \frac{1}{2}\sup t_{j})\right)^2 \\
				& = \frac{\beta_{jj}^2}{180} r_{j}^4 +  \sum_{l=1; l\not = j}^{s^*} \frac{\beta_{lj}^2 }{144} (3R_{j}^2r_{l}^2 + r_{j}^2r_{l}^2).
			\end{split}
		\end{equation}
		
		Also, by the uniform distribution assumption on $\boldsymbol{X}$, the above expression for (C), and that 
		$$\sum_{j=1}^{s^*}\sum_{\substack{l=1\\ l\not = j}}^{s^*} 1 = 2 \sum_{j=1}^{s^*-1}\sum_{l>j}^{s^*}1,$$
		it holds that 
		\begin{equation*}
			\sum_{j=1}^{s^*}\sum_{\substack{l=1\\ l\not = j}}^{s^*} \frac{\beta_{lj}^2 }{144} (3R_{j}^2r_{l}^2 + r_{j}^2r_{l}^2) = \sum_{j=1}^{s^*-1}\sum_{l>j}^{s^*} \beta_{lj}^2\textnormal{Var}(X_{l}X_{j}|\boldsymbol{X}\in\boldsymbol{t}) + \sum_{j=1}^{s^*-1}\sum_{l>j}^{s^*} \frac{\beta_{lj}^2}{144}r_{j}^2r_{l}^2,
		\end{equation*}
		which in combination with \eqref{new.example.4.9}--\eqref{new.example.4.12} leads to \eqref{new.example.4.2}.
		
		Lastly, we deal with \eqref{new.example.4.3}. By \eqref{new.example.4.11}, 
		\begin{equation*}
			\begin{split}
				H_{j}(\frac{1}{4}\inf t_{j} + \frac{3}{4}\sup t_{j})  - H_{j}(\frac{1}{2}\inf t_{j} + \frac{1}{2}\sup t_{j}) = \frac{\beta_{jj}r_{j}^2}{12}.
			\end{split}
		\end{equation*}
		Therefore, if $|H_{j}(\frac{1}{2}\inf t_{j} + \frac{1}{2}\sup t_{j})|\le \frac{|\beta_{jj}|r_{j}^2}{36}$, then
		\begin{equation*}
			\begin{split}
				& \left|\left(H_{j}(\frac{1}{4}\inf t_{j} + \frac{3}{4}\sup t_{j}) \right)^2 - \left(H_{j}(\frac{1}{2}\inf t_{j} + \frac{1}{2}\sup t_{j}) \right)^2\right|\\
				&= \frac{|\beta_{jj}| r_{j}^2}{12}\left| \frac{\beta_{jj} r_{j}^2}{12}+ 2H_{j}(\frac{1}{2}\inf t_{j} + \frac{1}{2}\sup t_{j}) \right|\\
				&\ge \frac{|\beta_{jj}| r_{j}^2}{12}\left( \frac{|\beta_{jj}| r_{j}^2}{12}- 2\left|H_{j}(\frac{1}{2}\inf t_{j} + \frac{1}{2}\sup t_{j})\right| \right) \ge \frac{\beta_{jj}^2 r_{j}^4}{432}.
			\end{split}
		\end{equation*}
		This implies that if $(H_{j}(\frac{1}{2}\inf t_{j} + \frac{1}{2}\sup t_{j}))^2\le (\frac{\beta_{jj}r_{j}^2}{36})^2 = \frac{\beta_{jj}^2r_{j}^4}{1296}$, then
		\begin{equation*}
			\Big(H_{j}\Big(\frac{1}{4}\inf t_{j} + \frac{3}{4}\sup t_{j}\Big)\Big)^2 \ge \frac{\beta_{jj}^2r_{j}^4}{648},
		\end{equation*}
		which concludes the desired result of \eqref{new.example.4.3}. We have finished the proofs for \eqref{new.example.4.2}--\eqref{new.example.4.4}.
		
	\end{proof}

	\subsection{Verifying Condition \ref{P} for Example~\ref{new.example3}} \label{SecE.5}

	The proof idea is straightforward: for each cell $\boldsymbol{t} = t_{1}\times \dots t_{p}$, we establish appropriate upper and lower bounds for $\textnormal{Var}( m(\boldsymbol{X})  \vert \boldsymbol{X} \in \boldsymbol{t})$ and $\sup_{j, c}(II)_{\boldsymbol{t}, \boldsymbol{t}(j,c)}$, respectively, to conclude the desired result, where we use $\sup_{j, c}(II)_{\boldsymbol{t}, \boldsymbol{t}(j,c)}$ to denote $\sup_{j\in \{1, \dots, p\}, c\in t_{j}}(II)_{\boldsymbol{t}, \boldsymbol{t}(j,c)}$ for simplicity. We begin with an upper bound of $\textnormal{Var}( m(\boldsymbol{X})  \vert \boldsymbol{X} \in \boldsymbol{t})$ as follows.
	\begin{equation}
		\begin{split}\label{variance.1}
			&\textnormal{Var}( m(\boldsymbol{X})  \vert \boldsymbol{X} \in \boldsymbol{t}) \\
			& = \mathbb{E}\Big(\big(m(\boldsymbol{X}) - \mathbb{E}(m(\boldsymbol{X})|\boldsymbol{X} \in \boldsymbol{t}) \big)^2| \boldsymbol{X} \in \boldsymbol{t}\Big)\\
			& \le \mathbb{E}\left(\left(m(\boldsymbol{X}) - \left( \frac{\sup_{\boldsymbol{z} \in \boldsymbol{t}} m(\boldsymbol{z})  + \inf_{\boldsymbol{z} \in \boldsymbol{t}} m(\boldsymbol{z})}{2}\right) \right)^2| \boldsymbol{X} \in \boldsymbol{t}\right)\\
			& \le \frac{1}{4}\mathbb{E}\left(\left(m(\boldsymbol{X}) - \sup_{\boldsymbol{z} \in \boldsymbol{t}} m(\boldsymbol{z}) - \inf_{\boldsymbol{z} \in \boldsymbol{t}} m(\boldsymbol{z}) + m(\boldsymbol{X}) \right)^2| \boldsymbol{X} \in \boldsymbol{t}\right)\\
			& \le \frac{1}{4}\mathbb{E}\left(\left(\sup_{\boldsymbol{z} \in \boldsymbol{t}} m(\boldsymbol{z}) - m(\boldsymbol{X})  + m(\boldsymbol{X}) - \inf_{\boldsymbol{z} \in \boldsymbol{t}} m(\boldsymbol{z}) \right)^2| \boldsymbol{X} \in \boldsymbol{t}\right)\\
			&= \frac{1}{4}(\sup_{\boldsymbol{z} \in \boldsymbol{t}} m(\boldsymbol{z})  - \inf_{\boldsymbol{z} \in \boldsymbol{t}} m(\boldsymbol{z}))^2.
		\end{split}
	\end{equation}
	
	Meanwhile, by the assumptions on $\frac{\partial m(\boldsymbol{z})}{\partial z_{j}}$'s, we can show that 
	$$|\sup_{\boldsymbol{z} \in \boldsymbol{t}} m(\boldsymbol{z})  - \inf_{\boldsymbol{z} \in \boldsymbol{t}} m(\boldsymbol{z})| \le \sum_{j\in S^*} |t_{j}|M_{2},$$
	and we omit the details for simplicity. By this and \eqref{variance.1}, 
	\begin{equation}
		\label{variance.2}
		\textnormal{Var}( m(\boldsymbol{X})  \vert \boldsymbol{X} \in \boldsymbol{t}) \le \frac{M_{2}^2}{4}(\sum_{j\in S^*} |t_{j}|)^2.
	\end{equation}
	
	Next, we deal with the lower bound of $\sup_{j, c}(II)_{\boldsymbol{t}, \boldsymbol{t}(j,c)}$, which is by definition larger than $(II)_{\boldsymbol{t}, \boldsymbol{t}(j^*, c^*)}$, where  
	$$j^* \coloneqq \arg\max_{j\in S^*} |t_{j}| \qquad\textnormal{ and } \qquad c^* \coloneqq \frac{\sup t_{j^*} + \inf t_{j^*}}{2}.$$ 
	Let $\boldsymbol{t}_{1}^*$ and $\boldsymbol{t}_{2}^*$ be the corresponding daughter cells of $\boldsymbol{t}$ such that the $j^*$th coordinate of $\boldsymbol{t}_{1}^*$ is $\big[\inf t_{j}, \frac{\sup t_{j} + \inf t_{j}}{2}\big)$. Recall that 
	\begin{equation*}\begin{split}
			& (II)_{\boldsymbol{t}, \boldsymbol{t}_{1}^*}= \mathbb{P} (\boldsymbol{X}\in\boldsymbol{t}_{1}^*  \ \vert \ \boldsymbol{X} \in \boldsymbol{t} )
			\Big(\mathbb{E} (m(\boldsymbol{X}) \ \vert \ \boldsymbol{X} \in \boldsymbol{t}_{1}^*)) - \mathbb{E} (m(\boldsymbol{X}) \ \vert \ \boldsymbol{X} \in \boldsymbol{t} )) \Big)^{2} \\
			& +  \mathbb{P} (\boldsymbol{X}\in\boldsymbol{t}_{2}^*  \ \vert \ \boldsymbol{X} \in \boldsymbol{t} )
			\Big(\mathbb{E} (m(\boldsymbol{X}) \ \vert \ \boldsymbol{X} \in \boldsymbol{t}_{2}^*)) - \mathbb{E} (m(\boldsymbol{X}) \ \vert \ \boldsymbol{X} \in \boldsymbol{t} )) \Big)^{2}.
	\end{split}\end{equation*}

	To establish a lower bound of $(II)_{\boldsymbol{t}, \boldsymbol{t}_{1}^*}$, it suffices to establish a lower bound on
	$$
	(\mathbb{E} (m(\boldsymbol{X}) | \boldsymbol{X} \in \boldsymbol{t}_{1}^{*}) - \mathbb{E} (m(\boldsymbol{X}) | \boldsymbol{X} \in \boldsymbol{t}_{2}^{*}) )^2.
	$$ 
	To this end, we first notice that by the fundamental theorem of calculus,
	$$m(\boldsymbol{z} + \frac{1}{2}|t_{j^*}| \boldsymbol{e}_{j^*}) = m(\boldsymbol{z}) + \int_{z_{j^*}}^{z_{j^*} + \frac{1}{2}|t_{j^*}|} \frac{\partial m(\boldsymbol{w})}{\partial w_{j^*}} d w_{j^*},$$
	where $\boldsymbol{e}_{j}$ is a unit vector with its $j$th coordinate being one and $\boldsymbol{z} = (z_{1}, \dots, z_{p})^{\top}$. 
	Now, suppose that the partial derivative along the $j^*$th coordinate is positive in the following arguments for simplicity; that is, $\frac{\partial m(\boldsymbol{w})}{\partial w_{j^*}} \ge M_{1}>0$. By these and the uniform distribution assumption on $\boldsymbol{X}$,
	\begin{equation}
		\begin{split}			
			\label{condi.expec.1}
			\mathbb{E} (m(\boldsymbol{X}) | \boldsymbol{X} \in \boldsymbol{t}_{2}^{*}) & = \frac{1}{|\boldsymbol{t}_{2}^*|} \int_{\boldsymbol{z} \in \boldsymbol{t}_{2}^*} m(\boldsymbol{z}) d\boldsymbol{z} \\
			&= \frac{1}{|\boldsymbol{t}_{1}^*|} \int_{\boldsymbol{z} \in \boldsymbol{t}_{1}^*} \left[ m(\boldsymbol{z}) + \int_{z_{j^*}}^{z_{j^*} + \frac{1}{2}|t_{j^*}|} \frac{\partial m(\boldsymbol{w})}{\partial w_{j^*}} d w_{j^*} \right] d\boldsymbol{z}\\
			& \ge \frac{1}{|\boldsymbol{t}_{1}^*|} \int_{\boldsymbol{x} \in \boldsymbol{t}_{1}^*}  m(\boldsymbol{z}) + \frac{1}{2}|t_{j^{*}}|M_{1} d\boldsymbol{z}.
		\end{split}
	\end{equation}

	With \eqref{condi.expec.1}, 
	$$\mathbb{E} (m(\boldsymbol{X}) | \boldsymbol{X} \in \boldsymbol{t}_{2}^{*}) - \mathbb{E} (m(\boldsymbol{X}) | \boldsymbol{X} \in \boldsymbol{t}_{1}^{*}) \ge \frac{1}{2}|t_{j^{*}}|M_{1},$$
	which along with simple calculations shows 
	{\small
		\begin{equation}\begin{split}\label{condi.expec.2}
				(II)_{\boldsymbol{t}, \boldsymbol{t}_{1}^*}&=\frac{1}{2}\Big(\Big(\mathbb{E} (m(\boldsymbol{X}) | \boldsymbol{X} \in \boldsymbol{t}) - \mathbb{E} (m(\boldsymbol{X}) | \boldsymbol{X} \in \boldsymbol{t}_{1}^{*}) \Big)^2 \\
				&\qquad+ \Big(\mathbb{E} (m(\boldsymbol{X}) | \boldsymbol{X} \in \boldsymbol{t}) - \mathbb{E} (m(\boldsymbol{X}) | \boldsymbol{X} \in \boldsymbol{t}_{2}^{*}) \Big)^2\Big)\\
				& \ge \frac{1}{16}|t_{j^{*}}|^2M_{1}^2.
		\end{split}\end{equation}
	}

	By \eqref{condi.expec.2}, \eqref{variance.2}, and $\frac{(\sum_{j\in S^*} |t_{j}|)^2}{|t_{j^*}|^2} \le (\# S^*)^2$ due to the definition of $j^*$, it holds that
	$$\left(4M_{2}^2M_{1}^{-2} (\# S^*)^2 \right)\sup_{j, c}(II)_{\boldsymbol{t}, \boldsymbol{t}(j,c)}\ge \left(4M_{2}^2M_{1}^{-2} (\# S^*)^2 \right) (II)_{\boldsymbol{t}, \boldsymbol{t}_{1}^*} \ge \textnormal{Var}( m(\boldsymbol{X})  \vert \boldsymbol{X} \in \boldsymbol{t}).$$
	
	This concludes the desired result for the case with a positive partial derivative. The same arguments apply to the other case, and hence we have finished the proof of the first assertion.

	As for the case with the additive model assumption, we have
	\begin{equation}
		\label{condi.expec.3}
		\textnormal{Var}( m(\boldsymbol{X})  \vert \boldsymbol{X} \in \boldsymbol{t})  =\sum_{j=1}^{s^*} \textnormal{Var}( m_{j}(X_{j})  \vert X_{j} \in t_{j}),
	\end{equation}
	where $\boldsymbol{t} = t_{1}\times \dots \times t_{p}$. By the arguments similar to those in \eqref{variance.1}--\eqref{variance.2}, 
	\begin{equation}
		\textnormal{Var}( m(\boldsymbol{X})  \vert \boldsymbol{X} \in \boldsymbol{t}) \le s^* \frac{1}{4}\max_{1\le j\le s^*}|t_{j}|^2 M_{2}^2,
	\end{equation}
	which in combination with \eqref{condi.expec.2} leads to
	$$\left(4M_{2}^2M_{1}^{-2} s^* \right)\sup_{j, c}(II)_{\boldsymbol{t}, \boldsymbol{t}(j,c)}\ge \left(4M_{2}^2M_{1}^{-2} s^* \right) (II)_{\boldsymbol{t}, \boldsymbol{t}_{1}^*} \ge \textnormal{Var}( m(\boldsymbol{X})  \vert \boldsymbol{X} \in \boldsymbol{t}).$$
	This concludes the second assertion and finishes the proof.

	\subsection{Verifying Condition \ref{P} for Example~\ref{new.example4}} \label{SecE.6}

	By the uniform distribution assumption on $\boldsymbol{X}$, for each cell $\boldsymbol{t} = t_{1} \times \dots \times t_{p}$,
	\begin{equation}
		\label{tool.5}
		\textnormal{Var}(m(\boldsymbol{X}) | \boldsymbol{X} \in \boldsymbol{t}) = \sum_{l=1}^{s^*} \textnormal{Var} (m_{l}(X_{l})|X_{l} \in t_{l}).
	\end{equation} 
	With \eqref{tool.5} and \eqref{tool.6}, to finish the proof of the first assertion, we show that the LHS of \eqref{tool.6} is proportional to the conditional bias decrease as follows. 
	
	A simple calculation shows that the conditional bias decrease $(II)_{\boldsymbol{t}, \boldsymbol{t}^{'}}$ given a split $(j, x)$ with $x\in t_{j}$ is bounded from below such that 
	\begin{equation}\label{tool.3}
		(II)_{\boldsymbol{t}, \boldsymbol{t}^{'}} \ge \mathbb{P} (\boldsymbol{X}\in\boldsymbol{t}^{'}| \boldsymbol{X} \in \boldsymbol{t})\mathbb{P} (\boldsymbol{X}\in\boldsymbol{t}^{''}| \boldsymbol{X} \in \boldsymbol{t}) \big(H_{j}(x)\big)^2,
	\end{equation}
	where
	$$H_{j}(x) \coloneqq \mathbb{E} (m(\boldsymbol{X}) | \boldsymbol{X} \in \boldsymbol{t}^{''}) - \mathbb{E} (m(\boldsymbol{X}) | \boldsymbol{X} \in \boldsymbol{t}^{'}),$$
	and that  $\boldsymbol{t}^{'} = t_{1}\times \dots \times t_{j}^{'}\times \dots \times t_{p}$ and $\boldsymbol{t}^{''} = t_{1}\times \dots \times t_{j}^{''}\times \dots \times t_{p}$ with $t_{j}^{'} = [\inf t_{j}, x)$ and $t_{j}^{''} = [x, \sup t_{j} ]\cap t_{j}$. For the reader's convenience, recall that
	\begin{equation*}\begin{split}
			& (II)_{\boldsymbol{t}, \boldsymbol{t}^{'}}= \mathbb{P} (\boldsymbol{X}\in\boldsymbol{t}^{'}   \vert  \boldsymbol{X} \in \boldsymbol{t} )
			\Big(\mathbb{E} (m(\boldsymbol{X})  \vert  \boldsymbol{X} \in \boldsymbol{t}^{'}) - \mathbb{E} (m(\boldsymbol{X})  \vert  \boldsymbol{X} \in \boldsymbol{t} ) \Big)^{2} \\
			& +  \mathbb{P} (\boldsymbol{X}\in\boldsymbol{t}^{''}   \vert  \boldsymbol{X} \in \boldsymbol{t} )
			\Big(\mathbb{E} (m(\boldsymbol{X})  \vert  \boldsymbol{X} \in \boldsymbol{t}^{''}) - \mathbb{E} (m(\boldsymbol{X})  \vert  \boldsymbol{X} \in \boldsymbol{t} ) \Big)^{2}.
	\end{split}\end{equation*}
	
	Recall that $m(\boldsymbol z) = \sum_{j=1}^p m_j(z_j)$.  Then
	$$
	\mathbb{E} (m(\boldsymbol{X}) | \boldsymbol{X} \in \boldsymbol{t}^{''}) = \frac{1}{\sup t_{j}-x} \int_{x}^{\sup t_{j}} m_j(z) dz + \sum_{l\not = j; l =1}^{s^*}\mathbb{E} (m_l(X_l) | \boldsymbol{X} \in \boldsymbol{t}^{''}),
	$$
	and
	$$\mathbb{E} (m(\boldsymbol{X}) | \boldsymbol{X} \in \boldsymbol{t}^{'}) = \frac{1}{x - \inf t_{j}} \int_{\inf t_{j}}^{x} m_j(z) dz + \sum_{l\not = j;l =1}^{s^*}\mathbb{E} (m_l(X_l) | \boldsymbol{X} \in \boldsymbol{t}^{''}).$$

	Thus, by the change of variables formula, 
	\begin{equation}
		\begin{split}			
			\label{tool.2}
			H_{j}(x)&= \frac{1}{\sup t_{j}-x} \int_{x}^{\sup t_{j}} m_j(z) dz - \frac{1}{x- \inf t_{j}} \int_{\inf t_{j}}^{x} m_j(z) dz \\
			& =  \frac{1}{x-\inf t_{j}} \int_{\inf t_{j}}^{x} \left(m_j\Big(\Big(\frac{\sup t_{j}-x}{x-\inf t_{j}}\Big)(z- \inf t_{j})+x\Big)-m_j(z)\right) dz .
		\end{split}
	\end{equation}
	
Suppose the  split is along $j\coloneqq \arg\max_{1\le l\le s^*} \textnormal{Var}( m(X_{l})|X_{l} \in t_{l})$. By \eqref{tool.2} and \eqref{tool.6}, 
	$$
	\sup_{x\in \Lambda(\inf t_{j}, \sup t_{j}) }(H_{j}(x))^2\geq c_{0}\max_{1\le l\le s^*}\textnormal{Var}(m_{l}(X_{l})|X_{l} \in t_{l}),
	$$
	which in combination with \eqref{tool.3},
	$$\sup_{l\in\{1, \dots, s^*\}, c\in t_{l}}(II)_{\boldsymbol{t}, \boldsymbol{t}(l,c)} \ge \sup_{c\in t_{j}}(II)_{\boldsymbol{t}, \boldsymbol{t}(j,c)} \ge  \lambda(1-\lambda)c_{0}\max_{1\le l\le s^*}\textnormal{Var}(m_{l}(X_{l})|X_{l} \in t_{l}).$$
	
		 By this and \eqref{tool.5}, it holds that
	$$m(\boldsymbol{X})\in \textnormal{SID}\left(\frac{s^*}{\lambda(1-\lambda)c_{0}} \right),$$
	which concludes the first assertion.
	
	For the second assertion, we note that if $m_{j}(z)$ is differentiable on $[0, 1]$, then
	\begin{equation*}
		\begin{split}
			&\textnormal{Var} (m_{j}(X_{j})|X_{j} \in t_{j}) \\
			& = \mathbb{E}\Big(\big(m_{j}(X_{j}) - \mathbb{E}(m_{j}(X_{j})|X_{j}\in t_{j}) \big)^2 \big| X_{j}\in t_{j} \Big)\\
			& \le \mathbb{E}\left(\left(m_{j}(X_{j}) - \left( \frac{\sup_{z \in t_{j}} m_{j}(z)  + \inf_{z \in t_{j}} m_{j}(z)}{2}\right) \right)^2\Big| X_{j}\in t_{j}\right)\\
			& \le \frac{1}{4}\mathbb{E}\left(\left(m_{j}(X_{j}) - \sup_{z\in t_{j}} m_{j}(z) - \inf_{z \in t_{j}} m_{j}(z) + m_{j}(X_{j}) \right)^2 \Big| X_{j} \in t_{j} \right)\\
			& \le \frac{1}{4}\mathbb{E}\left(\left(\sup_{z \in t_{j}} m_{j}(z) - m_{j}(X_{j})  + m_{j}(X_{j}) - \inf_{z\in t_{j}} m_{j}(z) \right)^2\Big| X_{j} \in t_{j}\right)\\
			&= \frac{1}{4}(\sup_{z\in t_{j}} m_{j}(z)  - \inf_{z\in t_{j}} m_{j}(z)  )^2\\
			&\le \frac{1}{4} (\sup_{t\in t_{j}}|m_{j}'(z)||t_{j}|)^2,
		\end{split}
	\end{equation*}
	where the last step is because of the mean  value theorem. This and \eqref{tool.7} lead to \eqref{tool.6} with some $c_{0}>0$, and hence we have finished the proof.


	\subsection{Proof for Remark~\ref{remark3}}\label{SecE.7}
	
	In this proof, we show that given the model 
	$$m(\boldsymbol{X}) = X_{1}X_{2} -0.5X_{1} - 0.5X_{2} + 0.25 = (X_{1} - 0.5)(X_{2} - 0.5)$$ 
	with uniformly distributed  $\boldsymbol{X}$, there exists a cell $\boldsymbol{t} = t_{1}\times \dots \times t_{p}$ such that a constant $\alpha_{1}\ge 1$ for SID does not exist. In fact, there are infinitely many such cells in this case. Consider a cell $\boldsymbol{t} = [0.5 - d_{1}, 0.5 + d_{1}]\times [0.5 - d_{2}, 0.5 + d_{2}]\times [a_{3}, b_{3}]\times \dots \times [a_{p}, b_{p }]$ for some positive $d_{1}\le0.5, d_{2}\le 0.5$ and $0\le a_{j}<b_{j}\le 1$. Let an arbitrary split be given and denote the two corresponding daughter cells by $\boldsymbol{t}^{'}$ and $\boldsymbol{t}^{''}$. Recall that 
	\begin{equation*}\begin{split}
			& (II)_{\boldsymbol{t}, \boldsymbol{t}^{'}}= \mathbb{P} (\boldsymbol{X}\in\boldsymbol{t}^{'}  \ \vert \ \boldsymbol{X} \in \boldsymbol{t} )
			\Big(\mathbb{E} (m(\boldsymbol{X}) \ \vert \ \boldsymbol{X} \in \boldsymbol{t}^{'})) - \mathbb{E} (m(\boldsymbol{X}) \ \vert \ \boldsymbol{X} \in \boldsymbol{t} )) \Big)^{2} \\
			& +  \mathbb{P} (\boldsymbol{X}\in\boldsymbol{t}^{''}  \ \vert \ \boldsymbol{X} \in \boldsymbol{t} )
			\Big(\mathbb{E} (m(\boldsymbol{X}) \ \vert \ \boldsymbol{X} \in \boldsymbol{t}^{''})) - \mathbb{E} (m(\boldsymbol{X}) \ \vert \ \boldsymbol{X} \in \boldsymbol{t} )) \Big)^{2}.
	\end{split}\end{equation*}
	
	Due to the assumption of independent features and the given regression model, 
	$$\mathbb{E} (m(\boldsymbol{X}) \ \vert \ \boldsymbol{X} \in \boldsymbol{t} ) = \mathbb{E} (m(\boldsymbol{X}) \ \vert \ \boldsymbol{X} \in \boldsymbol{t}^{'} )=\mathbb{E} (m(\boldsymbol{X}) \ \vert \ \boldsymbol{X} \in \boldsymbol{t}^{''} )= 0,$$
	which concludes that $(II)_{\boldsymbol{t}, \boldsymbol{t}^{'}} =0$. Since $\textnormal{Var}(m(\boldsymbol{X})|\boldsymbol{X}\in \boldsymbol{t})$ is obviously larger than zero for each cell, we conclude the desired result.

	\subsection{Verifying Condition \ref{P} for Example~\ref{new.example6}} \label{SecE.8}
	
	For the first case, 
	$\frac{\partial m(\boldsymbol{x})}{\partial x_{j}} = \frac{\beta_{j}\exp(\sum_{j\in S^*} \beta_{j}x_{j})}{(1+\exp(\sum_{j\in S^*} \beta_{j}x_{j}))^2}$ for each $j\in S^*$, and $\frac{\partial m(\boldsymbol{x})}{\partial x_{j}} = 0$ otherwise. It is seen that the vector-valued function $(\frac{\partial m(\boldsymbol{x})}{\partial x_{1}}, \dots, \frac{\partial m(\boldsymbol{x})}{\partial x_{p}})$ in this case is continuous in $[0, 1]^p$, and that $\frac{\partial m(\boldsymbol{x})}{\partial x_{j}}$ has the same sign as that of $\beta_{j}$. In light of $|\beta_{j}| \not= 0$ and $\boldsymbol{x} \in [0, 1]^p$, the parameters $M_{1}, M_{2}$ of Example~\ref{new.example3} are given as follows. If $\exp{(\sum_{j \in S^*} \beta_{j}x_{j})} \ge 1$, then 
	$$|\frac{\partial m(\boldsymbol{x})}{\partial x_{j}}| \ge \frac{(\min_{j\in S^*}|\beta_{j}|)\exp{(\sum_{j \in S^*} \beta_{j}x_{j})}}{(2\exp{(\sum_{j \in S^*} \beta_{j}x_{j})})^2} \ge \frac{\min_{j\in S^*}|\beta_{j}|}{ 4\exp{(\sum_{j\in S^*} |\beta_{j}|)}};$$
	otherwise, since $0<\exp{(\sum_{j \in S^*} \beta_{j}x_{j})} <1$, 
	$$|\frac{\partial m(\boldsymbol{x})}{\partial x_{j}}| \ge \frac{(\min_{j\in S^*}|\beta_{j}|)\exp{(\sum_{j \in S^*} \beta_{j}x_{j})}}{2^2} \ge \frac{\min_{j\in S^*}|\beta_{j}|}{4\exp{(\sum_{j\in S^*} |\beta_{j}|)}}.$$
	Therefore, we conclude that $M_{1} = \frac{\min_{j\in S^*}|\beta_{j}|}{4\exp{(\sum_{j\in S^*} |\beta_{j}|)}} \le \inf_{\boldsymbol{x} \in [0, 1]^p} |\frac{\partial m(\boldsymbol{x})}{\partial x_{j}}|$. Similarly, we have $M_{2} = \frac{1}{4}\max_{j\in S^*}|\beta_{j}|\ge \sup_{\boldsymbol{x} \in [0, 1]^p} |\frac{\partial m(\boldsymbol{x})}{\partial x_{j}}|$. By the results of Example~\ref{new.example3}, we concludes the proof for the first case.
	
	As for the second case, the partial derivatives of a polynomial are obviously continuous. Since all coefficients have the same sign and that $\boldsymbol{x}\in [0, 1]^p$, we let $M_{1} = \min_{j\in S^*}|\beta_{j}|$ and $M_{2} = (\max_{j,k}r_{jk}) \sum_{k=1}^{k_{1}}|\beta_{kk}| + \max_{j\in S^*}|\beta_{j}|$, which along with the results of Example~\ref{new.example3} concludes the proof of the second case, and hence the proof of Example~\ref{new.example6}.

	\subsection{Verifying Condition \ref{P} for Example~\ref{new.example7}} \label{SecE.9}

	Our goal is to show that the maximum conditional bias decrease $\sup_{j\in \{1, \dots, p\}, c\in t_{j}}(II)_{\boldsymbol{t}, \boldsymbol{t}(j,c)}$ is lower bounded in terms of $\textnormal{Var} (m(\boldsymbol{X})| \boldsymbol{X} \in \boldsymbol{t})$ for every cell $\boldsymbol{t} = t_{1}\times \dots \times t_{p}$. Since the case for $K=1$ is trivial, we consider the case with $K>1$ in the following.

	Let us start with an upper bound for $\textnormal{Var} (m(\boldsymbol{X})| \boldsymbol{X} \in \boldsymbol{t})$. By the assumptions of an additive model and a uniform distribution of $\boldsymbol{X}$, for every $\boldsymbol{t}$,
	\begin{equation}
		\label{piecewise.1}
		\textnormal{Var}(m(\boldsymbol{X})|\boldsymbol{X}\in\boldsymbol{t}) = \sum_{l=1}^{s^*} \textnormal{Var}(m_{l}(X_{l})|X_{l}\in t_{l}),
	\end{equation}
	and by the definition of $R$, for every $\boldsymbol{t} = t_{1} \times\dots \times t_{p}$ and $l\le s^*$,
	\begin{equation}
		\label{piecewise.12}
		\textnormal{Var}(m_{l}(X_{l})| X_{l} \in t_{l})\le (|t_{l}|R)^2.
	\end{equation}
	Define
	$$j \coloneqq \arg\max_{1\le l\le s^*} |t_{l}|,$$ 
	and hence by \eqref{piecewise.1}--\eqref{piecewise.12},
	\begin{equation}
		\label{piecewise.11}
		\textnormal{Var}(m(\boldsymbol{X})|\boldsymbol{X} \in\boldsymbol{t})\le s^*(|t_{j}|R)^2.
	\end{equation}

	Now, we proceed to deal with the lower bound of the conditional bias decrease. Let us introduce some notation for referring to the linear functions and splits on $t_{j}$. The rightmost linear function on $t_{j}$ is denoted by $h_{1}(x)$, with $l>0$ being the length of its domain on $t_{j}$; and the linear functin to the left of $h_{1}(x)$ is denoted by $h_{2}(x)$, with $L\ge0$ being the length of its domain on $t_{j}$. In addition, we consider four split points A, B, C, and D, which are respectively on the middle of the  domain of $h_{1}(x)$, the left-end of $h_{1}(x)$, the middle of the domain of $h_{2}(x)$, and the left-end of $h_{2}(x)$. A graphical illustration is in Figure~\ref{fig:RF2}. Moreover, we refer to the corresponding daughter cells of $t_{j}$ on the right hand side of the splits as $t_{jA}^{'}$, $t_{jB}^{'}$ $t_{jC}^{'}$, and $t_{jD}^{'}$, respectively. Let $\boldsymbol{t}_{s}^{'} = t_{1}\times \dots t_{j-1}\times t_{js}^{'}\times t_{j+1}\times \dots\times t_{p}$ for $s\in\{A, B, C, D\}$.
	\begin{figure}[t]		
		\centering
		\includegraphics[width=8cm]{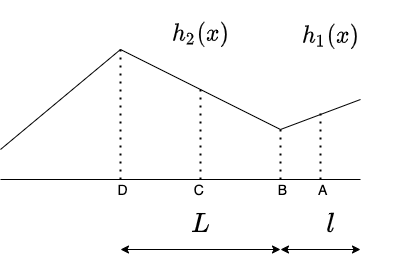}
		\caption{The part of $m_{j}(x)$ on $t_{j}$. In this example, there are three piecewise linear functions.}
		\label{fig:RF2}
	\end{figure}
	
	By the definition of $(II)_{\boldsymbol{t}, \boldsymbol{t}_{s}^{'}}$ and the assumption of a uniform distribution of $\boldsymbol{X}$,
	\begin{equation}
		\begin{split}
			\label{piecewise.3}
			& \max\{(II)_{\boldsymbol{t}, \boldsymbol{t}_{A}^{'}}, (II)_{\boldsymbol{t}, \boldsymbol{t}_{B}^{'}} \} \\
			& \ge \max\Big\{\frac{l}{2|t_{j}|} (\mathbb{E}(m_{j}(X_{j}) |X_{j}\in t_{jA}^{'}) - \mathbb{E}(m_{j}(X_{j}) |X_{j}\in t_{j}) )^2,\\ &\qquad\frac{l}{|t_{j}|} (\mathbb{E}(m_{j}(X_{j}) |X_{j}\in t_{jB}^{'}) - \mathbb{E}(m_{j}(X_{j}) |X_{j}\in t_{j}) )^2 \Big\}.
		\end{split}
	\end{equation}
	
	By the model assumptions and simple calculations, it holds that $|\mathbb{E}(m_{j}(X_{j}) |X_{j}\in t_{jA}^{'}) - \mathbb{E}(m_{j}(X_{j}) |X_{j}\in t_{jB}^{'})|\ge \frac{lr}{4}$. By this, if $|\mathbb{E}(m_{j}(X_{j}) |X_{j}\in t_{jA}^{'}) - \mathbb{E}(m_{j}(X_{j}) |X_{j}\in t_{j})  |\le \frac{lr}{8}$, then $|\mathbb{E}(m_{j}(X_{j}) |X_{j}\in t_{jB}^{'}) - \mathbb{E}(m_{j}(X_{j}) |X_{j}\in t_{j})  |\ge \frac{lr}{8}$; the other way round is also true. Hence
	\begin{equation}\label{piecewise.4}
		\textnormal{RHS of } \eqref{piecewise.3} \ge \frac{l}{2|t_{j}|}\times (\frac{lr}{8})^2 = \frac{l^3r^2}{128|t_{j}|}.
	\end{equation}
	Notice that the above arguments do not depend on the value of $L$.
	
	Next, by the definition of $(II)_{\boldsymbol{t}, \boldsymbol{t}_{s}^{'}}$ and the assumption of a uniform distribution of $\boldsymbol{X}$,
	\begin{equation}
		\begin{split}
			\label{piecewise.5}
			& \max\{(II)_{\boldsymbol{t}, \boldsymbol{t}_{C}^{'}}, (II)_{\boldsymbol{t}, \boldsymbol{t}_{D}^{'}} \} \\
			& \ge \max\Big\{\frac{1}{|t_{j}|}(\frac{L}{2} + l) (\mathbb{E}(m_{j}(X_{j}) |X_{j}\in t_{jC}^{'}) - \mathbb{E}(m_{j}(X_{j}) |X_{j}\in t_{j}) )^2,\\ &\qquad\frac{1}{|t_{j}|}(L+l) (\mathbb{E}(m_{j}(X_{j}) |X_{j}\in t_{jD}^{'}) - \mathbb{E}(m_{j}(X_{j}) |X_{j}\in t_{j}) )^2 \Big\}.
		\end{split}
	\end{equation}
	
	To lower bound the RHS of \eqref{piecewise.5}, we write
	$$\mathbb{E}(m_{j}(X_{j}) |X_{j}\in t_{jD}^{'}) = \frac{L}{2} (\frac{1}{L+l}) \mathbb{E}(m_{j}(X_{j}) |X_{j}\in [D, C)) + (\frac{L}{2}+l) (\frac{1}{L+l}) \mathbb{E}(m_{j}(X_{j}) |X_{j}\in t_{jC}^{'}),$$
	which follows from the uniform distribution assumption on the feature vector. Hence
	\begin{equation}
		\begin{split}\label{piecewise.6}
			&|\mathbb{E}(m_{j}(X_{j}) |X_{j}\in t_{jD}^{'}) - \mathbb{E}(m_{j}(X_{j}) |X_{j}\in t_{jC}^{'})|\\
			&= \frac{L}{2} (\frac{1}{L+l})\Big| \mathbb{E}(m_{j}(X_{j}) |X_{j}\in [D, C)) - \mathbb{E}(m_{j}(X_{j}) |X_{j}\in t_{jC}^{'})\Big|.
		\end{split}
	\end{equation}
	
	Since we have assumed that $m_{j}(x)$ is continuous with slope upper and lower bounds, if the value of $l$ is sufficiently small, then $|\mathbb{E}(m_{j}(X_{j}) |X_{j}\in [D, C)) - \mathbb{E}(m_{j}(X_{j}) |X_{j}\in t_{jC}^{'})|$ is sufficiently large. For example, if $lR\le \frac{Lr}{2}$, then 
	$\mathbb{E}(m_{j}(X_{j}) |X_{j}\in t_{jC}^{'}) \le m_{j}(C)$ and hence $\mathbb{E}(m_{j}(X_{j}) |X_{j}\in [D, C)) \ge \mathbb{E}(m_{j}(X_{j}) |X_{j}\in t_{jC}^{'}) + \frac{Lr}{4}$ in Figure~\ref{fig:RF2}. For general cases, if $lR\le \frac{Lr}{2}$, it holds that
	$$\textnormal{RHS of } \eqref{piecewise.6} \ge \frac{L}{2} (\frac{1}{L+l})\times \frac{Lr}{4} = \frac{L^2r}{8(L+l)}.$$
	By this and arguments similar to those for \eqref{piecewise.4},
	\begin{equation}
		\begin{split}\label{piecewise.7}
			\textnormal{RHS of } \eqref{piecewise.5} \ge \frac{1}{|t_{j}|}(\frac{L}{2} + l)(\frac{L^2r}{16(L+l)})^2\ge \frac{L^3r^2}{512|t_{j}|}.
		\end{split}
	\end{equation}
	
	To have our conclusion, we need an observation as follows. Due to the definition of $b^*$ and $K>1$, it holds that $b^*\le \frac{1}{2}.$ Also, if $\frac{L}{|t_{j}|}< b^*$, then $\frac{l}{|t_{j}|} \ge b^*$; to see this, notice that $\frac{L}{|t_{j}|}< b^*$ only if $h_{2}(\cdot)$ is the leftmost linear function on $t_{j}$, and recall that $h_{1}(\cdot)$ is the rightmost linear function on $t_{j}$. With this observation, we can establish the lower bound of the maximum conditional bias decrease by considering three cases in \eqref{piecewise.8}--\eqref{piecewise.10} below.
	
	If $\frac{L}{|t_{j}|}<b^*$, then by the observation and \eqref{piecewise.4},  \begin{equation}\label{piecewise.8}
		\sup_{j\in \{1, \dots, p\}, c\in t_{j}}(II)_{\boldsymbol{t}, \boldsymbol{t}(j,c)}\ge  \frac{(b^*)^3(r|t_{j}|)^2}{128}.
	\end{equation}
	If $\frac{L}{|t_{j}|}\ge b^*$ and $l> \frac{Lr}{2R}$, then by  \eqref{piecewise.4},
	\begin{equation}\label{piecewise.9}
		\sup_{j\in \{1, \dots, p\}, c\in t_{j}}(II)_{\boldsymbol{t}, \boldsymbol{t}(j,c)}\ge  \frac{(l)^3 r^2}{128|t_{j}|}\ge \frac{ r^5(b^*)^3(|t_{j}|)^2}{1024 R^3}.
	\end{equation}
	If $\frac{L}{|t_{j}|}\ge b^*$ and $l\le  \frac{Lr}{2R}$, then by \eqref{piecewise.7}, 
	\begin{equation}\label{piecewise.10}
		\sup_{j\in \{1, \dots, p\}, c\in t_{j}}(II)_{\boldsymbol{t}, \boldsymbol{t}(j,c)}\ge  \frac{(b^*)^3(r|t_{j}|)^2 }{512}.
	\end{equation}

	By \eqref{piecewise.11} and \eqref{piecewise.8}--\eqref{piecewise.10}, we conclude that 
	$$s^*(|t_{j}|R)^2\times \frac{1024 R^3}{ r^5(b^*)^3(|t_{j}|)^2}\sup_{j\in \{1, \dots, p\}, c\in t_{j}}(II)_{\boldsymbol{t}, \boldsymbol{t}(j,c)} \ge \textnormal{Var}(m(\boldsymbol{X})|\boldsymbol{X} \in\boldsymbol{t}),$$
	which leads to the desired result. 
	
	\subsection{Verifying Condition \ref{P} for Example~\ref{new.example8}} \label{SecE.10}

	The proof idea is to find an upper bound of $\textnormal{Var}(m(\boldsymbol{X})|\boldsymbol{X} \in\boldsymbol{t}) $ and a lower bound of maximum conditional bias decrease for each cell $\boldsymbol{t} = t_{1} \times\dots \times t_{p}$ such that $\sup_{j\in \{1, \dots, p\}, c\in t_{j}}(II)_{\boldsymbol{t}, \boldsymbol{t}(j,c)}$ is lower bounded in terms of $\textnormal{Var}(m(\boldsymbol{X})|\boldsymbol{X} \in\boldsymbol{t})$. We begin with noticing that if there are no jump points on $\boldsymbol{t}$, then $\textnormal{Var}(m(\boldsymbol{X})|\boldsymbol{X} \in\boldsymbol{t}) = 0$ and hence the proof is trivial. We therefore consider the case with at least one jump point on $\boldsymbol{t}$.

	In the following, we deal with the lower bound of the conditional bias decrease. Due to the regression function form, we can establish a lower bound for the conditional bias decrease on a cell $\boldsymbol{t}$ given a split at any jump point on $\boldsymbol{t}$ as follows. Let a cell $\boldsymbol{t}$ be given with a jump point on it; let $\boldsymbol{t}^{'},\boldsymbol{t}^{''}$ denote two daughter cells after the split at one of the jump points on $\boldsymbol{t}$. Then it holds that
	$$|\mathbb{E} (m(\boldsymbol{X}) | \boldsymbol{X} \in \boldsymbol{t}^{'}) - \mathbb{E} (m(\boldsymbol{X}) | \boldsymbol{X} \in \boldsymbol{t}^{''})|\ge \iota,$$
	which follows from the model assumptions, and that 
	\begin{equation}
		\begin{split}\label{linear.combi.2}
			(II)_{\boldsymbol{t}, \boldsymbol{t}^{'}} & = \zeta (\mathbb{E} (m(\boldsymbol{X}) | \boldsymbol{X} \in \boldsymbol{t}^{'}) - \mathbb{E} (m(\boldsymbol{X}) | \boldsymbol{X} \in \boldsymbol{t}))^2 \\
			&\qquad+ (1-\zeta) (\mathbb{E} (m(\boldsymbol{X}) | \boldsymbol{X} \in \boldsymbol{t}^{''}) - \mathbb{E} (m(\boldsymbol{X}) | \boldsymbol{X} \in \boldsymbol{t}))^2\\
			& \ge \inf_{x\in\mathbb{R}} \Big(\zeta (\mathbb{E} (m(\boldsymbol{X}) | \boldsymbol{X} \in \boldsymbol{t}^{'}) - x)^2 + (1-\zeta) (\mathbb{E} (m(\boldsymbol{X}) | \boldsymbol{X} \in \boldsymbol{t}^{''}) - x)^2\Big)\\
			&\ge \zeta(1 - \zeta) (\mathbb{E} (m(\boldsymbol{X}) | \boldsymbol{X} \in \boldsymbol{t}^{'}) - \mathbb{E} (m(\boldsymbol{X}) | \boldsymbol{X} \in \boldsymbol{t}^{''}))^2\\
			& = \zeta(1 - \zeta)  \iota^2,	\end{split}
	\end{equation}
	where $\zeta= \mathbb{P}(\boldsymbol{X} \in \boldsymbol{t}^{'}|\boldsymbol{X} \in \boldsymbol{t})$.

	On the other hand, to establish the upper bound for $\textnormal{Var}(m(\boldsymbol{X})|\boldsymbol{X} \in \boldsymbol{t})$, we separate the proof into two cases: (i) there are more than two jump points on some coordinates of $\boldsymbol{t}$, and (ii) there are at most two jump points on every coordinate of $\boldsymbol{t}$.

	We deal with the first case first. Write the regression function conditional on $\boldsymbol{t}$ as $\sum_{k=1}^{k_{0}}\beta^{(k)}\boldsymbol{1}_{\boldsymbol{X} \in \mathcal{C}^{(k)}}$ for some $k_0$, where $\mathcal{C}^{(1)}, \dots, \mathcal{C}^{(k_{0})}$ are all the subcells on $\boldsymbol{t}$, with their respective coefficients denoted by $\beta^{(1)}, \dots, \beta^{(k_{0})}$. Due to the coefficient assumptions,
	\begin{equation}
		\label{linear.combi.1}
		\max_{1\le l\le k_{0}}\beta^{(l)} - \min_{1\le l\le k_{0}}\beta^{(l)} \le 2M_{0}.
	\end{equation}
	Hence,
	\begin{equation}
		\begin{split}\label{linear.combi.3}
			\textnormal{Var}(m(\boldsymbol{X})|\boldsymbol{X} \in \boldsymbol{t}) &= \mathbb{E}\big((\sum_{k=1}^{k_{0}}\beta^{(k)}\boldsymbol{1}_{\boldsymbol{X} \in \mathcal{C}^{(k)}} - \mathbb{E} (m(\boldsymbol{X}) | \boldsymbol{X} \in \boldsymbol{t}))^2|\boldsymbol{X} \in \boldsymbol{t} \big)\\
			&\le \mathbb{E}\big((\sum_{k=1}^{k_{0}}\beta^{(k)}\boldsymbol{1}_{\boldsymbol{X} \in \mathcal{C}^{(k)}} - \beta^{(1)})^2|\boldsymbol{X} \in \boldsymbol{t} \big)\\
			& = \mathbb{E}\big((\sum_{k=1}^{k_{0}}\beta^{(k)}\boldsymbol{1}_{\boldsymbol{X} \in \mathcal{C}^{(k)}} - \sum_{k=1}^{k_{0}}\beta^{(1)}\boldsymbol{1}_{\boldsymbol{X} \in \mathcal{C}^{(k)}} )^2|\boldsymbol{X} \in \boldsymbol{t} \big)\\
			& = \mathbb{E}\big((\sum_{k=1}^{k_{0}}(\beta^{(k)} - \beta^{(1)})\boldsymbol{1}_{\boldsymbol{X} \in \mathcal{C}^{(k)}}  )^2|\boldsymbol{X} \in \boldsymbol{t} \big)\\
			& \le (2M_{0})^2,
		\end{split}
	\end{equation}
	where the last inequality is due to \eqref{linear.combi.1}.
	
	Suppose the $j$th coordinate has at least three jump points. Let us consider a split at any jump point in between the first and the last jump points on the $j$th coordinate; let $\boldsymbol{t}^{'}$ and $\boldsymbol{t}^{''}$ denote the resulting daughter cells. In light of the uniform distribution assumption on $\boldsymbol{X}$, $$\min\{\mathbb{P}(\boldsymbol{X}\in\boldsymbol{t}^{'}|\boldsymbol{X}\in\boldsymbol{t}), \mathbb{P}(\boldsymbol{X}\in\boldsymbol{t}^{''}|\boldsymbol{X}\in\boldsymbol{t})\}\ge \min_{j\le s^*, 1\le i\le k_{j}}c_{i}^{(j)} - c_{i-1}^{(j)}\eqqcolon c^*,$$
	which along with \eqref{linear.combi.2} and \eqref{linear.combi.3} shows that 
	\begin{equation}
		\begin{split}		
			\label{linear.combi.4}
			\left(\frac{2M_{0}}{\iota} \right)^2 \frac{1}{c^*(1-c^*)}\sup_{j\in \{1, \dots, p\}, c\in t_{j}}(II)_{\boldsymbol{t}, \boldsymbol{t}(j,c)} &\ge \left(\frac{2M_{0}}{\iota} \right)^2 \frac{1}{c^*(1-c^*)}(II)_{\boldsymbol{t}, \boldsymbol{t}^{'}} \\
			& \ge \textnormal{Var}(m(\boldsymbol{X})|\boldsymbol{X} \in \boldsymbol{t}),
		\end{split}
	\end{equation}
	which concludes the proof of case (i).

	Let us proceed to analyze case (ii). We again denote all the subcells and their respective coefficients by $\mathcal{C}^{(1)}, \dots, \mathcal{C}^{(k_{0})}$ and $\beta^{(1)}, \dots, \beta^{(k_{0})}$. In addition, define $\rho_{j}\ge 0$ such that
	\begin{enumerate}
	    \item $\rho_{j}= |t_{j}|^{-1}\max\{ a - \inf t_{j}, \sup t_{j} - b\}$ if the $j$th coordinate of $\boldsymbol{t}$ has two jump points at $a< b$,
	    \item $\rho_{j}= |t_{j}|^{-1}\min\{ a - \inf t_{j}, \sup t_{j} - a\}$ if the $j$th coordinate of $\boldsymbol{t}$ has only one jump point $a$,
	    \item $\rho_{j} = 0$ if the $j$th coordinate has no jumps.
	\end{enumerate}
	Next, we further separate the case (ii) into three subcases (ii.a)--(ii.c) as follows.
	
	The case (ii.a): there are two jump points on the $j$th coordinate and $\rho_{j}\ge \frac{1}{4}$. Then, let the split be at $(j, a)$ if $a - \inf t_{j} \ge \sup t_{j} - b$, and otherwise at $(j, b)$; due to the assumption of a uniform distribution of the feature vector,
	$$\min\{\mathbb{P} (\boldsymbol{X} \in \boldsymbol{t}^{'}|\boldsymbol{X} \in \boldsymbol{t}), \mathbb{P} (\boldsymbol{X} \in \boldsymbol{t}^{''}|\boldsymbol{X} \in \boldsymbol{t})\}\ge \min\{\frac{1}{4}, c^*\} = c^{\dagger},$$
	where $\boldsymbol{t}^{'}$ and $\boldsymbol{t}^{''}$ are the corresponding daughter cells. By this result, \eqref{linear.combi.2},  \eqref{linear.combi.3}, and the fact that $\zeta(1-\zeta)$ is increasing when $\zeta \le \frac{1}{2}$,
	\begin{equation}
		\begin{split}			
			\label{linear.combi.5}
			\left(\frac{2M_{0}}{\iota} \right)^2 \frac{1}{c^\dagger(1-c^\dagger)}\sup_{j\in \{1, \dots, p\}, c\in t_{j}}(II)_{\boldsymbol{t}, \boldsymbol{t}(j,c)} &\ge \left(\frac{2M_{0}}{\iota} \right)^2 \frac{1}{c^\dagger(1-c^\dagger)}(II)_{\boldsymbol{t}, \boldsymbol{t}^{'}} \\
			& \ge \textnormal{Var}(m(\boldsymbol{X})|\boldsymbol{X} \in \boldsymbol{t}).
		\end{split}
	\end{equation}
	
	The case (ii.b): there is only one jump point on the $j$th coordinate and $\rho_{j}\ge \frac{1}{4}$. Let us split at the only jump on the $j$th coordinate in this case, and denote the daughter cells of $\boldsymbol{t}$ by $\boldsymbol{t}^{'}$ and $\boldsymbol{t}^{''}$. By the assumption $\rho_{j}\ge \frac{1}{4}$ and the assumption of a uniform distribution of $\boldsymbol{X}$, 
	$$\min\{\mathbb{P} (\boldsymbol{X} \in \boldsymbol{t}^{'}|\boldsymbol{X} \in \boldsymbol{t}), \mathbb{P} (\boldsymbol{X} \in \boldsymbol{t}^{''}|\boldsymbol{X} \in \boldsymbol{t})\}\ge \frac{1}{4} \ge c^{\dagger}.$$
	We conclude similarly to \eqref{linear.combi.5} that 
	\begin{equation}
		\begin{split}			
			\label{linear.combi.6}
			\left(\frac{2M_{0}}{\iota} \right)^2 \frac{1}{c^\dagger(1-c^\dagger)}(II)_{\boldsymbol{t}, \boldsymbol{t}^{'}}  \ge \textnormal{Var}(m(\boldsymbol{X})|\boldsymbol{X} \in \boldsymbol{t}).
		\end{split}
	\end{equation}		
	Note that for a given $\boldsymbol{t}$, it is possible that there are two coordinates such that case (ii.a) holds for one coordinate and case (ii.b) hold for the other coordinate.
	
	Lastly, the case (ii.c) considers the remaining scenario, where there are only coordinates with at most two jump points and that $\max_{1\le l\le s^*}\rho_{l} < \frac{1}{4}$. Let $j =\arg\max_{1\le l\le s^*}\rho_{l}$. If there are two jump points on the $j$th coordinate, we consider a split the same as in the case (ii.a); otherwise, we consider the split the same as in the case (ii.b). Notice that $\max_{1\le l\le s^*}\rho_{l} > 0 $ since we have assumed a nontrivial case where there is at least one jump point on $\boldsymbol{t}$.
	
	Among $\mathcal{C}^{(1)}, \dots, \mathcal{C}^{(k_{0})}$, let us fix a subcell  $\mathcal{C}^{(k^*)}$  such that 1) if the $l$th coordinate  of $\boldsymbol{t}$ has two jump points at $0<a< b<1$, the $l$th coordinate of $\mathcal{C}^{(k^*)}$ is $[a, b)$, 2) if the $l$th coordinate  of $\boldsymbol{t}$ has only one jump point $a$, the $l$th coordinate of $\mathcal{C}^{(k^*)}$ is the longer one among $[\inf t_{l}, a)$ and $[a,  \sup t_{l}]\cap t_{l}$, and 3) if the $l$th coordinate has no jump points, the $l$th coordinate  of $\mathcal{C}^{(k^*)}$ is $t_{l}$. By the definition of $\mathcal{C}^{(k^*)}$ and $\rho_{j}$'s, it holds that
	$$\mathbb{P} (\boldsymbol{X}\not\in \mathcal{C}^{(k^*)}|\boldsymbol{X}\in\boldsymbol{t})\le 1 - (1 - 2\rho_{j})^{s^*}.$$
	We can use this result, \eqref{linear.combi.1}, and the definition of $\mathcal{C}^{(k^*)},\beta^{(k^*)}$ to refine the upper bound of $\textnormal{Var}(m(\boldsymbol{X})|\boldsymbol{X} \in \boldsymbol{t})$ derived in \eqref{linear.combi.3} as follows.
	\begin{equation}
		\begin{split}\label{linear.combi.7}
			\textnormal{Var}(m(\boldsymbol{X})|\boldsymbol{X} \in \boldsymbol{t}) 
			& \le \mathbb{E}\big((\sum_{k=1}^{k_{0}}(\beta^{(k)} - \beta^{(k^*)})\boldsymbol{1}_{\boldsymbol{X} \in \mathcal{C}^{(k)}}  )^2|\boldsymbol{X} \in \boldsymbol{t} \big)\\
			&=\mathbb{E}\big( \sum_{k\not= k^* }(\beta^{(k)} - \beta^{(k^*)})^2\boldsymbol{1}_{\boldsymbol{X} \in \mathcal{C}^{(k)}}  |\boldsymbol{X} \in \boldsymbol{t} \big)\\
			& \le (2M_{0})^2\times \big(1 - (1 - 2\rho_{j})^{s^*}\big),
		\end{split}
	\end{equation}
	where in the first equality, we use the fact that $\mathcal{C}^{(k)}\cap \mathcal{C}^{(l)} = \emptyset$ if $k\not=l$.

	With \eqref{linear.combi.2}, \eqref{linear.combi.7}, the choice of our split, the definition of $\rho_{j}$, and that $\rho_j < \frac{1}{4}$ in this scenario,
	\begin{equation*}
		\begin{split}			
			\textnormal{Var}(m(\boldsymbol{X})|\boldsymbol{X} \in \boldsymbol{t}) & \le \frac{\big(1 - (1 - 2\rho_{j})^{s^*}\big)}{\rho_{j}(1-\rho_{j})} \left(\frac{2M_{0}}{\iota}\right)^2(II)_{\boldsymbol{t}, \boldsymbol{t}^{'}} \\
			& \le \frac{4}{3}\left(\frac{1 - (1 - 2\rho_{j})^{s^*}}{\rho_{j}}\right) \left(\frac{2M_{0}}{\iota}\right)^2 (II)_{\boldsymbol{t}, \boldsymbol{t}^{'}}.
		\end{split}
	\end{equation*}
	
	In the following, we simplify the term $\frac{1 - (1 - 2\rho_{j})^{s^*}}{\rho_{j}}$. We need Bernoulli's inequality: for each $t\ge 1$ and $0\le x \le 1$,
	$$(1-x)^t \ge 1-tx.$$
	By this and that $0<\rho_{j} < \frac{1}{4}$,
	$$\frac{1 - (1 - 2\rho_{j})^{s^*}}{\rho_{j}}\le \frac{1 - (1 - s^* 2 \rho_{j})}{\rho_{j}} =  2s^*.$$
	
	Combining all these, we have 
	\begin{equation}
		\begin{split}\label{linear.combi.8}
			\textnormal{Var}(m(\boldsymbol{X})|\boldsymbol{X} \in \boldsymbol{t}) & \le \frac{8}{3} s^*\left(\frac{2M_{0}}{\iota}\right)^2(II)_{\boldsymbol{t}, \boldsymbol{t}^{'}}
		\end{split}
	\end{equation}
	in this scenario. 
	
	With \eqref{linear.combi.4}--\eqref{linear.combi.6} and \eqref{linear.combi.8} and the fact that $\frac{1}{c^{\dagger} (1-c^{\dagger})}\ge \frac{16}{3}$, we conclude that for every $\boldsymbol{t}$,
	$$\frac{s^*}{c^\dagger(1-c^\dagger)}  \left(\frac{2M_{0}}{\iota}\right)^2 \sup_{j\in\{1, \dots p\}, c\in t_{j}}(II)_{\boldsymbol{t}, \boldsymbol{t}(j, c)} \ge \textnormal{Var}(m(\boldsymbol{X})|\boldsymbol{X} \in \boldsymbol{t}).$$
	This implies the desired result
	$$m(\boldsymbol{X})\in\textnormal{SID}\Big(  \frac{s^*}{c^\dagger(1-c^\dagger)}  \left(\frac{2M_{0}}{\iota}\right)^2\Big).$$

\end{document}